\documentclass{article}

\usepackage{nameref}

\usepackage{stackengine}

\usepackage{amsmath,amscd,color}
\usepackage{amssymb}
\usepackage{mathrsfs}
\numberwithin{equation}{section}

\usepackage{relsize}

\usepackage{graphicx}
\newcommand*{\LargerCdot}{\raisebox{-0.25ex}{\scalebox{2}{$\cdot$}}}

\usepackage[top=1.4in, bottom=1.4in, left=1.4in, right=1.4in]{geometry}

\usepackage{amsthm}
\theoremstyle{plain}
\newtheorem{mydef}{Definition}[section]
\newtheorem*{mydef*}{Definition}

\newtheorem{mylemma}[mydef]{Lemma}
\newtheorem{mytheorem}[mydef]{Theorem}
\newtheorem{mycor}[mydef]{Corollary}
\newtheorem{myprop}[mydef]{Proposition}
\newtheorem{myassum}[mydef]{Assumption}

\newtheorem{myproblem}[mydef]{Problem}
\theoremstyle{definition}
\newtheorem{myexample}[mydef]{Example}
\newtheorem*{myexample*}{Example}
\newtheorem{myremark}[mydef]{Remark}
\newtheorem*{myremark*}{Remark}
\theoremstyle{remark}

\usepackage{enumitem}

\usepackage{comment}

\usepackage{cancel}

\newcommand{\Hom}{\mathrm{Hom}}
\newcommand{\Ker}{\mathrm{Ker}\;}
\newcommand{\Id}{\mathrm{id}}

\newcommand{\Dom}{\mathrm{Dom}}
\newcommand{\Image}{\mathrm{Im}}

\newcommand{\acts}{\curvearrowright}

\newcommand{\CAlg}{\mathcal{C}}
\newcommand{\Tab}{\mathlarger{\mathlarger{\sigma}}}

\newcommand{\LDiff}{\mathrm{Diff}_{\mathrm{loc}}}
\newcommand{\LBis}{\mathrm{Bis}_{\mathrm{loc}}}

\usepackage{titlesec}

\titleformat{\section}{\normalfont\fontsize{13}{15}\bfseries}{\thesection}{1em}{}
\titleformat{\subsection}{\normalfont\fontsize{11}{15}\bfseries}{\thesubsection}{1em}{}
  
\usepackage{mathtools}
\newcommand{\explain}[2]{\underset{\mathclap{\overset{\uparrow}{#2}}}{#1}}

\usepackage{tikz}
\usepackage{tikz-cd}
\usetikzlibrary{matrix,calc,arrows,decorations.pathmorphing}

\title{Lie Pseudogroups \`a la Cartan 
}
\author{Marius Crainic\footnote{Utrecht University; \texttt{m.crainic{@}uu.nl}} \hspace{1cm} Ori Yudilevich\footnote{KU Leuven; \texttt{ori.yudilevich{@}kuleuven.be} }}
\date{February 1, 2019}

\begin{document}

\maketitle

\begin{abstract} 
	We present a modern formulation of \'Elie Cartan's structure theory for Lie pseudogroups and prove a reduction theorem that clarifies the role of Cartan's systatic system. The paper is divided into three parts. In part one, using notions coming from the theory of Lie groupoids and algebroids, we introduce the framework of Cartan algebroids and realizations, structures that encode Cartan's structure equations and notion of a pseudogroup in normal form. In part two, we present a novel proof of Cartan's Second Fundamental Theorem, which states that any Lie pseudogroup is equivalent to a pseudogroup in normal form. In part three, we prove a new reduction theorem that states that, under suitable regularity conditions, a pseudogroup in normal form canonically reduces to a generalized pseudogroup of local solutions of a Lie-Pfaffian groupoid. 
\end{abstract}

\newpage

\setcounter{tocdepth}{2}
\tableofcontents

\newpage 

\section*{Introduction}
\addcontentsline{toc}{section}{Introduction}

In two classical papers \cite{Cartan1904,Cartan1905}, dating back to 1904-05, \'Elie Cartan introduced a structure theory for Lie pseudogroups, building on the then recent work of Sophus Lie. As part of this work, Cartan introduced several tools that have, with the years, become fundamental in differential geometry, including the theory of exterior differential systems, G-structures, the equivalence problem, and, to some extent, the very notion of differential forms. These tools, as it turned out, became much more influential than the original problem itself -- the study of Lie pseudogroups, the main reason being the gap between the notions and ideas that Cartan introduced and the mathematical language that he had at his disposal. Chern and Chevalley, in an obituary to Cartan (\cite{Chern1952}, 1952), wrote: ``We touch here a branch of mathematics which is very rich in results but which very badly needs clarification of its foundations.'' Singer and Sternberg, in their work titled \textit{The Infinite Groups of Lie and Cartan Part I (The Transitive Groups)} (\cite{Singer1965}, 1965), wrote: ``We must confess that we find most of these papers extremely rough going and we certainly cannot follow all arguments in detail. The best procedure is to guess at the theorems, then prove them, then go back to Cartan.'' 

In this paper, using modern differential geometric language, and in particular notions coming from the theory of Lie groupoids and algebroids and from the theory of geometric PDEs, we revisit Cartan's work and present a global, coordinate-free formulation of his theory of Lie pseudogroups, with the aim of providing further insight on a work that has proven to be so rich in content. The main contributions of this paper are:
\renewcommand\labelitemi{$\vcenter{\hbox{\tiny$\bullet$}}$}
\begin{itemize}
	\item The new framework of Cartan algebroids and realizations, the global objects behind Cartan's (local) structure equations of a Lie pseudogroup. 
	\item The equivalent but more intuitive point of view of Cartan pairs and realizations. 
	\item A coordinate-free and conceptual proof of Cartan's Second Fundamental Theorem.
	\item A canonical Lie algebroid, we call the \textit{systatic algebroid}, that is intrinsic to any Cartan algebroid, as well as its canonical action on realizations of the Cartan algebroid. This, in turn, leads us to a new reduction theorem for Lie pseudogroups. This part of the paper clarifies the role of Cartan's \textit{systatic system}, one of the least understood aspects of Cartan's theory. 
\end{itemize}

We would like to emphasize that our goal in this paper is not to promote Cartan's theory nor make claims concerning its strengths and weaknesses as compared to other theories, but rather to translate his ideas as faithfully as possible into modern language so that these ideas could be critically studied by means of modern mathematical tools. This approach has already proven to be very fruitful with the most notable example being the last item in the above list.

\subsubsection*{Lie Pseudogroups}

Let us begin with the basic object of interest -- a Lie pseudogroup. Let $M$ be a manifold and let 
\begin{equation*}
\LDiff(M):= \{\; \phi: U \to V\;|\; U,V\subset M\;\text{open subsets},\; \phi\;\text{a diffeomorphism} \;\}
\end{equation*}
denote the set of all locally defined diffeomoprhisms of $M$. 

\begin{mydef*}
	\index{pseudogroup}
	A \textbf{pseudogroup} on $M$ is a subset $\Gamma\subset \LDiff(M)$ that satisfies:
	\begin{enumerate}
		\item Group-like axioms:
		\begin{enumerate}[label=\arabic*)]
			\item if $\phi,\phi'\in\Gamma$ and $\Image(\phi')\subset\Dom(\phi)$, then $\phi\circ\phi'\in\Gamma$,
			\item if $\phi\in\Gamma$, then $\phi^{-1}\in\Gamma$,
			\item $\Id_M \in \Gamma$.
		\end{enumerate}
		\item Sheaf-like axioms:
		\begin{enumerate}[label=\arabic*)]
			\item if $\phi\in\Gamma$ and $U\subset\Dom(\phi)$ is open, then $\phi|_U\in\Gamma$,
			\item if $\phi\in\LDiff(M)$ and $\{U_i\}_{i\in I}$ is an open cover of $\Dom(\phi)$ s.t. $\phi|_{U_i}\in\Gamma\;\; \forall\;i\in I$, then $\phi\in\Gamma$.
		\end{enumerate}
	\end{enumerate}
	An \textbf{orbit} of $\Gamma$ is an equivalence class of points of $M$ given by the equivalence relation
	\begin{equation*}
	x \sim y \;\;\;\; \text{ if and only if }\;\;\;\;\; \exists\; \phi\in\Gamma \text{ such that } \phi(x)=y.
	\end{equation*}
	A pseudogroup is \textbf{transitive} if it has a single orbit, otherwise it is \textbf{intransitive}.
\end{mydef*}

\begin{myremark*}
	Many of Cartan's examples of pseudogroups arise in the following way: given a subset $\Gamma_0 \subset \text{Diff}_{\text{loc}}(M)$ satisfying the group-like axioms, there exists a smallest pseudogroup $\langle \Gamma_0 \rangle$ on $M$ containing $\Gamma_0$ which we call the \textbf{pseudogroup generated by} $\Gamma_0$. It is obtained by ``imposing'' the sheaf-like axioms, similar to sheafification in the theory of sheaves. \qedhere
\end{myremark*}

A \textbf{Lie pseudogroup}, loosely speaking, is a pseudogroup that is defined as the set of solutions of a system of partial differential equations (PDEs). For the precise definition, see Definition \ref{def:liepseudogroup}. For the moment, it suffices to keep in mind that Lie pseudogroups arise in differential geometry as the local symmetries of geometric structures. For example, the set of local symplectomorphisms of a symplectic manifold $(M,\omega)$, i.e. all locally defined diffeomorphisms $\phi\in\LDiff(M)$ that satisfy the partial differential equation $\phi^*\omega=\omega$, is a Lie pseudogroup. 

The work of Lie and Cartan, and a large part of the literature that followed, was restricted to the local study of Lie pseudogroups, i.e. Lie pseudogroups on open subsets of Euclidean spaces. In this case, given a Lie pseudogroup $\Gamma$ on $\mathbb{R}^n$ or an open subset thereof, one typically introduces coordinates $(x,y,...)$ on the copy of $\mathbb{R}^n$ on which the elements of $\Gamma$ are applied, coordinates $(X,Y,...)$ on the copy of $\mathbb{R}^n$ in which the elements of $\Gamma$ take value, and, with respect to these coordinates, every element of $\Gamma$ is represented by its component functions 
\begin{equation*}
X=X(x,y,...),\;Y=Y(x,y,...),...
\end{equation*}
Here are two simple but already very interesting examples of Lie pseudogroups cited from \cite{Cartan1937-1}:

\begin{myexample*}
The diffeomorphisms of $\mathbb{R}^2 \backslash \{y=0\}$ of the form
\begin{equation*}
X=x+ay,\;\;\;\;\; Y=y,
\end{equation*}
parametrized by a real number $a\in\mathbb{R}$, generate (see above Remark) a non-transitive Lie pseudogroup. It is characterized as the set of local solutions of the system of partial differential equations
\begin{equation*}
\frac{\partial X}{\partial x} = 1,\;\;\;\; \frac{\partial X}{\partial y} = \frac{X-x}{y},\;\;\;\; Y=y.
\end{equation*}
\end{myexample*}

\begin{myexample*}
The locally defined diffeomorphisms of $\mathbb{R}^2 \backslash \{y=0\}$ of the form
\begin{equation*}
X=f(x),\;\;\;\;\; Y=\frac{y}{f'(x)},
\end{equation*}
parametrized by a function $f\in\text{Diff}_{\text{loc}}(\mathbb{R})$ (locally defined diffeomorphisms of $\mathbb{R}$), generate a transitive Lie pseudogroup. This pseudogroup is characterized as the set of local solutions of the system of partial differential equations
\begin{equation*}
\frac{\partial X}{\partial x} = \frac{y}{Y},\;\;\;\; \frac{\partial X}{\partial y}=0,\;\;\;\; \frac{\partial Y}{\partial y}=\frac{Y}{y}.
\end{equation*}
\end{myexample*}

\subsubsection*{Cartan's Approach to Lie Pseudogroups}

The study of Lie pseudogroups was initiated by Sophus Lie in a three volume monograph written in collaboration with Friedrich Engel (\cite{Lie1888}, 1888-1893). In this work, Lie concentrated on the special class of Lie pseudogroups of \textit{finite type}. These are, loosely speaking, Lie pseudogroups whose elements are parametrized by a finite number of real variables (e.g. the first example above). They are substantially simpler to handle because they are (locally) encoded by their finite dimensional space of parameters which inherits the structure of a (local) Lie group. In fact, Lie's work on this ``special case'' marked the birth of Lie group theory. 

Lie's key idea was to study these objects by means of their induced set of infinitesimal transformations, or, in modern terms, to study a Lie group by means of its associated Lie algebra of invariant vector fields. Cartan sought to extend Lie's ideas to the general case and to develop a theory that also encompasses Lie pseudogroups of \textit{infinite type}. These are, loosely speaking, Lie pseudogroups whose elements may be also parameterized by arbitrary functions (e.g. the second example above). While a direct generalization of Lie's construction of a Lie algebra of vector fields proved difficult, Cartan showed that one can associate an infinitesimal structure with a Lie pseudogroup by passing to the dual picture of differential forms. Cartan's approach is depicted in the following diagram: 

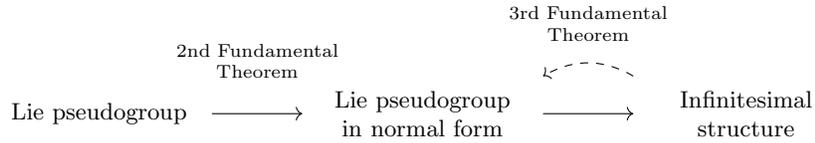
\begin{center}
	\begin{tikzpicture}
	\node at (0,0) {\small \begin{minipage}{2.5cm} \begin{center} Lie pseudogroup  \end{center} \end{minipage}};	
	\draw [->] (1.5,0) -- (2.7,0);
	\node at (2.1,0.7) {\scriptsize \begin{minipage}{2.5cm} \begin{center} 2nd Fundamental Theorem \end{center} \end{minipage}};
	\node at (4.3,0) {\small \begin{minipage}{3cm} \begin{center} Lie pseudogroup \\ in normal form  \end{center} \end{minipage}};
	\draw [->] (5.9,0) -- (7.1,0);
	\node at (8.6,0) {\small \begin{minipage}{3cm} \begin{center} Infinitesimal structure \end{center} \end{minipage}};
	\path [->, dashed, bend right] (7.1,0.5) edge (5.9,0.5);
	\node at (6.5,1.2) {\scriptsize \begin{minipage}{2.5cm} \begin{center} 3rd Fundamental Theorem \end{center} \end{minipage}};
	\end{tikzpicture}
\end{center}

In his \textit{Second Fundamental Theorem}, Cartan showed that any Lie pseudogroup can be replaced by an \textit{equivalent} pseudogroup that is in \textit{normal form}. This theorem introduces two new concepts. The first is that of equivalence, or what we call \textit{Cartan equivalence}, of pseudogroups. Cartan realized that what is important about a pseudogroup is its algebraic structure and not the space on which it acts, and that two pseudogroups should be allowed to be called ``the same'' even if they act on spaces of different dimensions. To achieve this, he introduced a relation that relates two pseudogroups $\widetilde{\Gamma}$ on $P$ and $\Gamma$ on $M$ if there exists a map $\pi:P\to M$ such that $\widetilde{\Gamma}$ is an action of $\Gamma$ on $P$ along $\pi$ (see Definition \ref{def:isomorphicprolongation}), in which case he said that $\widetilde{\Gamma}$ is an \textit{isomorphic prolongation} of $\Gamma$. Then, as equivalence, he took the equivalence relation generated by this. 
\begin{center}
	\begin{tikzpicture}[every node/.style={align=center,text depth=0ex,text height=2ex,text width=4em}]
	\matrix (m) [matrix of math nodes,column sep=-1.5em,row sep=1.5em, row 1/.style={nodes={align=left}}]
	{ 
		& \hspace{-0.3cm} \widetilde{\Gamma} \acts \; P & \\
		\hspace{-0.7cm} \Gamma \acts \; M & & \hspace{0.35cm} M' \;\, \text{\reflectbox{$\acts$}} \; \Gamma'\\
	};
	\path[->,font=\scriptsize]
	(m-1-2) edge node[auto] {} (m-2-1)
	(m-1-2) edge node[auto] {} (m-2-3);
	\end{tikzpicture}
\end{center}

The second new concept, and the essence of Cartan's theory, is that of a \textit{pseudogroup in normal form}. A pseudogroup is in normal form if it is characterized as the set of local symmetries of the following type of object: a collection of functions $I_1,..,I_n$ and 1-forms $\omega_1,...,\omega_r$ that satisfy the set of \textit{structure equations}\footnote{Here, and throughout the paper, we use the Einstein summation convention.}
\begin{equation*}
	d\omega_i + \frac{1}{2} c_i^{jk} \omega_j\wedge \omega_k = a_i^{\lambda j} \pi_\lambda \wedge \omega_j,
\end{equation*} 
where the coefficients $c_i^{jk}$ and $a_i^{\lambda j}$ are functions of the invariants $I_1,..,I_n$ and $\pi_1,...,\pi_p$ are auxiliary 1-forms that complete the $\omega_i$'s to a coframe (for Cartan's precise formulation, see Section \ref{section:cartansrealizationproblem}). Of course, these structure equations generalize the familiar Maurer-Cartan structure equations of a Lie group, the case in which the right-hand side is zero and the coefficients $c_i^{jk}$ are constant, namely the structure constants of the associated Lie algebra (see Example \ref{example:liegroupsliealgebras}).

As in the case of Lie groups, Cartan interpreted the structure functions $c_i^{jk}$ and $a_i^{\lambda j}$ as the infinitesimal structure associated with the Lie pseudogroup and posed the following integrability problem known as the \textit{realization problem} (see Section \ref{section:cartansrealizationproblem} for the precise formulation): starting from a set of functions $c_i^{jk}$ (antisymmetric in the top indices) and $a_i^{\lambda j}$, do these arise as the infinitesimal structure of a Lie pseudogroup? In the special case in which the $a_i^{\lambda j}$'s are zero and the $c_i^{jk}$'s are constant, the answer to the problem is given by Lie's well-known Third Fundamental Theorem: if the constants $c_i^{jk}$ satisfy the Jacobi identity 
\begin{equation*}
c_i^{mj} c_m^{kl} + c_i^{ml} c_m^{jk} + c_i^{mk} c_m^{lj} =0,
\end{equation*}
i.e. if they are the structure constants of a Lie algebra, then they are the structure constants of the Lie algebra of some Lie group. In the general case, Cartan identified a more intricate set of equations that play the role of the Jacobi identity (Equations (C1)-(C3) in Theorem \ref{theorem:necessaryconditions}), and gave a partial solution to the realization problem in what he called the \textit{Third Fundamental Theorem} for Lie pseudogroups: if the initial data is \textit{involutive}, then \textit{local} solutions exists in the \textit{real-analytic} category (see Theorem \ref{theorem:thirdfundamentaltheoremlocal}). The main ingredient of his proof is an analytic tool that he developed for this very purpose, a tool that has evolved into the modern day theory of Exterior Differential Systems. It is interesting to note that there have been no improvements on Cartan's results to date, namely there is no Third Fundamental Theorem in the smooth category nor one of a global nature as in the case of Lie groups. 

\subsubsection*{The Framework of Cartan Algebroids and Realizations}

The first step in modernizing Cartan's structure theory is to upgrade the local coordinate objects and equations to Lie-theoretic structures, in analogy to how Lie's structure constants gave rise to the notion of a Lie algebra. The theory of Lie groupoids and algebroids provides us with the appropriate language and tools to do this. 

In Section \ref{section:cartanalgebroidsrealizations}, we introduce the two main objects of our framework: \textit{Cartan algebroids} and their \textit{realizations}. A Cartan algebroid (Definition \ref{def:cartanalgebroid}) is a pair $(\CAlg,\Tab)$ consisting of a vector bundle $\CAlg$ over a manifold $N$ that is equipped with a Lie algebroid-like structure, but one in which the Jacobi identity fails. This failure is controlled by the second object $\Tab\subset\Hom(\CAlg,\CAlg)$, a vector bundle of fiberwise endomorphisms of $\CAlg$. A realization (Definition \ref{def:realization}) of a Cartan algebroid is a pair $(P,\Omega)$ consisting of a surjective submersion $I:P\to N$ together with a $\CAlg$-valued $1$-form $\Omega\in\Omega^1(P;\CAlg)$ that satisfies the Maurer-Cartan type equation
\begin{equation*}
d\Omega+\frac{1}{2}[\Omega,\Omega]=\Pi\wedge\Omega,
\end{equation*}
where $\Pi\in\Omega^1(P;\Tab)$ is an auxiliary $\Tab$-valued $1$-form that controls the failure of $\Omega$ to be a true Maurer-Cartan form. The structure is depicted in the following diagram:
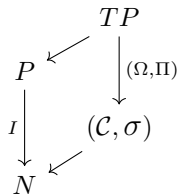
\begin{equation*}
\begin{tikzcd}[column sep=1em,row sep=0.6em]
& \node (1-2) [] {TP}; \\
\node (1-1) [] {P}; & \\ 
& \node (2-2) [] {(\CAlg,\Tab)}; \\
\node (2-1) [] {N};
\arrow[from=1-2, to=1-1] 
\arrow[from=2-2, to=2-1]
\arrow[from=1-1, to=2-1, "I" left]
\arrow[from=1-2, to=2-2, "(\Omega{,}\Pi)" right]
\end{tikzcd}
\end{equation*} 
Any realization has an associated pseudogroup of local symmetries
\begin{equation*}
\Gamma(P,\Omega) = \{ \; \phi \in \LDiff(P) \; | \; \phi^*I=I, \; \phi^*\Omega=\Omega \; \}.
\end{equation*}
Pseudogroups that arise in this way are said to be in \textit{normal form}. Together, a Cartan algebroid and a realization encode Cartan's notion of normal form, and the realization problem becomes the problem of whether a Cartan algebroid admits a realization. 

Section \ref{section:cartanalgebroidsrealizations} will be concluded by introducing an alternative point of view on Cartan algebroids and realizations. We show that, up to \textit{gauge equivalence}, Cartan algebroids are in 1-1 correspondence with \textit{Cartan pairs}, a notion which deviates slightly from Cartan, but which is simpler to handle and closer in nature to the well-understood notion of a Lie algebroid. We believe that this alternative description will lend itself more easily to modern Lie-theoretic methods.

\subsubsection*{Proof of the Second Fundamental Theorem and Lie-Pfaffian Groupoids}

In the framework of Cartan algebroids and realizations, Cartan's Second Fundamental Theorem (Theorem \ref{theorem:secondfundamentaltheorem}) states that: any Lie pseudogroup $\Gamma$ on $M$ admits an isomorphic prolongation in normal form $\Gamma(P,\Omega)$. The essential idea of the proof is due to Cartan: starting with a Lie pseudogroup $\Gamma$ of order $k$ (the order of the defining equations), one passes to the $k$'th jet space $J^k\Gamma$ of $k$-jets of elements of $\Gamma$, which is a geometric realization of the the defining system of PDEs of $\Gamma$ (jets were formalized by Ehresmann, but the idea is already present in Cartan's work). Like any jet space, $J^k\Gamma$ carries a tautological form $\omega$, known as the \textit{Cartan form}, from which one can construct the desired realization and its induced pseudogroup in normal form (jet spaces and the Cartan form are reviewed in Appendix \ref{section:jetgroupoids}). 

While the idea of the proof is simple, the proof itself becomes rather difficult to manage when one sets to work directly with jet spaces. This problem is overcome by the following observation: the pair
\begin{equation*}
(J^k\Gamma,\omega)
\end{equation*}
has the structure of a \textit{Lie-Pfaffian groupoid}, and this structure isolates the precise ingredients that are needed in proving the Second Fundamental Theorem. A Lie-Pfaffian groupoid is a Lie groupoid that is equipped with a $1$-form with coefficients that satisfies a few basic properties (Definition \ref{def:pfaffiangroupoid}). It was introduced in \cite{Salazar2013} with precisely this purpose in mind -- to isolate the essential ingredients that are needed when working with Lie pseudogroups and related objects. The necessary background material on Lie-Pfaffian groupoids is reviewed in Appendix \ref{section:Liepfaffaingroupoidandalgebroid}. 

In Section \ref{section:thesecondfundamentaltheorem}, we follow this path and present a novel proof of the Second Fundamental Theorem. We first show that the pair $(J^k\Gamma,\omega)$ is indeed a Lie-Pfaffian groupoid (Section \ref{section:liepseudogroups}) and then construct an isomorphic prolongation of $\Gamma$ that is in normal form purely out of the data of the Lie-Pfaffian groupoid (Section \ref{section:proofofthesecondfundamentaltheorem}).

\subsubsection*{The Systatic Space and Reduction}

The reduction of ``inessential'' invariants is probably the least understood part of Cartan's work on Lie pseudogroups (see \cite{Cartan1937-1}, pp. 18-24). By studying the stabilizers of a pseudogroup in normal form, Cartan uses the structure equations to derive the set of equations 
\begin{equation*}
a^i_{\lambda j}\omega^j=0,
\end{equation*}
which he calls the \textit{systatic system}. These equations determine an integrable distribution (and hence a foliation) on the space the pseudogroup acts on, which then allows him to determine which of the invariant coordinates are essential and which are inessential, and to argue that the inessential ones can be removed, resulting in a Cartan equivalent pseudogroup that acts on a space of lower dimension.

One of the difficulties in understanding Cartan's reduction procedure is its local nature. The very notion of an invariant coordinate only makes sense locally, and the three simple examples of reduction that Cartan gives in \cite{Cartan1937-1} (pp. 23-24) show that the process of separating the essential invariants from the inessential ones requires a smart choice of a coordinate transformation. Our framework of Cartan algebroids and realizations allows us to gain a deeper understanding of the systatic system and to prove a reduction theorem which is canonical, global and coordinate-free. 

Let us give an outline of the reduction procedure, which is the subject of Section \ref{section:systaticspacesmodern}. We show that, any Cartan algebroid $(\CAlg,\Tab)$ over $N$ contains a canonical Lie algebroid $\mathcal{S}\to N$, which we call the \textit{systatic space}. This Lie algebroid, in turn, acts canonically on any realization $(P,\Omega)$ of the Cartan algebroid in the sense that it acts on the surjective submersion $I:P\to N$ and the $1$-form $\Omega$ is invariant under this action (Propositions \ref{prop:systaticspace1action} and \ref{prop:reductionbasic}):
\begin{center}
	\begin{tikzpicture}[every node/.style={align=center,text depth=0ex,text height=2ex,text width=4em}]
	\matrix (m) [matrix of math nodes,column sep=-0.6em,row sep=2em, row 1/.style={nodes={align=left}}]
	{ 
		\hspace{0.6cm} \mathcal{S} \acts & (P,\Omega) \\
		N & \\
	};
	\path[->,font=\scriptsize]
	(m-1-1) edge node[auto] {} (m-2-1)
	(m-1-2) edge node[midway, right=-17] {$I$} (m-2-1);
	\end{tikzpicture}
\end{center}
The image of this action is an integrable distribution on $P$, which, in local coordinates, coincides precisely with Cartan's systatic system (see equation above). 

Reduction is then obtained by taking the quotient of $(P,\Omega)$ by the action of $\mathcal{S}$. This, however, must be done with care. The quotient, as we will show, does not give rise to a usual pseudogroup on a manifold, but rather to a \textit{generalized pseudogroup} -- a natural extension of the notion of a pseudogroup. Noting that a locally defined diffeomorphism on $M$ can be regarded as a local bisection of the pair groupoid $M\times M\rightrightarrows M$ leads us to the definition of a generalized pseudogroup, namely a set of local bisections of a Lie groupoid $\mathcal{G}\rightrightarrows M$, not necessarily the pair groupoid, that satisfies axioms analogous to those of a pseudogroup. In particular, generalized pseudogroups arise naturally in our setting as the set of local ``solutions'' of a Lie-Pfaffian groupoid (see Definition \ref{def:pfaffiangroupoid} and Example \ref{example:generalizedpseudogroupliepfaffiangroupoid}). Generalized pseudogroups will be discussed in Appendix \ref{section:generalizedpseudogroups}.

In Theorem \ref{theorem:reduction}, we prove that, under suitable regularity conditions, the quotient of a realization $(P,\Omega)$ by the action of the systatic space $\mathcal{S}$ is a Lie-Pfaffian groupoid $(\mathcal{G}_\mathrm{red},\Omega_\mathrm{red})$, and $\Gamma(P,\Omega)$ is an isomorphic prolongation of its generalized pseudogroup of local solutions $\Gamma_\mathrm{red}$. The reduction theorem together with the Second Fundamental Theorem give us the following picture:
\begin{center}
	\begin{tikzpicture}[every node/.style={align=center,text depth=0ex,text height=2ex,text width=4em}]
	\matrix (m) [matrix of math nodes,column sep=-0.5em,row sep=0em, row 1/.style={nodes={align=left}}]
	{ 
		& \hspace{-1.2cm} \Gamma(P,\Omega) \acts \, P & \\
	    & & \hspace{-0cm} (\mathcal{G}_\mathrm{red},\Omega_\mathrm{red}) \\
	    & & & \hspace{-0.6cm} \text{\reflectbox{$\acts$}} \; \Gamma_\mathrm{red}\\
		\hspace{-0.7cm} \Gamma \acts \, M & & \hspace{0.35cm} P_\mathrm{red} \; \\
	};
	\path[->,font=\scriptsize]
	(m-1-2) edge node[auto] {} (m-4-1)
	(m-2-3) edge [transform canvas={xshift=0ex}] node[auto] {} (m-4-3)
	(m-2-3) edge [transform canvas={xshift=1ex}] node[auto] {} (m-4-3)
	(m-1-2) edge node[auto] {} (m-4-3);
	\end{tikzpicture}
\end{center}
Indeed, by the Second Fundamental Theorem, a pseudogroup $\Gamma$ admits an isomorphic prolongation in normal form $\Gamma(P,\Omega)$, which can then be reduced to a generalized pseudogroup $\Gamma_\mathrm{red}$. Hence, $\Gamma$ and its reduction $\Gamma_\mathrm{red}$ are Cartan equivalent. Intuitively, we can think of $\Gamma_\mathrm{red}$ as a ``smaller'' representative of the equivalence class, one that is closer to the true abstract object underlying the pseudogroup.

\subsubsection*{A Brief Survey of the Existing Literature}

This paper is the result of our effort to read Cartan's classical work on Lie pseudogroups and understand his original ideas as faithfully as possible. Our main sources are Cartan's papers  \cite{Cartan1904,Cartan1905} from 1904-1905, and two later papers \cite{Cartan1937-1,Cartan1937-2} of Cartan from 1937 that present the same material in a more concise (and, in our opinion, more accessible) form. Our work also builds upon the work of many mathematicians that continued in Lie and Cartan's footsteps, many of which were motivated, like us, by the challenge of ``understanding Cartan''. The first big step in modernizing Cartan's work is attributed to Charles Ehresmann. Some of Ehresmann's important contributions to this field are: the modern definition of a pseudogroup, the theory of jet spaces, and the introduction of Lie groupoids into the theory (see \cite{Libermann2007} for a historical account). Ehresmann's work marked the beginning of the ``modern era'' of Lie pseudogroups, a renewed interest in the subject that had its peak in the 1950-1970's, but which continues until this very day. A full account of the literature that appeared on this subject over the years deserves a paper of its own. We mention here a small selection of this literature, with an emphasis on papers that were particularly influential to our work (some of which will also be mentioned throughout this paper). 

The structure equations of a Lie pseudogroup have been studied from various perspectives in \cite{Libermann1954,Matsushima1955,Kuranishi1961,Rodrigues1963,Kumpera1964,Singer1965,Guillemin1966,Malgrange1972-1,Malgrange1972-2,Kamran1989,Lisle1995,Stormark2000,Olver2009}. In particular, proofs of Cartan's Second Fundamental Theorem can be found in \cite{Libermann1954,Kuranishi1961,Guillemin1966,Kamran1989,Stormark2000}, a proof of Cartan's Third Fundamental Theorem can be found in \cite{Kumpera1964}, and \cite{Stormark2000} also contains an exposition of Cartan's systatic system and reduction of inessential invariants. Throughout the years, people have also taken several new approaches to the study of Lie pseudogroups. These include: the formal theory of Lie ($F$)-groups and Lie ($F$)-algebras \cite{Kuranishi1959,Kuranishi1961}; the infinitesimal point of view of sheaves of Lie algebras of vector fields \cite{Rodrigues1962,Singer1965}; the study of Lie pseudogroups via the defining equations of their infinitesimal transformations, known as Lie equations \cite{Malgrange1972-1,Malgrange1972-2}; an approach using Milnor's infinite dimensional Lie groups \cite{Kamran2004}; and the point of view of infinite jet bundles \cite{Olver2005} (\cite{Valiquette2008} compares the infinite jet bundle approach with Cartan's theory and examines Cartan's systatic system and the reduction of inessential invariants from this point of view). 

\subsubsection*{Structure of the Paper}

The recommended order for reading this paper is as follows:
\begin{equation*}
\text{Section \ref{section:cartanalgebroidsrealizations}} \hspace{0.1cm} \Rightarrow \hspace{0.1cm} 
\text{Appendix \ref{section:jetgroupoids}  - \ref{section:Liepfaffaingroupoidandalgebroid}} \hspace{0.1cm} \Rightarrow \hspace{0.1cm} \text{Section \ref{section:thesecondfundamentaltheorem}} \hspace{0.1cm} \Rightarrow \hspace{0.1cm} 
\text{Appendix \ref{section:generalizedpseudogroups} - \ref{appendix:basicforms}} \hspace{0.1cm}  \hspace{0.1cm} \Rightarrow \hspace{0.1cm} 
\text{Section \ref{section:systaticspacesmodern}}.
\end{equation*}
In Section \ref{section:cartanalgebroidsrealizations}, we introduce the framework of Cartan algebroids and realizations, as well as the alternative point of view of Cartan pairs and realizations. Appendix \ref{section:jetgroupoids} is a review of jet groupoids and algebroids, and Apendix \ref{section:Liepfaffaingroupoidandalgebroid} is a review of the abstract point of view of Lie-Pfaffian groupoids and algebroids. In Section \ref{section:thesecondfundamentaltheorem}, we present our proof of the Second Fundamental Theorem. Appendix \ref{section:generalizedpseudogroups} introduces the notion of a generalized pseudogroup, and in Appendix \ref{appendix:basicforms} we recall the notion of a basic form in the context of Lie groupoids, which is used in the reduction theorem. In Section \ref{section:systaticspacesmodern}, we study the systatic space of a Cartan algebroid, its action on realizations, and we present the reduction theorem for pseudogroup in normal form.

\subsubsection*{Prerequisites}

We assume familiarity with the theory of Lie groupoids and algebroids. Recommended references on the subject are \cite{Crainic2011-1,Moerdijk2003,Mackenzie2005}. We also assume a basic familiarity with the notion of jets. There are innumerable reference on this subject, e.g. \cite{Saunders1989,Olver1995,Bocharov1999}. See also \cite{Yudilevich2016-2} for a concise review of all the necessary background material.

\subsubsection*{Acknowledgments} The authors would like to thank Pedro Frejlich and Ioan Marcut for fruitful discussions. The second author would like to thank Henrique Bursztyn and IMPA for the opportunity to present a mini-course on this subject. 

The first author was financially supported by the Nederlandse Organisatie voor Wetenschappelijk Onderzoek - Vici grant no. 639.033.312. The second author by the ERC Starting Grant no. 279729 as a PhD student at Utrecht University and the long term structural funding - Methusalem grant of the Flemish Government as a postdoctoral researcher at KU Leuven.

\section{Cartan Algebroids, Realizations and Pseudogroups in Normal Form}
\label{section:cartanalgebroidsrealizations}

The central objects in our modern formulation of Cartan's theory of Lie pseudogroups are \textit{Cartan algebroids}, on the infinitesimal side, and their \textit{realizations}, on the global side. In this modern language, Cartan's realization problem asks whether a given Cartan algebroid admits a realization. As a prelude, let us recall the realization problem in Cartan's language. 

\subsection{Prelude: Cartan's Realization Problem}
\label{section:cartansrealizationproblem}

Fix two integers $N\geq n\geq 0$, and denote the coordinates on $\mathbb{R}^n$ by $(x_1,...,x_n)$, the coordinates on $\mathbb{R}^N$ by $(x_1,...x_n,x_{n+1},...,x_N)$, and the projection by
\begin{equation*}
I=(I_1,...,I_n):\mathbb{R}^N \to \mathbb{R}^n,\hspace{0.5cm} I_a(x)=x_a.
\end{equation*} 
Fix a third integer $r$ such that $N\geq r\geq n$, and set $p:=N-r$. 

\begin{myproblem} [Cartan's Realization Problem]
	\label{problem:cartansrealizationproblem}
	Given functions
	\begin{equation*}
	c_i^{jk},\; a_i^{\lambda j}\in C^\infty(U), \hspace{1.5cm} (1\leq i,j,k \leq r,\; 1\leq \lambda \leq p),
	\end{equation*}
	on an open subset $U\subset \mathbb{R}^n$ such that
	\begin{enumerate}
		\item $c_i^{jk} = -c_i^{kj}$,
		\item the matrices $A^\lambda = (a_i^{\lambda j})$ are linearly independent,
	\end{enumerate}
	find a set of linearly independent $1$-forms $\omega_1,...,\omega_r\in\Omega^1(V)$ on an open subset $V\subset\mathbb{R}^N$ satisfying $I(V)=U$, with
	\begin{equation}
	\omega_1=dI_1|_V,\; ... \; ,\; \omega_n=dI_n|_V,
	\label{eqn:anchoredlocal}
	\end{equation}
	that satisfy the following property: there exists another set of $1$-forms $\pi_1,...,\pi_p\in\Omega^1(V)$	such that $\{\omega_1,..,\omega_r,\pi_1,...\pi_p\}$ is a coframe of $V$ and
	\begin{equation}
	d\omega_i + \frac{1}{2} c_i^{jk} \omega_j\wedge \omega_k =  a_i^{\lambda j} \pi_\lambda \wedge \omega_j,
	\label{eqn:structureequationslocal}
	\end{equation} 
	where $c_i^{jk}, a_i^{\lambda j}$ are viewed as functions on $V$ that are constant along the fibers of $I$. 
\end{myproblem}

An immediate consequence of \eqref{eqn:anchoredlocal} and \eqref{eqn:structureequationslocal} is that 
\begin{equation}
c_i^{jk} = 0 \; \text{ and } \; a_i^{\lambda j} = 0, \hspace{1cm} \forall \; 1\leq i \leq n. \tag{C0}
\end{equation}
We call the initial data of the realization problem, i.e. the functions $(c_i^{jk},a_i^{\lambda j})$ on the open subset $U\subset\mathbb{R}^n$ that satisfy properties 1 and 2 as well as condition $(C0)$, an \textbf{almost Cartan data}. This plays the role of the infinitesimal structure. We call a solution of the realization problem $(I_a,\omega_i)$ a \textbf{realization} of the almost Cartan data. This plays the role of the global structure. Equations \eqref{eqn:structureequationslocal} are called the \textbf{structure equations}. Cartan's realization problem asks: \textit{does an almost Cartan data admit a realization?} 

In what Cartan calls \textit{the third fundamental theorem}, he gives a partial solution to the realization problem. The first step he takes in solving the realization problem is to identify a set of necessary conditions for the existence of a realization:

\begin{mytheorem}[necessary integrable conditions]
	\label{theorem:necessaryconditions}
	If an almost Cartan data $(c_i^{jk},a_i^{\lambda j})$ on $\mathbb{R}^n$ admits a realization, then there exist functions
	\begin{equation*}
	\nu_\lambda^{jk},\; \xi_{\lambda}^{\mu j},\; \epsilon_\lambda^{\eta\mu}\in C^\infty (\mathbb{R}^n) \hspace{1.5cm} (1\leq j,k \leq r,\; 1\leq \lambda,\eta,\mu \leq p),
	\end{equation*}
	with $\nu_\lambda^{jk} = -\nu_\lambda^{kj},\; \epsilon_\lambda^{\eta\mu} = -\epsilon_\lambda^{\mu\eta}$, such that
	\begin{gather}
	a_i^{\eta m} a_m^{\mu j} - a_i^{\mu m} a_m^{\eta j} = a_i^{\lambda j} \epsilon_\lambda^{\eta \mu} \tag{C1}, \\
	c_i^{mj} c_m^{kl} + c_i^{mk} c_m^{lj} + c_i^{ml} c_m^{jk} + \Big(\frac{\partial c_i^{kl}}{\partial x_j} + \frac{\partial c_i^{lj}}{\partial x_k} + \frac{\partial c_i^{jk}}{\partial x_l}\Big) = a_i^{\lambda l} \nu_\lambda^{jk} + a_i^{\lambda k} \nu_\lambda^{lj} + a_i^{\lambda j} \nu_\lambda^{kl} \tag{C2} \\
	a_m^{\lambda j} c_i^{mk}  - a_m^{\lambda k} c_i^{mj} + a_i^{\lambda m} c_m^{jk} + \Big( \frac{\partial a^{\lambda k}_i}{\partial x_j} - \frac{\partial a^{\lambda j}_i}{\partial x_k} \Big) = a_i^{\mu k} \xi_\mu^{\lambda j} - a_i^{\mu j} \xi_\mu^{\lambda k} \tag{C3},
	\end{gather}
	where terms that involve $\partial/\partial x_j$, with $j>n$, are understood to be zero. 
\end{mytheorem}

\begin{proof}
	Let $(I_a,\omega_i)$ be a realization, and choose $\pi_1,...,\pi_p$ such that $\{\omega_1,..,\omega_r,\pi_1,...,\pi_p\}$ is a coframe of $\mathbb{R}^N$ and such that \eqref{eqn:structureequationslocal} is satisfied. Decompose $d\pi_\lambda$ in terms of the coframe:
	\begin{equation*}
	d\pi_\lambda = \frac{1}{2} \nu_\lambda^{jk} \omega_j\wedge \omega_k + \xi_\lambda^{\mu j} \pi_\mu\wedge \omega_j + \frac{1}{2} \epsilon_\lambda^{\eta \mu} \pi_\eta\wedge \pi_\mu,
	\end{equation*}
	where $\nu_\lambda^{jk},\;\xi_\lambda^{\mu j},\;\epsilon_\lambda^{\eta \mu}\in C^\infty(\mathbb{R}^N)$ are the coefficients. Then differentiate \eqref{eqn:structureequationslocal} and replace all appearances of $d\pi_\lambda$ by this decomposition. This gives three sets of equations. Finally, $(C1)-(C3)$, which are equations on $\mathbb{R}^n$, are obtained by restricting these equations to the slice $\{x_{n+1}=0,...,x_N=0\}$ (or to any other section of the projection $I$). 
\end{proof}

We call an almost-Cartan data that satisfies $(C1)-(C3)$, for some set of functions $\nu_\lambda^{jk},\; \xi_{\lambda}^{\mu j},\; \epsilon_\lambda^{\eta\mu}$, a \textbf{Cartan data}. The latter theorem thus says that if an almost Cartan data admits a realization, then it is a Cartan data. 

Any realization induces a pseudogroup, namely the pseudogroup of its local symmetries:
\begin{equation*}
\Gamma(I_a,\omega_i) := \{\; \phi\in\LDiff(V)\;|\; \phi^*I_a=I_a,\; \phi^* \omega_i = \omega_i \;\}.
\end{equation*}
In general, there is no guarantee that this pseudogroup consists of more than just the identity diffeomorphism of $V$ and its restrictions to open subsets. A pseudogroup $\Gamma$ on an open subset $V$ of a Euclidean space is said to be in \textbf{normal form} if it is the pseudogroup of local symmetries of a realization $(I_a,\omega_i)$ on $V$ and if its orbits are the fibers of $I$. Theorem \ref{theorem:secondfundamentaltheoremconverse} will give sufficient conditions for the orbits of $\Gamma(I_a,\omega_i)$ to coincide with the fibers of $I$, and hence to be in normal form.

As a preparation for our coordinate-free definitions of these structures that underly Cartan's realization problem,  let us look at the following simple and familiar example of the realization problem: 

\begin{myexample}[Lie groups and Lie algebras]
	\label{example:liegroupsliealgebras}
	Let us consider the realization problem in the case $n=0$ and $p=0$. An almost Cartan data is simply a set of constants $c_i^{jk}$ ($1\leq i,j,k \leq r$) that are anti-symmetric in the upper indices. Fixing an $r$-dimensional vector space $\mathfrak{g}$ and a basis $X^1,...,X^r$, this data can be encoded in an anti-symmetric bilinear operation on $\mathfrak{g}$:
	\begin{equation*}
	[\cdot,\cdot]: \mathfrak{g} \times \mathfrak{g} \to \mathfrak{g},\hspace{1cm} [X^j,X^k] = c_i^{jk} X^i.
	\end{equation*} 
	Conditions (C1) and (C3) are vacuous while condition (C2) reduces to the well-known Jacobi identity 
	\begin{equation*}
	[[X,Y],Z] + [[Y,Z],X] + [[Z,X],Y] = 0,\hspace{1cm} \forall\;X,Y,Z\in\mathfrak{g}.
	\end{equation*}
	Hence, a Cartan data in this case is the same thing as a Lie algebra. 
	
	A Lie group integrating $\mathfrak{g}$ induces a solution to the realization problem as follows: any Lie group $G$ with Lie algebra $\mathfrak{g}$ comes with a canonical $\mathfrak{g}$-valued 1-form, the \textit{Maurer-Cartan form}:
	\begin{equation*}
	\Omega=\Omega_{\text{MC}}\in \Omega^1(G;\mathfrak{g}),\hspace{1cm} \Omega_g = (dL_{g^{-1}})_g : T_gG\to T_eG = \mathfrak{g}.
	\end{equation*}
	It satisfies two main properties: 
	\begin{enumerate}
		\item[1)] the Maurer-Cartan equation
	\begin{equation*}
	d\Omega + \frac{1}{2}[\Omega,\Omega] = 0,
	\end{equation*}
		\item[2)] it is pointwise an isomorphism, i.e. $\Omega_g:T_gG\xrightarrow{\simeq} \mathfrak{g}$ is a linear isomorphism for all $g\in G$. 
	\end{enumerate}	
	Writing $\Omega = \omega_i X^i$, for some uniquely defined $\omega_i\in\Omega^1(G)$, the Maurer-Cartan equation becomes
	\begin{equation*}
	d\omega_i + \frac{1}{2} c_i^{jk} \omega_j\wedge\omega_k=0,
	\end{equation*}
	and the second property is equivalent to requiring that $\{\omega_1,...,\omega_r\}$ be a coframe of $G$. Thus, we obtain a realization (where the projection $I$ is simply the map from $G$ to a point). We remark that any realization on an open subset $V\subset \mathbb{R}^r$ induces a local Lie group structure on $V$ (see e.g. \cite{Greub1973}, pp. 368-369). Hence, in this simple case, the realization problem is closely related to the problem of integrating a Lie algebra to a Lie group.
	
	Finally, the pseudogroup induced by the realization,
	\begin{equation*}
	\Gamma(G,\Omega) = \{ \; \phi\in\LDiff(G)\;|\; \phi^*\Omega=\Omega \; \},
	\end{equation*}
	is precisely the pseudogroup generated by left translations. It is in normal form, since its single orbit is $G$ itself.  
\end{myexample}

\subsection{Structure Equations (Realizations)}
\label{section:realizationsstructureequations}

We begin with Cartan's very basic idea: pseudogroups realized as the set of local symmetries of a system of functions and 1-forms. Globally, we start with a surjective submersion $I:P\to N$, a vector bundle $\CAlg\to N$ and a $\CAlg$-valued 1-form $\Omega\in\Omega^1(P;I^*\CAlg)$. Such data induces a pseudogroup on $P$, 
\begin{equation}
\Gamma(P,\Omega):= \{ \; \phi\in\LDiff(P)\;|\; \phi^*I=I,\; \phi^*\Omega=\Omega  \; \}.
\label{eqn:realizationinducedpseudogroup}
\end{equation}
Note that the first condition ensures that the second makes sense. One would like to understand the first order consequences of the defining equations of this pseudogroup (e.g. ``$\phi^*(d\Omega) = d\Omega$''). This becomes easier when $\CAlg$ is endowed with extra structure. 

\subsubsection{The Maurer-Cartan Expression}

\begin{mydef}
	\label{def:preliealgebroid}
	\index{Almost Lie algebroid}
	An \textbf{almost Lie algebroid} over a manifold $N$ is a vector bundle $\CAlg\to N$ equipped with a vector bundle map $\rho:\CAlg\to TN$ (``the anchor'') and a bilinear antisymmetric map $[\cdot,\cdot]:\Gamma(\CAlg)\times\Gamma(\CAlg)\to\Gamma(\CAlg)$ (``the bracket'') satisfying the Leibniz identity
	\begin{equation*}
	[\alpha,f\beta] = f[\alpha,\beta] + L_{\rho(\alpha)}(f)\beta,\hspace{1cm} \forall\;\alpha,\beta\in \Gamma(\CAlg),\; f\in C^\infty(N),
	\end{equation*}
	and 
	\begin{equation*}
	\rho([\alpha,\beta]) = [\rho(\alpha),\rho(\beta)],\hspace{1cm} \forall\;\alpha,\beta\in \Gamma(\CAlg).
	\end{equation*}
	An almost Lie algebroid $\CAlg$ is \textbf{transitive} if $\rho:\CAlg\to TN$ is surjective. 
\end{mydef}

\begin{myexample}
	The best known example of an almost Lie algebroid is a Lie algebroid: an almost Lie algebroid whose bracket satisfies the Jacobi identity. 
\end{myexample}

\begin{mydef}
	Given an almost Lie algebroid $\CAlg$ over $N$ and a surjective submersion $I:P\to N$, a 1-form $\Omega\in\Omega^1(P;I^*\CAlg)$ is called \textbf{anchored} if
	\begin{equation*}
	\rho\circ \Omega = dI.
	\end{equation*}
\end{mydef}

The anchored condition on $\Omega$ ensures that, although the expression $d\Omega$ does not make sense globally, the Maurer-Cartan type expression ``$d\Omega + \frac{1}{2}[\Omega,\Omega]$'' does. The construction is the same as for Lie algebroids: let $I:P\to N$ be a surjective submersion, $\CAlg\to N$ an almost Lie algebroid and $\nabla:\mathfrak{X}(N)\times \Gamma(\CAlg)\to \Gamma(\CAlg)$ a connection on $\CAlg$. The connection induces a de Rham-type operator 
\begin{equation*}
d_\nabla:\Omega^*(P;I^*\CAlg) \to \Omega^{*+1}(P;I^*\CAlg)
\end{equation*}
on the space of $\CAlg$-valued forms defined by the usual formula
\begin{equation*}
\begin{split}
(d_\nabla \Omega) (X_0,...,X_p) &= \sum_{i=0}^p (-1)^i (I^*\nabla)_{X_i} (\Omega(X_0,...,\hat{X}_i,...,X_p)) \\ 
& \;\;\;\;\;+ \sum_{0\leq i<j \leq p} (-1)^{i+j} \Omega([X_i,X_j],X_0,...,\hat{X}_i,...,\hat{X}_j,...,X_p),
\end{split}
\end{equation*}
where $\Omega\in\Omega^p(P;I^*\CAlg)$ and $X_0,...,X_p\in\mathfrak{X}(P)$. Note that $(d_\nabla)^2=0$ if and only if $\nabla$ is flat. Next, we also have the $\CAlg$-torsion of $\nabla$, which is the tensor $[\cdot,\cdot]_\nabla \in \Gamma(\Hom(\Lambda^2\CAlg,\CAlg))$ that is defined at the level of sections by $[\alpha,\beta]_\nabla = [\alpha,\beta] -  \nabla_{\rho(\alpha)}\beta + \nabla_{\rho(\beta)}\alpha$ for all $\alpha,\beta\in\Gamma(\CAlg)$. The $\mathcal{C}$-torsion, in turn, induces a graded bracket,
\begin{equation}
\label{eqn:torsionpairing}
[\cdot,\cdot]_\nabla:\Omega^p(P;I^*\CAlg)\times \Omega^q(P;I^*\CAlg) \to \Omega^{p+q}(P;I^*\CAlg),
\end{equation} 
which is defined by the following wedge-like formula:
\begin{equation*}
\begin{split}
&[\Omega,\Omega']_\nabla(X_1,...,X_{p+q}) = \\ & \hspace{2cm} \sum_{\sigma\in S_{p,q}} \text{sgn}(\sigma) \; [ \Omega(X_{\sigma(1)},...,X_{\sigma(p)}) ,\Omega'(X_{\sigma(p+1)},...,X_{\sigma(p+q)})]_\nabla,
\end{split}
\end{equation*}
where $S_{p,q}$ is the group of $(p,q)$-shuffles. 

\begin{myprop}
	\label{prop:maurercartan2form}
	Let $\mathcal{C}$ be an almost Lie algebroid over $N$, let $I:P\to N$ be a surjective submersion and let $\Omega\in\Omega^1(P;I^*\CAlg)$. If $\Omega$ is anchored, then the Maurer-Cartan 2-form
	\begin{equation*}
	\text{MC}_\Omega := d_\nabla\Omega + \frac{1}{2}[\Omega,\Omega]_\nabla \in \Omega^2(P;I^*\CAlg)
	\end{equation*}
	is independent of the choice of connection.
\end{myprop}

\begin{proof}
	Let $\nabla$ and $\nabla'$ be two connections on $\CAlg$ and set $\eta:=\nabla-\nabla'\in \Omega^1(N;\Hom(\CAlg,\CAlg))$. Let $p\in P$ and $X,Y\in T_pP$. The sum of the following two equations vanishes if $\Omega$ is anchored:
	\begin{equation*}
	\begin{split}
	(d_\nabla\Omega - d_{\nabla'}\Omega)(X,Y) &= \eta(dI(X))(\Omega(Y)) - \eta(dI(Y))(\Omega(X)) \\
	([\Omega,\Omega]_\nabla - [\Omega,\Omega]_{\nabla'})(X,Y) &= - \eta(\rho\circ\Omega(X))(\Omega(Y)) + \eta(\rho\circ\Omega(Y))(\Omega(X)) \qedhere
	\end{split}
	\end{equation*}
\end{proof}

\begin{myremark}
	From now on we suppress $\nabla$ from the notation and write $d\Omega+\frac{1}{2}[\Omega,\Omega]$ when $\Omega$ is anchored.
\end{myremark}

Intuitively, $\text{MC}_\Omega$ measures the failure of $\Omega:TP\to \CAlg$ to be a morphism of almost Lie algebroids. For example, when $P=N$ and $I$ is the identity, 
\begin{equation*}
\text{MC}_\Omega(X,Y) = - \Omega([X,Y]) + [\Omega(X),\Omega(Y)],\;\;\;\; \forall X,Y\in\mathfrak{X}(N).
\end{equation*}
When $\Omega$ is pointwise surjective, we have the following useful formula:

\begin{mylemma}
	\label{lemma:maurercartanexpression}
	Let $\mathcal{C}$ be an almost Lie algebroid over $N$, let $I:P\to N$ be a surjective submersion and let $\Omega\in\Omega^1(P;I^*\CAlg)$ be anchored and pointwise surjective. Given any $\alpha \in \Gamma(\CAlg)$, there exists $X_\alpha\in \mathfrak{X}(P)$ such that
	\begin{equation*}
	\Omega(X_\alpha) = I^*\alpha.
	\end{equation*}
	Given a pair $\alpha,\beta \in \Gamma(\CAlg)$, and $X_\alpha,X_\beta\in \mathfrak{X}(P)$ as above, 
	\begin{equation*}
	\text{MC}_\Omega (X_\alpha,X_\beta) = -\Omega([X_\alpha,X_\beta]) + I^*[\alpha,\beta].
	\end{equation*}
\end{mylemma}

\begin{proof}
	An $X_\alpha$ as in the statement can be obtained by choosing a splitting of the short exact sequence of vector bundles $0\to \text{ker}(\Omega)\to TP \xrightarrow{\Omega} I^*\CAlg\to 0$. By the anchored condition:
	\begin{equation*}
	\begin{split}
	(d\Omega + \frac{1}{2}[\Omega,\Omega])(X_\alpha,X_\beta) &= \cancel{(I^*\nabla)_{X_\alpha} (\Omega(X_\beta))} - \cancel{(I^*\nabla)_{X_\beta} (\Omega(X_\alpha))} - \Omega([X_\alpha,X_\beta]) \\
	& \;\;\; + I^*[\alpha,\beta] - \cancel{I^*(\nabla_{\rho(\alpha)}\beta)} + \cancel{I^*(\nabla_{\rho(\beta)}\alpha)}.\qedhere
	\end{split}
	\end{equation*}
\end{proof}

\subsubsection{Almost Cartan Algebroids}

In order to make sense of the structure equations \eqref{eqn:structureequationslocal} globally, one needs a little more than an almost Lie algebroid -- one needs to encode the space where the $\pi$'s live, and this is the role of the symbol space $\Tab$: 

\begin{mydef}
	\label{def:precartanalgebroid}
	\index{almost Cartan algebroid}
	An \textbf{almost Cartan algebroid} over a manifold $N$ is a pair $(\CAlg,\Tab)$, which consists of: 
	\begin{enumerate}
		\item a transitive almost Lie algebroid $\CAlg\to N$,
		\item a vector subbundle $\Tab\subset \Hom(\CAlg,\CAlg)$ (called the \textit{symbol space}),
	\end{enumerate}
	 such that $T(\CAlg)\subset \Ker\rho$ for all $T\in\Tab$ (i.e. $\Tab\subset \Hom(\CAlg,\Ker\rho)$), where $\rho$ is the anchor of $\CAlg$. 
\end{mydef} 

\begin{myremark}
	The vector bundle $\Tab$ is a tableau bundle in the sense of Definition \ref{def:tableaubundle}, and we can talk about its prolongations and associated Spencer cohomology. These play an important role in the theory, and in particular in questions of formal and real-analytic integrability (see e.g. Theorem \ref{theorem:secondfundamentaltheoremconverse}).
\end{myremark}

\begin{myexample}
	\label{example:preliealgebroidcartan}
	Locally, we are back to Cartan: almost Cartan algebroids locally correspond to the notion of an almost Cartan data (Section \ref{section:cartansrealizationproblem}) and they are encoded by functions $c_i^{jk}$ and $a_i^{\lambda j}$. Adapting to the notation of Section \ref{section:cartansrealizationproblem}: 
	\begin{itemize}
		\item $N = U\subset \mathbb{R}^n$, an open subset.
		\item $\CAlg\to N$ is the trivial vector bundle of rank $r$ (where $r\geq n$) with trivializing frame $\{e^1,...,e^r\}$ and endowed with the almost Lie algebroid structure determined by
		\begin{equation*}
		\rho(e^i) = \frac{\partial}{\partial x_i} \;\text{ for }\; 1\leq i \leq n,\;\;\;\; \rho(e^i)=0 \;\text{ for }\; i> n,
		\end{equation*}
		and
		\begin{equation*}
		[e^j,e^k] = c^{jk}_i e^i.
		\end{equation*}
		The fact that $\rho$ is a Lie algebra homomorphism is equivalent to the condition $c_i^{jk}=0$ for $i\leq n$ (first part of condition $(C0)$).
		\item $\Tab\to N$ is the trivial vector bundle of rank $p$ with trivializing frame denoted by $\{t^1,...,t^p\}$. Each element of the frame acts on $\CAlg$ by
		\begin{equation*}
		t^\lambda(e^j) =  a_i^{\lambda j} e^i,
		\end{equation*}
		and, extending by linearity, we obtain a map $\Tab \to \Hom(\CAlg,\CAlg)$. The injectivity of this map is equivalent to Cartan's condition that, at each point of $\mathbb{R}^n$, the matrices $A^\lambda = (a_i^{\lambda j})$ are linearly independent. The condition $\Tab\subset \Hom(\CAlg,\Ker\rho)$ is equivalent to the condition $a^i_{j\lambda}=0$ for $i\leq n$  (second part of condition $(C0)$).  \qedhere
	\end{itemize}
\end{myexample}

\subsubsection{Isomorphism and Gauge Equivalence}
\label{section:gaugeequivalence}

There is an obvious notion of isomoprhism of almost Cartan algebroids. First note that, given two vector bundles $\mathcal{C}$ and $\mathcal{C}'$ over $N$ and a vector subbundle $\Tab\subset \Hom(\mathcal{C},\mathcal{C})$, a vector bundle isomorphism $\psi:\mathcal{C}\to\mathcal{C}'$ maps $\Tab$ into $\Hom(\mathcal{C}',\mathcal{C}')$ by conjugation, i.e.
\begin{equation*}
\psi(\Tab) := \{ \psi\circ T\circ \psi^{-1}\;|\; S\in\Tab \}\subset \Hom(\mathcal{C}',\mathcal{C}').
\end{equation*}

\begin{mydef}
	Two almost Cartan algebroids $(\CAlg,\Tab)$ and $(\CAlg',\Tab')$ over $N$	are \textbf{isomorphic} if there exists a vector bundle isomorphism $\psi:\mathcal{C}\to\mathcal{C}'$ such that $\psi([\alpha,\beta]) = [\psi(\alpha),\psi(\beta)]$ for all $\alpha,\beta\in\Gamma(\mathcal{C})$, $\rho\circ\psi = \rho'$ and $\psi(\Tab)=\Tab'$.
\end{mydef}

\noindent However, this notion of an isomorphism turns out to be too strong, and the slightly weaker notion of gauge equivalence turns out to be the relevant one in the theory. The main evidence for this will come in Section \ref{section:thesecondfundamentaltheorem}, where we will see that the construction of a Cartan algebroid and a realization out of a Lie pseudogroup, both in the general algorithm as well as in examples, depends on a choice, and different choices lead to gauge equivalent structures. We begin now by defining the notion of a gauge equivalence of almost Cartan algebroids, and later show that realizations and Cartan algebroids behave well under such transformations. 

Given an almost Cartan algebroid $(\CAlg,\Tab)$, a choice of a vector bundle map $\eta:\CAlg\to\Tab$ induces a new bracket $[\cdot,\cdot]^\eta$ on $\CAlg$, 
\begin{equation*}
[\alpha,\beta]^\eta := [\alpha,\beta] + \eta(\alpha)(\beta)-\eta(\beta)(\alpha),\hspace{1cm} \forall\; \alpha,\beta\in\Gamma(\CAlg).
\end{equation*}
We denote by $\CAlg^\eta$ the vector bundle $\CAlg$ equipped with the new bracket $[\cdot,\cdot]^\eta$ but with the same anchor $\rho$.

\begin{mylemma}
	\label{lemma:gaugetransformation}
	Let $(\CAlg,\Tab)$ be an almost Cartan algebroid over $N$ and let $\eta\in\Gamma(\Hom(\CAlg,\Tab))$. Then $(\CAlg^\eta,\Tab)$ is an almost Cartan algebroid over $N$.
\end{mylemma}

\begin{proof}
	We only need to verify that $\CAlg^\eta$ is an almost Lie algebroid. The Leibniz identity is clear, and
	\begin{equation*}
	\begin{split}
	\rho[\alpha,\beta]^\eta = [\rho(\alpha),\rho(\beta)] + \cancel{\rho(\eta(\alpha)(\beta))} - \cancel{\rho(\eta(\beta)(\alpha))},\hspace{1cm} \forall\;\alpha,\beta\in\Gamma(\CAlg),
	\end{split}
	\end{equation*}
	because $\Tab\subset \Hom(\CAlg,\Ker\rho)$. 
\end{proof}

\begin{mydef}
	\label{def:gaugeequivalence}
	Two almost Cartan algebroids $(\CAlg,\Tab)$ and $(\CAlg',\Tab')$ over $N$ are \textbf{gauge equivalent} if there exists a vector bundle map $\eta: \CAlg\to \Tab$ s.t. $(\CAlg',\Tab')$ is isomorphic to $(\CAlg^\eta,\Tab)$.
\end{mydef}

Clearly, gauge equivalence defines an equivalence relation on the set of almost Cartan algebroids. 

\begin{myremark}
	Locally, the notion of gauge equivalence already appears in \cite{Kumpera1964}.
\end{myremark}

\subsubsection{Realizations and Structure Equations}

Given an almost Cartan algebroid $(\CAlg,\Tab)$ and a surjective submersion $I:P\to N$, the vector bundle $\Tab\subset \Hom(\CAlg,\CAlg)$ allow us to define a second wedge-like operation 
\begin{equation*}
\wedge : \Omega^p(P;I^*\Tab) \times \Omega^q(P;I^*\CAlg) \to \Omega^{p+q}(P;I^*\CAlg)
\end{equation*}
defined by:
\begin{equation*}
\begin{split}
& (\eta\wedge\phi)(X_1,...,X_{p+q}) = \sum_{\sigma\in S_{p,q}} \text{sgn}(\sigma) \; \eta(X_{\sigma(1)},...,X_{\sigma(p)}) (\phi(X_{\sigma(p+1)},...,X_{\sigma(p+q)})).
\end{split}
\end{equation*}

\begin{mydef}
	\label{def:realization} \index{realization!of an almost Cartan algebroid}
	A \textbf{realization} of an almost Cartan algebroid $(\CAlg,\Tab)$ is a pair $(P,\Omega)$ consisting of a surjective submersion $I:P\to N$ and a pointwise-surjective anchored 1-form 
	\begin{equation*}
	\Omega\in\Omega^1(P;I^*\CAlg)
	\end{equation*}
	such that, for some 1-form $\Pi\in\Omega^1(P;I^*\Tab)$,
	\begin{equation}
	d\Omega+\frac{1}{2}[\Omega,\Omega] = \Pi\wedge\Omega \label{eqn:structureequations}
	\end{equation}
	and
	\begin{equation}
	\label{eqn:vbisomorphism}
	(\Omega,\Pi):TP\xrightarrow{\simeq} I^*(\CAlg\oplus\Tab) 
	\end{equation}
	is vector bundle isomorphism. Equation \eqref{eqn:structureequations} is called the \textbf{structure equation}.
\end{mydef}

As we saw, a realization $(P,\Omega)$ induces a pseudogroup $\Gamma(P,\Omega)$ of its local symmetries (see \eqref{eqn:realizationinducedpseudogroup}). We say that:

\begin{mydef}
	\label{def:normalform}
	A pseudogroup $\Gamma$ on $P$ is in \textbf{normal form} \index{pseudogroup!in normal form} if there exists a realization $(P,\Omega)$ of an almost Cartan algebroid $(\CAlg,\Tab)$ over $N$ such that $\Gamma=\Gamma(P,\Omega)$ (see \eqref{eqn:realizationinducedpseudogroup}) and such that the orbits of $\Gamma$ coincide with the fibers of $I:P\to N$. 
\end{mydef}

The following theorem, due to Cartan, gives a criteria for when a pseudogroup induced by a realization is in normal form. Its proof uses the theory of exterior differential systems and the Cartan-K\"ahler Theorem, which is only valid in the real-analytic category and hence the assumption of real-analyticity. There is no known version of the theorem in the smooth category. In this theorem we see the first appearance of the Spencer cohomology of $\Tab$ and, in particular, the notion of involutivity (Definition \ref{def:acyclicinvolutivity}). 

\begin{mytheorem}[Cartan \cite{Cartan1904,Cartan1937-1}, see also Kumpera \cite{Kumpera1964}]
	Let $(P,\Omega)$ be a real-analytic realization of a real-analytic almost Cartan algebroid $(\CAlg,\Tab)$ (i.e., all manifolds and maps are real-analytic). If the tableau bundle $\Tab$ is involutive, then the orbits of $\Gamma(P,\Omega)$ coincide with the fibers of $I:P\to N$, and hence $\Gamma(P,\Omega)$ is in normal form.
	\label{theorem:secondfundamentaltheoremconverse}
\end{mytheorem}

\begin{myexample}
	\label{example:realizationcartan}
	Locally, a realization of an almost Cartan algebroid corresponds to Cartan's notion of a realization (see Section \ref{section:cartansrealizationproblem}). Continuing from Example \ref{example:preliealgebroidcartan}:
	\begin{itemize}
		\item $P=V\subset \mathbb{R}^N$ is an open subset with coordinates $(x_1,...,x_N)$, and $I:P\to N$ is the projection onto the first $n$ coordinates,
		\begin{equation*}
		I=(I_1,...,I_n):\mathbb{R}^N\to \mathbb{R}^n,\;\;\;\;\; I_a(x) = x_a.
		\end{equation*}
		\item $\Omega$ and $\Pi$ can be decomposed as
		\begin{equation*}
		\Omega = \omega_i \ I^*e^i ,\;\;\;\;\; \Pi= \pi_\lambda \ I^*t^\lambda,
		\end{equation*}
		with $\omega_i,\pi_\lambda\in\Omega^1(P)$. The anchored condition on $\Omega$ is equivalent to 
		\begin{equation*}
		\omega_1=dx_1,...,\omega_n=dx_n.
		\end{equation*}
		Equation \eqref{eqn:structureequations} becomes
		\begin{equation}
		\label{eqn:structureequationsframe}
		d\omega_i + \frac{1}{2} c_i^{jk}\;\omega_j\wedge\omega_k = a_i^{\lambda j} \; \pi_\lambda\wedge \omega_j,
		\end{equation}
		where $c_i^{jk}$ and $a_i^{\lambda j}$ are functions on $\mathbb{R}^n$ viewed as functions on $\mathbb{R}^N$ that are constant along the fibers of $I$. Condition \eqref{eqn:vbisomorphism} is equivalent to $\{\omega_1,...,\omega_r,\pi_1,...,\pi_p\}$ being a coframe. 
		\item The induced pseudogroup on $\mathbb{R}^N$ is
		\begin{equation}
		\begin{split}
		\Gamma(P,\Omega)= \{\; \phi\in\LDiff(\mathbb{R}^N) \;|\; \phi^*I_a=I_a,\;\phi^*\omega_i=\omega_i \;\}.
		\end{split}
		\end{equation}
	\end{itemize}
\end{myexample}

As we mentioned above, realizations of almost Cartan algebroids behave well under gauge equivalence:

\begin{myprop}
	\label{prop:realizationgauge}
	A realization $(P,\Omega)$ of an almost Cartan algebroid $(\CAlg,\Tab)$ serves as a realization of any gauge equivalent one $(\CAlg^\eta,\Tab)$, where $\eta\in\Gamma(\Hom(\CAlg,\Tab))$ (see Definition \ref{def:gaugeequivalence}). Moreover, if $\Pi\in\Omega^1(P;I^*\Tab)$ is a choice for the realization of $(\CAlg,\Tab)$ as in Definition \ref{def:realization}, then $\Pi^\eta\in\Omega^1(P;I^*\Tab)$ defined by
	\begin{equation*}
	\Pi^\eta(X) = \Pi(X) + (I^*\eta)(\Omega(X)),\hspace{1cm} \forall \; X\in\mathfrak{X}(P),
	\end{equation*}
	is a choice for the realization of $(\CAlg^\eta,\Tab)$. 
\end{myprop}

\begin{proof}
	Given $\alpha\in\Gamma(\CAlg)$ and $S\in\Gamma(\Tab)$, we write $X_\alpha,X_S\in\mathfrak{X}(P)$ for the unique vector fields that satisfy $(\Omega,\Pi)(X_\alpha)=I^*\alpha$ and $(\Omega,\Pi)(X_S)=I^*S$. One now easily checks that
	\begin{equation*}
	d\Omega + \frac{1}{2}[\Omega,\Omega]^\eta = \Pi^\eta\wedge\Omega
	\end{equation*}
	is satisfied by applying both sides of the equation on all pairs of the type $(X_\alpha,X_{\alpha'})$, $(X_\alpha,X_S)$, $(X_S,X_{S'})$. The formula for $\Pi^\eta$ implies that the vector fields $X^\eta_\alpha:=X_\alpha-X_{\eta(\alpha)}$ and $X^\eta_S:=X_S\in\mathfrak{X}(P)$ satisfy $(\Omega,\Pi^\eta)(X^\eta_\alpha)=I^*\alpha$ and $(\Omega,\Pi^\eta)(X^\eta_S)=I^*S$, from which we deduce that $(\Omega,\Pi^\eta):TP\to I^*(\CAlg\oplus\Tab)$ is an isomorphism. 
\end{proof}

\subsubsection{A Dual Point of View on Realizations} It is useful to keep in mind the following ``dual'' point of view of the notion of a realization, in which information is retained in the inverse of the pair $(\Omega,\Pi)$ rather than in $(\Omega,\Pi)$ itself.

Let $(P,\Omega)$ be a realization of an almost Cartan algebroid $(\CAlg,\Tab)$ over $N$. Given a choice of $\Pi$ as in Definition \ref{def:realization}, the inverse of the isomorphism \eqref{eqn:vbisomorphism} is the map
\begin{equation}
(\Omega,\Pi)^{-1}:I^*(\CAlg\oplus\Tab)\xrightarrow{\simeq} TP.
\label{eqn:infinitesimalactionofcartanalgebroid}
\end{equation}
Intuitively, one should view this map as an ``infinitesimal action'' of the object $\CAlg\oplus\Tab$ on the surjective submersion $I:P\to N$. This ``action map'' can be decomposed as the sum of the two vector bundle maps
\begin{equation*}
\Psi_{\CAlg,\Pi}:I^*\CAlg\to TP \hspace{1cm} \text{and} \hspace{1cm} \Psi_{\Tab,\Pi}:I^*\Tab\to TP.
\end{equation*}
We write
\begin{equation}
\label{eqn:realizationinversemap}
\begin{split}
&X_\alpha := (\Omega,\Pi)^{-1}(I^*\alpha)\in\mathfrak{X}(P), \hspace{1cm} \forall\;\alpha\in\Gamma(\CAlg) \\
&X_S = (\Omega,\Pi)^{-1}(I^*S)\in\mathfrak{X}(P), \hspace{1cm} \forall\;S\in\Gamma(\Tab).
\end{split}
\end{equation}
Thus, such vector fields are characterized by the conditions
\begin{equation*}
\begin{split}
&\Omega(X_\alpha) = I^*\alpha,\hspace{1cm}\Pi(X_\alpha) = 0, \\
&\Omega(X_S) = 0,\hspace{1.45cm} \Pi(X_S) = I^*S. \\
\end{split}
\end{equation*}
They should be thought of as the fundamental vector fields of the ``infinitesimal action'', and they provide the ``dual'' point of view. 

\begin{mylemma}
	\label{lemma:actionproperties}
	\label{lemma:structureequationsaction}
	Let $(P,\Omega)$ be a realization of an almost Cartan algebroid $(\CAlg,\Tab)$ and fix a choice of $\Pi$. Then,
	\begin{equation*}
	\begin{split}
	& \Omega([X_\alpha,X_{\alpha'}]) = I^*[\alpha,\alpha'], \\
	& \Omega([X_\alpha,X_S]) = I^*S(\alpha), \\
	& \Omega([X_S,X_{S'}]) = 0,
	\end{split}
	\end{equation*}
	for all $\alpha,\alpha'\in\Gamma(\CAlg)$ and $S,S'\in\Gamma(\Tab)$. In particular, $\Ker\Omega\subset TP$ is an involutive distribution. 
\end{mylemma}

\begin{proof}
	Follows directly from the structure equation \eqref{eqn:structureequations} together with lemma \ref{lemma:maurercartanexpression}.
\end{proof}

The fact that $\Ker\Omega$ is an involutive distribution is one first consequence of the structure equations. Another important consequence is: 

\begin{mylemma}
	\label{lemma:realizationsecondcomponentcanonical}
	Let $(P,\Omega)$ be a realization of an almost Cartan algebroid $(\CAlg,\Tab)$. The map
	\begin{equation}
	\label{eqn:canonicalinclusionisotropyalgebroid}
	\Psi_\Tab = \Psi_{\Tab,\Pi}:I^*\Tab \to  TP
	\end{equation}
	is independent of the choice of $\Pi$. Thus, there is a canonical isomorphism
	\begin{equation*}
	\Psi_\Tab:I^*\Tab\xrightarrow{\simeq} \Ker\Omega.
	\end{equation*}
\end{mylemma}

\begin{proof}
	Fix a choice of $\Pi$. We must show that if $\Pi'$ is another such choice, then $\Pi'(X_S)=I^*S$, or equivalently, that $\Pi'(X_S)(I^*\alpha)= \Pi(X_S)(I^*\alpha)$ for any $\alpha\in\Gamma(C)$. Subtracting the structure equations for $\Pi$ and $\Pi'$ from each other, we see that $(\Pi'-\Pi)\wedge\Omega=0$. Thus, for any $\alpha\in\Gamma(\CAlg)$,
	\begin{equation*}
	\begin{split}
	0&= \big((\Pi'-\Pi)\wedge\Omega\big) (X_S,X_\alpha) \\ &  = \Pi'(X_S)(\Omega(X_\alpha)) - \Pi'(X_\alpha)(\cancel{\Omega(X_S)}) - \Pi(X_S)(\Omega(X_\alpha)) + \Pi(X_\alpha)(\cancel{\Omega(X_S)}) \\
	&= \Pi'(X_S)(I^*\alpha) - \Pi(X_S)(I^*\alpha). \qedhere
	\end{split}
	\end{equation*}
\end{proof}

Thus, in the notation $X_\alpha$ and $X_S$, one should keep in mind that $X_\alpha$ depends on the choice of a $\Pi$, while $X_S$ does not. 

\begin{myremark}
	While condition \ref{eqn:structureequations} in the definition of a realization is rather natural, condition \ref{eqn:vbisomorphism} is less so. In the examples of realizations coming from Lie pseudogroups, this condition is always satisfied. Going beyond Lie pseudogroups, it would be interesting to relax this condition in two possible directions: 1) weaken the definition of a realization by dropping condition \ref{eqn:vbisomorphism} or requiring a weaker condition; 2) in the dual point of view, requiring of \eqref{eqn:infinitesimalactionofcartanalgebroid} to be an isomorphism is like requiring of the ``infinitesimal action'' to be free (injectivity) and transitive (surjectivity), and one may relax these conditions.  
\end{myremark}

\subsubsection{Freedom in Choosing $\Pi$}

In the definition of a realization (Definition \ref{def:realization}), we require the existence of a 1-form $\Pi$, which is, in general, not unique. The ambiguity in the choice of $\Pi$ can be be described in terms of the 1st prolongation of $\Tab$ (Definition \ref{def:tableaubundleprolongation}). For simplicity, let us assume that the first prolongation $\Tab^{(1)}$ is of constant rank. Fixing a $\Pi\in\Omega^1(P;I^*\Tab)$ that satisfies \eqref{eqn:structureequations} and \eqref{eqn:vbisomorphism} as a ``reference point'' (which also determines a choice of the maps \eqref{eqn:realizationinversemap}), we have a vector bundle isomorphism
\begin{equation}
\label{eqn:identification1stprolongation}
I^*\Hom(\CAlg,\Tab)\cong \{\; \hat{\xi}\in\Hom(TP;I^*\Tab)\;|\; \hat{\xi}(X_s)=0\;\;\forall S\in\Gamma(\Tab) \; \},\hspace{0.5cm} \xi\mapsto \hat{\xi}, 
\end{equation}
where $\hat{\xi}$ is uniquely determined by the conditions
\begin{gather}
\hat{\xi}(X_\alpha) = \xi(I^*\alpha)\hspace{0.5cm}  \forall\;\alpha\in\Gamma(\CAlg),\\
\hat{\xi}(X_S)=0\hspace{0.5cm} \forall\;S\in\Gamma(\Tab). \label{eqn:realizationconditions}
\end{gather}
The isomorphism \eqref{eqn:identification1stprolongation} restricts to the isomorphism 
\begin{equation}
\label{eqn:realization1prolongationg}
I^*\Tab^{(1)}\cong \{\; \hat{\xi}\in\Hom(TP;I^*\Tab)\;|\; \hat{\xi}(X_s)=0\;\;\forall S\in\Gamma(\Tab) \; \text{and}\; \hat{\xi}\wedge\Omega=0 \; \}.
\end{equation} 
At the level of sections, we obtain a linear isomorphism between sections $\xi\in\Gamma(I^*\Tab^{(1)})$ and 1-forms $\hat{\xi}\in\Omega^1(P;I^*\Tab)$ that satisfy both \eqref{eqn:realizationconditions} and $\hat{\xi}\wedge\Omega=0$. From now on, we write $\xi=\hat{\xi}$.

\begin{myprop}
	\label{prop:freedominchoosingpi}
	Let $(P,\Omega)$ be a realization of an almost Cartan algebroid $(\CAlg,\Tab)$ and assume that $\Tab^{(1)}$ is of constant rank. The subspace of $\Omega^1(P;I^*\Tab)$ consisting of elements $\Pi$ satisfying \eqref{eqn:structureequations} and \eqref{eqn:vbisomorphism} is an affine space modeled on $\Gamma(I^*\Tab^{(1)})$.
\end{myprop}

\begin{proof}
	Fix a choice of $\Pi\in\Omega^1(P;I^*\Tab)$ satisfying \eqref{eqn:structureequations} and \eqref{eqn:vbisomorphism}. Given any other choice $\Pi'$, Lemma \ref{lemma:realizationsecondcomponentcanonical} implies that the difference $\Pi'-\Pi\in\Omega^1(P;I^*\Tab)$ satisfies $(\Pi'-\Pi)(X_S)=0$ for all $S\in\Gamma(\Tab)$, and the structure equations imply that $(\Pi'-\Pi)\wedge\Omega=0$. Hence, $\Pi'-\Pi\in\Gamma(I^*\Tab^{(1)})$. Conversely, let $\xi\in\Gamma(I^*\Tab^{(1)})$. We claim that $\Pi+\xi$ satisfies conditions \eqref{eqn:structureequations} and \eqref{eqn:vbisomorphism}. Because $\xi$ satisfies $\hat{\xi}\wedge\Omega=0$, it follows that $\Pi+\xi\in\Omega^1(P;I^*\Tab)$ satisfies \eqref{eqn:structureequations}. Moreover, the composition
	\begin{equation*}
	\big( (\Omega,\Pi+\xi)\circ(\Omega,\Pi)^{-1} \big)(I^*\alpha,I^*S) = (I^*\alpha,I^*S+\xi(I^*\alpha))
	\end{equation*}
	is a vector bundle automorphism of $I^*(\CAlg\oplus\Tab)$, which implies that $(\Omega,\Pi+\xi)$ satisfies \eqref{eqn:vbisomorphism}.	
\end{proof}

\subsection{Cartan Algebroids}
\label{section:cartanalgebroids}

While the notion of an almost Cartan algebroid allows us to talk about structure equations, the correct underlying infinitesimal structure is more subtle. This is already clear in Cartan's local picture, where the coefficients $c^k_{ij}$ and $a^k_{i\lambda}$ suffice to write down the structure equations, but the fact that they come from structure equations implies that they should form a Cartan data (Theorem \ref{theorem:necessaryconditions}). Globally, these extra conditions are encoded in the notion of a \textit{Cartan algebroid}.

The following notion is standard for Lie algebroids, but it also makes sense for almost Lie algebroids. Let $\CAlg$ be an almost Lie algebroid over $N$. A \textbf{$\CAlg$-connection} on a vector bundle $\Tab$ over $N$ is a bilinear operation 
\begin{equation*}
\nabla:\Gamma(\CAlg)\times \Gamma(\Tab)\to \Gamma(\Tab),\hspace{1cm} (\alpha,T)\to \nabla_\alpha(T)
\end{equation*}
satisfying
\begin{equation*}
\nabla_{f\alpha}(T) = f\nabla_\alpha(T),\hspace{1cm} \nabla_\alpha(fT) = f\nabla_\alpha(T) + L_{\rho(\alpha)}(f)T,
\end{equation*}
for all $\alpha\in\Gamma(\CAlg),\;T\in\Gamma(\Tab)$ and $f\in C^\infty(N)$. For the following definition, recall also that $\Hom(\CAlg,\CAlg)$ is a bundle of Lie algebras with the fiberwise commutator bracket $[T,S] = T\circ S - S\circ T$. 

\begin{mydef}
	\label{def:cartanalgebroid}
	\index{Cartan algebroid}
	A \textbf{Cartan algebroid} is an almost Cartan algebroid $(\CAlg,\Tab)$ over $N$ such that:
	\begin{enumerate}
		\item $\Tab\subset \Hom(\CAlg,\CAlg)$ is closed under the commutator bracket,
		\item there exists a vector bundle map $t:\Lambda^2\CAlg\to\Tab,\; (\alpha,\beta)\mapsto t_{\alpha,\beta}$, such that
		\begin{equation}
		\label{eqn:cartanalgebroidconditionjacobi}
		[[\alpha,\beta],{\gamma}] + [[\beta,{\gamma}],\alpha] + [[{\gamma},\alpha],\beta] = t_{\alpha,\beta}({\gamma}) + t_{\beta,{\gamma}}(\alpha) + t_{{\gamma},\alpha}(\beta)
		\end{equation}
		for all $\alpha,\beta,\gamma\in\Gamma(\CAlg)$,
		\item there exists a $\CAlg$-connection $\nabla$ on $\Tab$ such that
		\begin{equation}
		\label{eqn:cartanalgebroidconditionactiong}
		T([\alpha,\beta]) - [T(\alpha),\beta] - [\alpha,T(\beta)] = \nabla_\beta(T)(\alpha) - \nabla_\alpha(T)(\beta)
		\end{equation}
		for all $\alpha,\beta\in\Gamma(\CAlg),\;T\in\Gamma(\Tab)$.
	\end{enumerate}
\end{mydef}

Thus, $t$ controls the failure of the Jacobi identity and $\nabla$ controls the failure of $\Tab$ to act on $\CAlg$ by derivations. Condition 1 can be restated as the condition that $\Tab\subset\Hom(\CAlg,\CAlg)$ must be a subbundle as a bundle of Lie algebras, or, equivalently, a Lie subalgebroid (since bundles of Lie algebras are the same thing Lie algebroids with zero anchor).

\begin{myremark}
	It is interesting to note that if we were to relax the definition of an almost Cartan algebroid and only require of $\Tab$ to be a vector subbundle of $\Hom(\CAlg,\CAlg)$, and not necessarily of $\Hom(\CAlg,\Ker\rho)$, then the fact that $\Tab$ actually lies in $\Hom(\CAlg,\Ker\rho)$ would follow from condition 3 in the above definition. Indeed, replacing $\beta$ by $f\beta$ in this condition, where $f\in C^\infty(N)$, one sees that $L_{\rho(T(\alpha))}(f)\beta = 0$, which implies that $\rho\circ T$ must vanish for all $T\in\Gamma(\Tab)$. 
\end{myremark}

\begin{myexample}
	\label{example:cartanalgebroidcartan}
	Locally, a Cartan algebroid is the same thing as a Cartan data. Continuing from Examples \ref{example:preliealgebroidcartan} and \ref{example:realizationcartan}:
	\begin{itemize}
		\item Condition $(1)$ is equivalent to the existence of functions $\epsilon^\lambda_{\eta\mu}$ on $N$ such that
		\begin{equation*}
		[t_\eta,t_\mu] = \epsilon^\lambda_{\eta\mu}t_\lambda.
		\end{equation*}
		This is precisely Equation (C1) in Theorem \ref{theorem:necessaryconditions}. 
		\item The bundle map $t:\Lambda^2\CAlg\to \Tab$ can be written as
		\begin{equation*}
		t(e_i,e_j) = \nu^\lambda_{ij}t_\lambda,
		\end{equation*}
		and a straightforward computation shows that condition $(2)$ is equivalent to (C2).
		\item A $\CAlg$-connection on $\Tab$ is determined by
		\begin{equation*}
		\nabla_{e_i}(t_\mu) = \xi^\lambda_{\mu i}t_\lambda,
		\end{equation*}
		and we readily verify that condition $(3)$ is equivalent to (C3). \qedhere
	\end{itemize}
\end{myexample}

Cartan algebroids behave well under gauge equivalences (see Section \ref{section:gaugeequivalence}):

\begin{mylemma}
	\label{lemma:gaugetransformationcartanalgebroid}
	Let $(\CAlg,\Tab)$ be a Cartan algebroid over $N$ and let $\xi\in\Gamma(\Hom(\CAlg,\Tab))$ be a gauge equivalence. Then $(\CAlg^\xi,\Tab)$ is again a Cartan algebroid over $N$. Furthermore, if an almost Cartan algebroid is gauge equivalent to a Cartan algebroid, then it is a Cartan algebroid. 
\end{mylemma}

\begin{proof}
	We know that $(\CAlg^\xi,\Tab)$ is an almost Cartan algebroid and we must verify that it is a Cartan algebroid by checking the three conditions of Definition \ref{def:cartanalgebroid}. The first condition is immediately satisfied. For the other two, choose $t$ and $\nabla$ for the Cartan algebroid $(\CAlg,\Tab)$. A straightforward computation shows that the remaining conditions are satisfied with $t^\xi:\Lambda^2\CAlg^\xi\to\Tab$ and $\nabla^\xi:\Gamma(\CAlg^\xi)\times\Gamma(\Tab)\to\Gamma(\Tab)$ defined by
	\begin{equation*}
	\begin{split}
	& t^\xi_{\alpha,\beta}:= t_{\alpha,\beta} - \nabla_\alpha(\xi(\beta)) + \nabla_\beta(\xi(\alpha)) - [\xi(\alpha),\xi(\beta)] + \xi([\alpha,\beta]) \\ & \hspace{1.5cm} -\xi(\xi(\beta)(\alpha)) + \xi(\xi(\alpha)(\beta)), \\
	& \nabla^\xi_\alpha(S):= \nabla_\alpha(S) + [\xi(\alpha),S] + \xi(S(\alpha)).
	\end{split}	
	\end{equation*}	
	The second assertion follows from the fact that gauge equivalence is an equivalence relation. 
\end{proof}

\subsubsection{Freedom in Choosing $t$ and $\nabla$}

Similar to realizations, the freedom in the choice of $t$ and $\nabla$ in the definition of a Cartan algebroid $(\CAlg,\Tab)$ can be expressed in terms of the Spencer complex of $\Tab$.

\begin{myprop}
	\label{prop:freedomcartanalgebroid}
	Let $(\CAlg,\Tab)$ be a Cartan algebroid over $N$. 
	\begin{enumerate}
		\item The subspace of $\Gamma(\Hom(\Lambda^2\CAlg,\Tab))$ consisting of elements $t$ satisfying \eqref{eqn:cartanalgebroidconditionjacobi} is an affine space modeled on $\Gamma(Z^{0,2}(\Tab))$.
		\item For each $S\in\Gamma(\Tab)$, the subspace of $\Gamma(\Hom(\CAlg,\Tab))$ consisting of elements $\nabla(S)$ satisfying \eqref{eqn:cartanalgebroidconditionactiong} is an affine space modeled on $\Gamma(\Tab^{(1)})$. 
	\end{enumerate}
\end{myprop}

\begin{proof}
	By \eqref{eqn:cartanalgebroidconditionjacobi}, the difference of two choices $t$ and $t'$ satisfies $\delta(t'-t) = 0$, where $\delta$ is the coboundary operator \eqref{eqn:spencercoboundaryoperator}. Conversely, given a choice of a $t$ and $\xi\in \Gamma(Z^{0,2}(\Tab))$, clearly \eqref{eqn:cartanalgebroidconditionjacobi} is satisfied when replacing $t$ by $t+\xi$. This proves items 1. Similarly, in item 2, given two choices $\nabla$ and $\nabla'$ and $S\in\Gamma(\Tab)$, $\delta(\nabla(S)) = 0$ by \eqref{eqn:cartanalgebroidconditionactiong}. Conversely, given a choice of $\nabla$ and $\xi\in\Gamma(\Tab^{(1)})$, \eqref{eqn:cartanalgebroidconditionactiong} is satisfied when replacing $\nabla(S)$ by $\nabla(S)+\xi$. 
\end{proof}

\subsubsection{The Need for Cartan Algebroids for the Existence of Realizations}

In the local picture, we saw that if an almost Cartan data admits a realization, then it is a Cartan data (Theorem \ref{theorem:necessaryconditions}). In the modern picture, this translates to:

\begin{mytheorem} 
	\label{theorem:realizationimpliescartanalgebroid}
	If an almost Cartan algebroid admits a realization, then it is a Cartan algebroid. 
\end{mytheorem}

\begin{proof}
	This proof is essentially a global version of the proof of Theorem \ref{theorem:necessaryconditions}. Choose $\Pi\in\Omega^1(P;I^*\CAlg)$ as in definition \ref{def:realization}. The 2-form $d\Omega+\frac{1}{2}[\Omega,\Omega] - \Pi\wedge\Omega\in \Omega^2(P;I^*\CAlg)$ vanishes, as well as its ``differential consequence'' $d_\nabla(d\Omega+\frac{1}{2}[\Omega,\Omega] - \Pi\wedge\Omega)\in \Omega^3(P;I^*\CAlg)$, for any choice of $\nabla$. Applying $d_\nabla(d\Omega+\frac{1}{2}[\Omega,\Omega] - \Pi\wedge\Omega)=0$ on a triple of vector fields of type $X_\alpha,X_{\alpha'},X_{\alpha''}\in\mathfrak{X}(P)$ and using Lemmas \ref{lemma:maurercartanexpression} and \ref{lemma:actionproperties} implies the identity
	\begin{equation*}
	\begin{split}
	& I^*([[\alpha,{\alpha'}],{\alpha''}]] + [[{\alpha'},{\alpha''}],\alpha]] + [[{\alpha''},\alpha],{\alpha'}]]) = \\ & \;\;\;\;\;\;\;\;\Pi([X_\alpha,X_{\alpha'}])(I^*{\alpha''}) + \Pi([X_{\alpha'},X_{\alpha''}])(I^*\alpha) + \Pi([X_{\alpha''},X_\alpha])(I^*{\alpha'}),
	\end{split}
	\end{equation*}   
	applying it on $X_\alpha,X_{\alpha'},X_S$ implies
	\begin{equation*}
	\begin{split}
	I^*(S([\alpha,{\alpha'}]) - [S(\alpha),{\alpha'}]] - [\alpha,S({\alpha'})]]) 
	= \Pi([X_{\alpha'},X_S])(I^*\alpha) - \Pi([X_\alpha,X_S])(I^*{\alpha'}),
	\end{split}
	\end{equation*}
	and applying it on $X_\alpha,X_S,X_{S'}$ implies
	\begin{equation}
	\label{eqn:realizationcartanalgebroidproof}
	\begin{split}
	I^*({S'}\circ S(\alpha)  - S\circ {S'}(\alpha)) = \Pi([X_S,X_{S'}])(I^*\alpha).
	\end{split}
	\end{equation}
	The three latter equations are equalities in $\Gamma(I^*\CAlg)$. Choosing a local section $\eta$ of $I:P\to N$ with domain $U\subset N$ and precomposing each of the equations with $\eta$ produces the three conditions in definition \ref{def:cartanalgebroid}, but restricted to $U$, where the maps $t$ and $\nabla$ at a point $x\in U$ are given by
	\begin{equation}
	\label{eqn:proofrealizationtocartanalgebroid}
	\begin{split}
	& (t_{\alpha,{\alpha'}})_x = \Pi([X_\alpha,X_{\alpha'}])_{\eta(x)}, \\
	& (\nabla_\alpha(S))_x = \Pi([X_\alpha,X_S])_{\eta(x)}.
	\end{split}
	\end{equation}
	The fact that $\Pi(X_\alpha)=\Pi(X_{\alpha'})=0$ and $\Pi(X_S) = I^*S$ implies that $t$ defines a tensor and $\nabla$ a connection. A standard partition of unity argument produces a global $t$ and $\nabla$.
\end{proof}

\begin{myremark}
	The notion of an almost Cartan algebroid contains the minimal amount of structure that is needed in order to define the notion of a realization. However, the above theorem shows that the relevant structure is actually that of a Cartan algebroid. Thus, from now one we will talk about realizations of Cartan algebroids rather than of almost Cartan algebroids. 
\end{myremark}

\begin{mycor}
	\label{cor:realizationimpliescartanalgebroid}
	If $(P,\Omega)$ is a realization of $(\CAlg,\Tab)$, then for any choice of $\Pi$, $t$ and $\nabla$, 
	\begin{equation*}
	\begin{split}
	&\big( (\alpha,\alpha')\mapsto \Pi([X_\alpha,X_{\alpha'}]) - I^*t_{\alpha,\alpha'}\big) \;\; \in  \Gamma(Z^{0,2}(\Tab)) ,  \\
	&\big( \alpha\mapsto \Pi([X_\alpha,X_S]) - I^*\nabla_\alpha(S) \big) \;\; \in \Gamma(\Tab^{(1)}),  \\
	&\Pi([X_S,X_{S'}]) + I^*[S,S'] = 0
	\end{split}
	\end{equation*}
	for all $S,S'\in\Tab$ (see Appendix \ref{section:tableaubundles} for the definition of $Z^{0,2}(\Tab)$ and $\Tab^{(1)}$).
\end{mycor}

\begin{proof}
	The proof of this corollary is contained in the proof of Theorem \ref{theorem:realizationimpliescartanalgebroid}. The third equation in the statement of the corollary is precisely \eqref{eqn:realizationcartanalgebroidproof}. The first two equations in the statement follow from Proposition \eqref{prop:freedomcartanalgebroid} together with \eqref{eqn:proofrealizationtocartanalgebroid}, first locally by choosing a local section of $I:P\to N$, and then globally by a standard partition of unity argument.  
\end{proof}

Lemma \ref{lemma:realizationsecondcomponentcanonical} together with the third identity in Lemma \ref{lemma:structureequationsaction} directly imply that the bundle of Lie algebras $\Tab$ of a Cartan algebroid $(\CAlg,\Tab)$ acts canonically on all realizations:

\begin{myprop}
	\label{prop:gactsonrealization}
	Let $(P,\Omega)$ be a realization of a Cartan algebroid $(\CAlg,\Tab)$ over $N$. The canonical vector bundle map
	\begin{equation*}
	\Psi_\Tab:I^*\Tab \to TP
	\end{equation*}
	defines a Lie algebroid action of $\Tab$ on $I:P\to N$. Moreover, the action is infinitesimally free (i.e. $\Psi_\Tab$ is injective) and its image is $\Ker\Omega\subset TP$.
\end{myprop}

\subsubsection{The Third Fundamental Theorem}

In this modern formulation, Cartan's realization problem (Problem \ref{problem:cartansrealizationproblem}) becomes the question of whether a Cartan algebroid admits a realization. The following theorem, due to Cartan, provides a partial solution to this problem. The main ingredient in the proof is the Cartan-K\"ahler Theorem (c.f. Theorem \ref{theorem:secondfundamentaltheoremconverse}). For the proof, see Cartan's original papers \cite{Cartan1904,Cartan1937-1} or Kumpera's proof in \cite{Kumpera1964}.

\begin{mytheorem}
	\label{theorem:thirdfundamentaltheoremlocal}
	(the third fundamental theorem) Let $(\CAlg,\Tab)$ be a real-analytic Cartan algebroid over $N$ (i.e. all manifolds and maps are analytic). If the tableau bundle $\Tab$ is involutive, then every $x\in N$ has a neighborhood $U\subset N$ such that the restricted Cartan algebroid $(\CAlg_U,\Tab_U)$ over $U$ (see Example \ref{example:restrictioncartanalgebroid}) admits a realization.  
\end{mytheorem}

\begin{myremark}
	Note that the existence of local solutions to the realization problem trivially implies the existence of a global solution, since realizations of $(\CAlg_U,\Tab_U)$ and $(\CAlg_V,\Tab_V)$, with $U,V\subset N$ open subsets, induce a realization of $(\CAlg_{U\cup V},\Tab_{U\cup V})$ by simply taking the disjoint union of the two realizations. More interesting is the question of whether there exists a global realization $(P,\Omega)$ with $P$ connected. This global problem is still open in the analytic case, while in the smooth case both the local and global problems are open. These problems have proven to be very difficult ones, and, at least in the smooth category, they may require new ideas and possibly new analytic tools, such as an analogue of the Cartan-K\"ahler theorem in the smooth setting. We hope that this modern formulation will provide new insights into this fascinating problem. In Chapter 7 of \cite{Yudilevich2016-2} (and see also \cite{Yudilevich2016}), we propose one possible new approach for tackling the realization problem that is based on a reformulation of the problem that will be discussed in the Section \ref{section:cartanpairs}. 
\end{myremark}

The main ingredients of Kumpera's proof in \cite{Kumpera1964} (which is along the lines of Cartan's proof, but presented in a rigorous and clear fashion)  are the theory of exterior differential systems and the Cartan-K\"ahler Theorem (see e.g. \cite{Bryant1991} for an introduction to these tools). Let us explain the general idea of the proof. Let $(\CAlg,\Tab)$ be a Cartan algebroid over $N$. Since we are looking for local solutions, we may assume that $N$ is an open subset of $\mathbb{R}^n$. Let $r$ and $p$ be the ranks of $\CAlg$ and $\Tab$, respectively, and let $\mathrm{pr}:\mathbb{R}^{r+p} \to \mathbb{R}^n$ be the projection onto the first $n$ coordinates. Set $P:=\mathrm{pr}^{-1}(N)\subset \mathbb{R}^{r+p}$ and denote the restriction of $\mathrm{pr}$ to $P$ by $I:P\to N$. Given $q\in P$, we would like to find a 1-form $\Omega^1(P;I^*\CAlg)$ defined locally around $q$ such that $(P,\Omega)$ is a realization of $(\CAlg,\Tab)$. Note that, since the problem is local, we may shrink $P$ to an arbitrarily small open neighborhood of $q$, and consequently shrink $N$ to $I(P)$.

The main idea is to consider the bundle of ``anchored frames'' of $P$. More precisely, recall that $\rho:\CAlg\to TN$ is the anchor of the almost Lie algebroid $\CAlg$ and let us also write $\rho:\CAlg\oplus\Tab\to TN$ for the map that sends $(\alpha,T)\mapsto \rho(\alpha)$. We consider the following space of linear isomorphisms:
\begin{equation*}
\mathrm{Fr}(P):= \{\; \xi:T_pP\xrightarrow{\simeq} (\CAlg\oplus \Tab)_{I(p)} \;|\; p\in P \text{ and } \rho\circ\xi=dI|_{T_pP}   \;\}.
\end{equation*}
We denote the natural projection from $\mathrm{Fr}(P)$ to $P$ by $\pi$ and the composition $I\circ \pi$ also by $I$, 
\begin{center}
	\begin{tikzpicture}[description/.style={fill=white,inner sep=2pt},bij/.style={above,sloped,inner sep=.5pt}]	
	\matrix (m) [matrix of math nodes, row sep=1.3em, column sep=2.5em, 
	text height=1.5ex, text depth=0.25ex]
	{ 
		\mathrm{Fr}(P)   \\
		P  \\
		N. \\
	};
	\path[->,font=\scriptsize]
	(m-1-1) edge node[right] {$\pi$} (m-2-1)
	(m-2-1) edge node[right] {$I$} (m-3-1)
	(m-1-1) edge[bend right=40] node[left] {$I$} (m-3-1);
	\end{tikzpicture}
\end{center}
The bundle of ``anchored frames'' $\mathrm{Fr}(P)$ comes equipped with two tautological 1-forms, one with values in $\CAlg$ and one with values in $\Tab$:
\begin{equation*}
\begin{split}
& \overline{\Omega} \in \Omega^1(\mathrm{Fr}(P);I^*\CAlg),\hspace{0.5cm} \overline{\Omega}_\xi =  \xi^\CAlg\circ d\pi, \\
& \overline{\Pi} \in \Omega^1(\mathrm{Fr}(P);I^*\Tab),\hspace{0.5cm} \overline{\Pi}_\xi =  \xi^\Tab\circ d\pi.
\end{split}
\end{equation*}
Here $\xi^\CAlg$ and $\xi^\Tab$ denote the $\CAlg$ and $\Tab$ components of $\xi\in\mathrm{Fr}(P)$, respectively. On $\mathrm{Fr}(P)$, we have the following 2-form:
\begin{equation}
\label{eqn:structureequation2form}
d\overline{\Omega} + \frac{1}{2}[\overline{\Omega},\overline{\Omega}] - \overline{\Pi}\wedge \overline{\Omega} \in \Omega^2(\mathrm{Fr}(P);I^*\CAlg).
\end{equation}
The key observation is that a local solution to the realization problem is the same thing as a local section $\eta$ of $\pi:\mathrm{Fr}(P)\to P$ that pulls-back \eqref{eqn:structureequation2form} to zero. Indeed, if this is the case, then $\eta^*\overline{\Omega}\in\Omega^1(P;I^*\CAlg)$ satisfies the structure equation as well as the coframe condition, and is, hence, a solution. 

The main challenge is to construct such a local section $\eta$. The strategy taken in \cite{Kumpera1964} (and by Cartan) is to consider the exterior differential system on $\mathrm{Fr}(P)$ spanned by the components of the vector bundle-valued 2-form \eqref{eqn:structureequation2form}. Integral manifolds of dimension $r+p$ of this exterior differential system that project diffeomorphically to $P$ correspond to the desired local sections. One proves that if $\Tab$ is involutive, then the assumptions of the Cartan-K\"ahler theorem are satisfied, and, hence, such integral manifolds exist. The three integrability conditions in the definition of a Cartan algebroid play a crucial role in proving this. 

\subsection{Examples}

The most important source of examples of Cartan algebroids and realizations is Cartan's Second Fundamental Theorem, which is the subject of Section \ref{section:thesecondfundamentaltheorem}. We will see that any Lie pseudogroup gives rise to a Cartan algebroid and a realization (explicit examples are computed in Section \ref{section:cartanexamples}), and, more generally, Lie-Pfaffian groupoids that satisfy certain conditions give rise to Cartan algebroids and realizations. In this section, we discuss some more general examples and constructions. 

\begin{myexample} [Truncated Lie algebras]
	\label{example:truncatedliealgebras}
	In \cite{Singer1965}, Singer and Sternberg study transitive Lie algebra sheaves (which are, at least morally, the infinitesimal counterpart of transitive Lie pseudogroups) and show that these gives rise to an algebraic structure which they call truncated Lie algebras (Definition 4.1 in \cite{Singer1965}). Truncated Lie algebras are the same thing as Cartan algebroids over a point modulo gauge equivalence. 
\end{myexample}

\begin{myexample}[Abstract Atiyah sequences]
	A transitive Lie algebroid $A$ over a manifold $N$ is the same thing as a Cartan algebroid $(A,0)$ over $N$. Transitive Lie algebroids are also known as ``abstract Atiyah sequences'' for the following reason. 
	
	Given a principal $G$-bundle $\pi:P\to N$ (where $G$ is a Lie group acting from the left), one has an associated exact sequence of vector bundles over $N$ known as the ``Atiyah sequence of P'',
	\begin{equation*}
	0\to P[\mathfrak{g}] \to TP/G \xrightarrow{d\pi} TN \to 0,
	\end{equation*}
	where $\mathfrak{g}$ is the Lie algebra of $G$ and $P[\mathfrak{g}] = (P\times \mathfrak{g}) / G$. The middle term
	\begin{equation*}
	A(P):= TP/G
	\end{equation*}
	has the structure of a transitive Lie algebroid: the anchor is induced by $d\pi$, while the bracket comes from the Lie bracket of vector fields on $P$ and the identification $\Gamma(A(P)) = \mathfrak{X}(P)^G$ (see \cite{Mackenzie2005}, Section 3.2, for more details). The relevance of this sequence comes from the fact that connections on $P$ are the same thing as splittings of the sequence, while the curvature of a connection appears as the failure of the splitting to preserve the Lie brackets. The quotient map $TP\to TP/G$ induces a tautological form
	\begin{equation}
	\label{eqn:atiyausequencetautologicalform}
	\Omega\in\Omega^1(P;A(P)),
	\end{equation}
	with which $(P,\Omega)$ becomes a realization of the Cartan algebroid $(A(P),0)$. 
	
	In general, given a general transitive Lie algebroid $A$ over $N$, there is an exact sequence 
	\begin{equation*}
	0\to \Ker(\rho) \to A\xrightarrow{\rho} TN \to 0
	\end{equation*}
	called an ``abstract Atiyah sequence''. The question of whether there exists a principal bundle $P$ so that $A$ is isomorphic to $A(P)$ is equivalent to the integrability of $A$ as a Lie algebroid. Hence, the integrability of the Lie algebroid $A$ is also closely related to the existence of a realization of the Cartan algebroid $(A,0)$; the only difference is that a general realization $P$ might have an induced action of $\mathfrak{g}$, but this action may fail to integrate to an action of $G$. 
\end{myexample}

\begin{myexample}[Lie groups as pseudogroups]
	\label{example:liegroupsaspseudogroups}
	Changing a bit the point of view of the previous example, any Lie group $G$ can be realized as a pseudogroup in normal form by making it act freely and properly on a space $P$. To be more precise, assume that $\pi:P\to N$ is a principal $G$-bundle, then the left multiplication by elements in $G$ induces a pseudogroup $\Gamma_{G,P}$ on $P$. To see that this is a pseudogroup in normal form, we just use the Lie algebroid $A(P)$ from the previous example and the associated tautological form $\Omega$ in \eqref{eqn:atiyausequencetautologicalform}. It is not difficult to see that $\Gamma_{G,P}$ is characterized by the invariance of $\pi$ and $\Omega$. Note that up to Cartan equivalence of pseudogroups (see Section \ref{section:equivalenceofpseudogroups}), the choice of $P$ is not so important. If $Q$ is another principal $G$-bundle, then $\Gamma_{G,P}$ and $\Gamma_{G,Q}$ admit a common isomorphic prolongation, namely $\Gamma_{G,P\times Q}$ (along the canonical projections from $Q\times P$ to $P$ and $Q$, respectively). The simplest choice for $P$, which we already saw in Example \ref{example:liegroupsliealgebras}, would be $P=G$ with the left action of $G$. Here, $N$ is a point, $A(P)$ is the Lie algebra $\mathfrak{g}$ of $G$ and
	\begin{equation*}
	\Omega=\Omega_\mathrm{MC}\in\Omega^1(G;\mathfrak{g})	
	\end{equation*}
	is the Maurer-Cartan form. 
\end{myexample}

\begin{myexample}
	Here is a general construction of Cartan algebroids that underlies Cartan's proof of the Second Fundamental Theorem (to be discussed in Section \ref{section:cartanalgebroidrealizationofpfaffiangroupoid}). Start with a Lie algebroid $A$ over $N$ and a connection $\nabla:\mathfrak{X}(N)\times\Gamma(A)\to\Gamma(A)$. Define
	\begin{equation*}
	\CAlg=\CAlg(A) := TM\oplus A ,\hspace{1cm} \Tab:=\Hom(TM,A).
	\end{equation*}
	The bracket of $\CAlg$ is defined by
	\begin{equation*}
	[(X,\alpha),(Y,\beta)]_\nabla := ([X,Y],[\alpha,\beta]_\nabla+\nabla_X(\beta) - \nabla_Y(\alpha)),
	\end{equation*}
	which uses the $A$-torsion of $\nabla$,
	\begin{equation*}
	[\alpha,\beta]_\nabla = [\alpha,\beta] - \nabla_{\rho(\alpha)}\beta + \nabla_{\rho(\beta)}\alpha/
	\end{equation*}
	The anchor of $\CAlg$ is just the projection onto the first component. The vector bundle $\Tab$ becomes a bundle of Lie algebras when endowed with the Lie bracket
	\begin{equation*}
	[T,S] := T\circ\rho\circ S - S\circ\rho \circ T,
	\end{equation*}
	where $\rho$ is the anchor of $A$. It can be realized as a subbundle
	\begin{equation*}
	\Tab\hookrightarrow \Hom(\CAlg,\CAlg),\hspace{1cm} T\mapsto \hat{T},
	\end{equation*}
	by setting
	\begin{equation*}
	\hat{T}(\alpha,X) := (0,T(\rho(\alpha)-X)),\hspace{1cm} \forall\; \alpha\in\Gamma(\CAlg),\;X\in\mathfrak{X}(M).
	\end{equation*}

	\begin{myprop}
		The pair $(\CAlg(A),\Tab)$ in the above example is a Cartan algebroid. Up to gauge equivalence, it is independent of the choice of a connection $\nabla$. 
	\end{myprop}
	
	\begin{proof}
		First, given two connections $\nabla$ and $\nabla'$ as in the example, we would like to show that their induced almost Cartan algebroids are gauge equivalent. Indeed, the difference of the two connections induces a vector bundle map $(\nabla'-\nabla):A\to \Hom(TM,A)$, which, in turn, induces a vector bundle map from $\CAlg=TM\oplus A$ to $\Tab=\Hom(TM,A)$ by acting trivially on $TM$. This is the desired gauge equivalence (simple computation). 
		
		There are two ways to prove that $(\CAlg,\Tab)$ is a Cartan algebroid \textendash\, a direct proof and an indirect one. The direct proof is to show that $(\CAlg,\Tab)$ is a Cartan algebroid by verifying the axioms of Definition \ref{def:cartanalgebroid}. A straightforward computation shows that axiom 1 is satisfied. We are left with finding a vector bundle map $\bar{t}:\Lambda^2\CAlg\to\Tab$ and a connection $\bar{\nabla}:\Gamma(\CAlg)\times \Gamma(\Tab)\to \Gamma(\Tab)$ with which axioms 2 and 3 of are satisfied. The idea is straightforward (but the computations are a bit tedious): compute the left hand side of the equations in both axioms and try to decompose the resulting expressions so as to obtain a suitable $\bar{t}$ and $\bar{\nabla}$. We will write down explicit solutions, i.e. expressions for $\bar{t}$ and $\bar{\nabla}$ that are obtained in this way, and leave it as an exercise to check that these satisfy axioms 2 and 3. To write down $\bar{t}$, we first define the tensors $R:\Lambda^2TM\to \Hom(A,A)$ and $\bar{R}:\Lambda^2A\to \Hom(TM,A)$ by
		\begin{equation*}
		\begin{split}
		R_{X,Y}(\alpha)&:= \nabla_{[X,Y]}\alpha - \nabla_X\nabla_Y\alpha + \nabla_Y\nabla_X\alpha,\\
		\bar{R}_{\alpha,\beta}(X)&:= \nabla_X[\alpha,\beta] - [\nabla_X\alpha,\beta] - [\alpha,\nabla_X\beta] + \nabla_{\rho(\nabla_X\alpha)}\beta - \nabla_{\rho(\nabla_X\beta)}\alpha  \\ & \hspace{1cm} - \nabla_{[X,\rho(\alpha)]}\beta + \nabla_{[X,\rho(\beta)]}\alpha,
		\end{split}
		\end{equation*}
		Then, for all $X,Y,Z\in\mathfrak{X}(M)$ and $\alpha,\beta,\gamma\in\Gamma(A)$,
		\begin{equation*}
		\begin{split}
		& \bar{t}_{(X,\alpha),(Y,\beta)}(Z,\gamma):= (0,\bar{R}_{\alpha,\beta}(\rho(\gamma)-Z) + \frac{1}{2}R_{\rho(\gamma)-Z,\rho(\alpha)-X}(\beta) + \frac{1}{2}R_{\rho(\beta)-Y,\rho(\gamma)-Z}(\alpha) ).
		\end{split}
		\end{equation*} 
		To write down $\bar{\nabla}$, we choose a torsion-free connection on $TM$, which (by abuse of notation) we denote by $\nabla:\mathfrak{X}(M)\times\mathfrak{X}(M)\to\mathfrak{X}(M)$. Thus, $[X,Y]=\nabla_X Y - \nabla_Y X$ for all $X,Y\in\mathfrak{X}(M)$. 
		\begin{equation*}
		\begin{split}
		& (\bar{\nabla}_{(X,\alpha)}\hat{T})(Y,\beta):= [(X,\alpha),\hat{T}(Y,\beta)]_\nabla + \big(0,T\big(\rho (\nabla_{\rho(\alpha)-X}\beta)  - \nabla_{\rho(\alpha)}\rho(\beta) + \nabla_X Y \big)\big),
		\end{split}
		\end{equation*}
		for all $X,Y\in\mathfrak{X}(M)$, $\alpha,\beta \in\Gamma(A)$ and $T\in\Gamma(\Tab)$.
		
		Alternatively, as an indirect proof, we can construct a realization of $(\CAlg(A),\Tab)$ by noting that this example is a simple case of the construction described in Section \ref{section:cartanalgebroidrealizationofpfaffiangroupoid} (we recommend returning to this example after reading that section). Assuming that $A$ integrates to a Lie groupoid $\mathcal{G}$ (in fact, it suffices to have an integration to a local groupoid which always exists), $(\CAlg(A),\Tab)$ is precisely the almost Cartan algebroid constructed out the of Lie-Pfaffian groupoid $(J^1\mathcal{G},\omega)$, where $\omega\in\Omega^1(J^1\mathcal{G};t^*A)$ is the Cartan form. Note that a linear connection on $A$ is the same thing as a choice of a splitting of \eqref{eqn:secondfundamentaltheoremsplitting}. Finally, $(J^1\mathcal{G},\omega)$ admits an integral Cartan-Ehresmann connection, which is the same thing as a section of the affine bundle $\text{pr}:J^2\mathcal{G}\to J^1\mathcal{G}$, and the proposition follows from Theorem \ref{theorem:piintegralcartanehresmannconnection}. 
	\end{proof}
	
\end{myexample}

\begin{myexample}(Restrictions of Cartan Algebroids)
	\label{example:restrictioncartanalgebroid}
	A Cartan algebroid $(\CAlg,\Tab)$ over $N$ can be restricted to any submanifold $S\subset N$, giving rise to a Cartan algebroid $(\CAlg_S,\Tab_S)$ over $S$, where $\Tab_S:= \Tab|_S$ (the restriction of the vector bundle to $S$), $\CAlg_S := \{\; \alpha\in\CAlg\;|\; \rho(\alpha)\in TS \;\}$ and the bracket is uniquely determined by
	\begin{equation*}
	[\alpha|_S,\beta|_S] = [\alpha,\beta]|_S,
	\end{equation*}
	for all $\alpha,\beta\in \Gamma(\CAlg)$. A realization $(P,\Omega)$ of $(\CAlg,\Tab)$ induces a realization $(P_S,\Omega_S)$ of $(\CAlg_S,\Tab_S)$ by taking the restrictions $P_S := I^{-1}(S)$, $I_S:=I|_{P_S}$ and $\Omega_S:=\Omega|_S$. This restriction operation underlies Cartan's trick of restricting a realization to a complete transversal that will be discussed in Section \ref{section:firstfundametnaltheorem}. 
\end{myexample}

\subsection{An Alternative Approach to Cartan Algebroids: Cartan Pairs}
\label{section:cartanpairs}

We conclude this section with an alternative but equivalent point of view on Cartan algebroids and the realization problem that is more intuitive and which has the advantage that the formulas (namely \eqref{eqn:cartanalgebroidconditionjacobi} and \eqref{eqn:cartanalgebroidconditionactiong}) become substantially simpler. We will see that Cartan algebroids up to gauge equivalence are the same thing as  \textit{Cartan pairs} up to isomorphism, and that realizations of one induce realizations of the other. 

\subsubsection{Cartan Pairs}

Let $A$ be an almost Lie algebroid over $N$. The Jacobiator of $A$ is the tensor $\text{Jac}_A\in \Gamma(\Hom(\Lambda^3 A,A))$ that is defined at the level of sections by 
\begin{equation*}
\text{Jac}_A(\alpha,\beta,\gamma) := [[\alpha,\beta],\gamma] + [[\beta,\gamma],\alpha] + [[\gamma,\alpha],\beta], \;\;\;\;\; \forall\; \alpha,\beta,\gamma\in \Gamma(A). 
\end{equation*}
We say that a vector subbundle $\Tab\subset A$ is \textbf{involutive} if $[\Gamma(\Tab),\Gamma(\Tab)]\subset \Gamma(\Tab)$.  

\begin{mydef}
	\label{def:cartanpair}
	A \textbf{Cartan pair} over a manifold $N$ is a pair $(A,\Tab)$ consisting of a transitive almost Lie algebroid $(A,[\cdot,\cdot],\rho)$ over $N$ and an involutive vector subbundle $\Tab\subset \Ker\rho \subset  A$ such that
	\begin{equation}
	\label{eqn:weakjacobiator}
	\mathrm{Jac}_A \equiv 0\;(\mathrm{mod}\; \Tab)
	\end{equation}  
\end{mydef}

Condition \eqref{eqn:weakjacobiator} can be rephrased as saying that $\mathrm{Jac}_A$, applied on any three sections of $A$, must take values in $\Tab$. When $\Tab=0$, a Cartan pair is simply a transitive Lie algebroid. 

A Cartan pair $(A,\Tab)$ has an associated vector bundle map
\begin{equation}
\label{eqn:cartanpairtableaubundle}
\iota: \Tab \to \Hom(A/\Tab,A/\Tab),
\end{equation}
which is defined at the level of sections by
\begin{equation*}
\iota(T)(\text{pr}(\alpha)) := \text{pr}([T,\alpha]),\hspace{1cm} \forall \; T\in\Gamma(\Tab), \; \alpha\in \Gamma(A),
\end{equation*}
where $\text{pr}:A\to A/\Tab$ is the quotient map. Note that the map is well defined (i.e. the formula does not depend on the representative $\alpha$) because $\Tab$ is involutive, and the right hand side is indeed $C^\infty(N)$-linear in both the $\alpha$ and $T$ slots because $\Tab$ is killed by both $\rho$ and $\text{pr}$. Equipping $\Hom(A/\Tab,A/\Tab)$ with the commutator bracket, we have that:

\begin{mylemma}
	\label{lemma:cartanpairpreservesbracket}
	Let $(A,\Tab)$ be a Cartan pair. The map \eqref{eqn:cartanpairtableaubundle} preserves the brackets, i.e.
	\begin{equation*}
	\iota([S,T]) = [\iota(S),\iota(T)],\hspace{1cm} \forall \; S,T\in\Gamma(\Tab).
	\end{equation*}  
\end{mylemma}

\begin{proof}
	Applying \eqref{eqn:weakjacobiator} on $\alpha\in\Gamma(A),\;S,T\in\Gamma(\Tab)$, 
	\begin{equation*}
	\begin{split}
	0 &= \text{pr}([[S,T],\alpha] + [[T,\alpha],S] + [[\alpha,S],T] ) \\
	&= \iota([S,T])(\alpha) - [\iota(S),\iota(T)](\alpha) \qedhere
	\end{split}
	\end{equation*}
\end{proof}

\begin{mydef}
	\index{Cartan pair!standard}
	A Cartan pair $(A,\Tab)$ is said to be \textbf{standard} if the map \eqref{eqn:cartanpairtableaubundle} is injective.
\end{mydef}

\begin{mylemma}
	If a Cartan pair $(A,\Tab)$ is standard, then $\Tab$ is a bundle of Lie algebras. 
\end{mylemma}

\begin{proof}
	Since $\iota$ preserves the bracket, then $\iota(\Tab)\subset \Hom(A/\Tab,A/\Tab)$ is closed under the bracket. Now, since the bracket of $\iota(\Tab)$ satisifies the Jacobi identity, it follows that $\iota\circ\text{Jac}_\Tab=0$, and so $\text{Jac}_\Tab=0$ by injectivity of $\iota$.
\end{proof}

As we will see, the following notion of an isomorphism of Cartan pairs will play the role that gauge equivalence has for Cartan algebroids. 

\begin{mydef}
	\label{def:isomorphismcartanpairs}
	Two Cartan pairs $(A,\Tab)$ and $(A',\Tab')$ over $N$ are isomorphic if there exists a vector bundle isomorphism $\psi:A\to A'$ such that $\psi(\Tab)=\Tab'$, $\rho\circ\psi=\rho'$ and
	\begin{equation}
	\label{eqn:isomorphismcartanpairs}
	\psi([\alpha,\beta]) \equiv [\psi(\alpha),\psi(\beta)] \;\;\;(\mathrm{mod}\; \Tab),\hspace{1cm} \forall\;\alpha,\beta\in\Gamma(A).
	\end{equation}
\end{mydef}

It is straightforward to check that, as a consequence of \eqref{eqn:isomorphismcartanpairs}, an isomorphism $\psi$ between two Cartan pairs commutes with the two maps $\iota$ and $\iota'$ (defined in \eqref{eqn:cartanpairtableaubundle}) induced by each of the Cartan pairs. 

\subsubsection{Cartan Algebroids vs. Cartan Pairs}

Up to isomorphism on the one side and gauge equivalence on the other,
Cartan pairs and Cartan algebroids are the same thing. We start by constructing  a Cartan pair out of a Cartan algebroid $(\CAlg,\Tab)$ over $N$. The construction is analogous to the construction of a non-abelian extension of a Lie algebroid (\cite{Mackenzie1987}, Chapter 4, Section 3), and depends on a choice of $t:\Lambda^2\CAlg\to \Tab$ and $\nabla:\Gamma(\CAlg)\times\Gamma(\Tab)\to\Gamma(\Tab)$ as in Definition \ref{def:cartanalgebroid} of a Cartan algebroid. Let $(t,\nabla)$ be such a pair. We set
\begin{equation*}
A:=\CAlg\oplus \Tab
\end{equation*}
and equip it with an anchor induced by the anchor of $\CAlg$,
\begin{equation}
\label{eqn:cartanpairprolongationanchor}
\rho:A\to TN,\hspace{1cm} \rho(\alpha,S) := \rho(\alpha),
\end{equation}
and a bracket 
\begin{equation}
\label{eqn:cartanpairprolongationbracket}
[\cdot,\cdot]:\Gamma(A)\times \Gamma(A)\to \Gamma(A)
\end{equation}
defined by
\begin{equation*}
\label{eqn:cartanpairprolongationbracketformula}
[(\alpha,S),(\beta,T)]:= ([\alpha,\beta] + S(\beta) - T(\alpha) ,\;-t_{\alpha,\beta} + \nabla_\alpha T - \nabla_{\beta} S + [S,T]),
\end{equation*}
for all $\alpha,\beta\in\Gamma(\CAlg),\; S,T\in\Gamma(\Tab)$. 

\begin{myprop}
	\label{prop:cartanalgebroidtocartanpair}
	Let $(\CAlg,\Tab)$ be a Cartan algebroid over $N$. The induced pair $(\CAlg\oplus\Tab,\Tab)$, for a fixed choice of $(t,\nabla)$, is a standard Cartan pair. Moreover, up to isomoprhism, the resulting Cartan pair is independent of the choice of $(t,\nabla)$. 
\end{myprop}

\begin{proof}
	Fix a choice of $t$ and $\nabla$. Clearly $\Tab\subset \Ker\rho$, and it is straightforward to verify that $A$ is an almost Lie algebroid. So, we are only left with checkgin that, for any $\alpha,\beta,\gamma\in\Gamma(\CAlg)$ and $S,T,U\in\Gamma(\Tab)$, 
	\begin{equation*}
	\begin{split}
	& \text{pr}\circ\text{Jac}_A\big((\alpha,0),(\beta,0),(\gamma,0)\big) \\ & \hspace{1cm} = [[\alpha,\beta],{\gamma}] + [[\beta,{\gamma}],\alpha] + [[{\gamma},\alpha],\beta] - t_{\alpha,\beta}({\gamma}) - t_{\beta,{\gamma}}(\alpha) - t_{{\gamma},\alpha}(\beta) =0 , \\
	& \text{pr}\circ\text{Jac}_A\big((\alpha,0),(\beta,0),(0,T)\big) \\ & \hspace{1cm} = -T([\alpha,\beta]) + [T(\alpha),\beta] + [\alpha,T(\beta)] + \nabla_\beta(T)(\alpha) - \nabla_\alpha(T)(\beta)=0, \\
	& \text{pr}\circ\text{Jac}_A\big((\alpha,0),(0,S),(0,T)\big)  =  [S,T](\alpha) - S(T(\alpha)) + T(S(\alpha))  = 0, \\
	& \text{pr}\circ\text{Jac}_A\big((0,S),(0,T),(0,U)\big)  =  0 .
	\end{split}
	\end{equation*}  
	Given another choice $(t',\nabla')$, the identity map $\Id:\mathcal{C}\oplus\Tab\to \mathcal{C}\oplus\Tab$ gives an isomorphism between the Cartan pair obtained using $(t,\nabla)$ and that obtained using $(t',\nabla')$. 
\end{proof}

\begin{myremark}
	If we relax the notion of a Cartan algebroid by requiring for there to be a map $\Tab\to \Hom(\CAlg,\CAlg)$ rather than an inclusion $\Tab\subset \Hom(\CAlg,\CAlg)$, then we will obtain Cartan pairs that are not necessarily standard. In the case of Cartan pairs, it is much more natural to impose the \textit{standard} property separately rather than add it to the initial definition. In the case of a Cartan algebroid, we chose to impose the stronger property in order to be consistent with Cartan's local picture. 
\end{myremark}

In the other direction, a standard Cartan pair $(A,\Tab)$ induces a Cartan algebroid $(A/\Tab,\Tab)$. The construction depends on a choice of a splitting of the short exact sequence
\begin{equation}
\label{eqn:cartanpairsplitting}
\begin{tikzpicture}[description/.style={fill=white,inner sep=2pt},bij/.style={above,sloped,inner sep=.5pt}]	
\matrix (m) [matrix of math nodes, row sep=2.5em, column sep=1.6em, 
text height=2.5ex, text depth=0.25ex]
{ 
	0 & \Tab & A & A/\Tab & 0. \\
};
\path[->,font=\scriptsize]
(m-1-1) edge node[auto] {} (m-1-2)
(m-1-2) edge node[auto] {} (m-1-3)
(m-1-3) edge node[below] {$ \text{pr} $} (m-1-4)
(m-1-4) edge node[auto] {} (m-1-5)
(m-1-4) edge[bend right=50] node[above] {$ \xi $} (m-1-3);
\end{tikzpicture}
\end{equation}
We equip the vector bundle $A/\Tab$ with the bracket
\begin{equation*}
[\cdot,\cdot]:\Gamma(A/\Tab)\times \Gamma(A/\Tab)\to \Gamma(A/\Tab),\hspace{1cm} [\alpha,\beta]:= \text{pr}([\xi(\alpha),\xi(\beta)]),
\end{equation*}
and the anchor
\begin{equation*}
\rho:A/\Tab\to TN,\hspace{1cm} \rho(\alpha) := \rho(\xi(\alpha)).
\end{equation*}
Note that the bracket depends on the choice of $\xi$, but the anchor does not. Since the Cartan pair is standard, we have an inclusion $\Tab \hookrightarrow \Hom(A/\Tab,A/\Tab)$. 

\begin{myprop}
	\label{prop:cartanpairtocartanalgebroid}
	Let $(A,\Tab)$ be a standard Cartan pair and let $\xi:A/\Tab\to A$ be a a choice of a splitting \eqref{eqn:cartanpairsplitting}. The pair $(A/\Tab,\Tab)$ equipped with the structure defined above is a Cartan algebroid. Moreover, up to gauge equivalence, the resulting Cartan algebroid is independent of the choice of $\xi$.
\end{myprop}

\begin{proof}
	We must show the existence of a vector bundle map $t:\Lambda^2(A/\Tab)\to \Tab$ and an $A/\Tab$-connection $\nabla:\Gamma(A/\Tab)\times \Gamma(\Tab)\to\Gamma(\Tab)$ as in Definition \ref{def:cartanalgebroid}. Let us denote by $\eta:A\to \Tab$ the left splitting induced by the right splitting $\xi$. We define $t$ and $\nabla$ by
	\begin{equation*}
	t_{\alpha,\beta}(\gamma):= - \eta([\xi(\alpha),\xi(\beta)]),\hspace{1cm} \nabla_\alpha(T):= \eta([\xi(\alpha),T]).
	\end{equation*}
	By the definition of a Cartan pair, $\text{pr}\circ \text{Jac}_A=0$. Applying this equality to the triples  $(\xi(\alpha),\xi(\beta),\xi(\gamma))$ and $(\xi(\alpha),\xi(\beta),T)$, where $\alpha,\beta,\gamma\in\Gamma(A/\Tab)$ and $T\in\Gamma(\Tab)$, a straightforward computation shows that \eqref{eqn:cartanalgebroidconditionjacobi} and \eqref{eqn:cartanalgebroidconditionactiong} are satisfied, and Lemma \ref{lemma:cartanpairpreservesbracket} implies the first axiom of Definition \ref{def:cartanalgebroid}. For the final assertion, given two splittings $\xi$ and $\xi'$, the difference $(\xi'-\xi):A/\Tab\to\Tab$ defines the desired gauge equivalence. 	
\end{proof}

The above constructions define the following correspondence:

\begin{mytheorem}
	\label{theorem:cartanpairscartanalgebroids}
	There is a $1-1$ correspondence (given by the constructions above) between Cartan algebroids over $N$ up to gauge equivalence (Definition \ref{def:gaugeequivalence}) and Cartan pairs over $N$ up to isomorphism (Definition \ref{def:isomorphismcartanpairs}). 
\end{mytheorem}

\begin{proof}
	Given a Cartan algebroid $(\mathcal{C},\Tab)$ and a gauge equivalence $\eta:\mathcal{C}\to \Tab$, the map $\psi:(\mathcal{C}\oplus\Tab)\to (\mathcal{C}^\eta\oplus\Tab),\; (\alpha,T)\mapsto (\alpha,T-\eta(\alpha))$, defines an isomorphism between the Cartan pair induced by $(\mathcal{C},\Tab)$ and the Cartan pair induced by the gauge equivalent one $(\mathcal{C}^\eta,\Tab)$. The remaining details are straightforward to verify. 
\end{proof}

\subsubsection{Realizations of Cartan Pairs}

The notion of a realization takes a more elegant form in the Cartan pair picture. The existence of realizations of a Cartan algebroid is equivalent to the existence of a reliazation of its induced Cartan pair, and vice versa. 

\begin{mydef}
	\label{def:realizationcartanpair}
	\index{realization!of a Cartan pair}
	A \textbf{realization} of a Cartan pair $(A,\Tab)$ over $N$ is a pair $(P,\Omega)$ consisting of a surjective submersion $I:P\to N$ and an anchored 1-form $\Omega\in\Omega^1(P;I^*A)$, such that 
	\begin{equation}
	\label{eqn:cartanpairrealizationstructureequation}
	d\Omega+\frac{1}{2}[\Omega,\Omega] \equiv 0\;(\mathrm{mod}\; \Tab)
	\end{equation}
	and such that $\Omega$ is pointwise an isomorphism. 
\end{mydef}

The proof of the following proposition is straightforward:

\begin{myprop}
	Given a realization $(P,\Omega)$ of a Cartan algebroid $(\CAlg,\Tab)$ with a fixed choice of $\Pi$ as in Definition \ref{def:realization}, the pair $(P,(\Omega,\Pi))$ is a realization of the induced standard Cartan pair  $(\CAlg\oplus\Tab,\Tab)$. Conversely, given a realization $(P,\Omega)$ of a standard Cartan pair $(A,\Tab)$, the pair $(P,\emph{pr}\circ\Omega)$ is a realization of the induced Cartan algebroid $(A/\Tab,\Tab)$. 
\end{myprop}

\begin{myremark}
	In the case of a Cartan pair $(A,0)$ (i.e. $A$ is simply a transitive Lie algebroid), one can obtain a solution to the realization problem by integrating the Lie algebroid to a Lie groupoid (when the Lie algebroid is integrable), in which case the Maurer-Cartan form on any source fiber of the Lie groupoid defines a solution. In \cite{Crainic2003-1}, the authors present a method for integrating Lie algebroids to Lie groupoids by constructing a Lie groupoid out of the space of so called $A$-paths of the Lie algebroid. A large part of this construction does not rely on the fact that the Lie algebroid one starts with satisfies the Jacobi identity. The point of view of Cartan pairs \textendash\ ``transitive Lie algebroids that satisfy the Jacobi identity modulo $\Tab$'' \textendash\ suggests a new method for tackling the realization problem: imitating the construction in \cite{Crainic2003-1} and pinpointing the precise role of the Jacobi identity along the way. In \cite{Yudilevich2016-2} (Chapter 7, and see also \cite{Yudilevich2016}), yet another method for solving the realization problem in the case of a transitive Lie algebroid is introduced, one which identifies the precise role of the Jacobi identity. Using the point of view of Cartan pairs, one may also attempt to use this method to tackle Cartan's realization problem.
\end{myremark}

\section{Cartan's Second Fundamental Theorem}
\label{section:thesecondfundamentaltheorem}

Cartan's Second Fundamental Theorem states that:

\begin{mytheorem} (the second fundamental theorem)
	\label{theorem:secondfundamentaltheorem}
	\index{fundamental theorem!second}
	Any Lie pseudogroup is Cartan equivalent to a pseudogroup in normal form (Definition \ref{def:normalform}).
\end{mytheorem}

This theorem shows that, up to Cartan equivalence, the study of Lie pseudogroups is the same as the study of pseudogroups in normal form, and hence to the study of Cartan algebroids and their realizations. In this section, after recalling the definitions of a Lie pseudogroup (Definition \ref{def:liepseudogroup}) and of Cartan equivalence of pseudogroups (Definition \ref{def:isomorphicprolongation}), we present a modern proof of the Second Fundamental Theorem. This proof is the result of our endeavor to understand Cartan's constructions (in \cite{Cartan1904,Cartan1937-1}) conceptually and in a global, coordinate-free fashion. In our proof, we use the language of jet groupoids and algebroids and, more abstractly, the language of Lie-Pfaffian groupoids and algebroids. These are recalled in Appendices \ref{section:jetgroupoids} and \ref{section:Liepfaffaingroupoidandalgebroid}. The Lie-Pfaffian groupoid framework isolates the essential properties of a Lie pseudogroup and, consequently, proofs become substantially simpler and more transparent. For the reader that is not familiar with these objects, we recommend reading Appendices \ref{section:jetgroupoids} and \ref{section:Liepfaffaingroupoidandalgebroid} in preparation for this section. 

In the course of our work, we found Cartan's examples of the Second Fundamental Theorem to be a useful guide to understanding the general theory. In Section \ref{section:cartanexamples}, we cite two such examples and run them through the machinery of the modern proof to ``rediscover'' Cartan's formulas. 
	
\subsection{Lie Pseudogroups and Cartan Equivalence}
\label{section:liepseudogroups}

Intuitively, a Lie pseudogroup is ``a pseudogroup that is defined by a system of partial differential equations''. The language of jet groupoids and algebroids allows us to make this definition precise. 

\subsubsection{Jet Groupoids and Algebroids}

To fix notation, let us review the main ingredients in the framework of jet groupoids and algebroids. For more details, see Appendix \ref{section:jetgroupoids}. With any manifold $M$, we associate the tower of jet groupoids
\begin{equation}
\label{eqn:towerofjetgroupoids}
... \xrightarrow{\pi} J^3M \xrightarrow{\pi} J^2M  \xrightarrow{\pi} J^1M \xrightarrow{\pi} J^0M,
\end{equation}
where $J^kM\rightrightarrows M$, the $k$-th jet groupoid of $M$, is the Lie groupoid whose space of arrows consists of all $k$-jets $j^k_x\phi$ of locally defined diffeomorphisms $\phi\in\LDiff(M)$ of $M$, and the projections $\pi:J^kM\to J^{k-1}M,\; j^k_x\phi\mapsto j^{k-1}_x\phi$, are Lie groupoid morphisms and surjective submersions. The Lie algebroid of $J^kM$, the $k$-th jet algebroid of $M$, is denoted by $A^kM$, and the induced projections by $l:A^kM\to A^{k-1}M$. Thus, at the infinitesimal level we have the tower a jet algebroids
\begin{equation}
\label{eqn:towerofjetalgebroids}
... \xrightarrow{\pi} A^3M \xrightarrow{l} A^2M  \xrightarrow{l} A^1M \xrightarrow{l} A^0M.
\end{equation}
The $k$-th jet groupoid $J^kM$ acts linearly on the $k-1$-th jet algebroid $A^{k-1}M$ by conjugation, giving rise to the adjoint representation (see \eqref{eqn:jetgroupoidadjointrepresentation}). The kernel 
\begin{equation*}
\Tab^kM=\Ker\; l\subset A^kM
\end{equation*}
of each projection in \eqref{eqn:towerofjetalgebroids} is called the $k$-th symbol space of $M$. Its elements are canonically identified with vector-valued homogeneous polynomials of degree $k$ on $M$ via the canonical isomorphism $\Tab^kM\cong S^kT^*M\otimes TM$. 

For our purposes, the most important piece of structure of a jet groupoid is its Cartan form, the tautological multiplicative 1-form 
\begin{equation*}
\omega\in\Omega^1(J^kM;t^*A^{k-1}M)
\end{equation*}
with values in the adjoint representation defined by the Formula \eqref{eqn:cartanformmultaiplicativepde}. The pair $(J^kM,\omega)$, i.e. the $k$-th jet groupoid equipped with the Cartan form, has the structure of a Lie-Pfaffian groupoid (see Definition \ref{def:pfaffiangroupoid}). At the infinitesimal level, the Cartan form on $J^kM$ induces (via \eqref{eqn:linearizingcartanform}) a connection-like operator 
\begin{equation*}
D:\mathfrak{X}(M)\times \Gamma(A^{k}M)\to \Gamma(A^{k-1}M)
\end{equation*}
on $A^kM$ called the Spencer operator. The pair $(A^kM,D)$ has the structure of a Lie-Pfaffian algebroid (Definition \ref{def:liepfaffianalgebroid}). 

\subsubsection{Lie Pseudogroups}

With any pseudogroup $\Gamma$ on $M$, we associate the tower
\begin{equation}
\label{eqn:towerofjetgroupoidspseudogroup}
... \xrightarrow{\pi} J^3\Gamma \xrightarrow{\pi} J^2\Gamma  \xrightarrow{\pi} J^1\Gamma \xrightarrow{\pi} J^0\Gamma,
\end{equation}
a subsequence of \eqref{eqn:towerofjetgroupoids}, where the \textbf{$k$-th jet groupoid} of $\Gamma$
\begin{equation*}
J^k\Gamma := \{\; j^k_x\phi\;|\; \phi\in \Gamma,\; x\in\Dom(\phi) \;\} \subset J^kM
\end{equation*}
is a subgroupoid of the $k$-th jet groupoid of $M$, and the restrictions of the projections from \eqref{eqn:towerofjetgroupoids}, which we also denote by $\pi:J^{k+1}\Gamma\to J^k\Gamma$, are surjective groupoid morphisms. In general, the sequence may fail to be smooth in the sense that the $J^k\Gamma$'s may fail to be submanifolds of the $J^kM$'s (if they are, then they are automatically Lie subgroupoids), and the projections may fail to be submersions. If $J^k\Gamma$ does turn out to be a Lie subgroupoid for some $k$, then it has an associated Lie subalgebroid 
\begin{equation*}
A^k\Gamma:= A(J^k\Gamma)\subset A^kM
\end{equation*}
of the $k$-th jet algebroid of $M$. In this case, we may define the \textbf{$k$-th symbol space} of $\Gamma$ to be
\begin{equation*}
\Tab^k\Gamma:= \Tab^kM\cap A^k\Gamma.
\end{equation*}
In general, $\Tab^k\Gamma$ is not a vector bundle (it may fail to be of constant rank) but a discrete vector bundle (see Section \ref{section:tableaubundles}), and it is a vector bundle if and only if it is of constant rank. 

\begin{mydef}
	\label{def:liepseudogroup}
	\index{Lie pseudogroup}
	A \textbf{Lie pseudogroup of (at least) order $k>0$} on a manifold $M$ is a pseudogroup $\Gamma$ on $M$ satisfying:
	\begin{enumerate}
		\item For any $\phi\in\LDiff(M)$, if $j^k_x\phi\in J^k\Gamma$ for all $x\in\Dom(\phi)$, then $\phi\in \Gamma$. 
		\item 
		\begin{enumerate}
			\item $J^k\Gamma\subset J^kM$ is a Lie subgroupoid,
			\item $J^{k-1}\Gamma\subset J^{k-1}M$ is a Lie subgroupoid,
			\item $\pi:J^k\Gamma\to J^{k-1}\Gamma$ is a submersion (hence $\Tab^k\Gamma$ is of constant rank), 
			\item $(\Tab^k\Gamma)^{(1)}$ is of constant rank. 
		\end{enumerate}
	\end{enumerate}
	A Lie pseudogroup is of \textbf{finite type} if $(\Tab^k\Gamma)^{(l)}=0$ for some $l>0$, and otherwise it is of \textbf{infinite type}. 
\end{mydef}

In this definition, the $k$-th jet groupoid $J^k\Gamma\subset J^kM$ should be interpreted as a system of PDEs and axiom 1 should be interpreted as the condition that $\Gamma$ consists of its full set of local solutions. Thus, one may study Lie pseudogroups by studying this special class of systems of PDEs. This was Lie's  original approach. Axiom 2 is a set of regularity conditions that allow us to study this system of PDEs geometrically. In particular, axioms 2a and 2b ensure that $J^k\Gamma$ and $J^{k-1}\Gamma$ have associated Lie algebroids, which we denote by $A^k\Gamma$ and $A^{k-1}\Gamma$, respectively, and axiom 2c implies that the projection
\begin{equation*}
l:A^k\Gamma\to A^{k-1}\Gamma
\end{equation*}
is surjective (and, hence, the symbol space $\Tab^k\Gamma$ is of constant rank).

\begin{myremark}
	While axiom 1 is standard, the regularity conditions that are imposed in axiom 2 vary in the literature (e.g. compare axioms 1 and 2 in Section 3 of \cite{Guillemin1966}, Definition IV.1 in \cite{Kuranishi1961} and Definition 3.1 in \cite{Olver2009}). In our definition, we impose sufficient conditions so as to allow us to detach $J^k\Gamma$, ``the defining system of PDEs'', from its ambient jet groupoid $J^kM$, and to handle it abstractly as a Lie-Pfaffian groupoid, as we will explain. These conditions, however, are sufficient but not necessary. Another possible definition would be to simply replace axiom 2 by the condition that $J^k\Gamma$ have the structure of a Lie-Pfaffian groupoid. We chose the current form of the definition in order to keep the conditions as explicit as possible.
\end{myremark}

In Section \ref{section:cartanexamples}, we will revisit Cartan's examples of Lie pseudogroups that we saw in the introduction (one of finite type and one of infinite type), and we will explicitly compute their jet groupoids and algebroids, their associated Cartan forms, Spencer operators, etc. As was also mentioned in the introduction, Lie pseudogroups arise in differential geometry as the local symmetries of geometric structures. Here is a general class of examples to keep in mind:

\begin{myexample}
	The pseudogroup $\Gamma$ of local automorphisms of any integrable $G$-structure is a Lie pseudogroup of order 1. Pseudogroups of local symmetries of foliations, symplectic structures, complex structures and integral affine structures are just a few of the examples that arise in this way (see \cite{Sternberg1983,Crainic2015} for more on $G$-structures). The assumption of integrability ensures that $\Gamma$ is ``large enough''. In fact, such pseudogroups are always transitive since these structures are locally homogeneous (they have a local normal form). The 1st jet groupoid $J^1\Gamma$ is in this case canonically isomorphic to the gauge groupoid of the $G$-structure (recall that any principal $G$-bundle $P\to M$ gives rise to a gauge groupoid $\mathcal{G}\text{auge}(P) \rightrightarrows M$, with $\mathcal{G}\text{auge}(P) = P\times P / G$). We also note that the restriction of the Cartan form on $J^1M$ to $J^1\Gamma$ is precisely the lift of the tautological form of the $G$-structure. 
\end{myexample}

\subsubsection{Lie Pseudogroups as Lie-Pfaffian Groupoids}
\label{section:liepseudogroupsasliepfaffiangroupoids}

In this section we prove that a Lie pseudogroup induces a Lie-Pfaffian groupoid, which, in turn, encodes it as its set of local holonomic bisections. In Section \ref{section:proofofthesecondfundamentaltheorem}, we use this abstract point of view in our proof of Cartan's Second Fundamental Theorem. The language of Lie-Pfaffian groupoids has the advantage of making the proof more tractable, and in many ways, more conceptual. In Appendix \ref{section:Liepfaffaingroupoidandalgebroid}, we have collected all the necessary background material on Lie-Pfaffian groupoids and algebroids. 

In our definition of a Lie pseudogroup (Definition \ref{def:liepseudogroup}) we have imposed certain regularity conditions. These were put into place precisely to ensure that the Cartan form of the ambient jet groupoid restricts nicely to the ``defining system of PDEs'': 

\begin{myprop}
	\label{prop:cartanformliepseudogroup}
	Let $\Gamma$ be a Lie pseudogroup on $M$ of order $k$. The adjoint representation $A^{k-1}M$ of $J^kM$ restricts to a representation $A^{k-1}\Gamma$ of $J^k\Gamma$ (which we call the \textbf{adjoint representation of $J^k\Gamma$}), and the Cartan form $\omega\in\Omega^1(J^kM;t^*A^{k-1}M)$ on $J^kM$ restricts to
	\begin{equation*}
	\omega = \omega|_{J^k\Gamma}\in\Omega^1(J^k\Gamma;t^*A^{k-1}\Gamma),
	\end{equation*}
	a multiplicative and pointwise surjective 1-form on $J^k\Gamma$ with values in the adjoint representation (which we call the \textbf{Cartan form of $J^k\Gamma$}).
\end{myprop}

\begin{proof}
	The first assertion follows directly from the defining formula \eqref{eqn:jetgroupoidadjointrepresentation} of the adjoint representation. Next, denoting the kernel of $\omega$ on $J^k\Gamma$ by $C_\omega=\Ker\omega\subset TJ^k\Gamma$, we construct an Ehresmann connection on $s:J^k\Gamma\to M$ by choosing a splitting of $d\pi:TJ^k\Gamma\to \pi^*TJ^{k-1}\Gamma$ and composing it at each point $j^k_x\phi\in J^k\Gamma$ with $(d(j^{k-1}\phi))_x$. This induces a decomposition $TJ^k\Gamma= H \oplus \Ker ds$, where $\omega$ kills the horizontal component $H$ (by the definition of $\omega$) and maps the second component surjectively onto $t^*A^{k-1}\Gamma$ by axiom 2c of a Lie pseudogroup. Finally, since $\omega$ is the restriction of a multiplicative form, it is multiplicative. 
\end{proof}

We this proposition, it is now simple to prove that: 

\begin{myprop}
	\label{prop:liepseudogroupasliepfaffiangroupoid}
	Let $\Gamma$ be a Lie pseudogroup on $M$ of order $k$. The pair $(J^k\Gamma,\omega)$ is a standard (Definition \ref{def:standard}) Lie-Pfaffian groupoid. Furthermore, there is a bijection 
	\begin{equation*}
	\Gamma \to \LBis(J^k\Gamma,\omega),\hspace{0.5cm} \phi\mapsto j^k\phi,
	\end{equation*}
	identifying $\Gamma$ with the generalized pseudogorup of local holonomic bisections $\LBis(J^k\Gamma,\omega)$. 
\end{myprop}

\begin{proof}
Having the previous proposition at hand, we are only left with checking axiom 2 of Definition \ref{def:pfaffiangroupoid}. This axiom is verified as in Example \ref{example:liepfaffiangroupoidjetgroupoids}, where we showed that a jet groupoid of a manifold, equipped with its Cartan form, is a Lie-Pfaffian groupoid. The symbol map is injective (and hence the Lie-Pfaffian groupoid is standard), because the symbol map of the ambient Lie-Pfaffian groupoid $J^kM$ (which is simply the inclusion \eqref{eqn:symbolspacetableau}) is injective. The final assertion is a consequence of Proposition \ref{prop:detectingholonomicsections} together with axiom 1 of the definition of a Lie pseudogroup. 
\end{proof}

The Lie-Pfaffian groupoid of a Lie pseudogroup satisfies the following property which will be crucial in our proof of Cartan's Second Fundamental Theorem:

\begin{mylemma}
	\label{lemma:liepseudogrouphasintegralcartanehresmannconnection}
	Let $\Gamma$ be a Lie pseudogroup on $M$ of order $k$. The induced Lie-Pfaffian groupoid $(J^k\Gamma,\omega)$ admits an integral Cartan-Ehresmann connection (Definition \ref{def:cartanehresmannconnection}).
\end{mylemma}

\begin{proof}
	The key idea is that the classical prolongation $P_\omega(J^k\Gamma)$ is the intersection of two affine subbundles of an affine bundle, and, in general, the intersection of two affine subbundles is again an affine bundle if and only if the intersection is non-empty in each fiber and if the intersection of the modeling vector bundles is of constant rank (see Proposition 1.1.6 in \cite{Yudilevich2016-2}). Indeed, here, $P_\omega(J^k\Gamma)$ is the intersection of $J^1(J^k\Gamma)\to J^k\Gamma$ (the first jets of sections of the source map of $J^k\Gamma$) and the restriction of $\pi:J^{k+1}M\to J^kM$ to $J^k\Gamma\subset J^kM$, both viewed as affine subbundles of the restriction of $\pi:J^1(J^kM)\to J^kM$ to $J^k\Gamma$. 
	
	Now, this intersection is an affine bundle because the intersection of the modeling vector bundles is precisely $(\Tab^k\Gamma)^{(1)}$, which is assumed to be of constant rank, and each fiber over $j^k_x\phi\in J^k\Gamma$ contains at least one point $j^{k+1}_x\phi$, where $\phi\in\Gamma$ is some representative of $j^k_x\phi\in J^k\Gamma$ (thus, the main point is that, by definition, the PDE $J^k\Gamma$ contains a solution through each point). Finally, since $P_\omega(J^k\Gamma)\to J^k\Gamma$ has sections (any affine bundle does), then integral Cartan-Ehresmann connections exist. 
\end{proof}

\subsubsection{Cartan Equivalence of Pseudogroups}
\label{section:equivalenceofpseudogroups}

Haefliger, in his work on the transverse structure of foliations (\cite{Haefliger1958}), recast the notion of a pseudogroup in the framework of Lie groupoids. He observed that pseudogroups are the same thing as \textit{effective \'etale Lie groupoids}. Recall that a Lie groupoid $\mathcal{G}\rightrightarrows M$ is called \textbf{\'etale} \index{Lie groupoid!effective \'etale} if its source map $s:\mathcal{G}\to M$ (and hence its target map $t:\mathcal{G}\to M$) is a local diffeomorphism, i.e.  each arrow $g \in \mathcal{G}$ has an open neighborhood $U$ such that $s|_U$ is a diffeomorphism onto its image. An \'etale groupoid is called \textbf{effective} if for any pair of local bisections $b$ and $b'$ with a common domain, $t\circ b=t\circ b'$ implies $b=b'$. Haefliger's correspondence is as follows: given a pseudogroup $\Gamma$ on $M$, one constructs the groupoid 
\begin{equation*}
\mathcal{G}\text{erm}(\Gamma) \rightrightarrows M
\end{equation*}
whose space of arrows consists of all germs of elements of $\Gamma$. We denote an arrow $\text{germ}_x\phi$, i.e. the germ at $x\in \Dom(\phi)$ of $\phi\in \Gamma$, and the structure maps are:
\begin{gather*}
s(\text{germ}_x\phi) = x,\hspace{1cm} t(\text{germ}_x\phi) = \phi(x), \hspace{1cm} 1_x = \text{germ}_x\Id, \\
\text{germ}_{\phi(x)}\phi'\cdot\text{germ}_x\phi = \text{germ}_x(\phi'\circ\phi),\hspace{1cm} (\text{germ}_x\phi)^{-1} = \text{germ}_{\phi(x)}\phi^{-1}.
\end{gather*}
Every element $\phi\in\Gamma$ gives rise to a local bisection $b_\phi$ of $\mathcal{G}\text{erm}(\Gamma)$ defined by $b_\phi(x) = \text{germ}_x\phi$ for all $x\in\Dom(\phi)$, and the smooth structure of $\mathcal{G}\text{erm}(\Gamma)$ (with a possibly non-Hausdorff nor second countable topology) is uniquely determined by the requirement that each such local bisection is a diffeomorphism onto its image. With this structure, $\mathcal{G}\text{erm}(\Gamma) \rightrightarrows M$ becomes an effective \'etale Lie groupoid. In the reverse direction, any effective \'etale Lie groupoid $\mathcal{G}\rightrightarrows M$ induces the pseudogroup 
\begin{equation*}
\Gamma(\mathcal{G}) := \{ \; \phi_b=t\circ b\;|\; b \text{ local bisection of } \mathcal{G} \; \} \subset \LDiff(M).
\end{equation*}
To summarize:

\begin{myprop}
	\label{prop:pseudogorupsetalegroupoids}
	Let $M$ be a manifold. There is a 1-1 correspondence between
	\begin{equation*}
	\{ \; \text{pseudogroups $\Gamma$ on $M$} \; \} \hspace{0.5cm} \longleftrightarrow \hspace{0.5cm} \{\; \text{effective \'etale Lie groupoids $\mathcal{G} \rightrightarrows M$} \;\}
	\end{equation*}
	given by $\Gamma\mapsto \mathcal{G}\mathrm{erm}(\Gamma)$ in the right direction and $\mathcal{G}\mapsto \Gamma(\mathcal{G})$ in the left.
\end{myprop}

We omit the proof, which is a straightforward exercise. For more on this correspondence, see also \cite{Haefliger1958} (Chapter I, Section 6) and \cite{Moerdijk2003} (Example 5.23). 

The correct notion of an ``isomorphism'' between pseudogroups is not at all obvious. Consider the pseudogroup on $\mathbb{R}$ generated by
\begin{equation*}
\{ \; \phi:\mathbb{R}\to \mathbb{R},\; x\mapsto x+a\;|\;   a\in\mathbb{R} \; \},
\end{equation*}
and the pseudogroup on $\mathbb{R}^2$ generated by
\begin{equation*}
\{ \; \phi:\mathbb{R}^2\to \mathbb{R}^2,\; (x,y)\mapsto (x+a,y)\;|\;   a\in\mathbb{R} \; \}.
\end{equation*}
Intuitively, these two pseudogroups should be ``isomorphic'', because there is a bijection between their generators that preserves the group-like structures. Thus, one would expect that the notion of ``isomorphism'' should be flexible enough to allow us to identify these two pseudogroups, and, in general, to identify pseudogroups that act on manifolds of varying dimension. In \cite{Cartan1937-1} (p. 1336), Cartan writes: \textit{``The notion of an `abstract group' does not lend itself to the theory of infinite Lie pseudogroups with the same level of purity as it does in the finite case, and it is for this reason that it has been proven difficult to find a simple analytic characterization for the notion of isomorphism. It is remarkable that M. Vessiot and I were simultaneously led to the same definition of an isomorphism of two Lie pseudogroups.'' } 

Cartan's notion of ``isomorphism'', which we call \textit{Cartan equivalence}, is best formulated in terms of Haefliger's point of view on pseudogroups. In the following definition, $\mathcal{G} \ltimes P\rightrightarrows P$ denotes the action groupoid associated with an action of a Lie groupoid $\mathcal{G}\rightrightarrows M$ on a surjective submersion $\pi:P\to M$.

\begin{mydef}
	\label{def:isomorphicprolongation}
	\index{isomorphic prolongation!pseudogroup}
	A pseudogroup $\widetilde{\Gamma}$ on $P$ is an \textbf{isomorphic prolongation} of a pseudogroup $\Gamma$ on $M$ along a surjective submersion $\pi:P\to M$ be if there exist an action of $\mathcal{G}\mathrm{erm}(\Gamma)\rightrightarrows M$ on $\pi:P\to M$ and an isomorphism of Lie groupoids
	\begin{equation*}
	\mathcal{G}\mathrm{erm}(\widetilde{\Gamma}) \cong \mathcal{G}\mathrm{erm}(\Gamma)\ltimes P.
	\end{equation*}
	A pseudogroup $\Gamma$ on $M$ is \textbf{Cartan equivalent} to a pseudogroup $\Gamma'$ on $M'$ if they admit a common isomorphic prolongation. In this case, we write $\Gamma \sim \Gamma'$. 
\end{mydef}

The notion of Cartan equivalence is illustrated in the following diagram:
\begin{center}
	\begin{tikzpicture}[every node/.style={align=center,text depth=0ex,text height=2ex,text width=4em}]
	\matrix (m) [matrix of math nodes,column sep=-1.5em,row sep=1.5em, row 1/.style={nodes={align=left}}]
	{ 
		& \hspace{-0.3cm} \widetilde{\Gamma} \acts \; P & \\
		\hspace{-0.7cm} \Gamma \acts \; M & & \hspace{0.35cm} M' \;\, \text{\reflectbox{$\acts$}} \; \Gamma'\\
	};
	\path[->,font=\scriptsize]
	(m-1-2) edge node[auto] {} (m-2-1)
	(m-1-2) edge node[auto] {} (m-2-3);
	\end{tikzpicture}
\end{center}

\begin{myremark}
	Cartan makes a distinction between two types of prolongations: \textit{holo\'edrique} and \textit{m\'eri\'edrique} prolongation. The former corresponds to the above notion of an \textit{isomorphic prolongation}, while the latter is like a ``covering map'' of pseudogroups (see \cite{Cartan1937-1}, p. 1336).
\end{myremark}

\begin{myprop}
	Cartan equivalence defines an equivalence relation on the set of all pseudogroups. 	
\end{myprop}

\begin{proof}
	Symmetry and reflexivity are clear. For transitivity, let $\Gamma,\Gamma',\Gamma''$ be pseudogroups on $M,M',M''$, respectively, such that $\Gamma\sim \Gamma'$ and $\Gamma'\sim \Gamma''$, i.e. $\Gamma$ and $\Gamma'$ have a common isomorphic prolongation on $P$, and $\Gamma'$ and $\Gamma''$ on $P'$. Using this data, we may construct an isomorphic prolongation of $\Gamma$ and $\Gamma'$ on the fibered product $Q:=P\times_{M'}P'$. 
	\begin{center}
		\begin{tikzpicture}[every node/.style={align=center,text depth=0ex,text height=2ex,text width=4em}]
		\matrix (m) [matrix of math nodes,column sep=-1.5em,row sep=1.5em]
		{ 
			& & P\times_{M'} P' \\
			& P & & P' \\
			M & & M' & & M''\\
		};
		\path[->,font=\scriptsize]
		(m-1-3) edge node[auto] {} (m-2-2)
		(m-1-3) edge node[auto] {} (m-2-4)
		(m-2-2) edge node[auto] {} (m-3-1)
		(m-2-2) edge node[auto] {} (m-3-3)
		(m-2-4) edge node[auto] {} (m-3-3)
		(m-2-4) edge node[auto] {} (m-3-5);
		\end{tikzpicture}
	\end{center}
	Define an action of $\mathcal{G}\text{erm}(\Gamma)$ on $Q$ as follows: let $\phi\in \Gamma$, $x\in\Dom(\phi)$ and $(p,p')\in Q$ such that $p$ projects to $x$. The isomorphism $\mathcal{G}\text{erm}(\Gamma)\ltimes P\xrightarrow{\simeq} \mathcal{G}\text{erm}(\Gamma')\ltimes P$ maps the pair $(\text{germ}_x(\phi),p)$ to a pair $(\text{germ}_x(\phi'),p)$. Set $\text{germ}_x(\phi)\cdot (p,p') := (\text{germ}_x(\phi')\cdot p,\text{germ}_x(\phi')\cdot p')$. Similarly, we can define an action of $\mathcal{G}\text{erm}(\Gamma'')$ on $Q$, and it is straightforward to verify that $\Gamma(\mathcal{G}\text{erm}(\Gamma)\ltimes Q)$ is an isomorphic prolongation of $\Gamma$ (and similarly for $\Gamma''$), and that $\mathcal{G}\text{erm}(\Gamma'')\ltimes Q \cong \mathcal{G}\text{erm}(\Gamma)\ltimes Q$. 
\end{proof}

\subsubsection{From Pseudogroups to Generalized Pseudogroups}

All the basic notions and constructions that we presented for pseudogroups extend rather straightforwardly to generalized pseudogroups (see Appendix \ref{section:generalizedpseudogroups}). For instance, given a generalized pseudogroup $\Gamma$ on $\mathcal{G}\rightrightarrows M$, we may construct $J^k\Gamma$, the groupoid of all $k$-jets of elements of $\Gamma$, and obtain the tower of jet groupoids as in \eqref{eqn:towerofjetgroupoidspseudogroup}. Of course, under suitable regularity conditions, we also obtain a tower of jet algebroids on the infinitesimal side. 

Also the notions of isomorphic prolongation and Cartan equivalence extend naturally. We may construct the germ groupoid $\mathcal{G}\text{erm}(\Gamma)\rightrightarrows M$ of $\Gamma$. It is an \'etale groupoid, but it may fail to be effective. However, this suffices for the following definition, which allows us to compare generalized pseudogroups that act on different Lie groupoids, and, in particular, to compare pseudogroups (viewed as generalized pseudogroups, see \ref{example:pseudogroupasgeneralizedpseudogroup}) with generalized pseudogroups:

\begin{mydef}
	\label{def:isomorphicprolongationgeneralized}
	\index{isomorphic prolongation!pseudogroup}
	A generalized pseudogroup $\widetilde{\Gamma}$ on $\widetilde{\mathcal{G}}\rightrightarrows P$ is an \textbf{isomorphic prolongation} of a generalized pseudogroup $\Gamma$ on $\mathcal{G}\rightrightarrows M$ along a surjective submersion $\pi:P\to M$ be if there exist an action of $\mathcal{G}\mathrm{erm}(\Gamma)\rightrightarrows M$ on $\pi:P\to M$ and an isomorphism of Lie groupoids
	\begin{equation*}
	\mathcal{G}\mathrm{erm}(\widetilde{\Gamma}) \cong \mathcal{G}\mathrm{erm}(\Gamma)\ltimes P.
	\end{equation*}
	A generalized pseudogroup $\Gamma$ on $\mathcal{G}\rightrightarrows M$ is \textbf{Cartan equivalent} to a generalized pseudogroup $\Gamma'$ on $\mathcal{G}'\rightrightarrows M'$ if they admit a common isomorphic prolongation. In this case, we write $\Gamma \sim \Gamma'$. 
\end{mydef}

\subsection{Proof of the Second Fundamental Theorem}
\label{section:proofofthesecondfundamentaltheorem}

We now turn to the proof of Theorem \ref{theorem:secondfundamentaltheorem}. We have seen that a Lie pseudogroup $\Gamma$ of order $k$ gives rise to a Lie-Pfaffian groupoid $(J^k\Gamma,\omega)$, which, in turn, encodes $\Gamma$ as its generalized pseudogroup of local solutions (Proposition \ref{prop:liepseudogroupasliepfaffiangroupoid}). A Lie-Pfaffian groupoid that arises in this way satisfies the following two properties:
\begin{itemize}
	\item[(a)] it is standard (Proposition \ref{prop:liepseudogroupasliepfaffiangroupoid}),
	\item[(b)] it admits an integral Cartan-Ehresmann connection (Lemma \ref{lemma:liepseudogrouphasintegralcartanehresmannconnection}).
\end{itemize}
The core of the proof of Theorem \ref{theorem:secondfundamentaltheorem} consists of two ``ingredients'', the following general facts about Lie-Pfaffian groupoids:
\begin{enumerate}
	\item The Canonical Prolongation: Any Lie-Pfaffian groupoid $(\mathcal{G},\omega)$ comes with a pseudogroup on its total space $\mathcal{G}$ which we call the \textit{canonical prolongation}. If $(\mathcal{G},\omega)$ satisfies property (a), then the canonical prolongation is characterized as the pseudogroup on $\mathcal{G}$ that consists of those elements in $\LDiff(\mathcal{G})$ that preserve the target map $t:\mathcal{G}\to M$ and $\omega\in\Omega^1(\mathcal{G};t^*E)$ (Theorem \ref{theorem:firstfundamentaltheorem}).
	\item Constructing a Realization: Any Lie-Pfaffian groupoid $(\mathcal{G},\omega)$ satisfying properties (a) and (b) gives rise to a Cartan algebroid and a realization whose induced pseudogroup of local symmetries is precisely the canonical prolongation (Theorem \ref{theorem:piintegralcartanehresmannconnection}).
\end{enumerate}
In Section \ref{section:secondfundamentaltheoremproof} we prove that Theorem \ref{theorem:secondfundamentaltheorem} follows from these two facts.

\subsubsection{Ingredient 1: The Canonical Prolongation}
\label{section:firstfundametnaltheorem}

Let $\mathcal{G}$ be a Lie groupoid over $M$ with source and target maps $s$ and $t$. Any local bisection $b\in\LBis(\mathcal{G})$, say with
\begin{equation*}
\phi_b := t\circ b: U\to V,\hspace{1cm} U,V\subset M,
\end{equation*}
induces a locally defined diffeomorphism $\psi_b\in\LDiff(\mathcal{G})$ given by:
\begin{equation}
\label{eqn:canonicalprolongation}
\psi_b:s^{-1}(U)\to s^{-1}(V),\hspace{1cm} g\mapsto g\cdot b(s(g))^{-1}. 
\end{equation}
We call $\psi_b$ the \textbf{prolongation} of $b$. From its definition, it follows that $s\circ\psi_b = \phi_b\circ s$, i.e. $\psi_b\in\LDiff(\mathcal{G})$ covers $\phi_b\in\LDiff(M)$ along the source map. 

Now, let $(\mathcal{G},\omega)$ be a Lie-Pfaffian groupoid over $M$, and recall that $\LBis(\mathcal{G},\omega)$ denotes the set local holonomic bisections of $(\mathcal{G},\omega)$, i.e. all local bisections $b$ of $\mathcal{G}$ that satisfy $b^*\omega=0$. We call the pseudogroup on $\mathcal{G}$ generated by all prolongations of elements of $\LBis(\mathcal{G},\omega)$, denoted by
\begin{equation}
\label{eqn:canonicalprolongationpseudogroup}
\mathrm{prol}(\LBis(\mathcal{G},\omega)) := \langle \, \{\; \psi_b\;|\; b\in\LBis(\mathcal{G},\omega) \;\} \, \rangle \subset \LDiff(\mathcal{G}),
\end{equation}
the \textbf{canonical prolongation} of $\LBis(\mathcal{G},\omega)$. 

\begin{myprop}
	\label{prop:prolongationiscartanequivalent}
	Let $(\mathcal{G},\omega)$ be a Lie-Pfaffian groupoid. The pseudogroup $\mathrm{prol}(\LBis(\mathcal{G},\omega))$ is an isomorphic prolongation of the generalized pseudogroup $\LBis(\mathcal{G},\omega)$ along the source map $s:\mathcal{G}\to M$, and hence the two are Cartan equivalent. 
\end{myprop}

\begin{proof}
	Given the action of the germ groupoid $\mathcal{G}\text{erm}(\LBis(\mathcal{G},\omega))\rightrightarrows M$ on $s:\mathcal{G}\to M$ by $\text{germ}_xb \cdot g =  g \cdot b(s(g))^{-1}$, it is straightforward to check that 
	\begin{equation*}
	\mathcal{G}\mathrm{erm}(\LBis(\mathcal{G},\omega))\ltimes \mathcal{G} \xrightarrow{\simeq} \mathcal{G}\mathrm{erm}(\mathrm{prol}(\LBis(\mathcal{G},\omega))),\hspace{1cm} (\mathrm{germ}_xb,g)\mapsto \mathrm{germ}_g\psi_b,
	\end{equation*}
	is an isomorphism of Lie groupoids. 
\end{proof}

\begin{myexample}
	\label{example:liepseudogroupprolongation}(Lie pseudogroups)
	Let $\Gamma$ be a Lie pseudogroup on $M$ of order $k$, and let $(J^k\Gamma,\omega)$ be its induced Lie-Pfaffian groupoid. Writing out \eqref{eqn:canonicalprolongation} in this case, any $\phi\in \Gamma$, say $\phi:U\to V$ with $U,V\subset M$, induces $\phi^k\in\LDiff(J^k\Gamma)$ given by
	\begin{equation}
	\label{eqn:kprolongationpseudogroup}
	\phi^k: s^{-1}(U) \to s^{-1}(V),\hspace{1cm} j^k_x\phi' \mapsto j^k_x\phi'\cdot (j^k_x\phi)^{-1}.
	\end{equation}
	This generates the pseudogroup 
	\begin{equation*}
	\Gamma^k:= \langle \{ \phi^k\;|\; \phi\in\Gamma \} \rangle \subset \LDiff(J^k\Gamma).
	\end{equation*}
	In the literature, $\Gamma^k$ is often called the \textbf{$k$-th prolongation} of $\Gamma$. 
\end{myexample}

In what he calls \textit{The Fundamental Theorem} (pp. 1337-1339 in \cite{Cartan1937-1}), Cartan shows that the $k$-th prolongation of a Lie pseudogroup of order $k$ (see example above) is characterized as the pseudogroup preserving a collection of functions and 1-forms. This theorem generalizes to the setting of Lie-Pfaffian groupoids as follows

\begin{mytheorem}
	\label{theorem:firstfundamentaltheorem}
	Let $(\mathcal{G},\omega)$ be a Lie-Pfaffian groupoid and assume that it is standard (Definition \ref{def:standard}). Then
	\begin{equation*}
	\mathrm{prol}(\LBis(\mathcal{G},\omega)) = \{\; \phi\in\LDiff(\mathcal{G})\;|\; \phi^*t=t,\;\phi^*\omega=\omega \;\}.
	\end{equation*}
\end{mytheorem}

\begin{proof}
	We first prove the right inclusion, i.e. that $\psi_b^*t=t$ and $\psi_b^*\omega=\omega$ for any $b\in\LBis(\mathcal{G},\omega)$. The first equality is clear from the defining formula \eqref{eqn:canonicalprolongation} of $\psi_b$. The second equality relies on the multiplicativity of $\omega$. Let $g\in \mathcal{G}$ and $X\in T_g\mathcal{G}$, then
	\begin{equation*}
	\begin{split}
	(\psi_b^*\omega)_g(X) &= (m^*\omega)_{(g,b(s(g))^{-1})}(X,di\circ db\circ ds(X)) \\
	&= \omega_g(X) + g\cdot(i^*\omega)_{b(s(g))}(db\circ ds(X)) \\
	&= \omega_g(X) - \psi_b(g)\cdot (\cancel{b^*\omega})_{s(g)}(ds(X)),
	\end{split}
	\end{equation*}
	where $m, i$ and $s$ are the multiplication, inverse and source maps, respectively. Here, the definition of $\psi_b$ is used in the first equality, the multiplicativity property \eqref{eqn:multiplicativeform} in the second and the following general identity for multiplicative forms in the third:
	\begin{equation}
	-(i^*\omega)_g = g^{-1}\cdot \omega_g, \hspace{1cm} \forall \; g\in\mathcal{G}.
	\label{eqn:multiplicativeformsinverse}
	\end{equation}
	The latter is obtained by applying \eqref{eqn:multiplicativeform} on a pair $(X,di(X))$ with $X\in T_g\mathcal{G}$. 
	
	For the left inclusion, we have to prove that if $\phi\in\LDiff(\mathcal{G})$ satisfies $\phi^*t=t$ and $\phi^*\omega=\omega$, then, locally, $\phi=\psi_b$ for some $b\in\LBis(\mathcal{G},\omega)$. Locally here mean that, for every $g\in\Dom(\phi)$, there exists $b\in\LBis(\mathcal{G},\omega)$ with $s(g)\in\Dom(b)$ and an open neighborhood $U\subset \Dom(\phi)$ of $g$ with $s(U)\subset \Dom(b)$ such that $\phi|_U=\psi_b|_U$. Spelling out $\psi_b$, this equality becomes 
	\begin{equation}
	\label{eqn:proofcanonicalprolongation1}
	b(s(h)) = \phi(h)^{-1}\cdot h,\hspace{1cm} \forall\; h\in U. 
	\end{equation}
	To prove this, we consider the map
	\begin{equation*}
	H=H_\phi:\Dom(\phi)\to \mathcal{G},\hspace{1cm} h\mapsto \phi(h)^{-1}\cdot h,
	\end{equation*}
	which is well defined because $\phi^*t=t$, and choose a local section $\eta$ of the source map $s:\mathcal{G}\to M$ with $g\in\Image(\eta)$ and $\Image(\eta)\subset \Dom(\phi)$. We set $b:=H\circ \eta$.  The left inclusion now follows from the following three claims that we will prove: 1) $H^*\omega=0$ (and hence $b^*\omega=0$); 2) if we sufficiently shrink the domain of $b$ around $s(g)$, then $b$ is a bisection; and 3) $H$ is locally constant along the $s$-fibers (and hence \eqref{eqn:proofcanonicalprolongation1} holds for a small enough neighborhood $U$ of $g$).
	
	1) Let $X\in T\mathcal{G}|_{\Dom(\phi)}$, then
	\begin{equation*}
	\begin{split}
	(H^*\omega)_g(X) &= (m^*\omega)_{(\phi(g)^{-1},g)} (di\circ d\phi(X),X) \\
	&= (i^*\omega)_{\phi(g)}(d\phi(X)) + \phi(g)^{-1}\cdot\omega_g(X) \\
	&= -\phi(g)^{-1}\cdot (\phi^*\omega)_g(X) + \phi(g)^{-1}\cdot\omega_g(X) \\
	&= 0,
	\end{split}
	\end{equation*}
	where the definition of $H$ was used in the first equality, the multiplicativity property in the second, \eqref{eqn:multiplicativeformsinverse} in the third and the assumption that $\phi^*\omega=\omega$ in the fourth.
	
	2) From the definition of $H$, it follows that $b=H\circ\eta$ is a local section of the source map $s$. We are left to show that $t\circ b$ is a diffeomorphism, where we are allowed to shrink the domain of $b$ to an arbitrarily small neighborhood of $s(g)$.  By the inverse function theorem, it is sufficient to check that $(d(t\circ b))_{s(g)}$ is a linear isomoprhism, or, by dimension count, that it is injective. Since $b$ is a section of $s$, it is injective, so it is enough to check that $\Ker dt\cap \Image\; db=\{0\}$. Now, since $H^*\omega=0$, then $\Image\;db\subset \Ker\omega$, and since $\Ker\omega\cap\Ker dt = \Ker\omega\cap\Ker ds$ (by the definition of a Lie-Pfaffian groupoid), then $\Ker dt\cap \Image\;db = \Ker ds\cap \Image\; db$. But $\Ker ds\cap \Image\; db=\{0\}$, because $b$ is a section of $s$, and so we are done. 
	
	3) The map $H$ is locally constant along the $s$-fibers if and only if $dH(X)=0$ for all $X\in \Ker ds|_{\Dom(\phi)}$. Let $X\in \Ker ds|_{\Dom(\phi)}$. Note that since $H^*\omega=0$ and $s\circ H=s$, then $dH(X)\in \Ker\omega\cap\Ker ds$. Now, since we are assuming that $(\mathcal{G},\omega)$ is standard (i.e. its symbol map is injective), the vanishing of $dH(X)$ follows from:
	\begin{equation*}
	\begin{split}
	\partial(dR_{g^{-1}}(dH(X)))(Y) &= \delta\omega(dH(X), db\circ(d\phi_b)^{-1})(Y))) \\ 
	&= \delta\omega(dH(X), dH((d(\eta\circ\phi_b^{-1}))(Y))) \\ 
	&= (H^*\delta\omega)(X,(d(\eta\circ \phi_b^{-1}))(Y))\\
	&=0,
	\end{split}
	\end{equation*}
	for all $Y\in T_{t(g)}M$, where $\phi_b=t\circ b$. Here, in the first equality Lemma \ref{lemma:symbolmapdifferentialcartanform} was used, while the last equality follows from the fact that 
	\begin{equation}
	\label{eqn:killingomegakillingdifferentialomega}
	H^*\omega=0 \hspace{1cm} \Rightarrow \hspace{1cm} H^*\delta\omega=0.
	\end{equation}
	The latter is a consequence of the fact that $\delta\omega=d_\nabla\omega|_{\Ker\omega}$ together with the basic fact that the pullback operation commutes with the differential, which, in the case of vector bundle-valued forms, means that $H^*d_\nabla\omega = d_{H^*\nabla} H^*\omega$, where $\nabla$ is a choice of a connection on the coefficients of $\omega$ and $H^*\nabla$ is the pull-back connection.
\end{proof}

\begin{myremark}
	This proof is a nice illustration of the advantage of working with Lie-Pfaffian groupoids. In the main application, when the Lie-Pfaffian groupoid is $(J^k\Gamma,\omega)$ with $\Gamma$ a Lie pseudogroup, this theorem can also be proven by induction on the order of the jets (this is done e.g. in Theorem 4.1 of \cite{Guillemin1966} in the case of transitive pseudogroups).  However, the above proof is substantially simpler, avoiding the need to work directly with jets and using solely the Cartan form and a small number of its essential properties. 
\end{myremark}

\subsubsection{Aside: Restricting to a Transversal}
\label{section:restrictingtotransversal}

Actually, Theorem \ref{theorem:firstfundamentaltheorem} can be improved, in some sense, by restricting the canonical prolongation to a \textit{complete transversal} of the underlying Lie-Pfaffian groupoid, thus obtaining a smaller prolongation than the canonical prolongation constructed above. This trick is employed by Cartan in examples (and even mentioned in his general proof of the Second Fundamental Theorem) in order to reduce the dimension of the realization he constructs out of a given Lie pseudogroup. We will see two examples of this in Section \ref{section:cartanexamples}. Let us explain this ``trick'' in detail. 

Let $(\mathcal{G},\omega)$ be a Lie-Pfaffian groupoid over $M$ and let $N\subset M$ be any submanifold. Because the orbits of the canonical prolongation $\mathrm{prol}(\LBis(\mathcal{G},\omega))$ are contained in the $t$-fibers of $\mathcal{G}$, then we may restrict each of its elements to the submanifold 
\begin{equation*}
\mathcal{G}_N:=t^{-1}(N)\subset \mathcal{G},
\end{equation*}
thus obtaining a pseudogroup on $\mathcal{G}_N$ which we denote by
\begin{equation*}
\mathrm{prol}(\LBis(\mathcal{G},\omega))|_N:= \{ \; \phi|_{\mathcal{G}_N} \; | \; \phi \in \mathrm{prol}(\LBis(\mathcal{G},\omega)) \; \} \subset \LDiff(\mathcal{G}_N).
\end{equation*}
We now ask whether Proposition \ref{prop:prolongationiscartanequivalent} and Theorem \ref{theorem:firstfundamentaltheorem} continue to hold for this pseudogroup. This is the case when the submanifold $N$ is ``nice'' in the following sense:

\begin{mydef}
	Let $\mathcal{G}$ be a Lie groupoid over $M$ with Lie algebroid $A$. A \textbf{transversal} to $\mathcal{G}$ is a submanifold $N\subset M$ that intersects the orbits of $\mathcal{G}$ transversely, i.e.
	\begin{equation*}
	TN + \rho(A)|_N = TM|_N.
	\end{equation*}
	A transversal is \textbf{complete} if it intersects each orbit at least once. 
\end{mydef}

An important consequence of this condition is that if $N$ is a complete transversal of a Lie groupoid $\mathcal{G}$, then the restriction of the source map $s_N:=s|_{\mathcal{G}_N}:\mathcal{G}_N\to M$ is a surjective submersion (the restriction of the target map $t_N:=t|_{\mathcal{G}_N}:\mathcal{G}_N\to M$ is a surjective submersion for any submanifold $N$). The proof of the following proposition is now a straightforward adaptation of the proof of Proposition \ref{prop:prolongationiscartanequivalent}, and will be omitted:

\begin{myprop}
	\label{prop:prolongationiscartanequivalenttransversal}
	Let $(\mathcal{G},\omega)$ be a Lie-Pfaffian groupoid over $M$ and let $N\subset M$ be a complete transversal of $\mathcal{G}$. The pseudogroup $\mathrm{prol}(\LBis(\mathcal{G},\omega))|_N\subset \LDiff(\mathcal{G}_N)$ is an isomorphic prolongation of the generalized pseudogroup $\LBis(\mathcal{G},\omega)$ along $s_N:\mathcal{G}_N\to M$, and hence the two are Cartan equivalent. 
\end{myprop}

Writing $\omega_N:=\omega|_{\mathcal{G}_N}$, we also have the following version of Theorem \ref{theorem:firstfundamentaltheorem}:

\begin{mytheorem}
\label{theorem:firstfundamentaltheoremtransversal}
Let $(\mathcal{G},\omega)$ be a Lie-Pfaffian groupoid over $M$ and assume that it is standard, and let $N\subset M$ be a complete transversal of $\mathcal{G}$. Then
\begin{equation*}
\mathrm{prol}(\LBis(\mathcal{G},\omega))|_N = \{\; \phi\in\LDiff(\mathcal{G}_N)\;|\; \phi^*t_N=t_N,\;\phi^*\omega_N=\omega_N \;\}.
\end{equation*}
\end{mytheorem}

\begin{proof}
	Observe that in Theorem \ref{theorem:firstfundamentaltheorem} (and its proof), $\mathcal{G}$ plays a double role: 1) it is the Lie groupoid underlying the Lie-Pfaffian groupoid $(\mathcal{G},\omega)$ whose generalized pseudogroup of local solutions  $\LBis(\mathcal{G},\omega)$ we consider, and 2) it is the space on which the classical prolongation $\mathrm{prol}(\LBis(\mathcal{G},\omega))$ acts. In the theorem we are currently proving, $\mathcal{G}$ remains unchanged in the first role and is replaced by $\mathcal{G}_N$ in the second role. Modulo this replacement (and  replacing $s$, $t$ and $\omega$ by their restrictions $s_N$, $t_N$ and $\omega_N$), the proof of Theorem \ref{theorem:firstfundamentaltheorem} can be copied verbatim. As explained above, the role of the condition of complete transversality is to ensure that $s_N$ is a surjective submersion.
\end{proof}

\begin{myremark}
	The double role played by $\mathcal{G}$, as explained in the above proof, suggests a more conceptual framework for this theorem, namely that of a Lie-Pfaffian groupoid $(\mathcal{G},\omega)$ acting on a ``Pfaffian bundle'' $(P,\theta)$, i.e. a surjective submersion $\mu:P\to N$ equipped with a vector-bundle valued 1-form $\theta$ (satisfying certain conditions). In our case, the ``Pfaffian bundle'' is $t_N:\mathcal{G}_N\to N$ equipped with $\omega_N$. This data is all that is needed in order to make sense of the ``prolongation'' of the generalized pseudogroup $\LBis(\mathcal{G},\omega)$ to a pseudogroup on $P$, and to prove that the latter is characterized as the local symmetries of the ``Pfaffian bundle'', generalizing Theorem \ref{theorem:firstfundamentaltheoremtransversal}. This direction is currently being investigated in \cite{Cattafi2020}.
\end{myremark}

\begin{myremark}
Ideally, to obtain the ``smallest'' prolongation, one would like to choose a transversal $N$ that crosses each orbit precisely once. In this case, $N$ can be regarded as the orbit space of $\Gamma$, but one which is obtained by choosing a slice rather than by taking a quotient. In fact, this is what Cartan does in local coordinates (as we will see in the examples of Section \ref{section:cartanexamples}). In the global setting, however, this is only possible if the orbits are ``nice enough''. For example, if the Lie pseudogroup is transitive, one takes $N$ to be a point in $M$. In \cite{Guillemin1966}, the authors study Cartan's structure theory in the case of transitive Lie pseudogroups, and, in particular, they prove Theorem \ref{theorem:firstfundamentaltheoremtransversal} in the case where $N$ is a point. 
\end{myremark}		
		
\subsubsection{Ingredient 2: Constructing a Realization}
\label{section:cartanalgebroidrealizationofpfaffiangroupoid}

The second main ingredient of the proof of the Second Fundamental Theorem is to construct a realization of a Cartan algebroid out of the data of a Lie-Pfaffian groupoid. We will present two different -- but equivalent -- constructions. In the current section, we simply provide a recipe, which, in local coordinates, coincides with Cartan's construction of a realization and depends on certain choices that may seem somewhat arbitrary. In Theorem \ref{theorem:piintegralcartanehresmannconnection}, the main result of this section, we will show that the existence of an auxiliary form $\Pi$ as in Definition \ref{def:realization} of a realizatoin is equivalent to the existence of an integral Cartan-Ehresmann connection (Definition \ref{def:cartanehresmannconnection}) on the Lie-Pfaffian groupoid we start with, thus giving a geometric interpretation to the notion of a realization. In the next section, we show how Cartan's structure equations can be understood more conceptually as the pullback of the canonical Maurer-Cartan equation on the prolongation of a Lie-Pfaffian groupoid, a notion that was introduced and studied in \cite{Salazar2013}. We will explain the relation between the two construction and how, in a sense, the latter construction clarifies the ``arbitrariness'' of the former. 

\subsubsection*{Constructing an Almost Cartan Algebroid}

The first step is to construct an almost Cartan algebroid out of a given standard Lie-Pfaffian groupoid. The construction depends only on the induced Lie-Pfaffian algebroid, and hence it suffices for this step to assume that we are given a standard Lie-Pfaffian algebroid $(A,D)$ over $M$ (Definition \ref{def:liepfaffianalgebroid}). This, as we recall, consists of two Lie algebroids $A$ and $E$ over $M$, a surjective Lie algebroid map $\l:A\to E$ and an $l$-connection
\begin{equation*}
D:\mathfrak{X}(M)\times \Gamma(A)\to \Gamma(E).
\end{equation*}
The construction will depend on a splitting $\xi$ of the short exact sequence of vector bundles
\begin{equation}
\begin{tikzpicture}[description/.style={fill=white,inner sep=2pt},bij/.style={above,sloped,inner sep=.5pt}]	
\matrix (m) [matrix of math nodes, row sep=2.5em, column sep=1.6em, 
text height=2.5ex, text depth=0.25ex]
{ 
	0 & \Tab & A & E & 0, \\
};
\path[->,font=\scriptsize]
(m-1-1) edge node[auto] {} (m-1-2)
(m-1-2) edge node[auto] {} (m-1-3)
(m-1-3) edge node[above,yshift=-0.05cm] {$ l $} (m-1-4)
(m-1-4) edge node[auto] {} (m-1-5)
(m-1-4) edge[bend right=50] node[above] {$ \xi $} (m-1-3);
\end{tikzpicture}
\label{eqn:secondfundamentaltheoremsplitting}
\end{equation}
where $\Tab=\Tab(A)$ is the symbol space of $(A,D)$. 

From this data, we construct an almost Cartan algebroid $(\CAlg,\Tab)$ as follows. We set
\begin{equation}
\label{eqn:secondfundamentaltheoremc}
\CAlg:=TM\oplus E.
\end{equation}
The bracket of $\CAlg$ depends on the splitting \eqref{eqn:secondfundamentaltheoremsplitting}. Such a splitting induces a linear connection on $E$ defined by
\begin{equation}
\label{eqn:cconnection}
\nabla^\xi:\mathfrak{X}(M)\times \Gamma(E)\to \Gamma(E),\hspace{1cm} \nabla^\xi_X(\alpha):= D_X(\xi\circ\alpha).
\end{equation}
We consider its torsion		
\begin{equation}
\label{eqn:ctorsion}
c^\xi\in\Gamma(\Hom(\Lambda^2E,E)),\hspace{1cm} c^\xi(\alpha,\beta):= [\alpha,\beta] - \nabla^\xi_{\rho(\alpha)}\beta - \nabla^\xi_{\rho(\beta)}\alpha,
\end{equation}
where $\rho$ is the anchor of $E$. The bracket of $\CAlg$,
\begin{equation}
\label{eqn:cbracket}
[\cdot,\cdot]^\xi:\Gamma(\CAlg)\times \Gamma(\CAlg)\to \Gamma(\CAlg),
\end{equation} 
is defined by
\begin{equation}
\label{eqn:cbracketdef}
[(X,\alpha),(Y,\beta)]^\xi:= \big([X,Y],\; c^\xi(\alpha,\beta) + \nabla^\xi_X(\beta) - \nabla^\xi_Y(\alpha) \big).
\end{equation}
The anchor, which is independent of the splitting, is simply the projection
\begin{equation}
\label{eqn:canchor}
\rho:\CAlg\to TM,\hspace{1cm} (X,\alpha)\mapsto X.
\end{equation}
It is straightforward to verify that $\CAlg$ is a transitive almost Lie algebroid. For the vector subbundle $\Tab\subset\Hom(\CAlg,\CAlg)$, we take the symbol space, where the inclusion is given by composing the symbol map \eqref{eqn:symbolmapliepfaffianalgebroid} (which is injective by the assumption that $(A,D)$ is standard) with the inclusion
\begin{equation}
\Hom(TM,E)\hookrightarrow \Hom(\CAlg,\CAlg), \hspace{1cm} T\mapsto \Big(\hat{T}: (X,\alpha) \mapsto \big(0,T(\rho(\alpha)-X)\big)\Big).
\label{eqn:secondinclusion}
\end{equation}
Indeed, $\hat{T}$ takes values in $\Ker\rho$, and hence $\Tab\subset\Hom(\CAlg,\Ker\rho)$. Note that if we equip $\Hom(TM,E)$ with the bracket
\begin{equation*}
[T,S]:= T\circ \rho\circ S - S\circ \rho\circ T
\end{equation*}
and $\Hom(\CAlg,\CAlg)$ with the commutator bracket, 
then \eqref{eqn:secondinclusion} becomes a Lie algebra map. 

\begin{myremark}
To generalize this construction to non-standard Lie-Pfaffian algebroids, we would need to relax the definition of a Cartan algebroid by requiring for there to be a map $\Tab\to \Hom(\mathcal{C},\mathcal{C})$ rather than an inclusion.
\end{myremark}

\begin{myprop}
	\label{myprop:cgprecartanalgebroid}
	The pair $(\CAlg,\Tab)$ defined above is an almost Cartan algebroid. Up to gauge equivalence, it is independent of the choice of splitting $\xi$.
\end{myprop}

\begin{proof}
	We have already seen that $(\CAlg,\Tab)$ is an almost Cartan algebroid. We are left with showing that, for any two choices of splittings $\xi:E\to A$ and $\xi':E\to A$, the resulting almost Cartan algebroids are gauge equivalent. Taking the difference, we get a map $(\xi-\xi'):E\to\Tab$, which we can interpret as a gauge equivalence by letting it act trivially on the first component, i.e.
	\begin{equation*}
	(\xi-\xi'):\CAlg=TM\oplus E\to \Tab,\hspace{1cm} (X,\alpha)\mapsto (\xi-\xi')(\alpha).
	\end{equation*}
	It is now straightforward to verify that gauge transforming the almost Cartan algebroid $(\CAlg,\Tab)$ with bracket $[\cdot,\cdot]^\xi$ by $\xi-\xi'$ yields the almost Cartan algebroid $(\CAlg_{\xi'},\Tab)$ with bracket $[\cdot,\cdot]^{\xi'}$
\end{proof}

\begin{myremark}
	Note that Propositions \ref{prop:realizationgauge}, together with Proposition \ref{myprop:cgprecartanalgebroid}, implies that if we manage to construct a realization of $(\CAlg,\Tab)$, then it will be independent of the choice of $\xi$. 
\end{myremark}

\begin{myremark}
	It would be more canonical to construct a Cartan pair rather than a Cartan algebroid (see Section \ref{section:cartanpairs} and, in particular, Theorem \ref{theorem:cartanpairscartanalgebroids}). The Cartan pair will consist of the pair $(TM\oplus A,\Tab)$, and will not require the choice of a splitting $\xi$. We have chosen the Cartan algebroid point of view in order to remain closer to Cartan's original constructions.
\end{myremark}

\subsubsection*{Constructing a Realization}
\label{section:constructingrealization}

The second step is to construct a realization of the almost Cartan algebroid $(\CAlg,\Tab)$. To this end, we assume that $(A,D)$ is the Lie-Pfaffian algebroid of a standard Lie-Pfaffian groupoid $(\mathcal{G},\omega)$, and as a candidate for the realization we take the pair $(\mathcal{G},\Omega)$ consisting of the target map $t:\mathcal{G}\to M$ and the 1-form 
\begin{equation*}
\Omega=(dt,\omega)\in\Omega^1(\mathcal{G};t^*\CAlg).
\end{equation*}
Here, we are using that $\mathcal{C}=TM\oplus E$ and we are viewing the differential of the target map $dt$ as a $1$-form on $\mathcal{G}$ with values in $t^*TM$. The problem at hand is to check whether there exists $\Pi\in\Omega^1(\mathcal{G};t^*\Tab)$ as in Definition \ref{def:realization} of a realization, and the idea is to show that such a $\Pi$ corresponds precisely to an integral Cartan-Ehresmann connection on the Lie-Pfaffian groupoid $(\mathcal{G},\omega)$ (see Definition \ref{def:cartanehresmannconnection}). 

To begin with, given a Lie groupoid $\mathcal{G}$, recall that an Ehresmann connection on the source map $s:\mathcal{G}\to M$ has two equivalent descriptions: a right splitting $H:s^*TM\to T\mathcal{G}$ of the short exact sequence 
\begin{equation}
\label{eqn:sesehresmannconnection}
0\to t^*A \cong T^s\mathcal{G} \to T\mathcal{G} \xrightarrow{ds} s^*TM \to 0,
\end{equation}
or a left splitting, also known as a connection $1$-form, which, due to the canonical isomorphism $t^*A \cong T^s\mathcal{G}$, can be viewed as an element of $\Omega^1(\mathcal{G};t^*A)$ that restricts to the Maurer-Cartan form (i.e. right translation) on $T^s\mathcal{G}$. 

Now, given a Lie-Pfaffian groupoid $(\mathcal{G},\omega)$, there is a special class of compatible Ehresmann connections on $s:\mathcal{G}\to M$ called Cartan-Ehresmann connections (Definition \ref{def:cartanehresmannconnection}). A Cartan-Ehresmann connection is an Ehresmann connection $H:s^*TM\to T\mathcal{G}$ that takes values in the Cartan distribution $C_\omega = \Ker\omega\subset T\mathcal{G}$, or, equivalently, a right splitting $H:s^*TM\to C_\omega$ of the short exact sequence
\begin{equation}
0\to t^*\Tab \to C_\omega \xrightarrow{ds} s^*TM \to 0.
\end{equation}
What is the corresponding notion in terms of connection $1$-forms? To answer this question, we use the splitting \eqref{eqn:secondfundamentaltheoremsplitting} that we have fixed earlier. This splitting induces an isomorphism $A\cong E\oplus \Tab$, which, in turn, induces a decomposition
\begin{equation}
\label{eqn:ehresmannconnectiondecomp}
\Omega^1(\mathcal{G};t^*A)\cong \Omega^1(\mathcal{G};t^*E) \oplus \Omega^1(\mathcal{G};t^*\Tab).
\end{equation}
Given a Cartan-Ehresmann connection $H$, the induced connection $1$-form has, therefore, two components. The first component is precisely $\omega$ (easy to see), and we denote the second component by $\Pi=\Pi_H\in \Omega^1(\mathcal{G};t^*\Tab)$. We thus have a map
\begin{equation}
\label{eqn:piintegralcartanehresmannconnection}
H\mapsto \Pi=\Pi_H,
\end{equation}
and:

\begin{mytheorem}
	\label{theorem:piintegralcartanehresmannconnection}
	Let $(\mathcal{G},\omega)$ be a standard Lie-Pfaffian groupoid over $M$. The map \eqref{eqn:piintegralcartanehresmannconnection} defines a 1-1 correspondence between
	\begin{equation}
	\mbox{
		\parbox[c][1cm][c]{4.5cm}{\centering Cartan-Ehresmann \\ connections $H$ on $\mathcal{G}$}
		\hspace{0.5cm} 
		$\longleftrightarrow$
		\hspace{0.5cm}
		\parbox[c][1cm][c]{4.5cm}{\centering $\Pi\in\Omega^1(\mathcal{G};t^*\Tab)$ \\ satisfying \eqref{eqn:vbisomorphism},}
	}
	\label{eqn:secondfundamentaltheoremcorrespondence1}
	\end{equation}
	that restricts to a 1-1 correspondence between
	\begin{equation}
	\mbox{
		\parbox[c][1cm][c]{4.5cm}{\centering integral Cartan-Ehresmann \\ connections $H$ on $\mathcal{G}$}
		\hspace{0.5cm} 
		$\longleftrightarrow$
		\hspace{0.5cm}
		\parbox[c][1cm][c]{4.5cm}{\centering $\Pi\in\Omega^1(\mathcal{G};t^*\Tab)$ \\ satisfying \eqref{eqn:vbisomorphism} and \eqref{eqn:structureequations}.}
	}
	\label{eqn:secondfundamentaltheoremcorrespondence2}
	\end{equation}
	Thus, the pair $(\mathcal{G},\Omega)$, consisting of the target map $t:\mathcal{G}\to M$ and $\Omega=(dt,\omega)$, is a realization of the almost Cartan algebroid $(\CAlg,\Tab)$ (which is then a Cartan algebroid) if and only if $(\mathcal{G},\omega)$ admits an integral Cartan-Ehresmann connection. 
\end{mytheorem}

\begin{proof}
	We begin with \eqref{eqn:secondfundamentaltheoremcorrespondence1}. Let $H$ be a Cartan-Ehresmann connection and let $\Pi=\Pi_H\in \Omega^1(\mathcal{G};t^*\Tab)$ be $1$-form given by \eqref{eqn:piintegralcartanehresmannconnection}. First note that
	\begin{equation}
	\label{eqn:3componentdecomposition}
	(\Omega,\Pi) = (dt,\omega,\Pi): T\mathcal{G} \xrightarrow{\simeq} t^*(TM\oplus E\oplus \Tab)
	\end{equation}
	is pointwise surjective, since $dt$ is surjective onto $t^*TM$ and $(\omega,\Pi)$ restricts to the Maurer-Cartan form on $T^s\mathcal{G}$, which is then surjective onto $t^*(E\oplus\Tab)\cong t^*A$. So, by dimension count, $(\Omega,\Pi)$ is pointwise an isomorphism and \eqref{eqn:vbisomorphism} is satisfied. In fact, we can explicitly describe the inverse of \eqref{eqn:3componentdecomposition}, which will serve us in the second part of the proof. Let us denote the map at the level of sections that is induced by the inverse of \eqref{eqn:3componentdecomposition} by:
	\begin{equation*}
	\mathfrak{X}(M)\to \mathfrak{X}(\mathcal{G}),\; X\mapsto Y_X; \hspace{0.5cm} 
	\Gamma(E)\to \mathfrak{X}(\mathcal{G}),\; \alpha\mapsto Y_\alpha; \hspace{0.5cm}
	\Gamma(\Tab)\to \mathfrak{X}(\mathcal{G}),\; S \mapsto Y_S.
	\end{equation*}
	Thus, $Y_X,Y_\alpha,Y_S\in\mathfrak{X}(\mathcal{G})$ are the unique vector fields that satisfy:
	\begin{equation}
	\label{eqn:3componentdecompositioneqns}
	\begin{split}
	&dt(Y_X) = t^*X,\hspace{1cm}\omega(Y_X) = 0,\hspace{1.3cm} \Pi(Y_X)=0, \\
	&dt(Y_\alpha) = 0,\hspace{1.5cm}\omega(Y_\alpha) = t^*\alpha,\hspace{1cm} \Pi(Y_\alpha)=0, \\
	&dt(Y_S) = 0,\hspace{1.5cm}\omega(Y_S) = 0,\hspace{1.32cm} \Pi(Y_S)=t^*S.
	\end{split}
	\end{equation}
	Given a section $\alpha\in\Gamma(A)$, we denote the induced right invariant vector field by $\widetilde{\alpha}\in\mathfrak{X}(\mathcal{G})$. Also recall that there is canonical isomorphism 
	\begin{equation}
	\label{eqn:secdonfundamentaltheoremstmttm}
	\psi: s^*TM\xrightarrow{\simeq} t^*TM
	\end{equation}
	of vector bundles over $\mathcal{G}$ which is equal to $dt\circ H$ (Lemma \ref{lemma:pfaffiangroupoidtriangle}). We define the map:
	\begin{equation*}
	\mathfrak{X}(M) \to \mathfrak{X}(J^k\Gamma),\hspace{1cm} X \mapsto X^H= H\circ (dt\circ H)^{-1}(t^*X).
	\end{equation*}
	One now readily verifies that
	\begin{equation}
	\label{eqn:3componentdecompositionformula}
	Y_X = X^H, \hspace{0.5cm} Y_\alpha = \widetilde{\xi(\alpha)} -  \rho(\alpha)^H, \hspace{0.5cm} Y_S = \widetilde{S}.
	\end{equation}
	
	In the other direction, choose $\Pi\in\Omega^1(\mathcal{G};t^*\Tab)$ that satisfies \eqref{eqn:vbisomorphism}, so we have an isomorphism \eqref{eqn:3componentdecomposition}. This induces a Cartan-Ehresmann connection $H:s^*TM\to T\mathcal{G}$ as follows: denote the restriction of the inverse of \eqref{eqn:3componentdecomposition} to $t^*TM$ by $H':t^*TM\to T\mathcal{G}$ and set $H=H'\circ \psi$, where $\psi$ is the isomorphism \eqref{eqn:secdonfundamentaltheoremstmttm}. Indeed,
	\begin{equation*}
	ds\circ H = \explain{ds}{=\psi^{-1} \circ dt}\circ H'\circ \psi = \psi^{-1} \circ \underbrace{dt \circ H'}_{=\Id} \circ \psi = \psi^{-1}\circ\psi = \Id.
	\end{equation*}
	It is easy to see that this construction is inverse to \eqref{eqn:piintegralcartanehresmannconnection}. 
	
	We move on to \eqref{eqn:secondfundamentaltheoremcorrespondence2}. Let $H$ be an integral Cartan-Ehresmann connection. We must show that the induced $\Pi$ satisfies the structure equation \eqref{eqn:structureequations}. For this, it is enough to verify that the expression
	\begin{equation}
	\label{eqn:structureequationexpression}
	d\Omega+\frac{1}{2}[\Omega,\Omega]-\Pi\wedge\Omega =  d(dt,\omega)+\frac{1}{2}[(dt,\omega),(dt,\omega)] - \Pi\wedge (dt,\omega)
	\end{equation}
	vanishes when applied to all possible pairs of the type \eqref{eqn:3componentdecompositionformula}. In the following computations, we use Lemma \ref{lemma:maurercartanexpression} to evaluate the Maurer-Cartan expression $\text{MC}_{\Omega}= d\Omega+\frac{1}{2}[\Omega,\Omega]$:
	
	\begin{enumerate}
		\item $\displaystyle
		\begin{aligned}[t]
		(d\Omega &+\frac{1}{2}[\Omega,\Omega] -\Pi\wedge\Omega)(Y_X,Y_{X'}) \\ & = \big(d(dt,\omega)+\frac{1}{2}[(dt,\omega),(dt,\omega)] - \Pi\wedge (dt,\omega)\big) (Y_X,Y_{X'}) \\
		& = - \big(dt([X^H,X'^H]),\underbrace{\cancel{\omega([X^H, X'^H])}}_{\text{$H$ is integral}}\big) + t^*[(X,0),(X',0)]  \\
		& = - \big(\underbrace{t^*[X,X']}_{\mathclap{\text{$Y_X,Y_{X'}$ are $t$-related to $X,X'$ }}},0\big) + t^*([X,X'],0) = 0
		\end{aligned}
		$
		\item $\displaystyle
		\begin{aligned}[t]
		(d\Omega&+\frac{1}{2}[\Omega,\Omega] -\Pi\wedge\Omega)(Y_X,Y_\alpha)\\ & =  - \big(\underbrace{\cancel{dt([Y_X,Y_\alpha])}}_{\mathclap{\text{$Y_\alpha$ is $t$-related to 0 }}},\omega([Y_X,Y_\alpha])\big) + t^*[(X,0),(0,\alpha)] \\ 
		& = -\Big(0,\underbrace{\omega([X^H,\widetilde{\xi(\alpha)}])}_{\mathclap{\text{$=t^*\nabla^\xi_X(\alpha)$ by Lemma \ref{lemma:cartanformspenceroperatorformula}}} } - \underbrace{\cancel{\omega([X^H, \rho(\alpha)^H])}}_{\text{$H$ is integral}} \Big) + t^*(0,\nabla^\xi_X(\alpha)) = 0
		\end{aligned}
		$
		\item $\displaystyle
		\begin{aligned}[t]
		(d\Omega&+\frac{1}{2}[\Omega,\Omega] -\Pi\wedge\Omega)(Y_X,Y_S) \\ & = - \big(dt([Y_X,Y_S]),\omega([Y_X,Y_S])\big) + \Pi(Y_S)((dt,\omega)(Y_X)) \\
		& = - \big(0,\underbrace{\omega([X^H, \widetilde{S}])}_{\mathclap{\text{$=t^*D_X(S)=-t^*S(X)$}}} \big) + t^*(0,S(-X)) = 0
		\end{aligned}
		$
		\item $\displaystyle
		\begin{aligned}[t]
		(d\Omega&+\frac{1}{2}[\Omega,\Omega] -\Pi\wedge\Omega)(Y_\alpha,Y_{\alpha'}) \\ & = - \big(dt([Y_\alpha,Y_{\alpha'}]),\omega([Y_\alpha,Y_{\alpha'}])\big) + t^*[(0,\alpha),(0,\alpha')]  \\
		& = - \big(0,\underbrace{\omega([\widetilde{\xi(\alpha)}, \widetilde{\xi(\alpha')}])}_{=t^*[\alpha,\alpha']} - \underbrace{\omega([\rho(\alpha)^H, \widetilde{\xi(\alpha')}])}_{=t^*\nabla^\xi_{\rho(\alpha)}(\alpha)} + \underbrace{\omega([\rho(\alpha')^H, \widetilde{\xi(\alpha)}])}_{=t^*\nabla^\xi_{\rho(\alpha')}(\alpha)} \big) \\
		& \hspace{0.35cm} + \big( 0,  \underbrace{\cancel{\omega([\rho(\alpha)^H, \rho(\alpha')^H])}}_{\text{$H$ is integral}} \big) - t^*(0,c^\xi(\alpha,\alpha')) = 0
		\end{aligned}
		$
		\item $\displaystyle
		\begin{aligned}[t]
		(d\Omega&+\frac{1}{2}[\Omega,\Omega] -\Pi\wedge\Omega)(Y_\alpha,Y_S) \\ & = - \big(dt([Y_\alpha,Y_S]),\omega([Y_\alpha,Y_S])\big)  + \Pi(Y_S)((dt,\omega)(Y_\alpha)) \\
		& = - \big(0,\underbrace{\cancel{\omega([\widetilde{\xi(\alpha)}, \widetilde{S}])}}_{ \mathclap{\text{ $= d\pi([\xi(\alpha),S]) =0$ } } } - \underbrace{\omega([\rho(\alpha)^H, \widetilde{S}])}_{= -t^*S(\rho(\alpha)) } \big) + t^*(0,S(\rho(\alpha)) =0
		\end{aligned}
		$
		\item $\displaystyle
		\begin{aligned}[t]
		(d\Omega&+\frac{1}{2}[\Omega,\Omega]-\Pi\wedge\Omega)(Y_S,Y_{S'}) \\ & = - \big(dt([Y_S,Y_{S'}]),\omega([Y_S,Y_{S'}])\big) \\ 
		& = - \big(0,\underbrace{\cancel{\omega([\widetilde{S}, \widetilde{S'}])}}_{ \mathclap{\text{ $= d\pi([S,S']) =0$ } } } \big) = 0
		\end{aligned}
		$
	\end{enumerate}
	
	Conversely, let $\Pi\in\Omega^1(\mathcal{G};t^*\Tab)$ satisfy \eqref{eqn:structureequations} and \eqref{eqn:vbisomorphism}. Thus, $\Pi$ induces a Cartan-Ehresmann connection $H$ on $\mathcal{G}$. Let $X,X'\in\mathfrak{X}(M)$ and set $X^H:=H\circ \psi^{-1}(t^*X)$ and $X'^H:=H\circ \psi^{-1}(t^*X')$. 
	Using \eqref{eqn:structureequations},
	\begin{equation*}
	\begin{split}
	0 &= (d\Omega+\frac{1}{2}[\Omega,\Omega]-\Pi\wedge\Omega)(X^H,X'^H) \\ 
	& = - \big(dt([X^H,X'^H]),\omega([X^H, X'^H])\big) + t^*[(X,0),(X',0)]  \\
	& = - \big(0,\delta \omega(H\circ \psi^{-1}(X), H\circ \psi^{-1}(X')) \big).
	\end{split}
	\end{equation*}
	We conclude that $\delta\omega(H(\cdot),H(\cdot))=0$ and, hence, that $H$ is integral. 
\end{proof}

\subsubsection{Aside: the Maurer-Cartan Equation of a Lie-Pfaffian Groupoid}
\label{section:marurercartanequationliepfaffiangroupoid}

The above construction of a realization, which may seem somewhat arbitrary at first, can be understood more conceptually in terms of the Maurer-Cartan equation on the classical prolongation of a Lie-Pfaffian groupoid, which was introduced and studied in \cite{Salazar2013} (but which comes from the classical notion of a prolongation, as its name suggests). 

Let us briefly review the construction of the classical prolongation. For more details, see Appendix \ref{section:Liepfaffaingroupoidandalgebroid}. Let $(\mathcal{G},\omega)$ be a Lie-Pfaffian groupoid, and let
\begin{equation*}
\widetilde{\mathcal{G}}\rightrightarrows M
\end{equation*}
be the subgroupoid of the first jet groupoid $J^1\mathcal{G}\rightrightarrows M$ consisting of all elements $j^1_xb$ that pull-back both $\omega$ and its differential $\delta\omega$ to zero (note that the notation $P_\omega(\mathcal{G})$ is used for $\widetilde{G}$ in Appendix \ref{section:Liepfaffaingroupoidandalgebroid}). We say that $\widetilde{\mathcal{G}}$ is smoothly defined (or 1-integrable)
if the projection $\pi:\widetilde{G}\to\mathcal{G},\;j^1_xb\mapsto b(x)$, admits a global section and if the symbol space of $(\mathcal{G},\omega)$ is of constant rank. This condition implies that $\widetilde{\mathcal{G}}\rightrightarrows M$ is a Lie subgroupoid and $\pi:\widetilde{\mathcal{G}}\to\mathcal{G}$ is an affine bundle. In this case, we may equip $\widetilde{\mathcal{G}}$ with the restriction of the Cartan form of $J^1\mathcal{G}$, which we denote by $\widetilde{\omega}\in\Omega^1(\widetilde{\mathcal{G}};t^*A)$, and the pair $(\widetilde{\mathcal{G}},\widetilde{\omega})$ becomes a Lie-Pfaffian groupoid which we call the \textit{classical prolongation} of $(\mathcal{G},\omega)$. 

The Cartan form $\widetilde{\omega}$ on of the classical prolongation satisfies a Maurer-Cartan equation. The equation is defined in complete analogy to the construction in Section \ref{section:realizationsstructureequations}. The differential of $\widetilde{\omega}$,
\begin{equation*}
d_{\omega}\widetilde{\omega}\in\Omega^2(\widetilde{G};t^*E),
\end{equation*}
is defined by the Koszul-type formula
\begin{equation*}
d_{\omega}\widetilde{\omega}(X,Y):= (t^*D)_X(\widetilde{\omega}(Y)) - (t^*D)_Y(\widetilde{\omega}(X)) - l\circ\widetilde{\omega}([X,Y]), \hspace{0.6cm} \forall\;  X,Y\in\mathfrak{X}(\widetilde{G}),
\end{equation*}
where $D=D_\omega:\mathfrak{X}(M)\times \Gamma(A)\to \Gamma(E)$ is the Spencer operator induced by $\omega$. The torsion 
\begin{equation*}
\{\cdot,\cdot\}_{\omega}\in\Gamma(\Hom(\Lambda^2A,E))
\end{equation*}
is defined by  
\begin{equation*}
\{ \alpha,\beta \}_{\omega}:= [\alpha,\beta] - D_{\rho(\alpha)}\beta + D_{\rho(\beta)}\alpha, \hspace{1cm} \forall\; \alpha,\beta\in\Gamma(A),
\end{equation*}
and we may use this pairing to define a graded bracket 
\begin{equation*}
\{\cdot,\cdot\}_{\omega}: \Omega^p(\widetilde{G};t^*A)\times \Omega^q(\widetilde{G};t^*A) \to \Omega^{p+q}(\widetilde{G};t^*E)
\end{equation*} 
by the usual wedge-like formula. In particular,
\begin{equation*}
\frac{1}{2}\{ \widetilde{\omega},\widetilde{\omega} \}_{\omega}(X,Y) = \{ \widetilde{\omega}(X),\widetilde{\omega}(Y) \}_{\omega}. 
\end{equation*}
Together, we obtain the Maurer-Cartan expression
\begin{equation*}
d_{\omega}\widetilde{\omega} + \frac{1}{2}\{ \widetilde{\omega},\widetilde{\omega} \}_{\omega}\in\Omega^2(\widetilde{G};t^*E).
\end{equation*}
Note that, in contrast to the construction in Section \ref{section:realizationsstructureequations}, this construction is canonical and does not require the choice of a connection. This, however, comes at the cost of ``going one level down'', in the sense that, while $\widetilde{\omega}$ is a form with values in $A$, the Maurer-Cartan expression takes values in $E$. Finally, we have that:

\begin{mytheorem}
	Let $(\mathcal{G},\omega)$ be a Lie-Pfaffian groupoid and assume that its classical prolongation is smoothly defined $(\widetilde{\mathcal{G}},\widetilde{\omega})$. Then
	\begin{equation}
	\label{eqn:maurercartanequationclassicalprolongation}
	d_{\omega}\widetilde{\omega} + \frac{1}{2}\{ \widetilde{\omega},\widetilde{\omega} \}_{\omega} = 0.
	\end{equation}
\end{mytheorem}

For the proof, see Theorem 6.2.17 together with Proposition 6.2.41 in \cite{Salazar2013}. 

\begin{myremark}
	In \cite{Salazar2013}, the abstract notion of a Lie prolongation of a Lie-Pfaffian groupoid $(\mathcal{G},\omega)$ is introduced, and it is proven that the classical prolongation is universal amongst all Lie prolongations. A Lie prolongation is, roughly speaking, a Lie-Pfaffian groupoid $(\widetilde{\mathcal{G}},\widetilde{\omega})$, together with a map $p:(\widetilde{\mathcal{G}},\widetilde{\omega})\to (\mathcal{G},\omega)$, that satisfies certain compatibility conditions between $\omega$ and $\widetilde{\omega}$. It is then proven that these compatibility conditions are equivalent to the above Maurer-Cartan equation, which can be interpreted as a compatibility condition between $\widetilde{\omega}$ and $\omega$. 
\end{myremark}

Let us now explain how the construction of a realization from the previous section relates to this Maurer-Cartan equation. Let $(\mathcal{G},\omega)$ be a Lie-Pfaffian groupoid and assume that its classical prolongation is smooth. In Appendix \ref{section:Liepfaffaingroupoidandalgebroid} it is shown that a section of the projection $\pi:\widetilde{G}\to \mathcal{G}$ is the same thing as an integral Cartan-Ehresmann connection. Given such a section, say $\eta:\mathcal{G}\to \widetilde{G}$, we can pull-back the Cartan form to obtain a 1-form 
\begin{equation*}
\eta^*\widetilde{\omega}\in\Omega^1(\mathcal{G};t^*A). 
\end{equation*}
Choosing a splitting \eqref{eqn:secondfundamentaltheoremsplitting} as in the construction of the previous section, we obtain a decomposition $A\cong E\oplus \Tab$, and $\eta^*\widetilde{\omega}$ decomposes into two components. The first is precisely $\omega\in \Omega^1(\mathcal{G};t^*E)$, while the second, which we denote by $\Pi\in\Omega^1(\mathcal{G};t^*\Tab)$, is precisely the 1-form which is obtained from the Cartan-Ehresmann connection $\eta$ via the map \eqref{eqn:piintegralcartanehresmannconnection}. Thus, the pair $(\omega,\Pi)$ is nothing but the pull-back of $\widetilde{\omega}$ by $\eta$. This, however, is not quite a realization yet, since Cartan chooses to complete $(\omega,\Pi)$ to a ``coframe'' of $\mathcal{G}$ by including the 1-form $dt\in\Omega^1(\mathcal{G};t^*TM)$, and to view $\omega$ and $dt$ as a single 1-form $\Omega=(dt,\omega)\in\Omega^1(\mathcal{G};t^*\mathcal{C})$, where recall that $\mathcal{C}=TM\oplus E$. Now, using the same ideas as in the proof of Theorem \ref{theorem:piintegralcartanehresmannconnection}, it is not hard to see that the Maurer-Cartan equation \eqref{eqn:maurercartanequationclassicalprolongation} for $\widetilde{\omega}$ is equivalent to the structure equation \eqref{eqn:structureequations} for the induced pair $(\Omega,\Pi)$.

\subsubsection{Proof of the Second Fundamental Theorem}
\label{section:secondfundamentaltheoremproof}

We are now ready to complete the proof of Theorem \ref{theorem:secondfundamentaltheorem}. The two ``ingredients'' have the following implication:

\begin{mycor}
	\label{cor:secondfundamentaltheorem}
	Let $(\mathcal{G},\omega)$ be a Lie-Pfaffian groupoid on $M$. If it is standard and admits an integral Cartan-Ehresmann connection, then the associated pair $(\mathcal{C},\Tab)$ (see Proposition \ref{myprop:cgprecartanalgebroid}) is a Cartan algebroid, $(\mathcal{G},\Omega)$ (see Theorem \ref{theorem:piintegralcartanehresmannconnection}) a realization, and the induced pseudogroup $\Gamma(\mathcal{G},\Omega)$ (see \eqref{eqn:realizationinducedpseudogroup}) is precisely the canonical prolongation of $\LBis(\mathcal{G},\omega)$ (see \eqref{eqn:canonicalprolongationpseudogroup}):
	\begin{equation}
	\label{eqn:inducedpseudogroupcanonicalprolongation}
	\Gamma(\mathcal{G},\Omega) = \mathrm{prol}(\LBis(\mathcal{G},\omega)).
	\end{equation}
	Thus, $\Gamma(\mathcal{G},\Omega)$ is in normal form and it is Cartan equivalent to the generalized pseudogroup $\LBis(\mathcal{G},\omega)$ of local solutions of $(\mathcal{G},\omega)$. 
\end{mycor}

\begin{proof}
	This is an immediate consequence of Proposition \ref{prop:prolongationiscartanequivalent}, Theorem \ref{theorem:firstfundamentaltheorem} and Theorem \ref{theorem:piintegralcartanehresmannconnection}. Note that, in \eqref{eqn:inducedpseudogroupcanonicalprolongation}, we are using the simple fact that, for $\phi\in\LDiff(\mathcal{G})$, $\phi^*t=t$ if and only if $\phi^*t=t$ and $\phi^*dt=dt$.
\end{proof}

This corollary can also be ``improved'' by restricting to a complete transversal of $\mathcal{G}$, as explained in Section \ref{section:restrictingtotransversal}. Using the notation from Example \ref{example:restrictioncartanalgebroid} for the restrictions of a Cartan algebroid and a realization, we have that:

\begin{mycor}
	\label{cor:secondfundamentaltheoremtransversal}
	Let $(\mathcal{G},\omega)$ be a Lie-Pfaffian groupoid on $M$ and let $N\subset M$ be a complete transversal of $\mathcal{G}$. If $(\mathcal{G},\omega)$ is standard and admits an integral Cartan-Ehresmann connection, then the associated pair $(\mathcal{C}_N,\Tab_N)$ is a Cartan algebroid, $(\mathcal{G}_N,\Omega_N)$ is a realization, and the induced pseudogroup $\Gamma(\mathcal{G}_N,\Omega_N)$ is precisely the canonical prolongation of $\LBis(\mathcal{G},\omega)$:
	\begin{equation}
	\Gamma(\mathcal{G}_N,\Omega_N) = \mathrm{prol}(\LBis(\mathcal{G},\omega)),
	\end{equation}
	Thus, $\Gamma(\mathcal{G}_N,\Omega_N)$ is in normal form and Cartan equivalent to the generalized pseudogroup $\LBis(\mathcal{G},\omega)$ of local solutions of $(\mathcal{G},\omega)$. 
\end{mycor}

\begin{proof}
	In the previous proof, replace Proposition \ref{prop:prolongationiscartanequivalent} and Theorem \ref{theorem:firstfundamentaltheorem} with their ``complete transversal counterparts'', Proposition \ref{prop:prolongationiscartanequivalenttransversal} and Theorem \ref{theorem:firstfundamentaltheoremtransversal}, and use the fact that Cartan algebroids and realizations can be restricted to submanifolds (Example \ref{example:restrictioncartanalgebroid}).
\end{proof}

And finally, the proof of the Second Fundamental Theorem:

\begin{proof}[Proof (Theorem \ref{theorem:secondfundamentaltheorem}).] 
Let $\Gamma$ be a Lie pseudogroup on $M$ of order $k$. The pair $(J^k\Gamma,\omega)$, consisting of the $k$-the jet groupoid of $\Gamma$ and its Cartan form, is a standard Lie-Pfaffian groupoid, and $\Gamma$ is Cartan equivalent to its generalized pseudogroup of local solutions $\LBis(J^k\Gamma,\omega)$ (see Section \ref{section:liepseudogroupsasliepfaffiangroupoids}). By Lemma \ref{lemma:liepseudogrouphasintegralcartanehresmannconnection}, $(J^k\Gamma,\omega)$ admits an integral Cartan-Ehresmann connection. Thus, we may apply Corollary \ref{cor:secondfundamentaltheorem} to obtain the realization $(J^k\Gamma,\Omega)$ of the Cartan algebroid $(TM\oplus A^{k-1}\Gamma,\Tab^k\Gamma)$, and the associated pseudogroup $\Gamma(J^k\Gamma,\Omega)$ is in normal form and Cartan equivalent to $\LBis(J^k\Gamma,\omega)$, which, in turn, is Cartan equivalent to $\Gamma$. Using Corollary \ref{cor:secondfundamentaltheoremtransversal} instead, we may also replace this realization by its restriction to any complete transversal to the orbits of $\Gamma$. 
\end{proof}

\subsection{Two Examples of Cartan}
\label{section:cartanexamples}

We conclude this section by looking at two explicit examples from Cartan's work (\cite{Cartan1937-1}, pp. 1344-1347) of the construction of a realization out of a given Lie pseudogroup. In each example, we start by citing Cartan, showing how he constructs a Lie pseudogroup $\widetilde{\Gamma}$ in normal form out of a given Lie pseudogroup $\Gamma$. Then, revisiting the example, we apply the algorithm of our proof, computing the Cartan algebroid and realization induced by Cartan's initial pseudogroup and arriving at Cartan's formulas. Some computations are performed here in a concise manner, and we refer the reader to \cite{Yudilevich2016-2} for more details. 

\subsubsection{Example 1 - Cartan}

Cartan: \textit{``Let $\Gamma$ be the pseudogroup of homographic transformations in one variable 
		\begin{equation}
		\label{eqn:pseudogroupexample1}
		X = \frac{ax+b}{cx+d}\;\;\;\; a,b,c,d\in\mathbb{R},\;\;\; ad-bc\neq 0.
		\end{equation}
		We know that the defining equation of the pseudogroup is
		\begin{equation*}
		X'X''' - \frac{3}{2} (X'')^2 = 0.
		\end{equation*}
		We set
		\begin{equation*}
		X'=u,\;\;\; X''=v,
		\end{equation*}
		and we have the system
		\begin{equation*}
		dX = \omega_1 = udx,\;\;\; du = vdx,\;\;\; dv = \frac{3}{2} \frac{v^2}{u} dx.
		\end{equation*}
		We have
		\begin{equation*}
		d\omega_1 = du\wedge dx = \frac{du-vdx}{u}\wedge udx = \frac{du-vdx}{u} \wedge \omega_1.
		\end{equation*}
		The form $\frac{du-vdx}{u}$ is thus invariant, we denote it by $\omega_2$,
		\begin{equation*}
		\omega_2 = \frac{du}{u} - \frac{v}{u} dx.
		\end{equation*}
		We compute
		\begin{equation*}
		\begin{split}
		d\omega_2 & = -\frac{1}{u} dv \wedge dx + \frac{v}{u^2} du \wedge dx = \big( -\frac{1}{u^2} dv + \frac{v}{u^3} du \big) \wedge \omega_1  \\
		& = \Big( -\frac{1}{u^2} \big( dv - \frac{3}{2} \frac{v^2}{u} dx \big) + \frac{v}{u^3} (du-vdx) \Big) \wedge \omega_1,
		\end{split}
		\end{equation*}
		from which we obtain the new invariant form
		\begin{equation*}
		\omega_3 = -\frac{1}{u^2} dv + \frac{v}{u^3} du + \frac{1}{2} \frac{v^2}{u^3} dx.
		\end{equation*}
		We compute
		\begin{equation*}
		d\omega_3 = \frac{1}{u^3} du\wedge dv + \frac{v}{u^3} dv\wedge dx - \frac{3}{2} \frac{v^2}{u^4} du\wedge dx = \omega_3\wedge\omega_2.
		\end{equation*}
		The structure equations are
		\begin{equation}
		\label{eqn:sl2rexamplethirdform}
		d\omega_1 = \omega_2\wedge\omega_1,\;\;\; d\omega_2 = \omega_3\wedge\omega_1,\;\;\; d\omega_3 = \omega_3\wedge\omega_2.\text{''}
		\end{equation}}

Thus, starting with a pseudogroup $\Gamma$ on $\mathbb{R}$ (with coordinate $x$) Cartan constructs a realization on $\mathbb{R}^3\backslash\{u=0\}$ (with coordinates $x,u,v$) consisting of the 1-forms $\omega_1,\omega_2,\omega_3$. We may now compute the induced pseudogroup $\widetilde{\Gamma}$ on $\mathbb{R}^3\backslash\{u=0\}$. It is generated by the transformations:
	\begin{equation*}
	\bar{x}= \frac{ax+b}{cx+d}, \;\; \bar{u}= u \frac{(cx+d)^2}{ad-bc} ,\;\; \bar{v}=  \frac{v(cx+d)^4+ 2 u c (cx+d)^3}{(ad-bc)^2},  
	\end{equation*}
	where $a,b,c,d\in\mathbb{R}$ and $ad-bc\neq 0$. It is clearly an isomorphic prolongation of (and hence Cartan equivalent to) $\Gamma_0$.

\subsubsection{Example 1 - Revisited}

We consider again the pseudogroup $\Gamma$ on $M=\mathbb{R}$ defined in \eqref{eqn:pseudogroupexample1}. Let $x$ be the coordinate on $M$. The pseudogroup is generated by the following locally defined diffeomorphisms (that are defined where $cx+d\neq 0$):
	\begin{equation*}
	\phi:x\mapsto \frac{ax+b}{cx+d},\;\;\;\;\; a,b,c,d\in\mathbb{R}\;\;\;\text{with}\;\;\; ad-bc\neq 0.
	\end{equation*}
	The first three derivatives of $\phi$ are
	\begin{equation}
	\label{eqn:liepseudogroup2equations}
	\frac{\partial \phi}{\partial x} = \frac{ad-bc}{(cx+d)^2},\;\;\;\;\; \frac{\partial^2\phi}{\partial x^2} = -2c \frac{ad-bc}{(cx+d)^3},\;\;\;\;\;\frac{\partial^3\phi}{\partial x^3} = \frac{3}{2} \left( \frac{\partial^2 \phi}{\partial x^2} \right)^2 \left(\frac{\partial \phi}{\partial x}\right)^{-1}.
	\end{equation}
	One can check that the third equation is the defining equation of $\Gamma$, i.e. $\Gamma$ is of order 3. We must compute $J^2\Gamma$ and $J^3\Gamma$. For the former, it is not difficult to show that $J^2\Gamma=J^2M$, on which we have the coordinates 
	\begin{equation*}
	J^2\Gamma = J^2M= \{\; (X,x,u,v)\;|\; X,x,u,v\in \mathbb{R},\; u\neq 0 \;\},
	\end{equation*}
	where a jet $j^2_x\phi$ is mapped to the coordinates $(\phi(x),x,\frac{\partial \phi}{\partial x}(x),\frac{\partial^2 \phi}{\partial x^2}(x))$. The source map is $s(X,x,u,v)=x$, and hence the Lie algebroid $A^2\Gamma=A^2M$ has a global frame
	\begin{equation*}
	\partial_X(x) := \frac{\partial}{\partial X}(x,x,1,0),\;\;\;\partial_u(x) := \frac{\partial}{\partial u}(x,x,1,0),\;\;\;\partial_v(x) := \frac{\partial}{\partial v}(x,x,1,0),
	\end{equation*}
	and the bracket is readily computed to be 
	\begin{equation*}
	[\partial_X,\partial_u] = 0,\;\;\;[\partial_X,\partial_v]=0,\;\;\;[\partial_u,\partial_v] = \partial_v,
	\end{equation*}
	and the anchor is
	\begin{equation*}
	\rho:A^2M\to TM,\hspace{1cm} \partial_X \mapsto \frac{\partial}{\partial x},\;\; \partial_u\mapsto 0,\;\;\partial_v\mapsto 0.
	\end{equation*}
	Turning to $J^3\Gamma$, the third equation in \eqref{eqn:liepseudogroup2equations} shows that each jet in $J^2\Gamma$ uniquely extends to a jet in $J^3\Gamma$. This implies that there is an isomorphism of Lie groupoids given by the projection $\pi:J^3\Gamma\xrightarrow{\simeq} J^2\Gamma$ ($\Gamma$ is of finite type). Next, one readily computes the Cartan form $\omega\in\Omega^1(J^3\Gamma;t^*A^2\Gamma)$, which takes the following form:
	\begin{equation*}
	\omega = (dX-udx)\,t^*  \partial_X + \frac{1}{u}(du-vdx)\,t^*  \partial_u + \frac{1}{u^2}(dv - \frac{v}{u} du-\frac{1}{2}\frac{v^2}{u}dx)\,t^*  \partial_v
	\end{equation*}
	(it is remarkable that the formulas for the components of the Cartan form precisely coincide with formulas that Cartan obtains using various tricks and manipulations, e.g. see \eqref{eqn:sl2rexamplethirdform}). The Spencer operator $D:\Gamma(A^3\Gamma)\to \Omega^1(M;A^2\Gamma)$ is
	\begin{equation*}
	D: \partial_X\mapsto 0,\;\;\; \partial_u\mapsto -dx\otimes\partial_X,\;\;\; \partial_v\mapsto -dx\otimes\partial_u. 
	\end{equation*}
	
	With this data, we can compute the induced Cartan algebroid and its realization. First,
	\begin{equation*}
	\CAlg=TM\oplus A^2\Gamma,
	\end{equation*}
	for which we take the frame (as before, we make these choices to conform with Cartan's choices)
	\begin{equation*}
	e^1 = -\partial_X,\;\;\; e^2=\partial_u,\;\;\; e^3=-\partial_v,\;\;\; e^4= \frac{\partial}{\partial x} + \partial_X.
	\end{equation*}
	In this example, $\Tab=0$. The bracket on $\CAlg$ is canonical, since there is no choice in splitting the projection from $A^3\Gamma$ to $A^2\Gamma$. Thus, the connection \eqref{eqn:cconnection} coincides with the Spencer operator $D$ and the associated torsion \eqref{eqn:ctorsion} is determined by
	\begin{equation*}
	c(\partial_X,\partial_u) = \partial_X,\;\;\; c(\partial_X,\partial_v) = \partial_u,\;\;\; c(\partial_u,\partial_v) = \partial_v.
	\end{equation*}
	From this, we compute the bracket of $\CAlg$,
	\begin{equation*}
	\begin{split}
	&[e^1,e^2] = e^1,\hspace{1cm} [e^1,e^3] = e^2,\hspace{1cm} [e^1,e^4]=0,\\
	&[e^2,e^3] = e^3,\hspace{1cm} [e^2,e^4] = 0,\hspace{1.1cm} [e^3,e^4]= 0,
	\end{split}
	\end{equation*}
	and the anchor
	\begin{equation*}
	\rho:\CAlg\to TM,\hspace{1cm} e^1\mapsto 0,\;\;e^2\mapsto 0,\;\; e^3\mapsto 0,\;\; e^4\mapsto \frac{\partial}{\partial x}.
	\end{equation*}
	The induced realization $(J^3\Gamma,\Omega)$ of $(\CAlg,0)$ consists of the target map $t:J^3\Gamma\to M$ and the extended Cartan form $\Omega=(dt,\omega)$, which, in terms of our choice of a frame, decomposes as
	\begin{equation*}
	\Omega = \omega_1\,t^*e^1 + \omega_2\,t^*e^2 + \omega_3\,t^*e^3 + \omega_4\,t^*e^4,
	\end{equation*}
	with
	\begin{equation*}
	\omega_1 = udx,\;\;\; \omega_2 = \frac{1}{u}(du-vdx),\;\;\;\omega_3 = -\frac{1}{u^2}(dv-\frac{v}{u}du-\frac{1}{2}\frac{v^2}{u}dx),\;\;\;\omega_4=dX.
	\end{equation*}
	In this case, $\Pi=0$, and 
	\begin{equation*}
	\Omega:J^3\Gamma\xrightarrow{\simeq} t^*\CAlg\hspace{1cm}\text{and}\hspace{1cm} d\Omega+\frac{1}{2}[\Omega,\Omega]=0,
	\end{equation*}
	or, in terms of components, $\omega_1,\omega_2,\omega_3,\omega_4$ is a coframe of $J^3\Gamma$ and
	\begin{equation*}
	\begin{split}
	&d\omega_1 + \omega_1\wedge\omega_2 = 0, \\
	&d\omega_2 + \omega_1\wedge\omega_3 =0, \\
	&d\omega_3 + \omega_2\wedge\omega_3  = 0, \\
	&d\omega_4  = 0.
	\end{split}
	\end{equation*}
	Restricting to the complete transversal $X=0$, we have that $\omega_4=0$, and we recover Cartan's forms and structure equations. 

\subsubsection{Example 2 - Cartan}
\label{example:infinitecartan}

	Cartan: \textit{``Let $\Gamma$ be the pseudogroup on $\mathbb{R}^2$ whose elements are given by
		\begin{equation}
		\label{eqn:example5}
		X=f(x),\;\;\; Y=\frac{y}{f'(x)},
		\end{equation}
		where $f$ is an arbitrary function of $x$ and $f'$ its derivative (nowhere vanishing). The defining equations are
		\begin{equation*}
		dX = \frac{y}{Y}dx ,\;\;\; dY = udx + \frac{Y}{y} dy =: \omega_2,
		\end{equation*}
		they are of 1st order. We set $Y=1$ on the right hand side of both equation, and obtain
		\begin{equation}
		\label{eqn:example2formscartan}
		\omega_1 = ydx,\;\;\; \omega_2 = udx + \frac{1}{y}dy,
		\end{equation}
		with the structure equations 
		\begin{equation*}
		d\omega_1 = \omega_2 \wedge \omega_1,\;\;\; d\omega_2 = \pi\wedge\omega_1,
		\end{equation*}
		where $\pi = \frac{1}{y}du \; (\mathrm{mod}\;dx)$. We remark here that the pseudogroup $\Gamma$ is the isomorphic prolongation of the pseudogroup $X=f(x)$, where the defining equation is $dX=udx$, with
		\begin{equation}
		\omega_1 = udx,\;\;\; d\omega_1 = \pi\wedge \omega_1. \text{''}
		\end{equation}}
	
Here, Cartan starts with a pseudogroup $\Gamma$ on $\mathbb{R}^2$ (with coordinates $x,y$), or, to be more precise, on $\mathbb{R}^2\backslash\{y=0\}$ (otherwise the equations are ill-defined). He then constructs a realization on $\mathbb{R}^3\backslash\{y=0\}$ (with coordinates $x,y,u$) consisting of the 1-forms $\omega_1,\omega_2$. To write the structure equations, he introduces the auxiliary form $\pi$. A computation now shows that the induced pseudogroup in normal form $\widetilde{\Gamma}$ on $\mathbb{R}^3\backslash\{y=0\}$ is generated by
	\begin{equation}
	\label{eqn:example2normalform}
	\bar{x}=f(x),\;\;\; \bar{y}=\frac{y}{f'(x)},\;\;\; \bar{u}=\frac{uf'(x) + f''(x)}{(f'(x))^2}, \;\;\;\;\; f\in \text{Diff}_{\text{loc}}(\mathbb{R}).\qedhere
	\end{equation}  
	As before, it is clearly an isomorphic prolongation of $\Gamma$. 

In this example (as in the previous one), Cartan simplifies the expressions by setting the target variable $Y$ to the fixed value $1$. This is an instance of the simplification obtained by \textit{restricting to a complete transversal}, as explained in Section \ref{section:restrictingtotransversal}. Cartan uses this simplification to reduce the dimension of the space on which the isomorphic prolongation acts, thus obtaining a smaller isomorphic prolongation. However, one may also skip this simplification to obtain the canonical prolongation. Indeed, prior to the simplification of setting $Y=1$, we had the 1-forms 
	\begin{equation}
	\label{eqn:example2forms1}
	\omega_1 =  \frac{y}{Y}dx,\;\;\; \omega_2 =  udx + \frac{Y}{y}dy.
	\end{equation}
	Adding to this data the projection functions 
	\begin{equation*}
	I_1=X,\;\;\; I_2=Y,
	\end{equation*}
	and their differentials 
	\begin{equation}
	\label{eqn:example2forms2}
	\omega_3=dX,\;\;\; \omega_4=dY,
	\end{equation}
	the structure equations are
	\begin{equation*}
	d\omega_1=\frac{1}{Y}(\omega_2-\omega_4)\wedge\omega_1 ,\;\;\; d\omega_2= \frac{1}{Y}\omega_4\wedge\omega_2 + \pi\wedge\omega_1 ,\;\;\; d\omega_3=0,\;\;\; d\omega_4=0,
	\end{equation*}
	with 
	\begin{equation*}
	\pi = \frac{Y}{y} du - \frac{u}{y}dY  \mod{dx}.
	\end{equation*}
	The isomorphic prolongation on $\mathbb{R}^5\backslash\{y=0 \text{ or }Y=0\}$ (with coordinates $x,y,X,Y,u$) is
	\begin{equation*}
	\begin{split}
	& \bar{x} = f(x),\;\; \bar{y}=\frac{y}{f'(x)},\;\; \bar{X} = X,\;\; \bar{Y} = Y,\;\;\bar{u} = \frac{uf'(x)+Yf''(x)}{(f'(x))^2}, \;\;\; f\in \text{Diff}_{\text{loc}}(\mathbb{R}).
	\end{split}
	\end{equation*}
	The restriction to the orbit $\{X=0,\; Y=1\}$ is precisely Cartan's isomorphic prolongation. 
	
\subsubsection{Example 2 - Revisited}
\label{example:infiniterevisited}

Consider again the pseudogroup $\Gamma$ on $M=\mathbb{R}^2\backslash\{y=0\}$ defined in \eqref{eqn:example5}. Let $(x,y)$ be coordinates on $M$. The pseudogroup is generated by 
	\begin{equation}
	\label{eqn:liepseudogroup1pseudogroup}
	\phi:(x,y)\mapsto (\phi_x(x,y),\phi_y(x,y)) = (f(x),\frac{y}{f'(x)}),\;\;\;\;\; f\in \text{Diff}_{\text{loc}}(\mathbb{R}).
	\end{equation}
	The pseudogroup is transitive and, hence, $J^0\Gamma=J^0M$ and $A^0\Gamma=A^0M$. On $J^0\Gamma$ we have coordinates 
	\begin{equation*}
	J^0\Gamma = \{ (X,Y,x,y)\;|\;y\neq 0 \text{ and } Y\neq 0\},
	\end{equation*}
	where $(x,y)$ are the source coordinates and $(X,Y)$ the target. We denote the induced frame on the Lie algebroid $A^0\Gamma$ by
	\begin{equation*}
	\partial_X,\partial_Y\in\Gamma(A^0\Gamma).
	\end{equation*}
	The bracket and anchor of $A^0\Gamma$ are
	\begin{equation*}
	[\partial_X,\partial_Y]=0\hspace{1cm}\text{and}\hspace{1cm} \rho:A^0\Gamma\to TM,\;\;\; \partial_X\mapsto \frac{\partial}{\partial x},\; \partial_Y\mapsto \frac{\partial}{\partial y}.
	\end{equation*}
	The first derivatives of the elements of $\Gamma$ are
	\begin{equation}
	\label{eqn:liepseudogroup1equations}
	\begin{split}
	&\frac{\partial \phi_x}{\partial x} = \frac{y}{\phi_y},\;\;\;\;\;\;\;\;\;\;\;\;\;\;\;\;\;\frac{\partial \phi_x}{\partial y}=0, \\
	&\frac{\partial \phi_y}{\partial x} = -\frac{f''(x)y}{(f'(x))^2},\;\;\;\;\; \frac{\partial \phi_y}{\partial y}=\frac{\phi_y}{y},
	\end{split}
	\end{equation}
	from which we deduce that 
	\begin{equation*}
	J^1\Gamma = \{ (X,Y,x,y,u)\;|\;y\neq 0 \text{ and } Y\neq 0\},
	\end{equation*}
	where a jet $j^1_{(x,y)}\phi$ is assigned the coordinates $(\phi_x(x,y),\phi_y(x,y),x,y,\frac{\partial \phi_y}{\partial x}(x,y))$. The source map is given by $s(X,Y,x,y,u) = (x,y)$, and so the Lie algebroid $A^1\Gamma$ has a global frame $e_X,e_Y,e_u$, where
	\begin{gather*}
	e_X(x,y) = \frac{\partial}{\partial X}(x,y,x,y,0),\;\;\;e_Y(x,y) = \frac{\partial}{\partial Y}(x,y,x,y,0), \\ e_u(x,y) = \frac{\partial}{\partial u}(x,y,x,y,0).
	\end{gather*}
	The projection is given by
	\begin{equation}
	d\pi:A^1\Gamma\to A^0\Gamma,\;\;\;\;\; e_X\mapsto \partial_X,\;e_y\mapsto \partial_Y,\; e_u\mapsto 0.
	\label{eqn:example2revisit-projection}
	\end{equation}
	Next, computing the Cartan form $\omega\in\Omega^1(J^1\Gamma;t^*A^0\Gamma)$, we readily find that
	\begin{equation*}
	\omega = (dX-\frac{y}{Y}dx)\,t^* \partial_X + (dY-udx - \frac{Y}{y}dy)\,t^*  \partial_Y,
	\end{equation*}
	and the Spencer operator $D:\Gamma(A^1\Gamma)\to \Omega^1(M;A^0\Gamma)$ takes the form
	\begin{equation*}
	D: e_X\mapsto 0,\;\;\; e_Y\mapsto \frac{1}{y} (dx\otimes \partial_X-dy\otimes\partial_Y),\;\;\; e_u\mapsto - dx\otimes \partial_Y. 
	\end{equation*}
	
	With this data, we can set out to compute the induced Cartan algebroid and its realization. First, for the Cartan algebroid,
	\begin{equation*}
	\CAlg=TM\oplus A^0\Gamma.
	\end{equation*}
	One natural frame for this vector bundle would be $\frac{\partial}{\partial x},\frac{\partial}{\partial y},\partial_X,\partial_Y$. Cartan's formulas, however, correspond to the frame
	\begin{equation*}
	e^1 = -\partial_X,\;\; e^2 = - \partial_Y,\;\; e^3 = \partial_X + \frac{\partial}{\partial x},\;\; e^4 = \partial_Y + \frac{\partial}{\partial y}.
	\end{equation*}
	For the tableau bundle
	\begin{equation*}
	\Tab=\Ker(d\pi:A^1\Gamma\to A^0\Gamma),
	\end{equation*}
	we choose the frame $t=e_u$. The bracket of $\CAlg$ depends on a choice of a splitting of \eqref{eqn:example2revisit-projection}, for which we choose
	\begin{equation*}
	\xi:A^0\Gamma\to A^1\Gamma,\;\;\;\;\; \partial_X\mapsto e_X,\;\partial_Y\mapsto e_Y.
	\end{equation*}
	The induced connection $\nabla^\xi$ on $A^0\Gamma$ defined in \eqref{eqn:cconnection} is given by
	\begin{equation*}
	\nabla^\xi_{\partial/\partial x}(\partial_X) = 0,\;\;\;\nabla^\xi_{\partial/\partial y}(\partial_X) = 0,\;\;\; \nabla^\xi_{\partial/\partial x}(\partial_Y) = \frac{1}{y} \partial_X,\;\;\; \nabla^\xi_{\partial/\partial y}(\partial_Y) = -\frac{1}{y} \partial_Y,\;\;\; 
	\end{equation*}
	and the torsion $c^\xi$ of $\nabla^\xi$ defined in \eqref{eqn:ctorsion} is given by
	\begin{equation*}
	c^\xi(\partial_X,\partial_Y) = -\frac{1}{y}\partial_X.
	\end{equation*}
	The bracket \eqref{eqn:cbracketdef} on $\CAlg$ is then
	\begin{equation*}
	\begin{split}
	&[e^1,e^2]=\frac{1}{y}e^1,\hspace{1cm} [e^1,e^3] = 0,\hspace{1cm} [e^1,e^4] = -\frac{1}{y}e^1,\\ 
	&[e^2,e^3] = 0,\hspace{1.4cm} [e^2,e^4] = \frac{1}{y}e^2,\hspace{0.55cm} [e^3,e^4] = 0,
	\end{split}
	\end{equation*}
	and the anchor \eqref{eqn:canchor} is
	\begin{equation*}
	\rho: \CAlg\to TM,\hspace{1cm} e^1\mapsto 0,\;\;e^2\mapsto 0, \;\;e^3\to \frac{\partial}{\partial x}, \;\;e^4\mapsto \frac{\partial}{\partial y}.
	\end{equation*}
	The action of $\Tab$ on $\CAlg$ is
	\begin{equation*}
	t(e^1) = e^2,\;\;\; t(e^2) = 0,\;\;\; t(e^3) = 0,\;\;\; t(e^4) = 0.
	\end{equation*}
	Thus, writing $[e^j,e^k]=\sum_{i=1}^4c^{jk}_ie^i$ and $t(e^j)=\sum_{i=1}^4a_i^je^i$, the non-zero structure functions are
	\begin{equation*}
	c^{12}_1=\frac{1}{y},\;\; c^{14}_1=-\frac{1}{y},\;\; c^{23}_2=\frac{1}{y},\;\; a_2^1=1. 
	\end{equation*}
	
	Finally, we describe the realization $(J^1\Gamma,\Omega)$ of $(\CAlg,\Tab)$. The realization consists of the target map $t:J^1\Gamma\to M$ and the extended Cartan form $\Omega=(dt,\omega)$. In terms of the frame of $\CAlg$, $\Omega$ decomposes as
	\begin{equation*}
	\Omega = \omega_1\,t^*  e_1 + \omega_2\,t^*  e_2 + \omega_3\,t^*  e_3 + \omega_4\,t^*  e_4,
	\end{equation*}
	with
	\begin{equation*}
	\omega_1 = \frac{y}{Y}dx,\;\;\; \omega_2 = udx+\frac{Y}{y}dy,\;\;\;\omega_3 = dX,\;\;\; \omega_4=dY.
	\end{equation*}
	These are precisely the forms \eqref{eqn:example2forms1} and \eqref{eqn:example2forms2}, and, when restricting to the orbit $X=0,Y=1$, these are precisely Cartan's forms \eqref{eqn:example2formscartan}.

\section{The Systatic Space and Reduction}
\label{section:systaticspacesmodern}

Cartan's \textit{systatic system} and procedure for reducing the \textit{inessential invariants} (pp. 197-202 in \cite{Cartan1904} and pp. 18-24 in \cite{Cartan1937-1}) are probably the least understood parts of his work on Lie pseudogroups. The language of Cartan algebroids and realizations allows us not only to gain a better understanding of this side of Cartan's work but also to go a step further and to obtain an improved reduction procedure. This, of course, comes at the cost of an increased level of abstraction. In particular, Lie-Pfaffian groupoids and generalized pseudogroups will play an important role. 

This section is divided into three parts. First, we show that any Cartan algebroid has an intrinsic Lie algebroid lying within it which we call the systatic space, and we prove that this Lie algebroid acts canonically on all realizations. Cartan's \textit{systatic system} arises as the involutive distribution at the image of this action. Next, we describe a reduction procedure that allows us to replace any pseudogroup in normal form by a Cartan equivalent ``reduced'' \textit{generalized} pseudogroup, namely one that acts on a space of smaller dimension. Roughly speaking, the procedure amounts to ``dividing out'' by the action of the systatic space. We conclude the section with two examples of the reduction procedure. 

\subsection{The Systatic Space of a Cartan Algebroid}
\label{section:systatic}

In Section \ref{section:cartansrealizationproblem}, we saw that, in Cartan's langauge, a pseudogroup in normal form is a pseudogroup that preserves the data of a realization, namely a family of functions and $1$-forms that satisfy the set of structure equations
\begin{equation*}
d\omega_i + \frac{1}{2} c_i^{jk} \omega_j\wedge \omega_k =  a_i^{\lambda j} \pi_\lambda \wedge \omega_j.
\end{equation*} 
In \cite{Cartan1937-1} (pp. 18-19), by examining the stabilizer groups of a pseudogroup in normal form, Cartan shows that the set of equations
\begin{equation}
\label{eqn:systaticsystem}
a^i_{\lambda j}\omega^j=0,\hspace{1cm} i=1,...,m, \;\lambda=1,...,p
\end{equation}
defines an involutive distirbution on the total space of the realization that consists of all vectors that are invariant under the infinitesimal generators of the stabilizer groups. He calls \eqref{eqn:systaticsystem} the \textit{systatic system}. The modern picture of Cartan algebroids and realizations allos us to clarify the geometric structure that is hidden behind \eqref{eqn:systaticsystem}, as we now explain

\subsubsection*{The Partial Systatic Space}

While the axioms of a Cartan algebroid (Definition \ref{def:cartanalgebroid}) may seem rather obscure, as they require the existence of $t$ and $\nabla$ which themselves are not part of the structure, there are several interesting consequences that are of an intrinsic nature and that turn out to be intimately related to Cartan's systatic system. 

\begin{mydef}
	\label{def:partialsystaticspace}
	\index{systatic space}
	The \textbf{partial systatic space} of a Cartan algebroid $(\CAlg,\Tab)$ over $N$ is the set-theoretical vector subbundle $\mathcal{S}^0\subset \CAlg$ whose fiber at $x\in N$ is
	\begin{equation}
	\label{eqn:systaticspace}
	\mathcal{S}^0_x := \{\; u\in\CAlg_x\;|\; T(u) = 0\;\;\forall \; T\in\Tab_x \;\} .
	\end{equation}
\end{mydef}

Note that $\mathcal{S}^0$ can be expressed as the kernel of the vector bundle map 
\begin{equation*}
\CAlg \to \Hom(\Tab,\CAlg),\hspace{1cm} u\mapsto (T\mapsto T(u)),
\end{equation*}
and hence it is a smooth vector subbundle if and only if it is of constant rank. 

\begin{myassum}
	We will assume that the partial systatic spaces of all Cartan algebroids appearing from now on are of constant rank.
\end{myassum}

As a consequence of Definition \ref{def:cartanalgebroid} of a Cartan algebroid, or more specifically, of conditions 2 and 3 of the definition:

\begin{myprop}
	\label{prop:systaticspaceliealgebroid}
	The systatic space $\mathcal{S}^0$ of a Cartan algebroid $(\CAlg,\Tab)$, equipped with the bracket and anchor inherited from $\CAlg$, is a Lie algebroid.
\end{myprop}

\begin{proof}
	Using the fact that all elements of $\mathcal{S}$ are killed by all elements of $\Tab$, \eqref{eqn:cartanalgebroidconditionactiong} implies that $\mathcal{S}$ is closed under the restriction of the bracket of $\CAlg$ and \eqref{eqn:cartanalgebroidconditionjacobi} reduces to the Jacobi identity.
\end{proof}

\subsubsection*{The Systatic Space}

Given $u,v \in\Gamma(\mathcal{S}^0)$, consider
\begin{equation*}
\text{J}_{u,v}:\CAlg\to \CAlg,
\end{equation*}
the vector bundle map that is defined at the level of sections by
\begin{equation*}
\text{J}_{u,v}(\alpha) = [[u,v],\alpha] + [[v,\alpha],u] + [[\alpha,u],v],\hspace{1cm} \forall \, \alpha\in\Gamma(\CAlg).
\end{equation*}
Given $u\in\Gamma(\mathcal{S}^0)$ and $T\in\Gamma(\Tab)$, consider also
\begin{equation*}
\text{Ad}_u(T) : \CAlg\to \CAlg,
\end{equation*} 
the vector bundle map that is defined at the level of sections by
\begin{equation*}
\text{Ad}_u(T)(\alpha) := [u,T(\alpha)] - T([u,\alpha]),\hspace{1cm} \forall\, \alpha\in\Gamma(\CAlg).
\end{equation*}
Another consequence of conditions 2 and 3 of Definition \ref{def:cartanalgebroid} is:
\begin{myprop}
	Let $(\CAlg,\Tab)$ be a Cartan algebroid. For all $u,v\in\Gamma(\mathcal{S}^0)$ and $T\in\Tab$,
	\begin{equation*}
	\text{J}_{u,v}\in\Gamma(\Tab) \hspace{1cm}\text{and}\hspace{1cm} \text{Ad}_u(T)\in\Gamma(\Tab).
	\end{equation*}
\end{myprop}

\begin{proof}
	Choose $t$ and $\nabla$ as in definition \ref{def:cartanalgebroid} of a Cartan algebroid. By \ref{eqn:cartanalgebroidconditionjacobi},
	\begin{equation*}
	\text{J}_{u,v}(\alpha) = t_{u,v}(\alpha),
	\end{equation*}
	and by definition \ref{eqn:cartanalgebroidconditionactiong},
	\begin{equation*}
	\text{Ad}_u(T) = \nabla_u(T). \qedhere
	\end{equation*}
\end{proof}

Thus, while the $\CAlg$-connection $\nabla:\Gamma(\CAlg)\times\Gamma(\Tab)\to\Gamma(\Tab)$ and the vector bundle map $t:\Lambda^2\CAlg\to \Tab$ on a Cartan algebroid $(\CAlg,\Tab)$ are non-canonical, their restrictions to $\mathcal{S}^0$, i.e. the $\mathcal{S}^0$-connection
\begin{equation*}
\text{Ad}:\Gamma(\mathcal{S}^0)\times\Gamma(\Tab)\to \Gamma(\Tab)
\end{equation*}
and 
\begin{equation*}
\text{J}:\Lambda^2\mathcal{S}^0\to \Tab,
\end{equation*}
are canonical.

The two previous propositions can be neatly packaged in a single object by using the following standard construction of a non-abelian extension of a Lie algebroid (see also \cite{Mackenzie1987}, Chapter 4, Section 3). 

\begin{mydef}
	\label{def:systaticspace}
	The \textbf{systatic space} of a Cartan algebroid $(\CAlg,\Tab)$ is the vector bundle
	\begin{equation*}
	\mathcal{S} := \mathcal{S}^0 \oplus \Tab \;\subset\; \CAlg\oplus\Tab.
	\end{equation*}
\end{mydef}

We equip $\mathcal{S}$ with the structure of a Lie algebroid. The anchor $\rho:\mathcal{S}\to TN$ is defined by
\begin{equation*}
\hspace{1cm} \rho(u,T) := \rho(u),\hspace{1cm} \forall\; u\in\Gamma(\mathcal{S}^0),\;T\in\Gamma(\Tab),
\end{equation*}
and the bracket $[\cdot,\cdot]:\Gamma(\mathcal{S})\times \Gamma(\mathcal{S})\to \Gamma(\mathcal{S})$ by the formula
\begin{equation*}
[(u,S),(v,T)]:= ([u,v],\text{J}_{u,v}  + \text{Ad}_u(T)-\text{Ad}_v(S) -[S,T]).
\end{equation*}
The prove that this indeed defines a Lie algebroid structure on $\mathcal{S}$ relies on the following lemma. We denote by $d_{\text{Ad}}$ the de Rham-like operator on $\Omega^*(\mathcal{S}^0;\Tab) = \Hom(\Lambda^*\mathcal{S}^0,\Tab)$ that is induced by the $\mathcal{S}^0$-connection $\text{Ad}$.  

\begin{mylemma}
	\label{eqn:lemmaliealgebroidextensionproperties}
	For all $u,v\in\Gamma(\mathcal{S}^0),\;S,T\in\Gamma(\Tab)$,
	\begin{enumerate}
		\item $d_{\text{Ad}}\text{J}  = 0$,
		\item $\text{Ad}_u (\text{Ad}_v(S)) - \text{Ad}_v (\text{Ad}_u(S)) - \text{Ad}_{[u,v]}(S) = [S,\text{J}_{u,v}]$,
		\item $\text{Ad}_u([S,T]) = [\text{Ad}_u(S),T] + [S,\text{Ad}_u(T)]$.
	\end{enumerate}
\end{mylemma}

\begin{proof}
	We explicitly compute the first identity. Let $u,v,w\in\Gamma(\mathcal{S}^0)$ and $\alpha\in\Gamma(\CAlg)$,
	\begin{equation*}
	\begin{split}
	d_{\text{Ad}}\text{J}(u,v,w)(\alpha) & = \text{Ad}_u(J_{v,w})(\alpha)  - J_{[u,v],w}(\alpha) + \text{ cyclic permutations of $u,v,w$} \\
	& = [u, [[v,w],\alpha] ] + [u, [[w,\alpha],v] ] + [u, [[\alpha,v],w] ] \\
	&\hspace{1cm} - [[v,w],[u,\alpha]] - [[[u,\alpha],v],w] - [[w,[u,\alpha]],v] \\
	&\hspace{1cm} - [[[u,v],w],\alpha] -  [[w,\alpha],[u,v]] - [[\alpha,[u,v]],w] \\ & \hspace{1cm} + \text{ cyclic permutations of $u,v,w$}.
	\end{split}
	\end{equation*}
	The $7$th term (together with its cyclic permutations) vanishes by the Jacobi identity of $\mathcal{S}^0$ and all other terms cancel pairwise. The other identities are dealt with similarly. We only point out that the remaining two identities do not rely on the Jacobi identity of $\mathcal{S}^0$.
\end{proof}

\begin{myprop}
	\label{prop:1stsystaticspaceliealgebroid}
	The systatic space $\mathcal{S}$ of a Cartan algebroid $(\CAlg,\Tab)$ is a Lie algebroid. 	
\end{myprop}

\begin{proof}
	The proof amounts to computing the Jacobiator of $\mathcal{S}$ and showing that the Jacobi identity holds if and only if the three identities in Lemma \ref{eqn:lemmaliealgebroidextensionproperties} are satisfied. This is a straightforward computation (see also \cite{Mackenzie1987}, Theorem 3.20).
\end{proof}

\subsubsection*{The Systatic Action on Realizations}
\label{section:systaticactionrealization}

The most important property of the systatic space $\mathcal{S}$ is that it acts canonically on all realizations of $(\CAlg,\Tab)$ in the following sense:

\begin{myprop}
	\label{prop:systaticspace1action}
	Let $(P,\Omega)$ be a realization of a Cartan algebroid $(\CAlg,\Tab)$. The map
	\begin{equation}
	\label{eqn:systaticspaceaction}
	a:=(\Omega,\Pi)^{-1}|_{I^*\mathcal{S}}:  I^*\mathcal{S} \to TP
	\end{equation}
	defines a canonical Lie algebroid action of $\mathcal{S}$ on $I:P\to N$ (thus, independent of the choice of $\Pi$) and its image is the involutive distribution
	\begin{equation}
	\label{eqn:systaticdistrbution}
	\{\; X\in TP\;|\; \Omega(X) \in \mathcal{S}^0  \;\}.
	\end{equation}
\end{myprop}

\begin{myremark}
	Using the local coordinate description of a realization from Example \ref{example:realizationcartan}, we immediately see that, locally, \eqref{eqn:systaticdistrbution} coincides with Cartan's systatic system \eqref{eqn:systaticsystem}. 
\end{myremark}

\begin{proof}
	By Lemma \ref{lemma:realizationsecondcomponentcanonical}, we already know that the map $S\in\Gamma(\Tab)\mapsto X_S\in\mathfrak{X}(P)$ is canonical. Fix a choice of $\Pi\in\Omega^1(P;I^*\Tab)$ as in Definition \ref{def:realization} and consider the induced map $\alpha\in\Gamma(\CAlg)\mapsto X_\alpha\in\mathfrak{X}(P)$. We would like to show that its restriction $u\in\Gamma(\mathcal{S}^0)\mapsto X_u\in\mathfrak{X}(P)$ is canonical, or equivalently, that $\Pi'(X_u)=0$ for any other choice of $\Pi'$ and $u\in\Gamma(\mathcal{S}^0)$. This follows from the structure equations (c.f. the proof of Lemma \ref{lemma:realizationsecondcomponentcanonical}): for any $\alpha\in\Gamma(\CAlg)$,
	\begin{equation*}
	\begin{split}
	0 & = \big( (\Pi'-\Pi)\wedge\Omega \big)(X_u,X_\alpha) \\ 
	& = \Pi'(X_u)(\Omega(X_\alpha)) - \Pi'(X_\alpha)(\Omega(X_u)) \\
	& = \Pi'(X_u)(I^*\alpha) - \cancel{\Pi'(X_\alpha)(I^*u)}. 
	\end{split}
	\end{equation*}
	
	Next, we prove that \eqref{eqn:systaticspaceaction} defines an action. Let $(u,S),(v,T)\in\Gamma(\mathcal{S})$ and let us write $X_{(u,S)} := X_u + X_S$ and $X_{(v,T)}:=X_v+X_T$. We claim that $[X_{(u,S)},X_{(v,T)}] = X_{[(u,S),(v,T)]}$. On the one hand, by the definition of the bracket of $\mathcal{S}$,
	\begin{equation*}
	\begin{split}
	& \Omega\big(X_{[(u,S),(v,T)]}\big) = I^*[u,v], \\
	& \Pi\big(X_{[(u,S),(v,T)]}\big) = I^*\big(\text{J}_{u,v} + \text{Ad}_u(T)-\text{Ad}_v(S) - [S,T] \big).
	\end{split}
	\end{equation*}
	On the other hand, by Lemma \ref{lemma:structureequationsaction} (which follows from the structure equations),
	\begin{equation*}
	\Omega\big( [X_{(u,S)},X_{(v,T)}] \big) = I^*[u,v],
	\end{equation*}
	Furthermore, note that for any $\xi\in\Gamma(\Tab^{(1)})$ and $\xi'\in\Gamma(Z^{0,2}(\Tab))$,
	\begin{equation*}
	\xi(u)(\alpha) = \xi(\alpha)(u)=0,\hspace{1cm} \xi'(u,v)(\alpha) = - \xi'(v,\alpha)(u) - \xi'(\alpha,u)(v) = 0,
	\end{equation*}
	for all $u,v\in\Gamma(\mathcal{S}^0)$ and $\alpha\in\Gamma(\CAlg)$. So, by Corollary \ref{cor:realizationimpliescartanalgebroid} (which also uses the structure equations),
	\begin{equation*}
	\Pi\big( [X_{(u,S)},X_{(v,T)}] \big) = I^*\big(\text{J}_{u,v} + \text{Ad}_u(T)-\text{Ad}_v(S) - [S,T] \big).
	\end{equation*}
	
	The last assertion concerning the image of the action is immediate.
\end{proof}

\begin{mydef}
	\label{def:systaticaction}
	Let $(P,\Omega)$ be a realization of a Cartan algebroid $(\CAlg,\Tab)$ over $N$. The action of the systatic space $\mathcal{S}$ of $(\CAlg,\Tab)$ on $I:P\to N$ is called the \textbf{systatic action}. 
\end{mydef}

\begin{myexample} (Lie groups)
	\label{example:liegroupssystatic}
	Let us consider the example of a realization coming from a Lie group $G$ with Lie algebra $\mathfrak{g}$ (see Example \ref{example:liegroupsaspseudogroups}), where $(G,\Omega_{\mathrm{MC}})$ with $\Omega_{\mathrm{MC}}\in\Omega^1(G;\mathfrak{g})$ the Maurer-Cartan form on $G$ is a realization of the Cartan algebroid $(\mathfrak{g},0)$. In this case: 
	\begin{equation*}
	\mathcal{S} = \mathcal{S}^0=\CAlg=\mathfrak{g}.
	\end{equation*}
	The systatic action is simply the inverse of the Maurer-Cartan form, which, at the level of sections, is given by the map
	\begin{equation}
	\label{eqn:exampleliegroupaction}
	\mathfrak{g}\to \mathfrak{X}(G),\hspace{1cm} X\mapsto X^R,
	\end{equation}
	which sends an element of the Lie algebra $X\in\mathfrak{g}$ to the induced right invariant vector field $X^R\in\mathfrak{X}(G)$, which is defined by $(X^R)_g:= (dR_g)_e(X)$. The pseudogroup $\Gamma(G,\Omega_{\text{MC}})$ is generated by right translations, and it is also characterized as the pseudogroup on $G$ consisting of all local diffeomorphisms that are invariant under the systatic action. By this we mean that $\psi\in\LDiff(G)$ belongs to $\Gamma(G,\Omega_{\text{MC}})$ if and only if
	\begin{equation*}
	d\psi(X^R)= X^R,\hspace{1cm} \forall\; X\in\mathfrak{g}. \qedhere
	\end{equation*}
\end{myexample}

In general, the pseudogroup $\Gamma(P,\Omega)$ is only contained in the pseudogroup of local diffeomorphisms that is invariant the systatic action:

\begin{myprop}
	\label{prop:reductioninvariancegamma}
	Let $(P,\Omega)$ be a realization of a Cartan algebroid $(\CAlg,\Tab)$ over $N$. The pseudogroup $\Gamma(P,\Omega)$ is invariant under the systatic action of $\mathcal{S}$ on $P$, i.e.
	\begin{equation*}
	d\psi(X_u)=X_u,\hspace{1cm} d\psi(X_S)=X_S,
	\end{equation*}
	for all $u\in\Gamma(\mathcal{S}^0),\;S\in\Gamma(\Tab),\;\psi\in\Gamma(P,\Omega)$.
\end{myprop}

\begin{proof}
	We need to show that
	\begin{equation*}
	\begin{split}
	& \Omega(d\psi(X_u)) = I^*u,\hspace{2cm}  \Pi(d\psi(X_u)) = 0, \\
	& \Omega(d\psi(X_S)) = 0, \hspace{2.35cm}  \Pi(d\psi(X_S)) = I^*S.
	\end{split}
	\end{equation*}
	First,
	\begin{equation*}
	\begin{split}
	& \Omega(d\psi(X_u)) = \Omega(X_u) = I^*u, \\
	& \Omega(d\psi(X_S)) = \Omega(X_S) = 0,
	\end{split}
	\end{equation*}
	because $\psi^*\Omega=\Omega$. Next, we note that the structure equation $d\Omega+\frac{1}{2}[\Omega,\Omega] = \Pi\wedge \Omega$ combined with the invariance condition $\psi^*\Omega=\Omega$ imply that $(\psi^*\Pi - \Pi)\wedge\Omega=0$. Applying this equation on pairs $(X_u,X_\alpha)$ and $(X_S,X_\alpha)$ and using the fact that $T(u)=0$ for all $T\in\Gamma(\Tab)$, we see that
	\begin{equation*}
	\begin{split}
	& (\psi^*\Pi-\Pi)(X_u)(I^*\alpha) =0,\\
	& (\psi^*\Pi-\Pi)(X_S)(I^*\alpha) =0.
	\end{split}
	\end{equation*}
	Since this is true for all $\alpha\in\Gamma(\CAlg)$, then 
	\begin{equation*}
	\begin{split}
	& \Pi(d\psi(X_u)) = \Pi(X_u) = 0, \\
	& \Pi(d\psi(X_S)) = \Pi(X_S) = I^*S. \qedhere
	\end{split}
	\end{equation*}
\end{proof}

\subsection{Reduction of a Pseudogroup in Normal Form}
\label{section:reductionofpseudogroupinnormalform}

In the previous section, we have seen that any Cartan algebroid $(\CAlg,\Tab)$ comes with an intrinsic Lie algebroid $\mathcal{S}$, which we call the systatic space (Definition \ref{def:systaticspace}). We have seen that $\mathcal{S}$ acts on all realizations $(P,\Omega)$ of the Cartan algebroid via the systatic action (Definition \ref{def:systaticaction}) and that the induced pseudogroup in normal form $\Gamma(P,\Omega)$ is invariant under this action. In this section, we study \textit{reduction}, the procedure of taking the quotient by the systatic action. For simplicity, we will do this under the assumption that the Lie algebroid $\mathcal{S}\to N$ integrates to a Lie groupoid $\Sigma\rightrightarrows N$, and that the action of $\mathcal{S}$ on $I:P\to N$ integrates to a sufficiently nice action of $\Sigma\rightrightarrows N$ on $I:P\to N$ (see Theorem \ref{theorem:reduction} for the precise assumptions). 

The key to reduction is to pass from the realization $(P,\Omega)$ to a \textit{Pfaffian groupoid} (Definition \ref{def:pfaffiangroupoid}) consisting of the gauge groupoid $\mathcal{G}(P)\rightrightarrows P$ of $I:P\to N$ together with a multiplicative $1$-form $\theta(\Omega)$ that is obtained as the natural lift of $\Omega$ to $\mathcal{G}(P)$. The action of $\Sigma\rightrightarrows N$ extends to an action on this Pfaffian groupoid. In this situation, the quotient can be taken, and we prove that the resulting object is a Lie-Pfaffian groupoid $(\mathcal{G}(P)_{\mathrm{red}},\theta(\Omega)_{\mathrm{red}})$ over the quotient manifold $P_{\mathrm{red}}= P/\Sigma$ and that the pseudogroup in normal form $\Gamma(P,\Omega)$ we started with is an isomorphic prolongation of the generalized pseudogroup $\Gamma_{\mathrm{red}} = \LBis(\mathcal{G}(P)_{\mathrm{red}},\theta(\Omega)_{\mathrm{red}})$ along the quotient map $\mathrm{pr}:P\to P_\mathrm{red}$ (Theorem \ref{theorem:reduction}). Thus, at the cost of passing to the more abstract setting of Lie-Pfaffian groupoids and generalized pseudogroups, the systatic action allows us to reduce a pseudogroup to a generalized pseudogroup that acts on a ``smaller'' space. In Section \ref{examples:reduction}, we will work out two examples of the reduction procedure.

\subsubsection*{The Pfaffian Groupoid of a Realization}
\label{section:realizationpfaffiangroupoid}

Let $(\mathcal{C},\Tab)$ be a Cartan algebroid over $N$ and $(P,\Omega)$ a realization. We denote the gauge groupoid (sometimes called the submersion groupoid) of $I:P\to N$ by
\begin{equation*}
\mathcal{G}=\mathcal{G}(P):=P\times_N P = \{ \; (p,q)\in P\times P\;|\; I(p)=I(q) \; \}\rightrightarrows P.
\end{equation*}
Let $s$ and $t$ denote the source and target maps of $\mathcal{G}$ and consider the $1$-form
\begin{equation}
\label{eqn:omegaongaugegroupoid}
\theta=\theta(\Omega):= s^*\Omega - t^*\Omega
\end{equation} 
on $\mathcal{G}$, which makes sense because $I(p)=I(q)$ for any $(p,q)\in\mathcal{G}$. A-priori $\theta$ takes values in $I^*\mathcal{C}$. However, given a vector $(X,Y)$ tangent to $\mathcal{G}$ at a point $(p,q)$, the fact that $\Omega$ is anchored ($dI=\rho\circ \Omega$) implies that $\rho(\theta(X,Y)) = \rho(\Omega(X)) - \rho(\Omega(Y)) = dI(X)-dI(Y)=0$, and hence
\begin{equation}
\label{eqn:reductionpfaffianform}
\theta\in\Omega^1(\mathcal{G};I^*E)
\end{equation}	
with
\begin{equation}
\label{eqn:reductionE}
E:=\Ker \rho\subset \CAlg.
\end{equation}
Finally, equipping $I^*E$ with the trivial action of $\mathcal{G}$, it is easy to check that $\theta$ is multiplicative. 

\begin{myprop}
	\label{prop:reductionpfaffiangroupoid}
	Let $(\CAlg,\Tab)$ be a Cartan algebroid over $N$ and $(P,\Omega)$ a realization. The pair $(\mathcal{G},\theta)$ is a Pfaffian groupoid (Definition \ref{def:pfaffiangroupoid}), and 
	\begin{equation*}
	\LBis(\mathcal{G},\theta) = \Gamma(P,\Omega),
	\end{equation*}
	where we view $\Gamma(P,\Omega)$ as a generalized pseudogroup (see Example \ref{example:pseudogroupasgeneralizedpseudogroup}).
\end{myprop}

\begin{myremark}
\label{remark:reductionnotation}
For this and later proofs, it will be useful to describe the module of vector fields on $\mathcal{G}$. Fix $\Pi$ for $(P,\Omega)$ as in the definition of a realization (Definition \ref{def:realization}), which, in turn, induces the map \eqref{eqn:realizationinversemap}. As a $C^\infty(P\times P)$-module, $\mathfrak{X}(P\times P)$ is generated by
\begin{equation*}
(X_\alpha+X_S,X_\beta+X_T),\hspace{1cm} \alpha,\beta\in\Gamma(\CAlg),\;S,T\in\Gamma(\Tab),
\end{equation*}
where $(X,Y)\in\mathfrak{X}(P\times P)$ with $X,Y\in\mathfrak{X}(P)$ denotes the vector field for which $(X,Y)_{(p,q)} = (X_p,Y_q)\in T_pP\oplus T_qP = T_{(p,q)}(P\times P)$. Since $dI=\rho\circ\Omega$, then  $\mathfrak{X}(\mathcal{G})$ is generated as a $C^\infty(\mathcal{G})$-module by 
\begin{equation*}
(X_\alpha+X_S,X_\beta+X_T),\hspace{1cm} \alpha,\beta\in\Gamma(\CAlg)\;\text{ with }\;\rho(\alpha)=\rho(\beta),\;S,T\in\Gamma(\Tab),
\end{equation*}
where by $(X_\alpha+X_S,X_\beta+X_T)$ we mean the restriction to $\mathcal{G}$. Applying $\theta$, we see that 
\begin{equation}
\label{eqn:reductiontheta}
\theta(X_\alpha+X_S,X_\beta+X_T) = (I \circ t)^*(\beta - \alpha),
\end{equation}
and hence that $\Gamma(\Ker\theta)\subset \mathfrak{X}(\mathcal{G})$ is generated by 
\begin{equation}
\label{eqn:reductionthetakernel}
(X_\alpha+X_S,X_\alpha+X_T),\hspace{1cm} \alpha\in\Gamma(\CAlg),\;S,T\in\Gamma(\Tab).
\end{equation}
\end{myremark}

\begin{proof}
	We only need to check that $\theta$ is pointwise surjective and that $\Ker \theta \cap \Ker ds$ is involutive. For the former, note that if $\alpha\in \Gamma(E)$, then $\theta(X_\alpha,0) = (I \circ t)^*\alpha$ by \eqref{eqn:reductiontheta}. For the latter, we see that $\Gamma(\Ker \theta\cap \Ker ds)$ is generated by vector fields $(X_S,0)$ with $S\in\Gamma(\Tab)$, and hence involutive because $\Ker \Omega$ is involutive (Lemma \ref{lemma:structureequationsaction}). The last assertion of the proposition follows from the definition of $\Gamma(P,\Omega)$. 
\end{proof}

\subsubsection*{The Systatic Action on the Pfaffian Groupoid}

The systatic action (Definition \ref{def:systaticaction}) extends to an action on the Pfaffian groupoid $(\mathcal{G},\theta)$ in the following sense. The Lie algebroid $\mathcal{S}$ acts on $\mathcal{G}$ via the diagonal action, which, using the notation of Remark \ref{remark:reductionnotation}, is given by
\begin{equation}
\label{eqn:reductiondiagonalaction}
\Gamma(\mathcal{S})\to \mathfrak{X}(\mathcal{G}),\hspace{1cm} (u,S)\mapsto (X_u+X_S,X_u+X_S),
\end{equation} 
for all $u\in\Gamma(\mathcal{S}^0),\, S\in\Gamma(\Tab)$. It is an action along the composition of the source (or target) map of $\mathcal{G}$ with $I:P\to N$, which is a surjective submersion that we also denote by $I:\mathcal{G} \to N$. It is well defined, i.e. takes values in $\mathfrak{X}(\mathcal{G})$, since $dI(X_\alpha+X_S) = \rho\circ\Omega(X_\alpha+X_S) = \rho(\alpha)$. 

\begin{myremark}
	The action of $\mathcal{S}$ on $P$ and $\mathcal{G}$ defines a Lie algebroid action on the Lie groupoid $\mathcal{G}\rightrightarrows P$, the infinitesimal counterpart of a Lie groupoid action on a Lie groupoid as defined in \cite{Higgins1990}, Definition 3.1.
\end{myremark}

 Furthermore:

\begin{myprop}
	\label{prop:reductionbasic}
	Let $(\CAlg,\Tab)$ be a Cartan algebroid and $(P,\Omega)$ a realization. The vector bundle $E$ (defined in \eqref{eqn:reductionE}), equipped with 
	\begin{equation}
	\label{eqn:reductions1connection}
	\nabla:\Gamma(\mathcal{S})\times \Gamma(E)\to \Gamma(E),\hspace{1cm} \nabla_{(u,S)}(\alpha) =  [u,\alpha] - S(\alpha),
	\end{equation}
	is a Lie algebroid representation of the systatic space $\mathcal{S}$. The 1-form $\theta$ defined in \eqref{eqn:reductionpfaffianform} is basic with respect to the action of $\mathcal{S}$ (Definition \ref{def:basicformliealgebroid}). 
\end{myprop}

\begin{proof}
	Clearly, $\nabla$ is an $\mathcal{S}$-connection, and it is flat since
	\begin{equation*}
	\begin{split}
		&(\nabla_{(u,S)}\nabla_{(v,T)} - \nabla_{(v,T)}\nabla_{(u,S)} - \nabla_{[(u,S),(v,T)]})(\alpha) \\
		&\hspace{1.5cm} = S(T(\alpha)) - T(S(\alpha)) - [S,T](\alpha) \\  
		&\hspace{1.7cm}  + J_{u,v}(\alpha) + [u,[v,\alpha]] - [[u,v],\alpha] - [v,[u,\alpha]]   \\
		&\hspace{1.7cm}  + T([u,\alpha]) - T([u,\alpha]) + S([v,\alpha])  - S([v,\alpha]) \\ 
		&\hspace{1.7cm} + [v,S(\alpha)] - [v,S(\alpha)] + [u,T(\alpha)]    - [u,T(\alpha)] \\
		& \hspace{1.5cm} = 0.
		\end{split}
	\end{equation*}
	Next, we show that $\theta$ is basic. First, from \eqref{eqn:reductionthetakernel} we see that all the vector fields at the image of \eqref{eqn:reductiondiagonalaction} lie in the kernel of $\theta$, which is therefore horizontal. Furthermore, $\theta$ satisfies the equivariance condition \eqref{eqn:basicequivariancealgebroid}:
	\begin{equation*}
	\begin{split}
	& (I^*\nabla)_{I^*(u,W)}\theta(X_\alpha+X_S,X_\beta+X_T)  \\ & \hspace{1.5cm}- \theta([(X_u+X_W,X_u+X_W),(X_\alpha+X_S,X_\beta+X_T)]) \\
	&\explain{=}{\theta=s^*\Omega-t^*\Omega} I^*\big(\nabla_{(u,W)}(\beta-\alpha)\big) - s^*\big(\Omega([X_u+X_W,X_\beta+X_T])\big) \\ & \hspace{1.5cm} + t^*\big(\Omega([X_u+X_W,X_\alpha+X_S])\big) \\
	&\explain{=}{\text{Lemma \ref{lemma:actionproperties}}} I^*\big([u,\beta-\alpha] - W(\beta-\alpha) - [u,\beta] - W(\beta) - [u,\alpha] + W(\alpha)\big) = 0. \qedhere
	\end{split}
	\end{equation*}
\end{proof}

\subsubsection*{Passing to the Quotient}
\label{section:reductionsystaticaction}

The final step in the reduction procedure is to pass to the quotient of the Pfaffian groupoid $(\mathcal{G},\theta)$ by the action of the systatic space $\mathcal{S}$. To this end, we make the following three regularity assumptions (in Remark \ref{remark:regularityconditionsreduction}, we explain how these assumptions can be relaxed). 

The first assumption is that the Lie algebroid $\mathcal{S}\to N$ integrates to an $s$-connected Lie groupoid $\Sigma\rightrightarrows N$. The second assumption is that the action of $\mathcal{S}$ on $I: P\to N$ integrates to a free and proper action of $\Sigma$. This implies that $\Sigma$ also acts freely and properly on $I:\mathcal{G}\to N$ via the diagonal action, and hence the orbit spaces
\begin{equation*}
P_\mathrm{red}:=P/\Sigma\hspace{0.5cm}\mathrm{and}\hspace{0.5cm} \mathcal{G}_\mathrm{red}=\mathcal{G}(P)_\mathrm{red}:=\mathcal{G}/\Sigma
\end{equation*}
have the unique smooth structure with which the projections
\begin{equation*}
\mathrm{pr}:P\to P_\mathrm{red} \hspace{0.5cm}\mathrm{and}\hspace{0.5cm} \mathrm{pr}:\mathcal{G}\to \mathcal{G}_\mathrm{red}
\end{equation*}
are surjective submersions. As a consequence,
\begin{equation*}
\mathcal{G}_\mathrm{red}\rightrightarrows P_\mathrm{red}
\end{equation*}
has the unique structure of a Lie groupoid with which the projections form a Lie groupoid morphism (depicted in Diagram \ref{eqn:reductiondiagram}), a fact which is easy to check directly, or, alternatively, noting that this is a special case of an action of a Lie groupoid on another Lie groupoid, one may apply Lemma 2.1 of \cite{Moerdijk2002}. 

The third assumption is that $E\to N$, as a representation of $\mathcal{S}$, integrates to a representation of $\Sigma$. This allows us to construct the associated vector bundle 
\begin{equation*}
E_\mathrm{red}\to P_\mathrm{red}
\end{equation*}
(see Appendix \ref{appendix:basicforms} for more details), which becomes a representation of $\mathcal{G}_\mathrm{red}\rightrightarrows P_\mathrm{red}$, since the trivial action of $\mathcal{G}$ on $I^*E$ descends to an action of $\mathcal{G}_\mathrm{red}$ on $E_\mathrm{red}$. In turn, the $1$-form $\theta\in\Omega^1(\mathcal{G},E)$ descends to a $1$-form
\begin{equation}
\label{eqn:reduced1form}
\theta_\mathrm{red}\in\Omega^1(\mathcal{G}_\mathrm{red},E_\mathrm{red}).
\end{equation}
Indeed, since $\theta$ is basic with respect to the action of $\mathcal{S}$ (Proposition \ref{prop:reductionbasic}), it is basic with resepct to the action of $\Sigma$ (Proposition \ref{prop:basicforms} together with our assumption of $s$-connectivity), and so it is the unique $1$-form satisfying $\mathrm{pr}^*\theta_\mathrm{red} = \theta$ (Proposition \ref{prop:basicformscorrespondence}). 

\begin{mytheorem}
	\label{theorem:reduction}
	Let $(\CAlg,\Tab)$ be a Cartan algebroid over $N$, $(P,\Omega)$ a realization, and consider the induced pseudogroup $\Gamma(P,\Omega)$. Let us assume that
	\begin{enumerate}
		\item there exists an $s$-connected Lie groupoid $\Sigma\rightrightarrows N$ integrating the Lie algebroid $\mathcal{S}\to N$,
		\item the action of $\mathcal{S}$ on $P$ integrates to a free and proper action of $\Sigma$,
		\item the representation $E$ of $\mathcal{S}$ integrates to a representation of $\Sigma$.
	\end{enumerate}
	Then, $(\mathcal{G}_\text{red},\theta_\text{red})$ is a Lie-Pfaffian groupoid over $P_\mathrm{red}$, and $\Gamma(P,\Omega)$ (viewed as $\LBis(\mathcal{G},\theta)$ by Proposition \ref{prop:reductionpfaffiangroupoid}) is an isomorphic prolongation of the generalized pseudogroup $\LBis(\mathcal{G}_\text{red},\theta_\text{red})$ (Definition \ref{def:pfaffiangroupoid}).
\end{mytheorem}

The theorem is depicted in the following diagram:
\begin{equation}
\label{eqn:reductiondiagram}
\begin{tikzpicture}[description/.style={fill=white,inner sep=2pt},bij/.style={above,sloped,inner sep=.5pt}]	
\matrix (m) [matrix of math nodes, row sep=0.3em, column sep=0.8em, 
text height=1.5ex, text depth=0.25ex]
{ 
	& & & \mathcal{G} \\
	& \mathcal{G}_\text{red} &  &  & \reflectbox{$\acts$} \; \LBis(\mathcal{G},\theta) = \Gamma(P,\Omega)  &  \\
	\LBis(\mathcal{G}_\text{red},\theta_\text{red}) \acts & &  & P \\
	& P_\text{red} \\
};
\path[->,font=\scriptsize]
(m-2-2) edge [transform canvas={xshift=-0.5ex}] node[auto] {} (m-4-2)
(m-2-2) edge [transform canvas={xshift=0.5ex}] node[auto] {} (m-4-2)
(m-1-4) edge [transform canvas={xshift=-0.5ex}] node[auto] {} (m-3-4)
(m-1-4) edge [transform canvas={xshift=0.5ex}] node[auto] {} (m-3-4)
(m-1-4) edge node[above] {$pr$} (m-2-2)
(m-3-4) edge node[above] {$pr$} (m-4-2);
\end{tikzpicture}
\end{equation}

\begin{myremark}
	\label{remark:regularityconditionsreduction}
	The purpose of the assumptions of the theorem is to ensure that the reduced objects are smooth, and these assumptions can be varied according to need. For example, if one assumes that the integration $\Sigma$ is the source-simply connected one, it automatically follows that the representation $E$ integrates to a representation of $\Sigma$. If one further assumes that the map $I:P\to N$ is proper, then the action of $\mathcal{S}$ on $P$ automatically integrates to an action of $\Sigma$, but one would still need to assume that the resulting action is free and proper. Note, however, that the freeness assumption is a rather weak one, since the action of $\mathcal{S}$ on $P$ is already infinitesimally free in the sense that the action map \eqref{eqn:systaticspaceaction} is injective. One can also relax this assumption at the possible cost of working with local Lie groupoids and local actions.
\end{myremark}

\begin{myremark}
	Note that, as an immediate corollary, $E_\mathrm{red}\to P_\mathrm{red}$ is a Lie algebroid (see item (c) of Remark \ref{remark:consequencesliepfaffiangroupoid}).
\end{myremark}

\begin{proof}
	\noindent \underline{Item 1:} We must verify the following three properties:
	
	1) $\theta_\text{red}$ is multiplicative: note that the multiplicativity expression 
	\begin{equation*}
	(m^*\theta_\text{red})_{(g,h)} - (\mathrm{pr}_1^*\theta_\text{red})_{(g,h)} - g\cdot (\mathrm{pr}_2^*\theta_\text{red})_{(g,h)},\hspace{1cm} (g,h)\in(\mathcal{G}_\text{red})_2,
	\end{equation*}
	defines an $E_\text{red}$-valued $1$-form on $(\mathcal{G}_\text{red})_2$, and that the pull-back of this form by 
	\begin{equation}
	\label{eqn:projectioncomposablearrows}
	\text{pr}: \mathcal{G}_2\to (\mathcal{G}_\text{red})_2
	\end{equation}
	is precisely the corresponding expression for $\theta$ (using the fact that $\theta = \mathrm{pr}^*\theta_\text{red}$ and that $\text{pr}$ is a Lie groupoid morphism). Now, if \eqref{eqn:projectioncomposablearrows} is a surjective submersion, then the pull-back is injective, and the multiplicativity of $\theta$ implies the multiplicativity of $\theta_\text{red}$. Let us verify that this is indeed the case. Consider a vector $(X,Y)$ in $(\mathcal{G}_\text{red})_2$ at the composable pair $([p,q],[q,r])$, with $X\in T_{[p,q]} \mathcal{G}_\text{red}$ and $Y\in T_{[q,r]} \mathcal{G}_\text{red}$ that satisfy $dt(Y)=ds(X) \in T_{[q]}P_\text{red}$. Because $\text{pr}:\mathcal{G}\to \mathcal{G}_\text{red}$ is a submersion, then there exist $X\in T_{(p,q)} \mathcal{G}$ and $Y\in T_{(q,r)} \mathcal{G}$ that project to $X$ and $Y$, respectively. The problem is that $Z:=ds(\widetilde{X})- dt(\widetilde{Y})\in T_qP $ may be non-zero. However, the projection $P\to P_\text{red}$ maps $Z$ to zero, and, hence, we may construct a $\widetilde{Z}\in T_{(q,r)}(\mathcal{G})$ that projects to zero under $\text{pr}:\mathcal{G}\to \mathcal{G}_\text{red}$ by means of the action on $P$. Thus, the pair $(\widetilde{X},\widetilde{Y}+\widetilde{Z})$ is a vector in $\mathcal{G}_2$ at $(p,q)$ that projects to $(X,Y)$. 
	
	2) $\theta$ is pointwise surjective: the fact that $\theta$ is surjective together with $\theta = \text{pr}^*\theta_\text{red}$ implies that $\theta_\text{red}$ is surjective, since $I^*E$ (the space in which $\theta$ takes values) is just the pullback of $E_\text{red}$ (the space in which $\theta_\text{red}$ takes values).
	
	3) $\Ker \theta_\mathrm{red}\cap \Ker dt=\Ker \theta_\mathrm{red}\cap \Ker ds$: we use the notation of Remark \ref{remark:reductionnotation}. Both sides consist of all equivalence classes of the type $[X_u+X_S,X_u+X_T]$, with $u\in\mathcal{S}^0$ and $S,T\in\Tab$, and hence coincide. 
	
	\vspace{0.3cm}
	
	\noindent \underline{Item 2:} There is a canonical action of $\mathcal{G}_\text{red}$ on $\text{pr}:P\to P_\text{red}$ given by $[p,q]\cdot q=p$, where $p$ is the first component of the unique representative of $[p,q]$ that has $q$ as the second component. Using this action, any element $\sigma\in\Gamma(\mathcal{G}_\text{red},\theta_\text{red})$ with domain $U=\Dom(\sigma)$ induces a bisection $\widetilde{\sigma}$ of $\mathcal{G}$ with domain $\text{pr}^{-1}(U)\subset P$ given by 
	\begin{equation}
	\label{eqn:reductionsigmatilde}
	\widetilde{\sigma}(p) = (\sigma(\text{pr}(p))\cdot p,p).
	\end{equation}
	We show that the induced generalized pseudogroup on $\mathcal{G}\rightrightarrows P$ is precisely $\LBis(\mathcal{G},\theta)$. Unraveling the definition of an isomorphic prolongation, this shows that $\LBis(\mathcal{G},\theta)$ is an isomorphic prolongation of $\LBis(\mathcal{G}_\text{red},\theta_\text{red})$. There are two things to verify: 
	\begin{enumerate}[label=\alph*)]
		\item Given $\widetilde{\sigma}$ induced by an element $\sigma\in\LBis(\mathcal{G}_\text{red},\theta_\text{red})$ by \eqref{eqn:reductionsigmatilde}, then $\widetilde{\sigma}^*\theta=0$ and hence $\widetilde{\sigma}\in \LBis(\mathcal{G},\theta)$.
		\item Locally, any $\widetilde{\sigma}\in \LBis(\mathcal{G},\theta)$ arises from some $\sigma\in\LBis(\mathcal{G}_\text{red},\theta_\text{red})$, i.e. for every $p\in\Dom(\widetilde{\sigma})$ there exists a neighborhood $U\subset \Dom(\widetilde{\sigma})$ of $p$ such that $\widetilde{\sigma}|_U$ is induced by some $\sigma$ via \eqref{eqn:reductionsigmatilde}. 
	\end{enumerate}
	
	Proof of a). Let $p\in P$ and let $q\in P$ be the element satisfying $\widetilde{\sigma}(p)=(q,p)$. Choose a local section $\eta$ of $\text{pr}:P\to P_\text{red}$ such that $p\in\Image\;\eta$. By \eqref{eqn:reductionsigmatilde}, $\sigma = \text{pr}\circ \widetilde{\sigma}\circ \eta$, and so
	\begin{equation}
	\label{eqn:reductionkillingtheta}
	0 = \sigma^*\theta_\text{red} = (\text{pr}\circ\widetilde{\sigma}\circ \eta)^*\theta_\text{red} = \eta^*(\widetilde{\sigma}^*\theta). 
	\end{equation}
	Hence, $\widetilde{\sigma}^*\theta$ vanishes on a horizontal subspace of $T_pP$ with respect to the projection $\text{pr}:P\to P_\text{red}$. Now, the vertical subbundle of $TP$ is spanned by vector fields of the type $X_u+X_s\in \mathfrak{X}(P)$, with $u\in\Gamma(\mathcal{S}^0)$ and $S\in\Gamma(\Tab)$ at the image of the action map of $\mathcal{S}$ on $P$. Choose a curve $g_\epsilon$ in $\Sigma$ such that $g_0=1_{I(p)}$ and $\frac{d}{d\epsilon}\big|_{\epsilon=0} \; g_\epsilon\cdot p = (X_u+X_S)_p$, and hence $\frac{d}{d\epsilon}\big|_{\epsilon=0} \; g_\epsilon\cdot q = (X_u+X_S)_q$. Then,
	\begin{equation}
	\begin{split}
	(d\widetilde{\sigma})_p(X_u+X_S) &= \frac{d}{d\epsilon}\Big|_{\epsilon=0} \widetilde{\sigma}(g_\epsilon\cdot p) \\
	&= \frac{d}{d\epsilon}\Big|_{\epsilon=0} (\sigma(\text{pr}(g_\epsilon\cdot p)))\cdot (g_\epsilon\cdot p),g_\epsilon\cdot p) \\
	&= \frac{d}{d\epsilon}\Big|_{\epsilon=0} (\sigma(\text{pr}(p)))\cdot (g_\epsilon\cdot p),g_\epsilon\cdot p) \\
	&= \frac{d}{d\epsilon}\Big|_{\epsilon=0} (g_\epsilon\cdot q,g_\epsilon\cdot p) \\
	& = (X_u+X_S,X_u+X_S)_{(q,p)}.
	\end{split}
	\end{equation}
	And so, by the definition of $\theta$,
	\begin{equation*}
	\begin{split}
	(\widetilde{\sigma}^*\theta)_p(X_u+X_S) &= (\theta)_{(q,p)}(X_u+X_S,X_u+X_S) \\ &= \Omega_p(X_u+X_S) - \Omega_q(X_u+X_S) \\ &= 0.
	\end{split}
	\end{equation*}
	
	Proof of b). Let $\widetilde{\sigma}\in \LBis(\mathcal{G},\theta)$, i.e. $\widetilde{\sigma}$ is a local bisection of $\mathcal{G}$ that satisfies $\widetilde{\sigma}^*\theta=0$. We first show that, locally, $\widetilde{\sigma}$ descends to a local bisection $\sigma$ of $P_\text{red}$. Since $\text{pr}:P\to P_\text{red}$ is a submersion, any point in $\Dom(\widetilde{\sigma})$ has a neighborhood $U\subset \Dom(\widetilde{\sigma})$ such that the fibers of $\text{pr}|_U:U\to P_\text{red}$ are connected, and we may, thus, assume that the domain of $\widetilde{\sigma}$ has this property. Let $p\in\Dom(\widetilde{\sigma})$, set $x:=I(p)$ and set $U_x:= \Dom(\widetilde{\sigma})\cap \text{pr}^{-1}([p])$, where $[p]=\text{pr}(p)\in P_\text{red}$. Since $\Sigma$ acts freely on $P$, the action map $\Sigma {}_s\!\times_I P\to P$ provides a diffeomorphism $U_x\cong V_x$ between $U_x$ and an open subset $V_x\subset s^{-1}(x)\subset \Sigma$ which maps a point $q\in U_x$ to the unique arrow $g\in\Sigma$ satisfying $g\cdot p=q$. We need to show that for any $g\in V_x$, $\widetilde{\sigma}(g\cdot p) = g\cdot \widetilde{\sigma}(p)$. Equivalently, passing to the pseudogroup point of view, writing $\widetilde{\sigma}(q)=(\phi(q),q)$ for all $q\in\Dom(\widetilde{\sigma})$ for some uniquely determined $\phi\in\LDiff(P)$, we must show that $\phi(g\cdot p)=g\cdot \phi(p)$. Since $V_x$ is connected, we can choose a path $g_\epsilon$ in $V_x$ such that $g_0=1_x$ and $g_1=g$. We show that
	\begin{equation*}
	\phi(g_\epsilon\cdot p) = g_\epsilon\cdot \phi(p),\hspace{1cm} \forall\;\epsilon,
	\end{equation*}
	by showing that both sides are integral curves of the same time dependent vector field, and since $\phi(g_\epsilon\cdot p) = g_\epsilon\cdot \phi(p)$ at $\epsilon=0$, this will imply that they are equal for all $\epsilon$. For each $\epsilon$, $\frac{d}{d\epsilon}g_\epsilon\in T_{g_\epsilon} s^{-1}(x)\subset T_{g_\epsilon}\Sigma$, and applying $dR_{g_\epsilon^{-1}}$ gives an element of $(\mathcal{S})_{t(g_\epsilon)}$, where $\mathcal{S}$ is the Lie algebroid of $\Sigma$. Thus, we may find a time dependent section $u_\epsilon\in\Gamma(\mathcal{S})$ such that if $\widetilde{u}_\epsilon\in\mathfrak{X}(\Sigma)$ is the induced right invariant vector field, then the value of $\widetilde{u}_\epsilon(g_\epsilon)=\frac{d}{d\epsilon}g_\epsilon$. Now, if $X_{u_\epsilon}\in\mathfrak{X}(P)$ is the image of $u_\epsilon$ under the infinitesimal action map $\Gamma(\mathcal{S})\to \mathfrak{X}(P)$, then $(X_{u_\epsilon})_{g_\epsilon\cdot p} = \frac{d}{d\epsilon}g_\epsilon\cdot p$. But by Proposition \ref{prop:reductioninvariancegamma}, $d\phi(X_{u_\epsilon})=X_{u_\epsilon}$, and so
	\begin{equation*}
	\frac{d}{d\epsilon}\phi(g_\epsilon\cdot p) = d\phi((X_{u_\epsilon})_{g_\epsilon\cdot p}) = (X_{u_\epsilon})_{\phi(g_\epsilon\cdot p)},
	\end{equation*} 
	and on the other hand,
	\begin{equation*}
	\frac{d}{d\epsilon}g_\epsilon\cdot \phi(p) = (X_{u_\epsilon})_{g_\epsilon\cdot \phi(p)}.
	\end{equation*}
	Thus, we may define a local section $\sigma$ of $\mathcal{G}_\text{red}$ with domain $\text{pr}(\Dom(\widetilde{\sigma}))$ by defining $\sigma([p]):=[\widetilde{\sigma}(p)]$ for some representative $p$ of $[p]$ (and one easily checks that it is smooth). The two are related by \ref{eqn:reductionsigmatilde}, and choosing a section $\eta$ of $\text{pr}:P\to P_\text{red}$, we see as before that $\sigma = \text{pr}\circ\widetilde{\sigma}\circ \eta$. Finally, reading \eqref{eqn:reductionkillingtheta} from right to left, we see that $\sigma^*\theta_\text{red} = 0$.
\end{proof}

\subsection{Examples of Reduction}
\label{examples:reduction}

Let us look at two examples of reduction. In the first example, we illustrate the full reduction procedure in the simple but instructive case of a Lie group. Starting from the Lie pseudogroup of right translations on a Lie group, the generalized pseudogroup obtained by reduction is the Lie group itself viewed as a generalized pseudogroup (see Example \ref{example:liegroupasgeneralizedpseudogroup}). The reduction procedure, thus, reveals in this case the ``true'' abstract object underlying the Lie pseudogroup. In the second example we sketch the reduction procedure in the case of a Lie pseudogroup of infinite type. 

\begin{myexample}[Lie Groups]
\label{example:liegroups}

Let us continue with Example \ref{example:liegroupssystatic}. Let $G$ be a Lie group with Lie algebra $\mathfrak{g}$. We saw that $(P=G,\Omega=\Omega_\text{MC})$ is a realization of the Cartan algebroid $(\CAlg=\mathfrak{g},\Tab=0)$ over a point. The Pfaffian groupoid induced by the realization is the pair $(\mathcal{G},\theta)$, with 
\begin{equation*}
\mathcal{G}=G\times G\hspace{0.7cm}\text{and}\hspace{0.7cm} \theta= s^*\Omega_\text{MC} - t^*\Omega_\text{MC},
\end{equation*}
and
\begin{equation*}
E=\Ker\rho=\mathfrak{g}
\end{equation*}
From \eqref{eqn:reductions1connection}, we see that $E=\mathfrak{g}$ becomes a representation of $\mathcal{S}=\mathfrak{g}$ by the map
\begin{equation}
\label{eqn:examplereductionliegroup1}
\mathfrak{g}\mapsto \text{End}(\mathfrak{g}),\hspace{1cm} X\mapsto \mathrm{ad}_X= [X,\cdot],
\end{equation}
i.e. the adjoint representation. 

To perform the reduction, we use the integration of $\mathcal{S}=\mathfrak{g}$ to the Lie group $\Sigma=G$. The action of $\mathcal{S}=\mathfrak{g}$ on $P=G$ integrates to the action of $\Sigma=G$ on $P=G$ given by
\begin{equation*}
(\Sigma=G)\times (P=G)\to (P=G),\hspace{1cm} (g,g')\mapsto g\cdot g',
\end{equation*}
and integrating \eqref{eqn:examplereductionliegroup1} results in the adjoint action of $\Sigma=G$ on $E=\mathfrak{g}$. With this integration, we can divide out by the action of $\Sigma$. The resulting quotient is the Lie-Pfaffian groupoid
\begin{center}
	\begin{tikzpicture}[description/.style={fill=white,inner sep=2pt},bij/.style={above,sloped,inner sep=.5pt}]	
	\matrix (m) [matrix of math nodes, row sep=1.5em, column sep=0em, 
	text height=1.5ex, text depth=0.25ex]
	{ 
		\mathcal{G}_\text{red} & \node[anchor=west]{\cong\; \;G}; \\
		P_\text{red}  &  \node[anchor=west]{\cong\;\{*\},}; \\
	};
	\path[->,font=\scriptsize]
	(m-1-1) edge [transform canvas={xshift=-0.5ex}] node[auto] {} (m-2-1)
	(m-1-1) edge [transform canvas={xshift=0.5ex}] node[auto] {} (m-2-1);
	\end{tikzpicture}
\end{center}
where the isomorphism $\mathcal{G}_\text{red}\cong G$ is canonical, and, under this isomorphism, $\theta_\text{red}$ coincides with the Maurer-Cartan form on $G$. Let us describe the isomorphism explicitly. Elements of $\mathcal{G}$ are pairs $(g,g')\in G\times G$, and $\mathcal{G}_\text{red}$ consists of equivalence classes of the form $[g,g']$, with $(g,g')\sim (hg,hg')$ for all $h\in G$. The isomorphism is then given by $\mathcal{G}_\text{red}\xrightarrow{\simeq} G,\; [g,g']\mapsto g^{-1}g'$. The result is that we recover the Lie group $G$ together with its Maurer-Cartan form. 

Finally, since a local solution of $(\mathcal{G}_\text{red},\theta_\text{red})$ is simply an element of $G$, then
\begin{equation*}
\LBis(\mathcal{G}_\text{red},\theta_\text{red}) = G,
\end{equation*}
and an element $R_{g^{-1}}\in\Gamma(G,\Omega_\text{MC})$ of $\Gamma(G,\Omega_\mathrm{MC})$ with $g\in G$ descends to the element $g\in \LBis(\mathcal{G}_\text{red},\theta_\text{red})$. 

Thus, the reduction procedure recovers the Lie group $G$ from the realization $(G,\Omega_\text{MC})$ and detects that $\Gamma(G,\Omega_\text{MC})$ comes from an action of the Lie group $G$, but viewed as a generalized pseudogroup, on the manifold $G$. 

\end{myexample}

\begin{myexample}[A Lie Pseudogroup of Finite Type] Consider Cartan's example (which was mentioned already in the introduction) of a Lie pseudogroup $\Gamma$ on the upper half plane $M=\{ (x,y)\in\mathbb{R}^2\;|\; y>0 \}$ generated by the diffeomorphisms
\begin{equation*}
\phi_a: M\to M,\; (x,y)\mapsto (x+ay,y),\hspace{1cm} \forall\;a\in\mathbb{R}.
\end{equation*}
Clearly, this pseudogroup comes from an action of the additive Lie group $\mathbb{R}$ on $M$. To apply reduction, we first apply the Second Fundamental Theorem and place this pseudogroup in normal form. We state here the result and leave it to the reader to fill in the details (see Section \ref{section:cartanexamples} for two fully worked out examples). The Cartan algebroid is the pair $(\mathcal{C},\Tab=0)$, with $\mathcal{C}$ a trivial vector bundle of rank 3 over $M$ spanned by global sections $e^1,e^2,e^3$ whose bracket and anchor are defined by
\begin{equation*}
[e^1,e^2]=0,\;\;\; [e^1,e^3]=0,\;\;\; [e^2,e^3] = -\frac{1}{y} e^3,\;\;\; \rho(e^1)=\frac{\partial}{\partial x},\;\;\; \rho(e^2)=\frac{\partial}{\partial y},\;\;\; \rho(e^3)=0. 
\end{equation*}
The realization is given by the manifold $P=\{(x,y,p)\in\mathbb{R}^3\;|\; y>0 \}$, the map
\begin{equation*}
I:P\to M,\; (x,y,p)\mapsto (x+py,y),
\end{equation*}
and the $1$-form $\Omega\in\Omega^1(P;I^*\mathcal{C})$ defined by
\begin{equation*}
\Omega: \frac{\partial}{\partial x} \mapsto I^*e^1,\;\; \frac{\partial}{\partial y} \mapsto p\,I^*e^1 + I^* e^2,\;\; \frac{\partial}{\partial p} \mapsto y\, I^*(e^1 + e^3).
\end{equation*}
The induced pseudogroup in normal form is the pseudogroup $\widetilde{\Gamma}$ on $P$ generated by the diffeomorphisms
\begin{equation*}
\widetilde{\phi}_a:P\to P,\; (x,y,p)\mapsto (x+ay,y,p-a),\hspace{1cm} \forall\; a\in\mathbb{R}.
\end{equation*}
Since $\Tab=0$, we immediately see that the systatic space $\mathcal{S}$ coincides with $\mathcal{C}$ and its action on $P$ is given by the inverse of $\Omega$, i.e.
\begin{equation*}
\Omega^{-1}: I^*\mathcal{S}\to TP,\;\; I^*e^1\mapsto \frac{\partial}{\partial x},\;\; I^*e^2\mapsto -p \frac{\partial}{\partial x} + \frac{\partial}{\partial y},\;\ I^*e^3\mapsto -\frac{\partial}{\partial x} +  \frac{1}{y} \frac{\partial}{\partial p}.
\end{equation*}
One can now perform the reduction by explicitly integrating the systatic algebroid and its action on $P$ (a nice exercise). The result, as one may expect, is precisely as in the previous example. The reduced Lie-Pfaffian groupoid is the Lie group $\mathbb{R}$ equipped with the Maurer-Cartan form, and, hence, reduction ``reveals'' that the pseudgroup $\widetilde{\Gamma}$ on $P$ is an isomorphic prolongation of this Lie group viewed as a generalized pseudogroup.

\end{myexample}

\begin{myexample}[A Lie Pseudogroup of Infinite Type]

Consider the pseudogroup in normal form \eqref{eqn:example2normalform} coming from Cartan's example of a Lie pseudogroup of infinite type. In this example, it is evident that the pseudogroup comes from an action of $\LDiff(\mathbb{R})$ on $\mathbb{R}^2$. Using our computation of the underlying Cartan algebroid and realization in Example \ref{example:infiniterevisited}, we see that the partial systatic space $\mathcal{S}^0$ is spanned by $e^2$, the systatic space $\mathcal{S}$ by $\{e^2,t\}$, and the action of $\mathcal{S}$ on $P$ is given by
\begin{equation*}
\mathcal{S}\to \mathfrak{X}(P),\hspace{1cm} e^2\mapsto y\frac{\partial}{\partial y} + uy\frac{\partial}{\partial u},\; t\mapsto y \frac{\partial}{\partial u}.
\end{equation*}
The orbits are the submanifolds of $P$ given by $x=\text{cnst}$, and hence $P_\text{red}=P/\mathcal{S}\cong \mathbb{R}$. Reduction in this case produces a Lie-Pfaffian groupoid over $\mathbb{R}$, whose local holonomic section are $\LDiff(\mathbb{R})$, viewed as a generalized pseudogroup. 
\end{myexample}

\appendix

\section*{Appendix}
\addtocounter{section}{1}
\addcontentsline{toc}{section}{Appendix}

\subsection{Jet Groupoids and Algebroids}
\label{section:jetgroupoids}

In this appendix, we review the framework of jet groupoids and algebroids. This framework allows for a geometric formulation of the notion of a system of PDEs in the setting of pseudogroups and, accordingly, of a Lie pseudogroup (see Section \ref{section:liepseudogroups}). We place a special emphasis on the role and properties of the Cartan form in encoding the essential structure. This will lead us in Appendix \ref{section:Liepfaffaingroupoidandalgebroid} to the notion of a Lie-Pfaffian groupoid, an axiomatization of the notion of a jet groupoid, which, in the spirit of Cartan, isolates the key role of the Cartan form. 

\subsubsection{Jet Groupoids}

Let $M$ be a manifold. For each integer $k\geq 0$, the \textbf{$k$-th jet groupoid} of $\LDiff(M)$ (or, for brevity, of $M$), denoted by $J^kM\rightrightarrows M$, is a Lie groupoid whose space of arrows consists of all $k$-jets of all locally defined diffeomorphisms of $M$,
\begin{equation}
J^kM := \{\; j^k_x\phi\;|\; \phi\in \LDiff(M),\; x\in\Dom(\phi) \;\}.
\label{eqn:jetgroupoidmanifold}
\end{equation}
Its structure maps are
\begin{gather*}
s(j^k_x\phi) = x,\hspace{1.3cm} t(j^k_x\phi) = \phi(x), \hspace{1.3cm} 1_x = j^k_x(\Id_M), \\
j^k_{\phi(x)}\phi'\cdot j^k_x\phi = j^k_x(\phi'\circ\phi),\hspace{1cm} (j^k_x\phi)^{-1} = j^k_{\phi(x)}\phi^{-1}.
\end{gather*}
Thus, the $k$-th jet groupoid encodes the $k$-th order Taylor polynomials of locally defined diffeomorphisms of $M$. For example, the $0$-th jet groupoid $J^0M\rightrightarrows M$ encodes the source and target points. It is canonically isomorphic to the pair groupoid $M\times M\rightrightarrows M$ by the map $j^0_x\phi \mapsto (\phi(x),x)$. The smooth structure of $J^kM$ is defined as usual for jet spaces (in fact, $J^kM$ is an open subset of the space of $k$-jets of all smooth maps from $M$ to $M$). Any $\phi\in \LDiff(M)$ induces a local bisection $j^k\phi$ of $J^kM\rightrightarrows M$, called a \textbf{local holonomic bisection}, whose domain is the domain of $\phi$ and which maps $x\mapsto j^k_x\phi$. The subset of local holonomic bisections (inside the set of all local bisections) is, one may say, the most important piece of structure of a jet groupoid. When one studies Lie subgroupoids of a jet groupoid, which play the role of PDEs (e.g. the defining PDEs of Lie pseudogroups), local holonomic bisections play the role of local solutions. 

The jet groupoids of $M$ form a sequence of Lie groupoids,
\begin{equation}
... \xrightarrow{\pi} J^3M \xrightarrow{\pi} J^2M  \xrightarrow{\pi} J^1M \xrightarrow{\pi} J^0M,
\label{eqn:jetgroupoidMsequence}
\end{equation}
where the projections, the Lie groupoid morphisms 
\begin{equation}
\pi:J^{k+1}M \to J^kM,\hspace{1cm} j^{k+1}_x\phi\mapsto j^k_x\phi,
\label{eqn:jetgroupoidprojection}
\end{equation}
are surjective submersions. 

\begin{myremark}
	\label{remark:jetgroupoidsofgroupoids}
	More generally, we may construct the $k$-th jet groupoid $J^k\mathcal{G}\rightrightarrows M$ of a Lie groupoids $\mathcal{G}$, consisting of $k$-jets $j^k_xb$ of local bisections $b\in\LBis(\mathcal{G})$. The jet groupoid $J^kM$ is the special case corresponding to the pair groupoid $M\times M\rightrightarrows M$. All the notions explained here and in Section \ref{section:jetgroupoids} generalize without extra effort to this setting. 
\end{myremark}

\subsubsection{Jet Algebroids and the Adjoint Representation}

At the infinitesimal level, each jet groupoid $J^kM\rightrightarrows M$ has an associated Lie algebroid $A^kM=A(J^kM)$. These fit in a sequence of Lie algebroids,
\begin{equation}
\label{eqn:sequenceofjetalgebroids}
... \xrightarrow{l} A^3M \xrightarrow{l} A^2M  \xrightarrow{l} A^1M \xrightarrow{l} A^0M,
\end{equation}
where the projections, induced by the projections \eqref{eqn:jetgroupoidprojection}, are surjective Lie algebroid morphisms. A well-known fact is that $A^kM$ is canonically isomorphic to the \textbf{$k$-th jet algebroid of $\mathfrak{X}(M)$} (or of $M$), the Lie algberoid $J^kTM$ consisting of $k$-jets of vector fields on $M$. The isomoprhism
\begin{equation}
J^kTM \cong A^kM,
\label{eqn:jetalgebroididentification}
\end{equation}
is given by mapping a $k$-jet $j^k_xX$ of a vector field $X\in\mathfrak{X}(M)$ at $x\in M$ to the vector $\frac{d}{d\epsilon}\big|_{\epsilon=0} \,j^k_x\varphi_X^\epsilon$, where $\varphi_X^\epsilon$ is the flow of $X$. In particular, $A^0M$ is canonically isomorphic to $TM$. We refer the reader to \cite{Guillemin1966}, \cite{Kumpera1975} or \cite{Yudilevich2016-2} for further details. We, thus, refer to $A^kM$ also as the $k$-th jet algebroid of $M$. 

For every $k>0$, there is a canonical action of the $k$-th jet groupoid $J^kM$ on the $(k-1)$-th jet algebroid $A^{k-1}M$, given by conjugation. Namely, an element $j^k_x\phi\in J^kM$ acts on the fiber $(A^{k-1}M)_x$ by:
\begin{equation}
j^k_x\phi\cdot \frac{d}{d\epsilon}\Big|_{\epsilon=0}\,j^{k-1}_x\psi_\epsilon := \frac{d}{d\epsilon}\Big|_{\epsilon=0} \, j^{k-1}_{\phi(x)} ( \phi\circ \psi_\epsilon\circ\phi^{-1}),
\label{eqn:jetgroupoidadjointrepresentation}
\end{equation} 
where $j^{k-1}_x\psi_\epsilon$ is a curve representing an element of $(A^{k-1}M)_x$. With this action, $A^{k-1}M$ becomes a representation of $J^kM$ -- the \textbf{adjoint representation}. At the infinitesimal level, any jet algebroid $A^kM$ (with $k>0$) has an adjoint representation $A^{k-1}M$  (see e.g. \cite{Crainic2005}).

\subsubsection{The Cartan Form}
\label{section:thecartanform}

Every jet groupoid $J^kM\rightrightarrows M$ (with $k>0$) comes equipped with a tautological form
\begin{equation}
\label{eqn:jetgroupoidcartanform}
\omega=\omega^k\in\Omega^1(J^kM;t^*A^{k-1}M),
\end{equation}
called the \textbf{Cartan form} of $J^kM$. It is a $1$-form on the total space of the jet groupoid with values in the adjoint representation. It is defined at a point $j^k_x\phi\in J^k M$ by the formula
\begin{equation}
\label{eqn:cartanformmultaiplicativepde}
\omega_{j^k_x\phi} :=  dR_{(j^{k-1}_x\phi)^{-1}} \circ(d\pi- (d(j^{k-1}\phi))_x\circ ds)_{j^k_x\phi},
\end{equation} 
which is depicted in the following diagram:
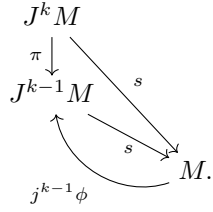
\begin{center}
	\begin{tikzpicture}[description/.style={fill=white,inner sep=2pt},bij/.style={above,sloped,inner sep=.5pt}]	
	\matrix (m) [matrix of math nodes, row sep=1.5em, column sep=2.5em, 
	text height=1.5ex, text depth=0.25ex]
	{ 
		J^kM &  \\
		J^{k-1}M & \\
		& M. \\
	};
	\path[->,font=\scriptsize]
	(m-1-1) edge node[left] {$\pi$} (m-2-1)
	(m-1-1) edge node[auto] {$s$} (m-3-2)
	(m-2-1) edge node[below] {$s$} (m-3-2)
	(m-3-2) edge[bend left=50] node[auto] {$j^{k-1}\phi$} (m-2-1);
	\end{tikzpicture}
\end{center}
Indeed, the image of the map $d\pi- (d(j^{k-1}\phi))_x\circ ds$ at a point $j^k_x\phi$ is tangent to the $s$-fiber of $J^{k-1}M$ at the point $j^{k-1}_x\phi$, and then right translation maps it to the fiber of the Lie algebroid $A^{k-1}M$ at the point $t(j^k_x\phi)=\phi(x)$. By restricting $\omega$ to $\Ker ds$, we see that $\omega$ is pointwise surjective. The kernel of the Cartan form, a non-involutive distribution, is called the \textbf{Cartan distribution} and is denoted by $C_\omega:=\Ker\omega\subset T\mathcal{G}$. The main property of the Cartan form is that it detects holonomic sections (for the proof see e.g. Proposition 1.3.3 in \cite{Yudilevich2016-2}):

\begin{myprop}
	Let $M$ be a manifold. A local bisection $b$ of $J^kM\rightrightarrows M$ is holonomic (i.e. of the form $b=j^k\phi$ for some $\phi\in\LDiff(M)$) if and only if $b^*\omega=0$. 
	\label{prop:detectingholonomicsections}
\end{myprop}

In fact, all jet spaces carry a tautological form, defined similarly to the Cartan form of a jet groupoid, and these satisfy a property similar to the one in the above proposition. However, the Cartan form of a jet groupoid satisfies the additional property of \textit{multiplicativity} that reflects its compatibility with the Lie groupoid structure of the jet groupoid. Recall that a differential form $\omega\in\Omega^*(\mathcal{G},t^*E)$ on a Lie groupoid $\mathcal{G}\rightrightarrows M$ with values in a representation $E\to M$ is said to be \textbf{multiplicative} if
\begin{equation}
(m^*\omega)_{(g,h)} = (\text{pr}_1 ^*\omega)_{(g,h)} + g\cdot(\text{pr}_2 ^*\omega)_{(g,h)},\hspace{1cm} \forall\; (g,h)\in\mathcal{G}_2,
\label{eqn:multiplicativeform}
\end{equation}
where $\mathcal{G}_2\subset \mathcal{G}\times\mathcal{G}$ is the space of composable arrows and $m,\text{pr}_i:\mathcal{G}_2\to \mathcal{G}$ are the multiplication and projection maps. For the proof that the Cartan form is multiplicative, see \cite{Crainic2012} (Proposition 3.4) or \cite{Yudilevich2016-2} (Proposition 2.4.3). To summarize: 

\begin{myprop}
	\label{prop:cartanformjetgroupoid}
	Let $M$ be a manifold. For any $k>0$, the Cartan form of $J^kM\rightrightarrows M$,
	\begin{equation*}
	\omega\in\Omega^1(J^kM;t^*A^{k-1}M),
	\end{equation*}
	is a pointwise-surjective multiplicative 1-form on $J^kM$ with values in the adjoint representation. 
\end{myprop}

\subsubsection{The Spencer Operator}

At the infinitesimal level, the Cartan form of $J^kM$ induces the \textbf{Spencer operator} of $A^kM$, a bilinear map
\begin{equation}
D=D^k:\mathfrak{X}(M)\times \Gamma(A^kM)\to \Gamma(A^{k-1}M), \hspace{1cm} (X,\alpha)\mapsto D_X(\alpha),
\label{eqn:spenceroperator}
\end{equation}
satisfying the connection-like properties
\begin{equation}
\label{eqn:connectionlikeproperties}
D_{fX}(\alpha) = fD_X(\alpha),\hspace{1cm} D_X(f\alpha) = fD_X(\alpha) + X(f)l(\alpha),
\end{equation}
for all $f\in C^\infty(M)$, $X\in\mathfrak{X}(M)$ and $\alpha\in\Gamma(A^kM)$, where $l:A^kM\to A^{k-1}M$, as we recall, is the projection. It is obtained from the Cartan form $\omega$ by differentiation: 
\begin{equation}
\label{eqn:linearizingcartanform}
D_X(\alpha)_x := \frac{d}{d\epsilon}\Big|_{\epsilon=0}\varphi^\epsilon_\alpha(x)^{-1}\cdot \omega(d\varphi^\epsilon_\alpha(X_x)),
\end{equation}
for all $x\in M$, $X\in\mathfrak{X}(M)$ and $\alpha\in\Gamma(A^kM)$. Here, $\varphi^\epsilon_\alpha$ is the flow of bisections of $J^kM$ associated with $\alpha$. Alternatively, under the identification \eqref{eqn:jetalgebroididentification}, the Spencer operator on $A^kM$ corresponds to the classical Spencer operator on the Lie algebroid $J^kTM$ of $k$-jets of vector fields on $M$, which is defined purely out of infinitesimal data (e.g. see Section 1.1.4 in \cite{Salazar2013}). The multiplicativity property \eqref{eqn:multiplicativeform} translates into a property known as \textit{infinitesimal  multiplicativity} (\textit{IM} in short). We will explain this property in the more general context of Lie-Pfaffian algebroids in Section \ref{section:Liepfaffaingroupoidandalgebroid} (in particular, see \eqref{eqn:liepfaffianalgebroid}). With this property, the Spencer operator becomes an \textit{IM-form}, the infinitesimal counterpart of a multiplicative form.

\subsubsection{The Symbol Space}

The kernel of the projection $l: A^kM\to A^{k-1}M$ (where $k>0$), 
\begin{equation}
\Tab^kM:= \Ker (l: A^kM\to A^{k-1}M),
\label{eqn:jetgroupoidsymbolspace}
\end{equation}
is called the \textbf{symbol space of $A^kM$}. It plays a key role in the theory. Being the kernel of a Lie algebroid morphism, the symbol space has the structure of a bundle of Lie algebras. For $k>1$ it is abelian. For notational purposes, we set $\Tab^0M:= A^0M = TM$. 

The restriction of the Spencer operator to the symbol space gives a canonical inclusion
\begin{equation}
\Tab^kM \hookrightarrow \Hom(TM,\Tab^{k-1}M),\hspace{1cm} T\mapsto \big(\hat{T}:X\mapsto D_{X}(T)\big).
\label{eqn:symbolspacetableau}
\end{equation}
For $k>1$, this identifies $\Tab^kM$ with the symmetric part of $\Hom(TM,\Tab^{k-1}M)$, in the sense that
\begin{equation}
\label{eqn:symbolspaceprolongation}
\Tab^kM \cong \{\; \hat{T}\in \Hom(TM,\Tab^{k-1}M)\;|\; \hat{T}(X)(Y) = \hat{T}(Y)(X)\;\;\; \forall \, X,Y\in TM  \;\},
\end{equation}
where the equation on the right-hand side uses the inclusion $\Tab^{k-1}M\hookrightarrow\Hom(TM,\Tab^{k-2}M)$. For $k=1$,
\begin{equation}
\label{eqn:firstsymbolspace}
\Tab^1M \cong \Hom(TM,TM).
\end{equation}
We note that, if we equip $\Hom(TM,\Tab^{k-1})$ with the zero bracket for $k>1$ and the commutator bracket for $k=1$, then these become isomorphisms of Lie algebroids. Applying \eqref{eqn:symbolspaceprolongation} and \eqref{eqn:firstsymbolspace} inductively, we obtain the well-known isomorphism
\begin{equation}
\Tab^kM \cong S^kT^*M\otimes TM,
\label{eqn:symbolspacepolynomials}
\end{equation}
which identifies the symbol space of $A^kM$ with the space of vector-valued monomials of degree $k$ on $M$ (this identification is often expressed in local coordinates, see e.g. p. 21 in \cite{Yudilevich2016-2}).  

\subsubsection{Aside: Tableau Bundles and the Spencer Cohomology}
\label{section:tableaubundles}

The inclusion \eqref{eqn:symbolspacetableau} of the symbol space as a vector subbundle of a $\Hom$-bundle is an instance of the abstract notion of a \textit{tableau bundle}. The important notions of prolongation, the Spencer cohomology and involutivity, that one typically associates with the symbol space, can be defined purely in terms of its tableau bundle structure. To construct and define these notions, it is sufficient to work with the following discrete version of the notion of a vector bundle, which allows us to include non-smooth vector bundles into the picture, such as kernels, images and cokernels of vector bundle maps. 

Given a manifold $M$, a \textbf{discrete vector bundle} over $M$ is a disjoint union of vector spaces indexed by $M$, i.e. a space $E = \sqcup_{x\in M} E_x$ where $\{E_x\}_{x\in M}$ is a collection of vector spaces. A \textbf{discrete vector subbundle} of a discrete vector bundle $F$ over $M$ is a subset $E\subset F$, such that $E_x\subset F_x$, the fiber of $E$ over $x$, is a vector subspace for each $x\in M$. Given two discrete vector bundles $E$ and $F$ over $M$, a map $\partial:E\to F$ covering the identity is a \textbf{discrete vector bundle map} if it restricts to a linear map on each fiber. A discrete vector subbundle can be viewed as a discrete vector bundle together with an injective discrete vector bundle map.  

\begin{mydef}
	\label{def:tableaubundle}
	\index{tableau bundle}
	Let $E,F$ be discrete vector bundles over a manifold $M$. A \textbf{tableau bundle} relative to $(E,F)$ is a discrete vector bundle $\Tab$ over $M$ together with a discrete vector bundle map
	\begin{equation}
	\label{eqn:tableaumap}
	\partial:\Tab\to \Hom(E,F).
	\end{equation}
\end{mydef}

\begin{myexample}
	The symbol space $\Tab^kM$, together with \eqref{eqn:symbolspacetableau}, is a tableau bundle relative to $(TM,\Tab^{k-1}M)$. Note that, in this case, the map \eqref{eqn:tableaumap} is injective. 
\end{myexample}  

We denote a tableau bundle by $(\Tab,\partial)$, or simply by $\Tab$ when it is clear what the map $\partial$ is. While the map $\partial$ is injective in most applications, such as in the previous example, we will see in Section \ref{section:Liepfaffaingroupoidandalgebroid} and in Section \ref{section:thesecondfundamentaltheorem} that a great deal of the theory does not rely on this property and that, in the setting of Lie-Pfaffian groupoids, it is natural to consider non-injective maps.

\begin{mydef}
	\label{def:tableaubundleprolongation}
	Let $(\Tab,\partial)$ be a tableau bundle relative to $(E,F)$. The \textbf{1st prolongation} \index{prolongation!tableau bundle} of $\Tab$ is the tableau bundle given by the discrete vector subbundle
	\begin{equation*}
	\Tab^{(1)}\subset \Hom(E,\Tab),
	\end{equation*}
	whose fiber at $x\in M$ is
	\begin{equation*}
	\Tab^{(1)}_x := \{ \; \xi\in \Hom_x(E,\Tab)\;\;|\;\; \partial(\xi(u))(v)=\partial(\xi(v))(u)\;\;\;\forall \;u,v\in E_x \; \}.
	\end{equation*}
\end{mydef}

\begin{myremark}
	Even when the initial data of a tableau bundle is smooth, i.e. $E$, $F$ and $\Tab$ are vector bundles and $\partial$ a vector bundle map, its 1st prolongation may fail to be smooth. This is the main reason we resort to the language of discrete vector bundles. However, we note that $\Tab^{(1)}$ is the kernel of the map \eqref{eqn:spencercoboundaryoperator} defined below (with $m=1$ and $l=0$), and, hence, if the data of the tableau bundle is smooth, then $\Tab^{(1)}$ is smooth (i.e. a vector subbundle of $\Hom(E,\Tab)$) if and only if it is of constant rank.
\end{myremark}

Since the 1st prolongation is again a tableau bundle, we can continue and define the higher prolongations in the following inductive manner:

\begin{mydef}
	Let $(\Tab,\partial)$ be a tableau bundle relative to $(E,F)$. The \textbf{$l$-th prolongation} of $\Tab$, with $l>0$ an integer, is the discrete vector subbundle $\Tab^{(l)}\subset \Hom(E,\Tab^{(l-1)})$ defined by
	\begin{equation*}
	\Tab^{(l)} := (\Tab^{(l-1)})^{(1)},
	\end{equation*}
	where we set $\Tab^{(-1)}:= F$ and $\Tab^{(0)}:= \Tab$.
\end{mydef}

\begin{myexample}
	By \eqref{eqn:symbolspaceprolongation} and \eqref{eqn:firstsymbolspace}, we see that $$\Tab^kM = (\Tab^{k-1}M)^{(1)}=(\Tab^{k-2}M)^{(2)}= ... =\Hom(TM,TM)^{(k-1)}.$$
\end{myexample}

A tableau bundle $(\Tab,\partial)$ and its prolongations fit into the following sequence of cochain complexes called the \textbf{Spencer complex} of $(\Tab,\partial)$: 

\begin{equation}
\label{eqn:spencercomplex}
\begin{split}
& \hspace{0.9cm} \text{0} \hspace{1.5cm} \text{1} \hspace{2.6cm} \text{2} \hspace{2.9cm}\text{3} \hspace{2.3cm} \text{4} \\
& \text{0} \hspace{0.7cm} \Tab \hspace{0.40cm} \hookrightarrow \Hom(E,F) \\
& \text{1} \hspace{0.7cm}  \Tab^{(1)} \hookrightarrow \Hom(E,\Tab) \hspace{0.4cm} \xrightarrow{\delta} \Hom(\Lambda^2 E,F) \\
& \text{2} \hspace{0.7cm}  \Tab^{(2)} \hookrightarrow \Hom(E,\Tab^{(1)}) \xrightarrow{\delta} \Hom(\Lambda^2 E,\Tab) \hspace{0.35cm} \xrightarrow{\delta} \Hom(\Lambda^3 E,F)  \\
& \text{3} \hspace{0.7cm}  \Tab^{(3)} \hookrightarrow \Hom(E,\Tab^{(2)}) \xrightarrow{\delta} \Hom(\Lambda^2 E,\Tab^{(1)}) \xrightarrow{\delta} \Hom(\Lambda^3 E,\Tab) \hspace{0.1cm} \xrightarrow{\delta} \Hom(\Lambda^4 E,F) \\
& \hspace{7cm} \LargerCdot \\
& \hspace{7cm} \LargerCdot \\
& \hspace{7cm} \LargerCdot 
\end{split}
\end{equation}
The coboundary operator $\delta:\Hom(\Lambda^m E,\Tab^{(l)})\to \Hom(\Lambda^{m+1} E,\Tab^{(l-1)})$ is defined by
\begin{equation}
\delta(\xi)(u_0,...,u_m) := \sum_{i=0}^m (-1)^i \partial(\xi(u_0,...,\hat{u}_i,...,u_m))(u_i),
\label{eqn:spencercoboundaryoperator}
\end{equation}
where $\hat{u}_i$ denotes the removal of the $i$-th term. It is straightforward to verify that $\delta\circ\delta=0$. We denote the cocycles at $\Hom(\Lambda^m E,\Tab^{(l)})$ by $Z^{l,m}(\Tab,\partial)$, the coboundaries by $B^{l,m}(\Tab,\partial)$, and the resulting cohomology group by
\begin{equation*}
H^{l,m}(\Tab,\partial) = \frac{Z^{l,m}(\Tab)}{B^{l,m}(\Tab)}.
\end{equation*} 
The resulting cohomology theory is called the \textbf{Spencer cohomology} of $(\Tab,\partial)$. Note that, by definition, $H^{l,1}(\Tab,\partial)=0$ for all $l\geq 0$.

\begin{mydef}
	\label{def:acyclicinvolutivity}
	Let $r\geq 1$ be an integer. A tableau bundle $\Tab$ is said to be \textbf{$r$-acyclic} if 
	\begin{equation*}
	H^{l,m}(\Tab,\partial)=0 \hspace{1cm} \forall \; 1 \leq m \leq r,\; 0 \leq l,
	\end{equation*}
	and $\textbf{involutive}$ if it is $r$-acyclic for all $r\geq 1$. 
\end{mydef} 

\begin{myremark}
	The notion of a tableau  was introduced by Cartan in the context of PDEs and his theory of Exterior Differential Systems. In that work, Cartan gave a different characterization of the notion of involutivity (see e.g. \cite{Bryant1991}). Many years later, in \cite{Spencer1962-1}, Spencer introduced the Spencer complex and cohomology in the context of deformations of geometric structures, and Guillemin, Singer and Sternberg conjectured and Serre proved (\cite{Guillemin1964,Singer1965}) the equivalence between Cartan's notion of involutivity and involutivity in terms of the Spencer cohomology as defined above. 
\end{myremark}

\subsection{Lie-Pfaffian Groupoids and Algebroids}
\label{section:Liepfaffaingroupoidandalgebroid}

The essential structure of a jet groupoid is encoded in its Cartan form. For example, Proposition \ref{prop:detectingholonomicsections} shows us that the notion of a holonomic section can be formulated purely in terms of the Cartan form. With this in mind, we can think of a jet groupoid more abstractly as a Lie groupoid equipped with a vector bundle-valued 1-form that satisfies the essential properties of the Cartan form, and similarly of a jet algebroid as a Lie algebroid equipped with an operator that satisfies the essential properties of the Spencer operator. 

This is the idea behind the theory of Lie-Pfaffian groupoids and algebroids that was introduced in \cite{Salazar2013}. These structures isolate the minimal set of essential ingredients that are present in jet groupoids and algebroids. Working with them has the advantage that constructions and proofs become substantially simpler and more transparent as compared to working directly with jets. For example, it is shown in \cite{Salazar2013} that the construction of prolongations of systems of PDEs and the proof of a formal integrability theorem can be formulated in terms of this structure alone, allowing one to avoid the messy computations and book-keeping that one typically encounters when working directly with jets. This appendix comes as a preparation for Section \ref{section:thesecondfundamentaltheorem}, where we use Lie-pfaffian groupoids and algebroids in our proof of Cartan's Second Fundamental Theorem, and for Section \ref{section:systaticspacesmodern}, where they will prove to be an essential ingredient in the reduction procedure. 

\subsubsection{Lie-Pfaffian Groupoids}

Let $\mathcal{G}\rightrightarrows M$ be a Lie groupoid with source and target maps $s$ and $t$, respectively. Recall that a differential form $\omega\in\Omega^*(\mathcal{G};t^*E)$ on $\mathcal{G}$ with values in a representation $E\to M$ of $\mathcal{G}$ is said to be multiplicative if it satisfies \eqref{eqn:multiplicativeform}.

\begin{mydef}
	\label{def:pfaffiangroupoid}
	\index{Pfaffian groupoid}
	\index{Lie-Pfaffian groupoid}
	A \textbf{Lie-Pfaffian groupoid} over a manifold $M$ is a Lie groupoid $\mathcal{G}\rightrightarrows M$ equipped with a pointwise-surjective multiplicative 1-form 
	\begin{equation*}
	\omega\in\Omega^1(\mathcal{G};t^*E)
	\end{equation*}
	(called the \textbf{Cartan form}) with values in a representation $E\to M$ of $\mathcal{G}$, such that:
	\begin{enumerate}
		\item $C_\omega \cap \Ker ds$ is an involutive distribution,
		\item $C_\omega\cap \Ker dt = C_\omega\cap \Ker ds$,
	\end{enumerate}
	where $C_\omega:=\Ker \omega$ (called the \textbf{Cartan distribution}). We call $(\mathcal{G},\omega)$ a \textbf{Pfaffian groupoid} if axiom 2 is omitted. 
	
	A \textbf{(local) holonomic bisection} (or a \textbf{(local) solution}) of $(\mathcal{G},\omega)$ is a (local) bisection $\eta$ of $\mathcal{G}$ satisfying $\eta^*\omega=0$. The set of local holonomic bisections is denoted by $\LBis(\mathcal{G},\omega)$.
\end{mydef}

\begin{myremark}
	\label{remark:consequencesliepfaffiangroupoid}
	\begin{enumerate}[label=(\alph*), wide, labelwidth=!, labelindent=0pt]
		\item In the majority of the paper, we will only encounter Lie-Pfaffian groupoids. The weaker notion of a Pfaffian groupoid only appears as a ``step'' in the reduction procedure of Section \ref{section:reductionofpseudogroupinnormalform}. For more applications of Pfaffian groupoids, see \cite{Salazar2013}. 
		\item Both Pfaffian and Lie-Pfaffian groupoids satisfy the property that 
	\begin{equation}
	C_\omega + \Ker ds = T\mathcal{G},
	\label{eqn:pfaffiangroupoidstransversal}
	\end{equation}
	or, equivalently, that
	\begin{equation}
	ds\big|_{C_\omega}:C_\omega\to s^*TM
	\label{eqn:pfaffiangroupoidstransversalequivalent}
	\end{equation}
	is pointwise-surjective. To verify this, one first shows that it holds at the units of $\mathcal{G}$ using the fact that the unit map is a holonomic section (which follows from multiplicativity of $\omega$). Then, one ``transfers'' the property to other points $g\in\mathcal{G}$ using the fact that $\mathcal{C}_\omega$ is of constant rank and that $(dR)_g(\mathcal{C}_\omega\cap \Ker ds)_{t(g)} = (\mathcal{C}_\omega\cap \Ker ds)_{g}$ (again, due to multiplicativity). Note that this condition is imposed in the definition given in \cite{Salazar2013}, but this is not necessary. 
	\item A consequence of axiom 2 is that $E$ inherits a Lie algebroid structure, namely the unique one with which 
	\begin{equation}
	\label{eqn:liepfaffiangroupoidliealgebroidmap}
	\omega|_{A}:A\to E,
	\end{equation}
	is a Lie algebroid map (see Proposition 6.1.8 in \cite{Salazar2013}). 
	\item In the case of a Lie-Pfaffian groupoid, axiom 1 is actually not needed, it follows from the other axioms. One way to see this is to note that, since \eqref{eqn:liepfaffiangroupoidliealgebroidmap} is a Lie algebroid map, then its kernel is closed under the bracket, and its bracket is precisely induced by the bracket of right invariant vector fields in $\mathcal{C}_\omega\cap \Ker ds$. Another argument is given in Remark 6.1.9 in \cite{Salazar2013}. For a Pfaffian groupoid, this axiom must be imposed separately. 
	\end{enumerate}
\end{myremark}

\begin{myexample}[Jet groupoids]
	\label{example:liepfaffiangroupoidjetgroupoids}
	Given a manifold $M$ and an integer $k>0$, the pair $(J^k M, \omega)$, consisting of the jet groupoid $J^k M\rightrightarrows M$ and the Cartan form $\omega\in\Omega^1(J^kM;t^*A^{k-1}M)$ with values in the adjoint representation $A^{k-1}M$, is a Lie-Pfaffian groupoid. Due to Proposition \ref{prop:cartanformjetgroupoid}, we are only left to verify axiom 2. For this, we note that there is a (unique) isomorphism $s^*TM\cong t^*TM$ of vector bundles over $J^kM$ with which the diagram
	\begin{equation}
	\label{eqn:commutativediagramliepseudogroup}
	\begin{tikzpicture}[description/.style={fill=white,inner sep=2pt},bij/.style={above,sloped,inner sep=.5pt}]	
	\matrix (m) [matrix of math nodes, row sep=0.75em, column sep=0.75em, 
	text height=1.5ex, text depth=0.25ex]
	{ 
		& C_\omega & \\
		&    \\
		s^*TM & & t^*TM \\
	};
	\path[->,font=\scriptsize]
	(m-1-2) edge node[above,xshift=-0.5em] {$ds$} (m-3-1)
	(m-1-2) edge node[above,xshift=0.5em] {$dt$} (m-3-3)
	(m-3-1) edge node[above] {$\simeq$} (m-3-3);
	\end{tikzpicture}
	\end{equation}
	commutes given at a point $j^k_x\phi$ by $(d\phi)_x$. This implies that $C_\omega\cap\Ker ds = C_\omega\cap\Ker dt$. Note that the isomorphism \eqref{eqn:commutativediagramliepseudogroup} defines an action with which $TM$ becomes a representation of $J^kM$. 
\end{myexample}

Our main examples of interest of a Lie-Pfaffian groupoid are the defining equations of Lie pseudogroups (which are Lie subgroupoids of jet groupoids). These will be discussed in Section \ref{section:liepseudogroupsasliepfaffiangroupoids}.

In the above example, we saw that $TM$ is canonically a representation of any jet groupoid. This is a general feature of Lie-Pfaffian groupoids: 

\begin{mylemma}
	\label{lemma:pfaffiangroupoidtriangle}
	Let $(\mathcal{G},\omega)$ be a Lie-Pfaffian groupoid. There exists a unique isomorphism $s^*TM\cong t^*TM$ with which the following diagram commutes:
	\begin{equation*}
	\begin{tikzpicture}[description/.style={fill=white,inner sep=2pt},bij/.style={above,sloped,inner sep=.5pt}]	
	\matrix (m) [matrix of math nodes, row sep=0.75em, column sep=0.75em, 
	text height=1.5ex, text depth=0.25ex]
	{ 
		& C_\omega & \\
		&    \\
		s^*TM & & t^*TM. \\
	};
	\path[->,font=\scriptsize]
	(m-1-2) edge node[above,xshift=-0.5em] {$ds$} (m-3-1)
	(m-1-2) edge node[above,xshift=0.5em] {$dt$} (m-3-3)
	(m-3-1) edge node[above] {$\simeq$} (m-3-3);
	\end{tikzpicture}
	\end{equation*}	
	This defines an action of $\mathcal{G}$ on $TM$ with which $TM$ becomes a representation of $\mathcal{G}$. 
\end{mylemma}

\begin{proof}
	The isomorphism is given by choosing a splitting $H:s^*TM\to C_\omega$ of \eqref{eqn:pfaffiangroupoidstransversalequivalent} and composing it with $dt$. It is independent of the choice, since the difference of any two connections takes values in $C_\omega\cap\Ker ds$, which is killed by $dt$ by axiom 2.
\end{proof}

\subsubsection{Lie-Pfaffian Algebroids}
\label{section:liepfaffianalgebroids}

The infinitesimal counterpart of a Lie-Pfaffian groupoid is \textit{a Lie-Pfaffian algebroid}.  Let $A\to M$ and $E\to M$ be vector bundles and let $l:A\to E$ be a surjective vector bundle map. An \textbf{$l$-connection} $D$ on $A$ is a bilinear map
\begin{equation*}
D:\mathfrak{X}(M)\times \Gamma(A)\to \Gamma(E)
\end{equation*}
satisfying the connection-like properties 
\begin{equation*}
D_{f X}(\alpha) = fD_X(\alpha),\hspace{1cm} D_X(f\alpha) = fD_X(\alpha) + X(f)l(\alpha),
\end{equation*}
for all $X\in\mathfrak{X}(M),\alpha\in\Gamma(A)$ and $f\in C^\infty(M)$. Equivalently, we can view $D$ as a linear map $D:\Gamma(A)\to \Omega^1(M;E)$ satisfying the Leibniz condition $D(f\alpha) = fD(\alpha) + df\otimes l(\alpha)$. If $A$ and $E$ are Lie algebroids and $l:A\to E$ is a surjective Lie algebroid map, then an $l$-connection induces an $A$-connection on $E$,
\begin{equation*}
\nabla^D:\Gamma(A)\times \Gamma(E)\to\Gamma(E),
\end{equation*}
defined by 
\begin{equation*}
\nabla^D_\alpha(\beta) = [l(\alpha),\beta] + D_{\rho(\beta)}(\alpha),\hspace{1cm} \forall \; \alpha\in \Gamma(A),\beta\in\Gamma(E).
\end{equation*}
Fixing a section $\alpha\in \Gamma(A)$, the connection induces a Lie derivative operation on $E$-valued forms,
\begin{equation*}
\mathcal{L}^D_\alpha:\Omega^*(M;E)\to \Omega^*(M;E).
\end{equation*}
We will only need the formula for 1-forms, which is
\begin{equation*}
\mathcal{L}^D_\alpha \omega (X) = \nabla^D_\alpha(\omega(X)) - \omega([\rho(\alpha),X]), \hspace{1cm} \forall \; \omega\in\Omega^1(M;E).
\end{equation*}
For the formula for arbitrary degrees, see \cite{Crainic2012}.

\begin{mydef}
	\label{def:liepfaffianalgebroid}
	\index{Lie-Pfaffian algebroid}
	A \textbf{Lie-Pfaffian algebroid} over a manifold $M$ consists of a pair of Lie algebroids $A$ and $E$ over $M$, a surjective Lie algebroid map $l:A\to E$ and an $l$-connection 
	\begin{equation*}
	D:\mathfrak{X}(M)\times \Gamma(A)\to \Gamma(E),
	\end{equation*}
	(called the \textbf{Spencer operator}) such that
	\begin{equation}
	\label{eqn:liepfaffianalgebroid}
	D([\alpha,\beta]) = \mathcal{L}^D_\alpha (D(\beta)) - \mathcal{L}^D_\beta (D(\alpha)),\hspace{1cm} \forall \; \alpha,\beta\in\Gamma(A).
	\end{equation}
	A \textbf{(local) holonomic section} of $(A,D)$ is a (local) section $\alpha\in\Gamma(A)$ such that $D(\alpha)=0$. 
\end{mydef}

\begin{myremark}
	Due to \eqref{eqn:liepfaffianalgebroid}, the $A$-connection $\nabla^D$ is flat and $E$ is a representation of $A$. 
\end{myremark}

A Lie-Pfaffian groupoid $(\mathcal{G},\omega)$ induces a Lie-Pfaffian algebroid $(A(\mathcal{G}),D_\omega)$, where $A=A(\mathcal{G})$ is the Lie algebroid of $\mathcal{G}$, $l:A\to E$ is the restriction of $\omega\in\Omega^1(\mathcal{G};t^*E)$ to $A$, and 
\begin{equation*}
D=D_\omega:\mathfrak{X}(M)\times\Gamma(A)\to \Gamma(E)
\end{equation*}
is obtained from $\omega$ by the differentiation formula \eqref{eqn:linearizingcartanform}. The following theorem is due to \cite{Salazar2013} (Theorem 6.1.23 and Proposition 6.1.25):

\begin{mytheorem}
	\label{theorem:liepfaffiangroupoidalgebroid}
	If $(\mathcal{G},\omega)$ is a Lie-Pfaffian groupoid, then $(A(\mathcal{G}),D_\omega)$ is a Lie-Pfaffian algebroid. Conversely, if $(A,D)$ is a Lie-Pfaffian algebroid such that $A$ is integrable and $\mathcal{G}$ is the $s$-simply connected integration of $A$, then there is a unique Lie-Pfaffian groupoid $(\mathcal{G},\omega)$ integrating $(A,D)$. 
\end{mytheorem}

\begin{myexample}[Jet algebroids]
	\label{example:liepfaffianalgebroidjetalgebroid}
	Continuing from Example \ref{example:liepfaffiangroupoidjetgroupoids}, the jet algebroid $A^kM$ of a manifold $M$, equipped with the Spencer operator $D:\mathfrak{X}(M)\times\Gamma(A^kM)\to \Gamma(A^{k-1}M)$, is a Lie-Pfaffian algebroid. 
\end{myexample}

The following lemma provides a useful formula for computing the Spencer operator induced by a Lie-Pfaffian groupoid: 

\begin{mylemma}
	\label{lemma:cartanformspenceroperatorformula}
	Let $(\mathcal{G},\omega)$ be a Lie-Pfaffian groupoid, and let $(A,D)$ be its Lie-Pfaffian algebroid. Then
	\begin{equation}
	\label{eqn:cartanformspenceroperatorformula}
	\omega([\widehat{X},\widetilde{\alpha}]) = t^*\big(D_X(\alpha)\big),
	\end{equation}
	for all $\alpha\in \Gamma(A)$ and $X\in\mathfrak{X}(M)$, where $\widetilde{\alpha}\in\mathfrak{X}(\mathcal{G})$ is the right invariant vector field induced by $\alpha$ and $\widehat{X}\in\mathfrak{X}(\mathcal{G})$ is a lift of $X$ that satisfies $dt(\widehat{X})=X$ and $\omega(\widehat{X})=0$. 
\end{mylemma} 

\begin{proof}
	First note that a lift $\widehat{X}$ always exist by Lemma \ref{lemma:pfaffiangroupoidtriangle} together with the fact that \eqref{eqn:pfaffiangroupoidstransversalequivalent} admits splittings. Let us write $x=t(g)$. One easily shows that $\varphi_{\widetilde{\alpha}}^\epsilon(g) = \varphi^\epsilon_\alpha(x)\cdot g$, and hence, by replacing $g$ with a curve representing $\widehat{X}_g$, we see that $d\varphi^\epsilon_{\widetilde{\alpha}}(\widehat{X}_g) = dm(d\varphi^\epsilon_\alpha(X_x),\widehat{X}_g)$. Applying $\omega$ on both sides and using the multiplicativity of $\omega$, 
	\begin{equation*}
	\omega(d\varphi^\epsilon_{\widetilde{\alpha}}(\widehat{X}_g)) = \omega(d\varphi^\epsilon_\alpha(X_x)) + \varphi^\epsilon_\alpha(x)\cdot \cancel{\omega(\widehat{X}_g))}.
	\end{equation*}
	With these identities, the right hand side of \eqref{eqn:linearizingcartanform} can be re-expressed as
	\begin{equation*}
	\varphi_\alpha^{\epsilon}(x)^{-1} \cdot \omega(d\varphi_\alpha^\epsilon (X_x)) = g\cdot \varphi_{\widetilde{\alpha}}^{\epsilon}(g)^{-1} \cdot \omega(d\varphi^\epsilon_{\widetilde{\alpha}}(\widehat{X}_g)).
	\end{equation*}
	Using this,
	\begin{equation*}
	\begin{split}
	D_X(\alpha)_x &= \frac{d}{d\epsilon}\Big|_{\epsilon=0}\; \varphi_\alpha^{\epsilon}(x)^{-1} \cdot \omega(d\varphi_\alpha^\epsilon (X_x)) 
	=\frac{d}{d\epsilon}\Big|_{\epsilon=0}\; g\cdot \varphi_{\widetilde{\alpha}}^{\epsilon}(g)^{-1} \cdot \omega(d\varphi^\epsilon_{\widetilde{\alpha}}(\widehat{X}_g)) \\
	&= \lim_{\epsilon\to 0} \frac{g\cdot \varphi_{\widetilde{\alpha}}^{\epsilon}(g)^{-1} \cdot \omega (d\varphi^\epsilon_{\widetilde{\alpha}}(\widehat{X}_g)) - \omega(\widehat{X}_g)}{\epsilon} \\
	&= \lim_{\epsilon\to 0} \frac{g\cdot\varphi_{\widetilde{\alpha}}^{\epsilon}(g)^{-1} \cdot \omega (d\varphi^\epsilon_{\widetilde{\alpha}} (d\varphi^{-\epsilon}_{\widetilde{\alpha}} \;\widehat{X}_{\varphi^\epsilon_{\widetilde{\alpha}}(x)})) - \omega(d\varphi^{-\epsilon}_{\widetilde{\alpha}} \;\widehat{X}_{\varphi^\epsilon_{\widetilde{\alpha}}(g)})}{\epsilon} \\
	&= \lim_{\epsilon\to 0} \frac{g\cdot \varphi_{\widetilde{\alpha}}^{\epsilon}(g)^{-1} \cdot \cancel{\omega (\widehat{X}_{\varphi^\epsilon_{\widetilde{\alpha}}(g)})} - \omega(d\varphi^{-\epsilon}_{\widetilde{\alpha}} \;\widehat{X}_{\varphi^\epsilon_{\widetilde{\alpha}}(g)})}{\epsilon} \\
	&= - \omega\Big( \frac{d}{d\epsilon}\Big|_{\epsilon=0}\; d\varphi^{-\epsilon}_{\widetilde{\alpha}} \;\widehat{X}_{\varphi^\epsilon_{\widetilde{\alpha}}(g)} \Big) 
	= \omega ([\widehat{X},\widetilde{\alpha}]_g)
	\end{split}
	\end{equation*}
	In the fourth equality, $\hat{X}_g$ can be replaced by $d\varphi^{-\epsilon}_{\widetilde{\alpha}} \;\widehat{X}_{\varphi^\epsilon_{\widetilde{\alpha}}(g)}$ since they coincide in the limit. 
\end{proof}

\subsubsection{The Symbol Space and the Symbol Map}
\label{section:symbolspaceandmap}

Let $(A,D)$ be a Lie-Pfaffian algebroid over $M$. The kernel of $l$,
\begin{equation}
\label{eqn:symbolspaceliepfaffainalgebroid}
\Tab=\Tab(A,D):=\Ker(l:A\to E)\subset A
\end{equation}
is called the \textbf{symbol space} of $(A,D)$. Being the kernel of a Lie algebroid morphism, it has the structure of a bundle of Lie algebras. Given a Lie-Pfaffian groupoid $(\mathcal{G},\omega)$ with associated Lie-Pfaffian algebroid $(A,D)$, we define  its symbol space to be the symbol space of $(A,D)$, thus $\Tab=\Tab(\mathcal{G},\omega):=\Tab(A,D)$. Note that, by right translation,
\begin{equation}
\label{eqn:symbolspaceliepfaffiangroupoid}
t^*\Tab(\mathcal{G},\omega)\cong C_\omega\cap \Ker ds.
\end{equation}
The restriction of the Spencer operator to the symbol space induces a map
\begin{equation}
\label{eqn:symbolmapliepfaffianalgebroid}
\partial=\partial_D:\Tab(A)\to \Hom(TM,E), \hspace{1cm} T\mapsto (\hat{T}:X\mapsto D_X(T)),
\end{equation}
called the \textbf{symbol map}. The pair $(\Tab,\partial)$ is a tableau bundle, in the sense of Definition \ref{def:tableaubundle}, and hence we can construct its prolongations and Spencer cohomology (see Section \ref{section:tableaubundles}).

\begin{mydef}
	\label{def:standard}
	A Lie-Pfaffian algebroid $(A,D)$ is \textbf{standard} if its symbol map is injective. A Lie-Pfaffian groupoid $(\mathcal{G},\omega)$ is \textbf{standard} if its associated Lie-Pfaffian algebroid is standard. 
\end{mydef}

\begin{myexample}
	The Lie-Pfaffian groupoids coming from jet groupoids (Example \ref{example:liepfaffiangroupoidjetgroupoids}) and from Lie pseudogroups (Section \ref{section:liepseudogroupsasliepfaffiangroupoids}) are standard, as well as their associated Lie-Pfaffian algebroids.
\end{myexample}

\subsubsection{The Differential of the Cartan Form}

The main problem in the study of PDEs is that of integrability, i.e. the existence of solutions. Obstructions to integrability are obtained by looking at what are known as prolongations of PDEs, which are first and higher order differential consequences of the equations. The construction of prolongations can also be thought of as the construction of formal solutions (see Remark \ref{remark:formalsolutions}). In the framework of Lie-Pfaffian groupoids, prolongations are encoded in the differential of the Cartan form, as we now explain.

While the differential of a vector bundle-valued form is not canonically defined, its restriction to the kernel of the form is, and it is precisely this part that contains the relevant information. Let $(\mathcal{G},\omega)$ be a Lie-Pfaffian groupoid, with $\omega\in\Omega^1(\mathcal{G};t^*E)$, where $E$ is a representation of $\mathcal{G}$. A choice of a connection $\nabla$ on $E$ induces a de-Rham type operator $d_\nabla$ on the space of $E$-valued forms (by the usual Koszul-type formula), and we denote the restriction of $d_\nabla\omega\in\Omega^2(\mathcal{G};t^*E)$ to the Cartan distribution $C_\omega$, the kernel of $\omega$, by 
\begin{equation}
\label{eqn:differentialcartanform}
\delta\omega:=d_\nabla\omega\big|_{C_\omega}:\Lambda^2C_\omega\to t^*E.
\end{equation}
At the level of sections, 
\begin{equation*}
\delta\omega(X,Y)=-\omega([X,Y]),\hspace{1cm} \forall\, X,Y\in\Gamma(C_\omega).
\end{equation*} 
As an immediate consequence of Lemma \ref{lemma:cartanformspenceroperatorformula}, we obtain the following formula that relates the symbol map with the differential of the Cartan form:

\begin{mylemma}
	\label{lemma:symbolmapdifferentialcartanform}
	Let $(\mathcal{G},\omega)$ be a Lie-Pfaffian groupoid over $M$. Then
	\begin{equation}
	\label{eqn:symbolmapdifferentialcartanform}
	\delta\omega(T^r,\hat{X}) = t^*(\delta(\widetilde{T})(X)),
	\end{equation}
	for all $T\in \Gamma(\Tab)$ and $X\in\mathfrak{X}(M)$, where $\widetilde{T}\in\mathfrak{X}(\mathcal{G})$ is the right invariant vector field induced by $T$ and $\widehat{X}\in\mathfrak{X}(\mathcal{G})$ is a lift of $X$ that satisfies $dt(\widehat{X})=X$ and $\omega(\widehat{X})=0$. Because $\delta\omega$ is a tensor, \eqref{eqn:symbolmapdifferentialcartanform} holds pointwise.  
\end{mylemma}

The prolongation of a Lie-Pfaffian groupoid, which we call the \textit{classical prolongation}, is the space of all first order solutions. Recall that a local solution (or a local holonomic bisection) $b\in\LBis(\mathcal{G},\omega)$ of a Lie-Pfaffian groupoid $(\mathcal{G},\omega)$ over $M$ is a local bisection $b$ that satisfies 
\begin{equation}
\label{eqn:holonomicequation}
b^*\omega=0.
\end{equation}
As a consequence of \eqref{eqn:holonomicequation}, the first order approximation of a local solution $b$ at a point $x\in\Dom(b)\subset M$, i.e. its differential $\xi:=(db)_x:T_xM\to T_{b(x)}\mathcal{G}$, satisfies the equations
\begin{equation}
\label{eqn:firstordersolution}
\xi^*\omega=0 \hspace{1cm}\text{and}\hspace{1cm} \xi^*\delta\omega=0.
\end{equation}
This leads us to the following definition. Let $g\in\mathcal{G}$ and let $\xi:T_{s(g)}M\to T_g\mathcal{G}$ be a linear map such that $(ds)_g\circ \xi=\Id$ and $(dt)_g\circ\xi:T_{s(g)}M\to T_{t(g)}M$ is a linear isomorphism (note that any such linear map can be written as $\xi=(db)_{s(g)}$ for some local bisection $b$ of $\mathcal{G}$). Such a linear map $\xi$ is called a \textbf{1st order solution} of $(\mathcal{G},\omega)$ if it satisfies \eqref{eqn:firstordersolution}. Of course, first order solutions do not necessarily arise as the differentials of local solutions, but the existence of a first order solutions is a necessary condition for the existence of local solutions. 

\begin{mydef}
	Let $(\mathcal{G},\omega)$ be a Lie-Pfaffian groupoid. The space of all 1st order solutions,
	\begin{equation*}
	P_\omega(\mathcal{G}):= \{ \; \xi=(db)_x\;|\; b\in\LBis(\mathcal{G}),\; x\in\Dom(b),\; \xi^*\omega=0, \; \xi^*\delta\omega=0  \; \},
	\end{equation*}
	is called the \textbf{classical prolongation} of $(\mathcal{G},\omega)$.
\end{mydef}

Relaxing the two conditions in \eqref{eqn:firstordersolution} one by one, we get two inclusions,
\begin{equation*}
P_\omega(\mathcal{G})\subset J^1_\omega\mathcal{G} \subset J^1\mathcal{G},
\end{equation*}
where
\begin{equation}
\label{eqn:partialprolongation}
J^1_\omega(\mathcal{G}):= \{ \; \xi=(db)_x\;|\; b\in\LBis(\mathcal{G}),\; x\in\Dom(b),\; \xi^*\omega=0  \; \}
\end{equation}
is called the \textbf{partial prolongation} of $(\mathcal{G},\omega)$, and
\begin{equation}
\label{eqn:firstjetgroupoidofbisections}
J^1\mathcal{G} = \{\; (db)_x\;|\; b\in\LBis(\mathcal{G}),\; x\in\Dom(b) \;\}
\end{equation}
is precisely the 1st jet groupoid of local bisections of $\mathcal{G}$ (where first jets $j^1_xb$ of local bisections are canonically identified with the differentials $(db)_x$ of local bisections). The classical prolongation inherits its structure from these ambient spaces, as we explain in the next two sections. 

\subsubsection{The Partial Prolongation and its Affine Structure}

Let us examine more closely the partial prolongation $J^1_\omega\mathcal{G}$ of a Lie-Pfaffian groupoid $(\mathcal{G},\omega)$. The 1st jet groupoid $J^1\mathcal{G}\rightrightarrows M$ of a Lie groupoid $\mathcal{G}\rightrightarrows M$, as defined in  \eqref{eqn:firstjetgroupoidofbisections} (and see also Remark \ref{remark:jetgroupoidsofgroupoids}), is a Lie groupoid over $M$. The source and target maps send $(db)_x$ to $x$ and $\phi_b(x)$, where $\phi_b=t\circ b$, multiplication is induced by the composition of local bisections, i.e. $(db')_y\cdot(db)_x := (d(b'\cdot b))_x$, and the inverse and unit maps are induced by the inverse operation on bisections and the identity bisection. The smooth structure is the usual one for jet spaces, and the projection,
\begin{equation}
\label{eqn:firstjetgroupoidprojection}
\pi:J^1\mathcal{G}\to \mathcal{G}, \hspace{1cm} (db)_x\mapsto b(x),
\end{equation}
is a surjective Lie groupoid morphism and a submersion. The Cartan form, 
\begin{equation}
\label{eqn:firstjetgroupoidcartanform}
\omega\in\Omega^1(J^1\mathcal{G};t^*A),
\end{equation}
which takes with values in the Lie algebroid $A$ of $\mathcal{G}$, is defined by the formula (c.f. \eqref{eqn:cartanformmultaiplicativepde})
\begin{equation*}
\omega_{(db)_x}:= dR_{b(x)^{-1}}\cdot (d\pi-(db)_x\circ ds)_{(db)_x}.
\end{equation*}
Also here, $A$ is a representation of $J^1\mathcal{G}$, the adjoint representation (the action is given by \eqref{eqn:jetgroupoidadjointrepresentation}), and $\omega$ is multiplicative. With this structure, $(J^1\mathcal{G},\omega)$ is a Lie-Pfaffian groupoid. 

The partial prolongation $J^1_\omega\mathcal{G}\subset J^1\mathcal{G}$ inherits this Lie-Pfaffian groupoid structure. The main step in showing this is to show that this inclusion is smooth. This follows from the following important observation: the restriction of the projection \eqref{eqn:firstjetgroupoidprojection}, which we also denote by
\begin{equation}
\label{eqn:partialprolongationprojection}
\pi:J^1_\omega\mathcal{G}\to \mathcal{G},
\end{equation}
has the structure of an affine bundle modeled on $t^*\Hom(TM,\Tab)$, where $\Tab$ is the symbol space of $(\mathcal{G},\omega)$. To describe this structure, recall that, by Lemma \ref{lemma:pfaffiangroupoidtriangle} and \eqref{eqn:symbolspaceliepfaffiangroupoid}, we have canonical isomorphisms $s^*TM\cong t^*TM$ and $C_\omega\cap \Ker ds\cong t^*\Tab$ of vector bundles over $\mathcal{G}$. Together, these give the identification
\begin{equation}
\label{eqn:modelingvectorbundle}
t^*\Hom(TM,\Tab)\cong \Hom(s^*TM,C_\omega\cap\Ker ds).
\end{equation}
Let us first describe the affine space structure of a single fiber of \eqref{eqn:partialprolongationprojection} over a point $g\in\mathcal{G}$. The difference $(db')_{s(g)}-(db)_{s(g)}$ of two points in the fiber is a linear map $T_{s(g)}M\to T_g\mathcal{G}$ which takes values in $C_\omega\cap\Ker ds$, and hence, modulo \eqref{eqn:modelingvectorbundle}, belongs to $\Hom_{t(g)}(TM,\Tab)$. Conversely, if $(db)_{s(g)}$ is in this fiber and $\zeta\in\Hom_{t(g)}(TM,\Tab)$, then the sum $(db)_{s(g)}+\zeta$ is again in this fiber, where axiom 2 of Definition \ref{def:pfaffiangroupoid} ensures that the composition of $(db)_{s(g)}+\zeta$ with $dt$ is a linear isomorphism, and hence $(db)_{s(g)}+\zeta$ comes from a local bisection. We thus have a collection of affine spaces parametrized by $\mathcal{G}$. These, in turn, glue smoothly to an affine bundle, since \eqref{eqn:partialprolongationprojection} has a smooth global section (smooth as a section of \eqref{eqn:firstjetgroupoidprojection}). Indeed, a section of \eqref{eqn:partialprolongationprojection} is the same thing as a splitting of \eqref{eqn:pfaffiangroupoidstransversalequivalent}, which always exists.

Finally, equipped with the restriction of \eqref{eqn:firstjetgroupoidcartanform}, which we also denote by
\begin{equation}
\label{eqn:partialprolongationcartanform}
\omega\in\Omega^1(J^1_\omega\mathcal{G};t^*A),
\end{equation}
it is straightforward to show that the partial prolongation is a Lie-Pfaffian groupoid over $M$. For example, the fact that the restrictions of the multiplication and inverse maps are well-defined and that the unit map is surjective is a consequence of the multiplicativity of $\omega$. We refer to Section 6.2.3 in \cite{Salazar2013} for more details. To summarize:

\begin{myprop}
	Let $(\mathcal{G},\omega)$ be a Lie-Pfaffian groupoid. The partial prolongation $(J^1_\omega\mathcal{G},\omega)$ is a Lie-Pfaffian groupoid and the projection \eqref{eqn:partialprolongationprojection} is an affine bundle modeled on $t^*\Hom(TM,\Tab)$ and a morphism of Lie groupoids.
\end{myprop}

\subsubsection{The Classical Prolongation and its Affine Structure}

While the partial prolongation $J^1_\omega\mathcal{G}$ of a Lie-Pfaffian groupoid $(\mathcal{G},\omega)$ is always smooth, the classical prolongation $P_\omega(\mathcal{G})$ may fail to be so by failing to be an affine subbundle of $J^1_\omega(\mathcal{G})$. Understanding when it is smooth is a first step in the problem of integrability. In the case it is smooth, it is note hard to show that it is a Lie subgroupoid of $J^1_\omega\mathcal{G}$ and a Lie-Pfaffian groupoid (when equipped with the restriction of $\omega$ on $J^1_\omega\mathcal{G}$)).

Consider the projection 
\begin{equation*}
\pi:P_\omega(\mathcal{G})\to \mathcal{G},
\end{equation*}
the restriction of \eqref{eqn:partialprolongationprojection}, and recall that $\Tab^{(1)}$ denotes the 1st prolongation of the symbol space $\Tab$ (see Section \ref{section:tableaubundles}).

\begin{mylemma}
	Let $(\mathcal{G},\omega)$ be a Lie-Pfaffian groupoid. If the fiber of $\pi:P_\omega(\mathcal{G})\to \mathcal{G}$ at $g\in\mathcal{G}$ is non-empty, then it is an affine subspace of the respective fiber of \eqref{eqn:partialprolongationprojection}, and it is modeled on the vector subspace $\Tab^{(1)}_{t(g)}\subset \Hom_{t(g)}(TM,\Tab)$. 
\end{mylemma}

\begin{proof}
	Fix $g\in\mathcal{G}$. To simplify notation, we treat the identifications $s^*TM \cong t^*TM$, $C_\omega\cap \Ker ds \cong t^*\Tab$ and \eqref{eqn:modelingvectorbundle} as equalities. First, let $\xi,\xi'\in P_\omega(\mathcal{G})$ in the fiber over $g$. We prove that, for every $X,Y\in T_{t(g)}M$ the difference $\xi'-\xi$, which is a priori an element of $\Hom_{t(g)}(TM,\Tab)$, is an element of $\Tab^{(1)}$:
	\begin{equation*}
	\begin{split}
	\partial((\xi'-\xi)(X))(Y) &= \delta\omega((\xi'-\xi)(X),\xi(Y))
	= \delta\omega(\xi'(X),\xi(Y)) \\
	&= \delta\omega(\xi'(X),(\xi-\xi')(Y)) 
	= \delta\omega((\xi'-\xi)(Y),\xi'(X)) 
	= \partial((\xi'-\xi)(Y))(X).
	\end{split}
	\end{equation*}
	Lemma \ref{lemma:symbolmapdifferentialcartanform} was used in the first and last equality, the fact that $\xi^*\delta\omega=\xi'^*\delta\omega=0$ (since they are elements of $P_\omega(\mathcal{G})$) in the second and third, and anti-symmetry of $\delta\omega$ in the fourth. 
	
	Next, let $\xi\in P_\omega(\mathcal{G})$ in the fiber over $g$ and let $\zeta\in\Tab^{(1)}_{t(g)}$. We know already that $\xi+\zeta\in J^1_\omega\mathcal{G}$, in the fiber over $g$, and we prove that it is in $P_\omega(\mathcal{G})$. Clearly, $(\xi+\zeta)^*\omega=0$. Furthermore, $(\xi+\zeta)^*\delta\omega=0$, since for every $X,Y\in T_{t(g)}M$,
	\begin{equation*}
	\begin{split}
	&\delta\omega((\xi+\zeta)(X),(\xi+\zeta)(Y)) \\
	& \hspace{0.5cm} = \delta\omega(\xi(X),\xi(Y)) + \delta\omega(\zeta(X),\xi(Y)) - \delta\omega(\zeta(Y),\xi(X))  + \delta\omega(\zeta(X),\zeta(Y)) \\
	& \hspace{0.5cm} = 0,
	\end{split}
	\end{equation*}
	where the first term vanishes because $\xi^*\delta\omega=0$, the sum of the second and third are equal to $\partial(\zeta(X))(Y) - \partial(\zeta(Y))(X)$, which vanish because $\xi\in\Tab^{(1)}$, and the third vanishes because $C_\omega\cap\Ker ds$ is involutive. 
\end{proof}

So we can conclude that:

\begin{myprop}
	\label{prop:smoothnessclassicalprolongation}
	Let $(\mathcal{G},\omega)$ be a Lie-Pfaffian groupoid over $M$. The projection
	\begin{equation}
	\label{eqn:prolongationprojection}
	\pi:P_\omega(\mathcal{G})\to \mathcal{G},
	\end{equation}
	is an affine bundle (modeled on $t^*\Tab^{(1)}$) if and only if $\Tab^{(1)}$ is of constant rank and $\pi$ has a smooth global section (smooth as a section of \eqref{eqn:partialprolongationprojection}). In this case, the pair $(P_\omega(\mathcal{G}),\omega)$ is a Lie-Pfaffian groupoid and $\pi$ is a morphism of Lie groupoids.
\end{myprop}

\begin{myremark}
	In \cite{Salazar2013}, it is shown that \eqref{eqn:prolongationprojection} is, in fact, a \textit{morphism} of Lie-Pfaffian groupoids, in a sense that they make precise. Furthermore, they define an abstract notion of a \textit{Lie prolongation}, which, roughly speaking, is a morphism $p:(\widetilde{\mathcal{G}},\widetilde{\omega})\to (\mathcal{G},\omega)$ of Lie-Pfaffian groupoids such that $\widetilde{\omega}$ ``extends'' $\omega$. It is then proven that \eqref{eqn:prolongationprojection} is a Lie prolongation and that it is ``universal'' in some sense (see Proposition 6.2.42 in \cite{Salazar2013} for the precise statement).
\end{myremark}

\begin{myremark}
	\label{remark:formalsolutions}
	In the study of formal integrability, the classical prolongation of a Lie-Pfaffian groupoid $(\mathcal{G},\omega)$ is also called the \textit{1st prolongation}. If it is smooth, i.e. it is a Lie-Pfaffian groupoid and the projection is an affine bundle, then we can proceed and construct its classical prolongation, which is called the \textit{2nd prolongation} of $(\mathcal{G},\omega)$. Its elements correspond to \textit{2nd order solutions} of $(\mathcal{G},\omega)$. Proceeding inductively (where at each step there may be obstructions to smoothness), we obtain, at the $k$-th step, the \textit{$k$-th prolongation} of $(\mathcal{G},\omega)$ consisting of \textit{$k$-th order solutions}. If there are no obstructions and we can continue indefinitely, then the inverse limit of the resulting infinite tower of prolongations is called the \textit{$\infty$-prolongation} of $(\mathcal{G},\omega)$ and its elements correspond to \textit{formal solutions}. In this case, $(\mathcal{G},\omega)$ is said to be formally integrable. Theorem 6.3.13 in \cite{Salazar2013} gives criteria for formal integrability.
\end{myremark}

\subsubsection{Cartan-Ehresmann Connections}
\label{section:cartanehresmannconnection}

The existence of sections of the projection $\pi:P_\omega(\mathcal{G})\to\mathcal{G}$ is an obstruction to prolongation and to formal integrability. Geometrically, such sections can be interpreted as special type of connections. We first note that sections of $\pi:J^1_\omega\mathcal{G}\to \mathcal{G}$ are the same thing as splittings of the vector bundle map \eqref{eqn:pfaffiangroupoidstransversalequivalent}, i.e. splittings $H$ of $ds:T\mathcal{G}\to s^*TM$ that satisfy the condition $H^*\omega=0$. In other words, they are are Ehresmann connections that take value in the Cartan distribution. These, in turn, correspond to sections of the classical prolongation if and only if they satisfy the extra condition $H^*\delta\omega=0$. This motivates the following terminology:

\begin{mydef}
	\label{def:cartanehresmannconnection}
	Let $(\mathcal{G},\omega)$ be a Lie-Pfaffian groupoid over $M$. A \textbf{Cartan-Ehresmann connection} is a splitting
	\begin{equation*}
	H:s^*TM\to C_\omega
	\end{equation*}
	of \eqref{eqn:pfaffiangroupoidstransversalequivalent}. It is said to be \textbf{integral} if $H^*\delta\omega=0$. 
\end{mydef}

Cartan-Ehresmann connections play an important role in our proof of Cartan's Second Fundamental Theorem in Section \ref{section:thesecondfundamentaltheorem}. 

\subsection{Generalized Pseudogroups}
\label{section:generalizedpseudogroups}

A locally defined diffeomorphism of a manifold $M$ can be interpreted as a local bisection of the pair groupoid $M\times M \rightrightarrows M$. This point of view leads us to the notion of a \textit{generalized pseudogroup} by replacing the pair groupoid with any Lie groupoid $\mathcal{G}$. Generalized pseudogroups arise naturally when one studies the space of local holonomic bisections (or ``local solutions'') of a Lie-Pfaffian groupoid (Appendix \ref{section:Liepfaffaingroupoidandalgebroid}) and they are central to the reduction procedure of Section \ref{section:systaticspacesmodern}.

We denote the set of local bisections of a Lie groupoid $\mathcal{G}$ by $\LBis(\mathcal{G})$. Recall that any two local bisections of a Lie groupoid $\mathcal{G}$ can be composed if their domains and codomains are compatible, that a local bisection has an inverse and that there exists a unit bisection, which we denote by $1\in \LBis(\mathcal{G})$. 

\begin{mydef}
	\label{def:generalized_pseudogroup}
	\index{generalized pseudogroup}
	A \textbf{generalized pseudogroup} on a Lie groupoid $\mathcal{G}\rightrightarrows M$ is a subset $\Gamma\subset \LBis(\mathcal{G})$ that satisfies the following axioms:
	\begin{enumerate}[label=\Alph*)]
		\item Group-like axioms:
		\begin{enumerate}[label=\arabic*)]
			\item if $\sigma,\sigma'\in\Gamma$ and $\Image(t \circ \sigma')\subset \Dom(\sigma)$, then $\sigma\cdot \sigma'\in\Gamma$,
			\item if $\sigma \in\Gamma$, then $\sigma^{-1}\in\Gamma$,
			\item $1 \in \Gamma$.
		\end{enumerate}
		\item Sheaf-like axioms:
		\begin{enumerate}[label=\arabic*)]
			\item if $\sigma \in\Gamma$ and $U\subset\Dom(\sigma)$ is an open subset, then $\sigma|_U\in\Gamma$,
			\item if $\sigma \in\LBis(\mathcal{G})$ and $\{U_i\}_{i\in I}$ is an open cover of $\Dom(\sigma)$ such that $\sigma|_{U_i}\in\Gamma$ for all $i\in I$, then $\sigma\in\Gamma$.
		\end{enumerate}
	\end{enumerate}
\end{mydef}

\begin{myremark}
	$\LBis(\mathcal{G})$ has the structure of a groupoid over $\mathrm{Open}(M)$, the set of open subsets of $M$. Using this observation, the first three axioms can be rephrased as saying that $\Gamma\subset \LBis(\mathcal{G})$ is a wide subgroupoid. 
\end{myremark}

\begin{myexample}[pseudogroups as generalized pseudogroups]
	\label{example:pseudogroupasgeneralizedpseudogroup}
	Pseudogroups on $M$ are the same thing as generalized pseudogroups on the pair groupoid $M\times M\rightrightarrows M$, where a locally defined diffeomorphism $\phi$ of $M$ is viewed as the local bisection $x\mapsto (\phi(x),x)$ of $M\times M$. 
\end{myexample}

\begin{myexample}(Lie groups as generalized pseudogroups)
	\label{example:liegroupasgeneralizedpseudogroup}
	While realizing a Lie group $G$ as a Lie pseudogroup depends on the choice of a space $M$ and an action of $G$ on $M$, the generalized point of view allows us to make sense of a Lie group as a (generalized) pseudogroup in a canonical way: taking the Lie groupoid $G\rightrightarrows \{*\}$, its group of local bisections coincides with $G$ as a group. This is a trivial yet conceptually important interpretation. For instance, using the reduction procedure of Section \ref{section:systaticspacesmodern}, certain Lie pseudogroups can be reduced to Lie groups when interpreting them as generalized pseudogroups in this sense (see e.g. Example \ref{example:liegroups}).
\end{myexample}

\begin{myexample}[Lie-Pfaffian groupoids]
	\label{example:generalizedpseudogroupliepfaffiangroupoid}
	The set $\LBis(\mathcal{G},\omega)$ of local holonomic bisections of a Lie-Pfaffian groupoid $(\mathcal{G},\omega)$ is a generalized pseudogroup, which is a consequence of the multiplicativity property of $\omega$. In particular, this holds for the jet groupoids $(J^kM,\omega)$ (Example \ref{example:liepfaffiangroupoidjetgroupoids}), where, by Proposition \ref{prop:detectingholonomicsections}, we have a bijection
	\begin{equation*}
	\LDiff(M) \xrightarrow{\simeq} \LBis(J^kM,\omega),\hspace{1cm} \phi\mapsto j^k\phi.
	\end{equation*}
\end{myexample}

\begin{myexample} (the classical shadow)
	Any generalized pseudogroup $\Gamma$ on a Lie groupoid $\mathcal{G}\rightrightarrows M$ induces a pseudogroup $\Gamma_{\text{cl}}$ on the base $M$ by ``projecting'' the elements, i.e.
	\begin{equation*}
	\Gamma_{\text{cl}}:= \{\; t\circ \sigma\;|\;\sigma\in\Gamma \;\} \subset \LDiff(M).
	\end{equation*}
	We call $\Gamma_{\text{cl}}$ the \textbf{classical shadow} of $\Gamma$. ``Classical'' pseudogroups often arise as the classical shadows of generalized pseudogroups that are more natural than the actual pseudogroups one is interested in studying. Consider, for instance, the pseudogroup on $\mathbb{R}$ generated by the set of diffeomorphisms
	\begin{equation*}
	\phi:\mathbb{R}\to \mathbb{R},\hspace{1cm} \phi(x) = (ax+b),
	\end{equation*}
	parametrized by $a\in\mathbb{R}\backslash\{0\}$ and $b\in\mathbb{R}$. It clearly comes from the action of a Lie group. Indeed, we take the Lie groups $(\mathbb{R},+)$ and $(\mathbb{R}\backslash\{0\},\times)$, and the map $\varphi:\mathbb{R}\backslash\{0\}\to \text{Aut}(\mathbb{R}),\; a\mapsto (b\mapsto ab)$. From this data we construct the semi direct product 
	\begin{equation*}
	\mathbb{R}\backslash\{0\}{}_\varphi\!\!\ltimes \mathbb{R},
	\end{equation*}
	where the product of two elements $(a,b)$ and $(a',b')$ is given by
	\begin{equation*}
	(a',b')\cdot(a,b)=(a'a,b' + a'b).
	\end{equation*}
	We then consider the action groupoid induced the action of $\mathbb{R}\backslash\{0\}{}_\varphi\!\ltimes \mathbb{R}$ on $\mathbb{R}$, where the action is given by $(a,b)\cdot x = ax+b$. The above pseudogroup is the classical shadow of the generalized pseudogroup of ``constant'' bisections of this action groupoid. 
\end{myexample}

\begin{myexample} (classical pseudogroups made nicer)
	Cartan's approach to the study of Lie pseudogroups works well under suitable regularity conditions (such as in Definition \ref{def:liepseudogroup}) and is best understood in the transitive case. Generalized pseudogroups often allow one replace ill-behaved pseudogroups by well-behaved generalized pseudogroups. Here is an illustration of this phenomenon. Consider the pseudogroup $\Gamma_{\text{cl}}$ of rotations on $\mathbb{R}^2$, i.e. the one generated by the set of diffeomorphisms
	\begin{equation*}
	\phi:\mathbb{R}^2\to \mathbb{R}^2,\hspace{1cm} \phi(x,y) = (x\cos \,2\pi\theta + y\sin\,2\pi\theta, -x\sin\,2\pi\theta+y\cos\,2\pi\theta),	
	\end{equation*}
	parametrized by $\theta\in \mathbb{R}/\mathbb{Z}$. Since it is has a singular orbit, it is not a Lie pseudogroup, but it is the classical shadow of the generalized pseudogroup $\Gamma_{\text{gen}}$ of ``constant'' bisections of the action groupoid $\mathbb{R}/\mathbb{Z} \ltimes \mathbb{R}^2$ associated with the action of the Lie group $\mathbb{R}/\mathbb{Z}$ on $\mathbb{R}^2$ by rotations, i.e.,
	\begin{equation*}
	\mathbb{R}/\mathbb{Z} \times \mathbb{R}^2\to \mathbb{R}^2,\hspace{0.5cm} \theta\cdot(x,y) = (x\cos \,2\pi\theta + y\sin\,2\pi\theta, -x\sin\,2\pi\theta+y\cos\,2\pi\theta).
	\end{equation*}
\end{myexample}

\subsection{Basic Forms on Principal $\mathcal{G}$-Bundles}
\label{appendix:basicforms}

In this appendix, we recall the notion of a basic form in the setting of Lie groupoid and Lie algebroid actions, which will be used in the reduction procedure in Section \ref{section:reductionofpseudogroupinnormalform}. Let us begin with the more familiar case of Lie groups. Let $\pi:P\to B$ be a principal $G$-bundle, with $G$ a Lie group, and let $V$ be a representation of $G$. A $V$-valued differential form $\theta\in \Omega^*(P;V)$ is said to be basic if it is horizontal and $G$-equivariant. Horizontal means that $\theta$ vanishes if applied to at least one vertical tangent vector, while $G$-equivariance means that
\begin{equation*}
L_g^* \theta = g \cdot\theta,\hspace{1cm} \forall\;g\in G,
\end{equation*}
where $L_g:P\to P,\; p\mapsto g\cdot p$. Basic $V$-valued forms on $P$, which we denote by $\Omega^*_{\text{bas}}(P;V)$, are precisely those $V$-valued forms on $P$ that come from forms on the base $B$. More precisely, recalling that the associated vector bundle $E=E(P,V)$ on $B$ is the vector bundle obtained as the quotient of the trivial vector bundle $P\times V\to P$ by the induced action of $G$ on $P\times V$ given by $g\cdot(p,v)=(g\cdot p,g\cdot v)$, the pull-back by $\pi$ gives a linear isomorphism 
\begin{equation*}
\pi^*:\Omega^*(B;E)\xrightarrow{\simeq} \Omega^*_{\text{bas}}(P;V).
\end{equation*}
Note that, on the right hand side, we are implicitly using the canonical isomorphism $P\times V\xrightarrow{\simeq} \pi^* E$ which, at a point $p$, maps $v \mapsto [p,v]$. See e.g. \cite{Kobayashi1963} (Section II.5) for more details. 

The notion of a basic form generalizes naturally to the setting of Lie groupoids. For simplicity, we restrict to the case of $1$-forms, which is of relevance to us. To keep the notation clean, given a surjective submersion $\pi:P\to M$ and a vector bundle $E\to M$, we write $\Omega^1(P;E)$ for the space of $\pi^*E$-valued 1-forms on $P$. 

\begin{mydef}
	\index{Lie groupoid!basic form}
	Let $\pi:P\to M$ be a surjective submersion equipped with an action of a Lie groupoid $\mathcal{G}\rightrightarrows M$ and let $E\to M$ be a representation of $\mathcal{G}$. A $1$-form $\theta\in\Omega^1(P;E)$ is \textbf{horizontal} if it vanishes on all vectors that are tangent to the orbits of the action of $\mathcal{G}$. An $E$-valued $1$-form $\theta\in\Omega^1(P;E)$ is \textbf{basic} if it is both horizontal and satisfies
	\begin{equation}
	\label{eqn:basicequivariance}
	\theta(g\cdot X) = g\cdot\theta(X), 
	\end{equation}
	for all $g\in\mathcal{G}$ and $X\in T_pP$ for which $s(g)=\pi(p)$. We denote the space of basic $1$-forms by $\Omega^1_\text{bas}(P;E)$.
\end{mydef}

Let us clarify the left hand side of \eqref{eqn:basicequivariance}. This is best understood in terms of the action groupoid $\mathcal{G}\ltimes P\rightrightarrows P$ associated with the action of $\mathcal{G}$ on $P$. We regard $\theta$ as a form on the base of the action groupoid. The fact that $\theta$ is horizontal is equivalent to the condition that the restriction of $\theta$ to any orbit $\mathcal{O}\subset P$ of $\mathcal{G}\ltimes P$ must vanish. If $\theta$ is horizontal, then it descends to a map $\theta:N\mathcal{O}\to E|_\mathcal{O}$ on the normal bundle to an orbit. With this in mind, \eqref{eqn:basicequivariance} should be read as $\theta((g,p)\cdot[X]) = g\cdot\theta([X])$, for all $(g,p)\in \mathcal{G}\ltimes P$ and $X\in T_pP$, where $(g,p)$ acts on $[X]$ via the normal representation (see e.g. Section 2.1 in \cite{Yudilevich2016-2} for a description of the normal representation). 

The notion of a basic form can also be defined at the infinitesimal level. 

\begin{mydef}
	\label{def:basicformliealgebroid}
	\index{Lie algebroid!basic form}
	Let $\pi:P\to M$ be a surjective submersion equipped with an action $a:\pi^*A \to TP$  of a Lie algebroid $A\to M$ and let $\nabla:\Gamma(A)\times \Gamma(E)\to \Gamma(E)$ be a representation of $A$. A $1$-form $\theta\in\Omega^1(P;E)$ is \textbf{horizontal} if $\theta(a(\alpha))=0$ for all $\alpha\in\Gamma(A)$. A $1$-form $\theta\in\Omega^1(P;E)$ is \textbf{basic} if it is both horizontal and satisfies
	\begin{equation}
	\label{eqn:basicequivariancealgebroid}
	\theta([a(\alpha),X]) = (\pi^*\nabla)_{\pi^*\alpha}\theta(X),
	\end{equation}
	for all $\alpha\in\Gamma(A)$ and $X\in\mathfrak{X}(P)$. 
\end{mydef}

Let clarify the right-hand side of \eqref{eqn:basicequivariancealgebroid}. The action of $A$ on $P$ induces the action algebroid $\pi^*A\to P$, and the representation $\nabla:\Gamma(A)\times \Gamma(E)\to \Gamma(E)$ of $A$ induces a representation $\pi^*\nabla:\Gamma(\pi^*A)\times \Gamma(\pi^*E)\to \Gamma(\pi^*E)$ of $\pi^*A$. The construction is analogous to the construction of the pull-back of a usual connection. Namely, on pull-back sections one defines $(\pi^*\nabla)_{\pi^*\alpha}(\pi^*\sigma)|_p:=\nabla_\alpha(\sigma)|_{\pi(p)}$, and then one extends the definition by the Leibniz identity. The connection is easily seen to be flat. Note that if the action of $A$ on $P$ and the representation $\nabla$ come from an action of $\mathcal{G}$ on $P$ and a representation $E$, then the representation $\pi^*\nabla$ is induced by a representation $\pi^*E$ of $\mathcal{G}\ltimes P$. Namely, the one in which an arrow $(g,p)\in\mathcal{G}\ltimes P$ acts on a vector $v\in (\pi^*E)_p$ by $(g,p)\cdot v = g\cdot v \in (\pi^*E)_{g\cdot p}$. 

\begin{myprop}
	\label{prop:basicforms}
	Let $\mathcal{G}\rightrightarrows M$ be a Lie groupoid acting on a surjective submersion $\pi:P\to M$ and let $E\to M$ be a representation of $\mathcal{G}$. Given any $\theta\in\Omega^1(P;E)$, the following are equivalent:
	\begin{enumerate}
		\item $s^*\theta - t^*\theta = 0$, where $s, t$ is the source and target maps of $\mathcal{G}\ltimes P$. 
		\item $\theta - (t\circ\sigma)^*\theta =0$ for all $\sigma\in\LBis(\mathcal{G}\ltimes P)$.
		\item $\theta$ is basic with respect to $\mathcal{G}$.
	\end{enumerate}
	If $A$ is the Lie algebroid of $\mathcal{G}$ (together with the induced representation of $\mathcal{G}$ and action on $P$), then conditions 1-3 imply that:
	\begin{enumerate}
		\setcounter{enumi}{3}
		\item $\theta$ is basic with respect to $A$.
	\end{enumerate}
	If $\mathcal{G}$ is $s$-connected, then conditions 1-4 are equivalent.
\end{myprop}

\begin{myremark}
	\label{remark:basicforms}
	A few words of explanation are in order. In condition 2, we view both $s^*\theta$ and $t^*\theta$ as elements of $\Omega^1(\mathcal{G}\ltimes P;t^*E)$. The latter is clear, while for the former one makes use of the representation, namely $(s^*\theta)_g(X) = g\cdot\theta(ds(X))\in E_{t(g)}$. In condition 3, we view both $\theta$ and $(t\circ\sigma)^*\theta$ as locally defined elements of $\Omega^1(P;E)$ whose domains of definitions are the domain of $\sigma$. For the former one simply restricts $\theta$ to $\Dom(\phi)$, while in the latter one uses again the representation. Namely, $((t\circ\sigma)^*\theta)_p = \sigma(p)^{-1}\cdot((t\circ\sigma)^*\theta)_p$ for all $p\in\Dom(\sigma)$. 
\end{myremark}

\begin{proof}
	We first prove the equivalence of 1-3. To go from 1 to 2, simply pull back by $\sigma$. Next, assume 2. First, let $X\in T_pP$ such that $X$ is tangent to an orbit. One can always find an $\widetilde{X}\in T_{1_p}(\mathcal{G}\ltimes P)$ such that $ds(\widetilde{X}) = -dt(\widetilde{X}) = X$ (since $X$ is tangent to an orbit, there exist vectors in $T_{1_p}(s^{-1}(p))$ and $T_{1_p}(t^{-1}(p))$ projecting to $X$, so simply take their difference) and a local bisection $\sigma$ of $\mathcal{G}\ltimes P$ such that $\sigma(p) = 1_p$ and $d\sigma(X) = \widetilde{X}$. Applying $\theta-(t\circ \sigma)^*\theta$ on $X$, we see that $2\theta(X)=0$, and hence $\theta$ is horizontal. To prove \eqref{eqn:basicequivariance}, for every $g\in \mathcal{G}$ and $X\in T_pP$ such that $s(g)=\pi(p)$, simply choose a bisection $\sigma$ such that $\sigma(p)=g$ and note that $\theta(g\cdot X) = \theta (d(t\circ \sigma)(X))$. Finally, assume 3. Let $(g,p)\in\mathcal{G}\ltimes P$ and $X\in T_{(g,p)}(\mathcal{G}\ltimes P)$, then $(s^*\theta - t^*\theta)(X) = g\cdot \theta(ds(X)) - \theta(dt(X))$, which vanishes by \eqref{eqn:basicequivariance}, since, by the definition of the normal representation, $[dt(X)] = [g\cdot ds(X)]$. 
	
	The most direct approach for proving that 1-3 imply 4 is to go from 3 to 4 by differentiating \eqref{eqn:basicequivariance} and discovering \eqref{eqn:basicequivariancealgebroid}. More conceptually, we will go from 1 to 4 by appealing to the notion of multiplicative forms and their infinitesimal counterparts, the Spencer operators. In \cite{Crainic2012}, the notion of a representation-valued multiplicative form on a Lie groupoid is studied, and in Theorem 1 of that paper it is shown that such a form linearizes to a so called Spencer operator on the associated Lie algebroid  and that the map sending the multiplicative form to the Spencer operator is injective if the Lie groupoid is source-connected (note that while the authors actually assume source-simply connectedness in the statement of the theorem, source-connectedness is sufficient for the injectivity assertion). To apply the theorem to our problem, one first notes that the form $\omega:=s^*\theta-t^*\theta\in\Omega^1(\mathcal{G}\ltimes P;t^*E)$ is multiplicative (e.g. see Proposition \ref{prop:reductionpfaffiangroupoid}). Next, one computes the induced Spencer operator $D_\omega:\mathfrak{X}(P)\times \Gamma(\pi^*A)\to \Gamma(\pi^*E)$ on the Lie algebroid $\pi^*A$ of $\mathcal{G}\ltimes P$, obtaining the formula
	\begin{equation*}
	(D_\omega)_X(\pi^*\alpha):= \theta([a(\alpha),X]) - (\pi^*\nabla)_{\pi^*\alpha}\theta(X),\hspace{1cm} \forall\;\alpha\in\Gamma(A),\;X\in\mathfrak{X}(P).
	\end{equation*}
	This formula is given in Example 2.10 in \cite{Crainic2012} (and it can be computed by the same method used to prove Lemma \ref{lemma:cartanformspenceroperatorformula}). Now, since the vanishing of $\omega$ implies that $D_\omega$ vanishes, and hence that \eqref{eqn:basicequivariancealgebroid} is satisfied (and we already saw that 1 implies that $\theta$ is horizontal), then 1 implies 4. Conversely, if $\mathcal{G}$ is $s$-connected, and hence $\mathcal{G}\ltimes P$ is $s$-connected, then, by the injectivity of the map described above, the vanishing of $D_\omega$ implies the vanishing of $\omega$ and hence 4 implies 1. 
\end{proof}

Now, generalizing the case of Lie groups, if the action of $\mathcal{G}\rightrightarrows M$ on $P\to M$ is free and proper, then the basic forms on $P$ with values in a representation $E\to M$ are precisely those that descend to forms on the orbit space $P_{\text{red}}:= P/\mathcal{G}$ with values in the associated vector bundle $E_{\text{red}}\to P_{\text{red}}$. Here, $P_\mathrm{red}$ has the unique smooth structure with which the projection $pr:P\to P_{\text{red}}$ is a surjective submersion, and the \textbf{associated vector bundle} $E_\mathrm{red}\to P_\mathrm{red}$ is the quotient of $\pi^*E\to P$ by the action of $\mathcal{G}$ given by $g\cdot (p,v) = (g\cdot p,g\cdot v)$, for all $p\in P$, $v\in E_{\pi(p)}$ and $g\in s^{-1}(\pi(p))$ (thus, $[g\cdot p,v] = [p,g^{-1}\cdot v]\in (E_\text{red})_{[p]}$). The proof of the following proposition is in complete analogy to the case of Lie groups:

\begin{myprop}
	\label{prop:basicformscorrespondence}
	Let  $\pi:P\to M$ be a surjective submersion equipped with a free and proper action of a Lie groupoid $\mathcal{G}\rightrightarrows M$, and let $E\to M$ be a representation of $\mathcal{G}$. The pull-back by the projection $pr:P\to P_\text{red}$ gives a 1-1 correspondence
	\begin{equation*}
	pr^*: \Omega^1(P_\text{red};E_\text{red}) \xrightarrow{\simeq} \Omega^1_\text{bas}(P;E).
	\end{equation*}
\end{myprop}

\bibliographystyle{plain}

\begin{thebibliography}{10}
	
	\bibitem{Bocharov1999}
	A.~V. Bocharov, V.~N. Chetverikov, S.~V. Duzhin, N.~G. Khor{\cprime}kova, I.~S.
	Krasil{\cprime}shchik, A.~V. Samokhin, Yu.~N. Torkhov, A.~M. Verbovetsky, and
	A.~M. Vinogradov.
	\newblock {\em Symmetries and conservation laws for differential equations of
		mathematical physics}, volume 182 of {\em Translations of Mathematical
		Monographs}.
	\newblock American Mathematical Society, Providence, RI, 1999.
	\newblock Edited and with a preface by Krasil{\cprime}shchik and Vinogradov,
	Translated from the 1997 Russian original by Verbovetsky [A. M.
	Verbovetski{\u\i}] and Krasil{\cprime}shchik.
	
	\bibitem{Bryant1991}
	R.~L. Bryant, S.~S. Chern, R.~B. Gardner, H.~L. Goldschmidt, and P.~A.
	Griffiths.
	\newblock {\em Exterior differential systems}, volume~18 of {\em Mathematical
		Sciences Research Institute Publications}.
	\newblock Springer-Verlag, New York, 1991.
	
	\bibitem{Cartan1904}
	{\'E}lie Cartan.
	\newblock Sur la structure des groupes infinis de transformation.
	\newblock {\em Ann. Sci. \'Ecole Norm. Sup. (3)}, 21:153--206, 1904.
	
	\bibitem{Cartan1905}
	{\'E}lie Cartan.
	\newblock Sur la structure des groupes infinis de transformation (suite).
	\newblock {\em Ann. Sci. \'Ecole Norm. Sup. (3)}, 22:219--308, 1905.
	
	\bibitem{Cartan1937-1}
	{\'E}lie Cartan.
	\newblock La structure des groupes infinis.
	\newblock {\em S\'eminaire de Math.}, 4e ann{\'e}e, 1936-1937.
	
	\bibitem{Cartan1937-2}
	{\'E}lie Cartan.
	\newblock La structure des groupes infinis (suite).
	\newblock {\em S\'eminaire de Math.}, 4e ann{\'e}e, 1936-1937.
	
	\bibitem{Cattafi2020}
	Francesco Cattafi.
	\newblock Thesis in progress.
	\newblock Thesis.
	
	\bibitem{Chern1952}
	Shiing-Shen Chern and Claude Chevalley.
	\newblock Obituary: {E}lie {C}artan and his mathematical work.
	\newblock {\em Bull. Amer. Math. Soc.}, 58:217--250, 1952.
	
	\bibitem{Crainic2005}
	M.~Crainic and R.~L. Fernandes.
	\newblock Secondary characteristic classes of {L}ie algebroids.
	\newblock In {\em Quantum field theory and noncommutative geometry}, volume 662
	of {\em Lecture Notes in Phys.}, pages 157--176. Springer, Berlin, 2005.
	
	\bibitem{Crainic2015}
	Marius Crainic.
	\newblock Mastermath course differential geometry (lecture notes), (available
	for download at http://www.staff.science.uu.nl/~crain101/dg-2015/main10.pdf),
	2015.
	
	\bibitem{Crainic2003-1}
	Marius Crainic and Rui~Loja Fernandes.
	\newblock Integrability of {L}ie brackets.
	\newblock {\em Ann. of Math. (2)}, 157(2):575--620, 2003.
	
	\bibitem{Crainic2011-1}
	Marius Crainic and Rui~Loja Fernandes.
	\newblock Lectures on integrability of {L}ie brackets.
	\newblock In {\em Lectures on {P}oisson geometry}, volume~17 of {\em Geom.
		Topol. Monogr.}, pages 1--107. Geom. Topol. Publ., Coventry, 2011.
	
	\bibitem{Crainic2012}
	Marius Crainic, Maria~Amelia Salazar, and Ivan Struchiner.
	\newblock Multiplicative forms and {S}pencer operators.
	\newblock {\em Math. Z.}, 279(3-4):939--979, 2015.
	
	\bibitem{Greub1973}
	Werner Greub, Stephen Halperin, and Ray Vanstone.
	\newblock {\em Connections, curvature, and cohomology. {V}ol. {II}: {L}ie
		groups, principal bundles, and characteristic classes}.
	\newblock Academic Press [A subsidiary of Harcourt Brace Jovanovich,
	Publishers], New York-London, 1973.
	\newblock Pure and Applied Mathematics, Vol. 47-II.
	
	\bibitem{Guillemin1966}
	Victor Guillemin and Shlomo Sternberg.
	\newblock Deformation theory of pseudogroup structures.
	\newblock {\em Mem. Amer. Math. Soc. No.}, 64:80, 1966.
	
	\bibitem{Guillemin1964}
	Victor~W. Guillemin and Shlomo Sternberg.
	\newblock An algebraic model of transitive differential geometry.
	\newblock {\em Bull. Amer. Math. Soc.}, 70:16--47, 1964.
	
	\bibitem{Haefliger1958}
	Andr{\'e} Haefliger.
	\newblock Structures feuillet\'ees et cohomologie \`a valeur dans un faisceau
	de groupo\"\i des.
	\newblock {\em Comment. Math. Helv.}, 32:248--329, 1958.
	
	\bibitem{Higgins1990}
	Philip~J. Higgins and Kirill Mackenzie.
	\newblock Algebraic constructions in the category of {L}ie algebroids.
	\newblock {\em J. Algebra}, 129(1):194--230, 1990.
	
	\bibitem{Kamran1989}
	Niky Kamran.
	\newblock Contributions to the study of the equivalence problem of \'{E}lie
	{C}artan and its applications to partial and ordinary differential equations.
	\newblock {\em Acad. Roy. Belg. Cl. Sci. M\'em. Collect. 8$^{\rm o}$ (2)},
	45(7):122, 1989.
	
	\bibitem{Kamran2004}
	Niky Kamran and Thierry Robart.
	\newblock An infinite-dimensional manifold structure for analytic {L}ie
	pseudogroups of infinite type.
	\newblock {\em Int. Math. Res. Not.}, (34):1761--1783, 2004.
	
	\bibitem{Kobayashi1963}
	Shoshichi Kobayashi and Katsumi Nomizu.
	\newblock {\em Foundations of differential geometry. {V}ol {I}}.
	\newblock Interscience Publishers, a division of John Wiley \& Sons, New
	York-London, 1963.
	
	\bibitem{Kumpera1964}
	A.~Kumpera.
	\newblock A theorem on {C}artan pseudogroups.
	\newblock In {\em Topologie et {G}\'eom\'etrie {D}iff\'erentielle
		({S}\'eminaire {C}h. {E}hresmann, {V}ol. {VI}, 1964)}, page~12. Inst. Henri
	Poincar\'e, Paris, 1964.
	
	\bibitem{Kumpera1975}
	A.~Kumpera.
	\newblock Invariants diff\'erentiels d'un pseudogroupe de {L}ie. i.
	\newblock {\em J. Differential Geometry}, 10(2):289--345, 1975.
	
	\bibitem{Kuranishi1959}
	Masatake Kuranishi.
	\newblock On the local theory of continuous infinite pseudo groups. {I}.
	\newblock {\em Nagoya Math. J}, 15:225--260, 1959.
	
	\bibitem{Kuranishi1961}
	Masatake Kuranishi.
	\newblock On the local theory of continuous infinite pseudo groups. {II}.
	\newblock {\em Nagoya Math. J.}, 19:55--91, 1961.
	
	\bibitem{Libermann1954}
	Paulette Libermann.
	\newblock Sur le probl\`eme d'\'equivalence de certaines structures
	infinit\'esimales.
	\newblock {\em Ann. Mat. Pura Appl. (4)}, 36:27--120, 1954.
	
	\bibitem{Libermann2007}
	Paulette Libermann.
	\newblock Charles {E}hresmann's concepts in differential geometry.
	\newblock In {\em Geometry and topology of manifolds}, volume~76 of {\em Banach
		Center Publ.}, pages 35--50. Polish Acad. Sci., Warsaw, 2007.
	
	\bibitem{Lie1888}
	Sophus Lie and Friedrich Engel.
	\newblock {\em {Theorie der Transformationsgruppen}}, volume 1-3.
	\newblock B.G. Teubner, Leipzig, 1888-93.
	
	\bibitem{Lisle1995}
	Ian~G. Lisle and Gregory~J. Reid.
	\newblock Cartan structure of infinite {L}ie pseudogroups.
	\newblock In {\em Geometric approaches to differential equations ({C}anberra,
		1995)}, volume~15 of {\em Austral. Math. Soc. Lect. Ser.}, pages 116--145.
	Cambridge Univ. Press, Cambridge, 2000.
	
	\bibitem{Mackenzie1987}
	K.~Mackenzie.
	\newblock {\em Lie groupoids and {L}ie algebroids in differential geometry},
	volume 124 of {\em London Mathematical Society Lecture Note Series}.
	\newblock Cambridge University Press, Cambridge, 1987.
	
	\bibitem{Mackenzie2005}
	Kirill C.~H. Mackenzie.
	\newblock {\em General theory of {L}ie groupoids and {L}ie algebroids}, volume
	213 of {\em London Mathematical Society Lecture Note Series}.
	\newblock Cambridge University Press, Cambridge, 2005.
	
	\bibitem{Malgrange1972-1}
	Bernard Malgrange.
	\newblock Equations de {L}ie. {I}.
	\newblock {\em J. Differential Geometry}, 6:503--522, 1972.
	\newblock Collection of articles dedicated to S. S. Chern and D. C. Spencer on
	their sixtieth birthdays.
	
	\bibitem{Malgrange1972-2}
	Bernard Malgrange.
	\newblock Equations de {L}ie. {II}.
	\newblock {\em J. Differential Geometry}, 7:117--141, 1972.
	
	\bibitem{Matsushima1955}
	Y.~Matsushima.
	\newblock Pseudo-groupes de {L}ie transitifs.
	\newblock In {\em S\'eminaire {B}ourbaki, {V}ol.\ 3}, pages Exp.\ No.\ 118,
	183--196. Soc. Math. France, Paris, 1995.
	
	\bibitem{Moerdijk2003}
	I.~Moerdijk and J.~Mr{\v{c}}un.
	\newblock {\em Introduction to foliations and {L}ie groupoids}, volume~91 of
	{\em Cambridge Studies in Advanced Mathematics}.
	\newblock Cambridge University Press, Cambridge, 2003.
	
	\bibitem{Moerdijk2002}
	Ieke Moerdijk and Janez Mr{\v{c}}un.
	\newblock On integrability of infinitesimal actions.
	\newblock {\em Amer. J. Math.}, 124(3):567--593, 2002.
	
	\bibitem{Olver1995}
	Peter~J. Olver.
	\newblock {\em Equivalence, invariants, and symmetry}.
	\newblock Cambridge University Press, Cambridge, 1995.
	
	\bibitem{Olver2005}
	Peter~J. Olver and Juha Pohjanpelto.
	\newblock Maurer-{C}artan forms and the structure of {L}ie pseudo-groups.
	\newblock {\em Selecta Math. (N.S.)}, 11(1):99--126, 2005.
	
	\bibitem{Olver2009}
	Peter~J. Olver, Juha Pohjanpelto, and Francis Valiquette.
	\newblock On the structure of {L}ie pseudo-groups.
	\newblock {\em SIGMA Symmetry Integrability Geom. Methods Appl.}, 5:Paper 077,
	14, 2009.
	
	\bibitem{Rodrigues1962}
	A.~M. Rodrigues.
	\newblock The first and second fundamental theorems of {L}ie for {L}ie pseudo
	groups.
	\newblock {\em Amer. J. Math.}, 84:265--282, 1962.
	
	\bibitem{Rodrigues1963}
	A.~M. Rodrigues.
	\newblock On {C}artan pseudo groups.
	\newblock {\em Nagoya Math. J.}, 23:1--4, 1963.
	
	\bibitem{Salazar2013}
	Maria~Amelia Salazar.
	\newblock Pfaffian groupoids, 2013.
	\newblock Thesis.
	
	\bibitem{Saunders1989}
	D.~J. Saunders.
	\newblock {\em The geometry of jet bundles}, volume 142 of {\em London
		Mathematical Society Lecture Note Series}.
	\newblock Cambridge University Press, Cambridge, 1989.
	
	\bibitem{Singer1965}
	I.~M. Singer and Shlomo Sternberg.
	\newblock The infinite groups of {L}ie and {C}artan. {I}. {T}he transitive
	groups.
	\newblock {\em J. Analyse Math.}, 15:1--114, 1965.
	
	\bibitem{Spencer1962-1}
	D.~C. Spencer.
	\newblock Deformation of structures on manifolds defined by transitive,
	continuous pseudogroups. {I}. {I}nfinitesimal deformations of structure.
	\newblock {\em Ann. of Math. (2)}, 76:306--398, 1962.
	
	\bibitem{Sternberg1983}
	Shlomo Sternberg.
	\newblock {\em Lectures on differential geometry}.
	\newblock Chelsea Publishing Co., New York, second edition, 1983.
	\newblock With an appendix by Sternberg and Victor W. Guillemin.
	
	\bibitem{Stormark2000}
	Olle Stormark.
	\newblock {\em Lie's structural approach to {PDE} systems}, volume~80 of {\em
		Encyclopedia of Mathematics and its Applications}.
	\newblock Cambridge University Press, Cambridge, 2000.
	
	\bibitem{Valiquette2008}
	Francis Valiquette.
	\newblock Structure equations of {L}ie pseudo-groups.
	\newblock {\em J. Lie Theory}, 18(4):869--895, 2008.
	
	\bibitem{Yudilevich2016-2}
	Ori Yudilevich.
	\newblock Lie pseudogroups \`a la cartan from a modern perspective (thesis).
	
	\bibitem{Yudilevich2016}
	Ori Yudilevich.
	\newblock The role of the {J}acobi identity in solving the {M}aurer--{C}artan
	structure equation.
	\newblock {\em Pacific J. Math.}, 282(2):487--510, 2016.
	
\end{thebibliography}

\def\cprime{$'$} \def\cprime{$'$} \def\cprime{$'$}

\end{document}